\documentclass[
oneside, openright, titlepage, numbers=noenddot, headinclude, 
lineheaders footinclude=true, cleardoublepage=empty, 
BCOR=5mm, paper=a4, fontsize=12pt
]{scrbook}

\usepackage[parts, eulerchapternumbers, beramono]{classicthesis}
\usepackage[brazil, portuguese, english]{babel}
\usepackage[utf8]{inputenc}
\usepackage[T1]{fontenc}
\usepackage{
  amsmath, amssymb, amsthm, cancel, graphicx, hyperref, wrapfig,
  imakeidx, listings, mathtools, microtype, multicol, physics, tikz-cd,
  verbatim, xcolor, pdfpages, breqn, amsfonts, blindtext, epigraph
}
\usepackage[most]{tcolorbox}
\usetikzlibrary{positioning}

\tcbuselibrary{listings}
\newcommand{\nc}{\newcommand}
\nc{\dmo}{\DeclareMathOperator}
\nc{\rnc}{\renewcommand}

\dmo{\alt}{Alt}
\dmo{\aut}{Aut}
\dmo{\arr}{Ar}
\dmo{\euc}{Euc}
\dmo{\gl}{GL}
\dmo{\Gr}{Gr}
\dmo{\hess}{Hess}
\dmo{\hhom}{Hom}
\dmo{\hor}{Hor}
\dmo{\Id}{Id}
\dmo{\im}{Im}
\dmo{\Inn}{Inn}
\dmo{\obj}{Ob}
\dmo{\orb}{Orb}
\dmo{\pt}{PT}
\dmo{\re}{Re}
\dmo{\riem}{Riem}
\dmo{\sll}{SL}
\dmo{\spec}{Spec}
\dmo{\specm}{Specm}
\dmo{\St}{St}
\dmo{\stab}{Stab}
\dmo{\syl}{Syl}
\dmo{\sym}{Sym}
\dmo{\aux}{aux}
\dmo{\charr}{char}
\dmo{\coker}{coker}
\dmo{\gradd}{grad}
\dmo{\id}{id}
\dmo{\lift}{lift}
\dmo{\mdc}{mdc}
\dmo{\mmc}{mmc}
\dmo{\Spann}{span}
\dmo{\Pl}{Pl}
\dmo{\proj}{proj}
\dmo{\sen}{sen}
\dmo{\sign}{sign}
\dmo{\supp}{supp}
\dmo{\vecc}{vec}
\dmo{\unvec}{unvec}
\dmo{\hstack}{hstack}
\dmo{\vstack}{vstack}

\rnc{\cp}{\bc\bp}
\rnc{\grad}{\gradd}
\rnc{\hom}{\hhom}
\rnc{\l}{\ell}
\rnc{\vec}{\vecc}
\nc{\0}{\{0\}}
\nc{\ceq}{\coloneqq}
\nc{\eps}{\varepsilon}
\nc{\eqc}{\eqqcolon}
\nc{\hra}{\hookrightarrow}
\nc{\lra}{\longrightarrow}
\nc{\mb}{\mathbf}
\nc{\mc}{\mathcal}
\nc{\mf}{\mathfrak}
\nc{\ob}{\overbrace}
\nc{\ol}{\overline}
\nc{\ola}{\overleftarrow}
\nc{\ora}{\overrightarrow}
\nc{\os}{\overset}
\nc{\rp}{\br\bp}
\nc{\ssm}{\smallsetminus}
\nc{\tbf}{\textbf}
\nc{\tit}{\textit}
\nc{\ttc}{\textcolor}
\nc{\ttt}{\texttt}
\nc{\ub}{\underbrace}
\nc{\ul}{\underline}
\nc{\us}{\underset}
\nc{\what}{\widehat}
\nc{\wtil}{\widetilde}
\nc{\x}{\times}
\nc{\intl}{\int\limits}
\nc{\wten}{\what{\otimes}}
\nc{\NA}{N_{\alpha}}
\nc{\NB}{N_{\beta}}
\nc{\NC}{N_{\gamma}}

\nc{\oll}[1]{\ol{\ol{#1}}}
\nc{\e}[2]{\exp_{#1} (#2)}
\nc{\inn}[2]{\left\langle #1,#2 \right\rangle}
\nc{\nab}[2]{\nabla_{#1} #2}
\nc{\pars}[1]{\left( #1 \right)}
\nc{\ten}[2]{#1 \what{\otimes} \ldots \what{\otimes} #2}
\nc{\tenw}[2]{#1 \otimes \ldots \otimes #2}
\nc{\wed}[2]{#1 \wedge \ldots \wedge #2}
\nc{\Span}[1]{\Spann\qty{#1}}
\nc{\gr}[2]{\text{Gr}(#1, #2)}
\nc{\st}[2]{\text{St}(#1, #2)}
\nc{\mat}[2]{\text{Mat}(#1 \x #2)}

\nc{\bc}{\mb{C}}
\nc{\bbf}{\mb{F}}
\nc{\bk}{\mb{K}}
\nc{\bn}{\mb{N}}
\nc{\bo}{\mb{O}}
\nc{\bp}{\mb{P}}
\nc{\bq}{\mb{Q}}
\nc{\br}{\mb{R}}
\nc{\bs}{\mb{S}}
\nc{\bz}{\mb{Z}}

\nc{\ie}{i.e.}
\nc{\tiff}{if, and only if, }
\nc{\eg}{e.g.}

\newtcblisting{algbox}[2]{
  title=#1,
  label={#2},
  colback=white,
  colframe=black!100,
  left=6mm,
  size=fbox,
  listing only,
  listing options={language=Python,
    basicstyle=\scriptsize,
    mathescape=true,
    numbers=left,
    showstringspaces=false,
    keywordstyle=\bfseries,
    morekeywords={solve, precompute},
    numberstyle=\tiny\color{red!75!black}
  }
}

\theoremstyle{definition}
\numberwithin{equation}{section}
\newtheorem{thm}{Theorem}[chapter]
\newtheorem{lem}[thm]{Lemma}
\newtheorem{prop}[thm]{Proposition}
\newtheorem{cor}[thm]{Corollary}
\newtheorem{defi}[thm]{Definition}
\newtheorem{obs}[thm]{Observation}
\newtheorem{exm}[thm]{Example}
\newtheorem{algo}[thm]{Algorithm}

\makeindex[columns=3, title=Index]

\begin{document}
\title{\huge \scshape \bfseries Riemannian Optimization and the Hartree--Fock Method}
\author{Caio Oliveira da Silva}
\date{\today}
\maketitle

I dedicate this thesis to my grandmother (\emph{in memoriam}), to my father and to Yuri.

\chapter*{Acknowledgements}

First of all, I'd like to thank Yuri for being the best supervisor I've ever seen and for being an amazing and inspiring human being.
Also my friends that are still by my side, for making this weird thing called life way more enjoyable.
And last, but not least, I thank my father and my grandmother for supporting me throughout my academic journey without questioning my dubious choices.

\chapter*{}
\vfill
\begin{flushright}
I was actually chasing a special sort of buzz, a special moment that comes sometimes. One teacher called these moments “mathematical experiences.” What I didn’t know then was that a mathematical experience was aesthetic in nature, an epiphany in Joyce’s original sense. These moments appeared in proof-completions, or maybe algorithms. Or like a gorgeously simple solution to a problem you suddenly see after filling half a notebook with gnarly attempted solutions. It was really an experience of what I think Yeats called “the click of a well-made box.” Something like that. The word I always think of it as is “click.” \\
— David Foster Wallace
\end{flushright}


\chapter*{Abstract}

In the present work we studied a subfield of Applied Mathematics called Riemannian Optimization.
The main goal of this subfield is to generalize algorithms, theorems and tools from Mathematical Optimization to the case in which the optimization problem is defined on a Riemannian manifold.

As a case study, we implemented some of the main algorithms described in the literature (Gradient Descent, Newton--Raphson and Conjugate Gradient) to solve an optimization problem known as Hartree--Fock.
This method is extremely important in the field of Computational Quantum Chemistry and it is a good case study because it is a problem somewhat hard to solve and, as a consequence of this, it requires many tools from Riemannian Optimization.
Besides, it is also a good example to see how these algorithms perform in practice.

\vspace{.5cm}
\textbf{Keywords}: Riemannian Optimization, Hartree--Fock Method, Mathematical Optimization, Quantum Chemistry, Riemannian Geometry.


\tableofcontents

\mainmatter


\chapter{Introduction} \label{intro}

Riemannian Optimization is a field that aims to study optimization problems given by
\begin{equation}
  f:M \to \br,
\end{equation}
being $M$ a Riemannian manifold.
(In Chapter \ref{geometry} we will define precisely what a manifold is and provide some examples, but for now the reader that is not familiar with this concept can assume $M$ is a sphere.)
From a theoretical perspective, one of the reasons why studying this field is important is that it generalizes unconstrained optimization to a class of spaces called Riemannian manifolds.
So, that means we can study and understand a wider number of problems because the Euclidean space is just one example of Riemannian manifold.
Another important reason is that this field can also be thought of as an alternative to constrained optimization as long as the constraints of the problem form a manifold.
A good example of this is, again, the sphere.
If one has an optimization problem given by
\begin{equation}
  \text{minimize} \ \ f:\br^n \to \br \ \ \text{subject to} \ \ \braket{x} = 1,
\end{equation}
then the constraint is actually a manifold and Riemannian Optimization provides mathematical tools to study this problem from a theoretical perspective (such as convergence guarantees) and algorithms to tackle it computationally.
The field is also important from a practical perspective because it increases the tool bag we can use to model, study and solve real-world problems.
That said, in the present work we will study a problem from Quantum Chemistry, but in \cite{boumal2022, edelman1998} the reader can find other applications ranging from Chemistry and Physics to Computer Vision and Signal Processing.
Now let us motivate our problem.

In the late 1950s and early 1960s a medication called \emph{Thalidomide} was marketed and prescribed to treat sleeping trouble and morning sickness in pregnant women.
However, when they started to take the drug, it passed across the placenta and harmed the developing fetus, leading to thousands of infants dying in the time of birth and many others surviving with debilitating malformations.
The cause of this unexpected effect was that Thalidomide existed inside the body in two different \emph{molecular geometries} (\ie, the three-dimensional arrangement of the atoms that constitute a molecule) and, while one of the geometries helped the pregnant women, the other harmed the fetus \cite{vargesson2015}.
This is an extreme but very illustrative example of why computing and understanding molecular geometry is important.
Now, to compute the many possible geometries of a molecule and the energy necessary to change from one geometry to the other chemists created the concept of \emph{Potential Energy Surface} (PES), which is a function that maps the geometries of a molecule to the energy that the molecule has in that geometry.
This function is important because the physically stable geometries are its minima and this is where the method we are going to study enters the picture.
Simplifying things a little bit, the function $F(\qty{\vb{R}_i}, \qty{\phi_j(\vb{r}_j)})$ that describes a PES depends on two sets of parameters: the arrangement of the atoms in space, which we denote by $\vb{R}_i$ (each $\vb{R}_i \in \br^3$ determines the position of the nucleus of the $i$-th atom after we fix an order for the atoms, for example), and the wave functions that describes the electrons orbiting these atoms, which we denote by $\phi_j(\vb{r}_j)$ (again, $\vb{r}_j \in \br^3$, but electrons are never fixed, this is why we need to consider a wave function instead of its ``position'').
The optimization problem we are interested in is the \emph{energy minimization} part, \ie, we want to find the wave functions that minimizes the energy once the nuclei are fixed.
This should not be confused with the problem of minimizing the whole function $F$, which is usually called \emph{geometry optimization} in the literature.
In mathematical terms, we are interested in \emph{currying} the function $F$ and optimizing just $F(\qty{\vb{R}_i}, \cdot)$ for a fixed geometry $\qty{\vb{R}_i}$.\footnote{This is also known as a \emph{parametric dependence} on the nuclei coordinates.}

There are many ways to solve this problem of energy minimization and the \emph{Hartree--Fock Method} (HF) is one of them.
We discuss this method in-depth in Chapter \ref{qm}, but for now let us just describe its importance and say that in this work we define HF as the optimization problem of energy minimization with the constraint that the wave function should be a \emph{Slater determinant}.\footnote{Observe that the word \emph{method} is a little misleading if we adopt this definition of Hartree--Fock because in this work method is usually meant as a synonym of \emph{algorithm}. However, since in the literature Hartree--Fock is always referred as a method, we will also adopt this convention.}
So, as already mentioned, there are multiple ways of modelling the problem of energy minimization and Hartree--Fock is the simplest method that takes into account wave functions, the Schrödinger equation and the antisymmetry of electrons.
These methods are usually called \emph{ab initio} and they are the most accurate methods we have in Computational Quantum Chemistry nowadays, although the most expensive ones.
It is important to mention that Hartree--Fock is not the most accurate method for today's standards, but many of the state-of-the-art methods aims to improve the Slater determinant obtained by HF (which is called \emph{Hartree--Fock wave function}).
These methods are called \emph{post-Hartree--Fock} and the fact that they use HF as a starting point shows why robust implementations of algorithms that solves HF are still highly desirable.
And this is where Riemannian Optimization enters the picture.

The Hartree--Fock Method, as the reader may have guessed, can be seen as a Riemannian optimization problem because the set of all Slater determinants is a manifold known as \emph{Grassmann manifold} or, for short, the \emph{Grassmannian}.
Consequently, we can use tools from Riemannian Optimization to tackle this problem and in the present work we implemented three Riemannian algorithms to solve HF: Gradient Descent (GD), Newton--Raphson (NR) and Conjugate Gradient (CG).
The main reason for this choice of algorithms is that they were all described in the seminal papers \cite{edelman1998, smith2014}, but the field has many other algorithms that can be used to solve HF and other problems, see \cite{boumal2014}.
Now, one of the important aspects of using Riemannian Optimization to obtain a robust implementation for HF is that many algorithms with different strategies and complexities can be described using a common language.
This is good because we can, for example, use algorithms such as Gradient Descent or Conjugate Gradient when the starting point is too far (which usually means that the gradient of the cost function is big) and then switch to Newton--Raphson when the algorithm seems to be near convergence (\ie, the gradient is small and the Hessian is positive-definite).
In practice, combining CG with NR resulted in the best algorithm in terms of molecules converged (see Section \ref{results}), but it should be pointed out that we did not analyze this and other heuristics systematically.

All the Riemannian methods previously described were implemented using the programming language Python and the reader can see it in the GitHub repository \cite{aotograssmann}.
We also implemented two other methods to compare with the Riemannian algorithms: a version of Newton--Raphson using Lagrange multipliers that can be found in any textbook on constrained optimization, and the Self-Consistent Field (SCF) method.
Among all algorithms implemented, the best one (without combining with others) in terms of molecules converged was CG, which converged for 93.2\% of the molecules in the dataset.
The second best algorithm was SCF, which converged for 91.2\% of the molecules.
Another interesting comparison we made was with respect to both Newton--Raphson methods.
The Riemannian Newton--Raphson converged for 83.8\% of the molecules while the regular Newton--Raphson with Lagrange multipliers converged only for 53.4\% of the molecules.
This is a very big improvement and it shows a real gain in using Riemannian algorithms, at least for Hartree--Fock.
With that, hopefully we convinced the reader why studying Hartre--Fock and Riemannian Optimization is worthwhile.

\begin{wrapfigure}[13]{R}{0.6\textwidth}
  \vspace{-1cm}
  \begin{tikzpicture}[
    squarednode/.style={rectangle, thick, draw=black, align=center},
    sregulararrow/.style={->, >=stealth, shorten >= 0.1cm, shorten <= 0.1cm},
    dregulararrow/.style={->, >=stealth, shorten >= 0.6cm, shorten <= 0.6cm}
    ]
    \node[squarednode] (topalg)                             {{\small Topology} \\[-2.5mm] {\scriptsize Appendix \ref{top}}};
    \node[squarednode] (linalg)     [right=2.6cm of topalg]   {{\small Linear Algebra} \\[-2.5mm] {\scriptsize Appendix \ref{linalg}}};
    \node[squarednode] (riemgeo)    [below=1.1cm of topalg]   {{\small Riemannian Geometry} \\[-2.5mm] {\scriptsize Chapter \ref{geometry}}};
    \node[squarednode] (qm)         [below=1.1cm of linalg]   {{\small Quantum Mechanics} \\[-2.5mm] {\scriptsize Chapter \ref{qm}}};
    \node[squarednode] (riemopt)    [below=1.1cm of qm]       {{\small Riemannian Optimization} \\[-2.5mm] {\scriptsize Chapter \ref{optimization}}};
    \node[squarednode] (matrices)   [below=1.1cm of riemgeo]  {{\small Matrix identities} \\[-2.5mm] {\scriptsize Appendix \ref{matrix_ids}}};
    \node[squarednode] (lagrange)   [below=1.1cm of matrices] {{\small Lagrange multipliers} \\[-2.5mm] {\scriptsize Appendix \ref{lagrange_multipliers}}};
    \node[squarednode] (conclusion) [below=1.1cm of riemopt]  {{\small Conclusion} \\[-2.5mm] {\scriptsize Chapter \ref{conclusion}}};

    \draw[sregulararrow] (topalg.south)   -- (riemgeo.north);
    \draw[sregulararrow] (matrices.east)  -- (riemopt.west);
    \draw[sregulararrow] (linalg.south)   -- (qm.north);
    \draw[dregulararrow] (topalg.south)   -- (qm.north);
    \draw[dregulararrow] (linalg.south)   -- (riemgeo.north);
    \draw[dregulararrow] (riemopt.south)  -- (lagrange.north);
    \draw[sregulararrow] (matrices.south) -- (lagrange.north);
    \draw[sregulararrow] (qm.south)       -- (riemopt.north);
    \draw[sregulararrow] (riemopt.south)  -- (conclusion.north);
    \draw[dregulararrow] (riemgeo.south)  -- (riemopt.north);
  \end{tikzpicture}
    \caption{Graph illustrating chapters dependencies.}
\end{wrapfigure}
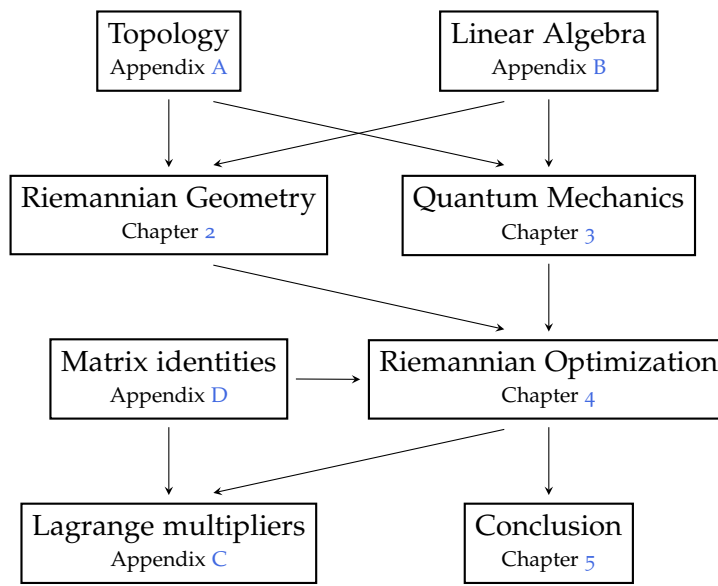
Now let us talk about the structure of this work.
It is organized as follows: in Chapter \ref{geometry} we define all the mathematical objects necessary to implement the three Riemannian algorithms mentioned above and we also compute everything we need to implement these algorithms for the Grassmannian, which is the manifold we are mainly interested in this work.
In Chapter \ref{qm} we present the axioms and objects of Quantum Mechanics we need to understand the Hartree--Fock Method.
We also motivate this method from a physical and computational perspective.
Chapter \ref{optimization} is where we define and write pseudocodes for the Riemannian algorithms already mentioned.
The pseudocodes were implemented for arbitrary cost functions defined in the product of two Grassmannians\footnote{We need two Grassmannians to take into account both spins of electrons, see Definition \ref{hf}.}, but we also filled the blank spaces with the cost function of the Hartree--Fock Method (\ie, the energy of a molecule) in Section \ref{hf_ch3}.
To conclude Chapter \ref{optimization}, we present the results obtained and analyze them.
And last, but not least, in Chapter \ref{conclusion} we provide a conclusion for the present work and point towards what is next.
Now, the appendices.
The Topology and Linear Algebra appendices (\ref{top} and \ref{linalg}) are quite dry and they were written to set the set the language and contain the definitions and theorems we need in the main chapters.
In the Lagrange multipliers appendix (\ref{lagrange_multipliers}) there is a pseudocode for the Newton--Raphson Method using Lagrange multipliers because we wanted to compare it with the Riemannian Newton--Raphson.
The result obtained by this algorithm is also in Section \ref{results}.
To conclude, the Matrix identities appendix \ref{matrix_ids} provides some matrix identities and computations that are used in Chapter \ref{optimization} and in the Lagrange multipliers appendix.

To conclude, a few remarks about the content of this work and how to actually read it.
One of the goals of this text is to be as self-contained as possible because different people with very distinct backgrounds will probably read it.
With that said, this work was written in such a way that the only prerequisite necessary to read it is Multivariable Calculus.
However, to be fair, the presentation of the fields involved is far from comprehensive, otherwise this text would have thousands of pages.
Also, knowing that fields such as Riemannian Geometry and Quantum Mechanics are not easy to grasp if the reader is encountering them for the first time in this work, the main chapters were written to be read independently and in a top down approach, which means it is possible to read the text starting from any chapter.
It is important to mention, though, that to achieve the independence of Chapter \ref{optimization} we had to define and write some objects twice.
This was a deliberate choice.


\chapter{Riemannian Geometry} \label{geometry}

\epigraph{Young man, in mathematics you don't understand things. You just get used to them.}{John von Neumann}

\section{Introduction}

One of the goals of this chapter is to define and compute all the mathematical objects necessary to implement the Riemannian algorithms we will see in Chapter \ref{optimization}.
Since the content of this chapter is, to a great extent, easily found in the literature, we opted for leaving most of the lemmas, propositions and theorems that are not directly related to the present work without proof.
Another goal of the present chapter is to show how Riemannian objects can be computed using the Grassmannian and the Stiefel manifold as examples because by doing this we hopefully provide a guide to the reader interested in working with other manifolds.
For a mathematical perspective of the subject containing all the proofs omitted, the references are \cite{lee2012} for the Differential Topology section and \cite{lee2018} for the Riemannian Geometry section.
However, the main reference used was \cite{boumal2022}.

\section{Differential Topology}

The reader that is not familiar with Topology should read Appendix \ref{top} first because we will use many definitions and results from Topology throughout this whole chapter.
Now, since all the fields involved in this chapter are quite advanced, a brief explanation of what they are all about is required.
Roughly speaking, the goal of Topology is to abstract the concept of space and study continuous deformations of these spaces.
However, sometimes the space has more structure than just a topology and this allows us to use tools from Linear Algebra and Calculus to study them.
That said, while pure mathematicians are interested in using these tools to study different kinds of deformations (smooth deformations in Differential Topology and isometries in Riemannian Geometry), in our case we are interested in these fields because they provide powerful tools to compute stuff and, as this whole work will show, computing objects such as gradients, Hessians, geodesics etc is quite useful to implement optimization algorithms, for example.
So, in essence the goal of Differential Topology and Riemannian Geometry is to provide a framework in which we can do Calculus and Linear Algebra on spaces that are a little more exotic than the Euclidean space and it is this framework that we will study in this chapter, with a focus on two ``exotic'' spaces called Stiefel manifold and Grassmannian.

\begin{defi} \label{topological_manifold}
  A \emph{$d$-dimensional topological manifold}\index{manifold!topological} is a topological space that is second-countable, Hausdorff and locally Euclidean of dimension $d$.
\end{defi}

\begin{defi} \label{chart}
  Given a $d$-dimensional topological manifold $X$, a \emph{chart of $X$}\index{chart} is a pair $(U, \phi)$, being $U \subset X$ an open set and $\phi:U \to V$ a homeomorphism between $U$ and an open subset $V$ of $\br^d$.
\end{defi}

\begin{obs}
  We can replace the open set $V$ in the definition above by $\br^d$ itself or by an open ball $B(x, r) \subset \br^d$, but it is usually easier to find a homeomorphism with an arbitrary open set $V$.
\end{obs}

\begin{defi}
  Let $X_1, \ldots, X_m, Y_1, \ldots, Y_n$ be arbitrary sets and $f:X_1 \x \ldots \x X_m \to Y_1 \x \ldots \x Y_n$ be a function.
  The \emph{components of $f$} are the composites $\pi_i \circ f$, being $\pi_i:Y_1 \x \ldots \x Y_n \to Y_i$ the canonical projection $\pi_i(y_1, \ldots, y_n) = y_i$.
\end{defi}

\begin{defi} \label{smooth_euclidean}
  Let $U \subset \br^d$ be an open subset and $f:U \to \br^m$ be an arbitrary function.
  We say that $f$ is \emph{smooth (in the Euclidean sense)} if all the partial derivatives of order $k$ computed at $x$ of all the component functions exist for every $k \in \bn$ and every $x \in U$.
\end{defi}

\begin{defi}
  Two charts $(U_1, \phi_1)$ and $(U_2, \phi_2)$ of $X$ are said to be \emph{compatible}\index{chart!compatible} if $U_1 \cap U_2 = \emptyset$ or $\phi_1 \circ \phi_2^{-1}:\phi_2(U_1 \cap U_2) \to \phi_1(U_1 \cap U_2)$ is a diffeomorphism (\ie, a smooth bijection whose inverse is also smooth).
\end{defi}

\begin{defi}
  A set $\mc{A} = \qty{(U_i, \phi_i) : i \in I}$ of charts of $X$ is said to be an \emph{atlas for $X$}\index{atlas} if $X = \bigcup_{i \in I} U_i$ and the charts are all compatible with each other.
\end{defi}

\begin{defi}
  Two atlases $\mc{A}_1$ and $\mc{A}_2$ are said to be \emph{compatible}\index{atlas!compatible} if $\mc{A}_1 \cup \mc{A}_2$ is again an atlas.
\end{defi}

\begin{defi}
  An atlas $\mc{A}$ is said to be \emph{maximal}\index{atlas!maximal} if any other atlas that is compatible with $\mc{A}$ is contained in $\mc{A}$.
\end{defi}

\begin{defi}
  A \emph{($d$-dimensional) smooth manifold}\index{manifold!smooth}\index{smooth!manifold} (also called \emph{differentiable manifold}) is a pair $(X, \mc{A})$, being $X$ a ($d$-dimensional) topological manifold and $\mc{A}$ a maximal atlas.
\end{defi}

\begin{obs}
  When the atlas is not important, we will denote a smooth manifold just by $X$.
\end{obs}

\begin{obs}
  From now on all manifolds are assumed to be smooth.
\end{obs}

\begin{obs}
  It is always possible to extend an atlas $\mc{A}$ to a unique maximal atlas $\ol{\mc{A}}$, just let $\ol{\mc{A}}$ be the set of all charts of $X$ that are compatible with all charts in $\mc{A}$.
\end{obs}

This observation is very useful because we can now provide an infinite amount of examples of manifolds quite easily:
\begin{exm}[Open sets]
  Any open subset $U$ of the Euclidean space $\br^d$ with the subspace topology and the atlas $\qty{(U, \id_U)}$ is a smooth manifold.
  Consequently, $\br^d$ is a manifold and open balls are manifolds.
\end{exm}

\begin{exm}[Vector spaces]
  Finite-dimensional vector spaces $X$ over $\br$ and $\bc$ (with the topology defined in Example \ref{vector_space_topology}) are manifolds,\footnote{It is possible to generalize manifolds to cover the infinite-dimensional case, but we will not do so in the present work.} just consider the chart $\qty{(X, T)}$, being $T:X \to \br^d$ an isomorphism.
  When the underlying field of the vector space is $\bc$, we interpret $\bc^d$ as $\br^{2d}$.
  It is important to mention that any other isomorphism $S:X \to \br^d$ is compatible with $T$ because $T \circ S^{-1}$ is obviously a bijection and, since it can be represented by an invertible matrix, it is also a smooth function (its components are polynomials) and its inverse is smooth by the Inverse Function Theorem.
\end{exm}

\begin{exm}[Symmetric matrices]
  A particular example of vector space that will be useful later is $\sym(N) \ceq \qty{M \in \mat{N}{N} : \text{$M$ is symmetric, \ie, $M = M^{\top}$}}$.
  It is also important to mention that the topology on $\sym(N)$ defined in Example \ref{vector_space_topology} is the same as the subspace topology because this allows us to see $\sym(N)$ as an embedded manifold inside $\mat{N}{N}$ (we will define what embedded means in a moment, just keep that in mind).
\end{exm}

\begin{defi}
  A function between manifolds $f:X \to Y$ is said to be \emph{smooth at $x \in X$}\index{smooth!function}\index{function!smooth} if there exists charts $(U_X, \phi_X)$ and $(U_Y, \phi_Y)$ of $X$ and $Y$, respectively, such that $x \in U_X$, $f(U_X) \subset U_Y$ and $\phi_Y \circ f \circ \phi_X^{-1}:\phi_X(U_X) \to \phi_Y(U_Y)$ is smooth in the Euclidean sense.
  If $f$ is smooth at every $x \in X$, then we say $f$ is a \emph{smooth function}.
\end{defi}

\begin{obs}
  We will denote the set of all smooth functions $f:X \to \br$ by $C^{\infty}(X)$.
\end{obs}

Now let us define the tangent space to a manifold at a fixed point.
Given $X$ and $x \in X$, let
\begin{equation}
  C_x \ceq \qty{\gamma:(-\eps, \eps) \to X : \text{$\eps > 0$, $\gamma(0) = x$, $\gamma$ is smooth}}.
\end{equation}
Consider the following equivalence relation in $C_x$: given a chart $(U, \phi)$ and $\gamma_1, \gamma_2 \in C_x$, then $\gamma_1 \sim_{\phi} \gamma_2$ \tiff $(\phi \circ \gamma_1)'(0) = (\phi \circ \gamma_2)'(0)$.
This equivalence relation does not depend on the choice of chart, \ie, for any other chart $(\ol{U}, \ol{\phi})$, we have $\gamma_1 \sim_{\phi} \gamma_2$ \tiff $\gamma_1 \sim_{\ol{\phi}} \gamma_2$.
With that said, we will denote $\sim_{\phi}$ just as $\sim$ and we can define the tangent space:

\begin{defi}
  Given a smooth manifold $X$ and $x \in X$, the \emph{tangent space (to $X$ at $x$)}\index{tangent!space} is the quotient $C_x/\sim$.
  We will denote this quotient by $T_xM$ and call its elements \emph{tangent vectors}\index{tangent!vector}.
\end{defi}

\begin{obs}
  For any chart $(U, \phi)$ containing $x \in X$ it is possible to define the following bijection:
  \begin{align}
    \begin{split}
      \theta_x^{\phi}:T_xM & \to \br^d \\
      [\gamma] & \mapsto (\phi \circ \gamma)'(0).
    \end{split}
  \end{align}
  Then, we induce a vector space structure on $T_xM$ in the following way:
  \begin{equation}
    r[\gamma_1] + s[\gamma_2]
    \ceq (\theta_x^{\phi})^{-1}\qty(r\theta_x^{\phi}([\gamma_1])
    + s\theta_x^{\phi}([\gamma_2])).
  \end{equation}
  This structure does not depend on the choice of chart.
\end{obs}

\begin{exm}
  If $X$ is a finite-dimensional vector space, then $T_xX \cong X$ for every $x \in X$.
  The idea behind the proof of this example is that we can replace (equivalence classes of) curves passing through $x$ by the derivative $\gamma'(0)$ because the linear structure of $X$ allows us to define the derivative $\gamma'$ using limits as in Calculus.
\end{exm}

\begin{thm}[Chain Rule] \label{chain_rule}
  If $f:X \to Y$ and $g:Y \to Z$ are smooth functions, then
  \begin{equation}
    T_x(g \circ f) = T_{f(x)}g \circ T_xf.
  \end{equation}
\end{thm}

\begin{defi}
  Given a smooth function $f:X \to Y$ and $x \in X$, the \emph{directional derivative (of $f$ at $x$)} is the linear transformation defined by
  \begin{align}
    \begin{split}
      T_xf:T_pX & \to T_{f(x)}Y \\
      [\gamma] & \mapsto [f \circ \gamma].
    \end{split}
  \end{align}
\end{defi}

\begin{exm}
  When $X$ and $Y$ are vector spaces, the directional derivative can be computed as we learn in Calculus:
  \begin{equation}
    T_xf(v) = \lim_{t \to 0} \frac{f(x + tv) - f(x)}{t}.
  \end{equation}
  Observe that this only makes sense because $T_xX \cong X$ and $T_{f(x)}Y \cong Y$.
\end{exm}

\begin{defi}
  Given a manifold $X$, its \emph{tangent bundle}\index{tangent!bundle} is the set
  \begin{equation}
    TX \ceq \bigcup_{x \in X} \qty{x} \x T_xX = \qty{(x, [\gamma]) : x \in X, \ [\gamma] \in T_xX}.
  \end{equation}
  The tangent bundle comes with a natural function $\pi:TX \to X$, which we call \emph{projection}, that is defined by $\pi(x, [\gamma]) = x$.
\end{defi}

\begin{obs}
  The tangent bundle $TX$ is actually a $2d$-dimensional manifold, being $d = \dim{X}$.
  Its charts are constructed as follows: given a chart $(U, \phi)$ of $X$, the charts of $TX$ are $(\pi^{-1}(U), \wtil{\phi})$, being $\pi^{-1}(U)$ the preimage of $U$ and $\wtil{\phi}(x, [\gamma]) \ceq (\phi(x), \theta_x^{\phi}([\gamma]))$.
\end{obs}

\begin{defi}
  A \emph{vector field on $X$}\index{vector field} is a smooth function $V:X \to TX$ such that $\pi \circ V = \id_X$.
  The set of all vector fields on $X$ will be denoted by $\mf{X}(X)$.
\end{defi}

\begin{obs}
  Most of the time we will use $V(x)$ or $Vx$ to represent just the tangent vector instead of the pair $(x, V(x))$.
\end{obs}

\begin{obs} \label{vector_fields_module}
  $\mf{X}(X)$ is actually a vector space over $\br$ and a module (a vector space over a ring instead of a field) over $C^{\infty}(X)$.
  That means we can multiply $V \in \mf{X}(X)$ by a function $f:X \to \br$ by doing $(fV)(x) \ceq f(x)V(x)$ because $V(x) \in T_xX$, $f(x) \in \br$ and $T_xX$ is a vector space over $\br$.
\end{obs}

\begin{defi}
  A smooth function $f:X \to Y$ is called an \emph{immersion}\index{immersion} if the directional derivative $T_xf:T_xX \to T_{f(x)}Y$ is injective for every $x \in X$.
\end{defi}

\begin{defi}
  A smooth function $f:X \to Y$ is called a \emph{submersion}\index{submersion} if the directional derivative $T_xf:T_xX \to T_{f(x)}Y$ is surjective for every $x \in X$.
\end{defi}

\begin{exm}
  The canonical inclusion $i:\sym(N) \to \mat{N}{N}$ is an immersion because, since it is a linear transformation, we have $T_Mi = i$ for every $M \in \sym(N)$, which is obviously an injective function.
\end{exm}

\begin{defi}
  An immersion $f:X \to Y$ is called an \emph{embedding (of $X$ into $Y$)}\index{embedding} if $f$ is also a topological embedding, \ie, $X$ and $f(X)$ are homeomorphic when we endow $f(X)$ with the subspace topology of $Y$.
\end{defi}

\begin{defi}
  An \emph{embedded manifold}\index{manifold!embedded} is a subset $S \subset X$ such that $S$ with the subspace topology and a maximal atlas for this topology makes the inclusion $i:S \to X$ an embedding.
\end{defi}

\begin{exm}
  $\sym(N)$ is embedded in $\mat{N}{N}$.
\end{exm}

\begin{obs} \label{embedded_inside_embedded}
  If $S$ is embedded in $X$ and $S'$ is embedded in $S$, then $S'$ is embedded in $X$.
\end{obs}

\begin{prop} \label{tangent_space_embedding}
  If $X$ is embedded in $\br^d$ (or any other vector space, as a matter of fact), then we can describe the tangent space more concretely as
  \begin{equation}
    T_xX = \qty{\gamma'(0) \in \br^d : \text{$\gamma:(-\eps, \eps) \to X$ is smooth and $\gamma(0) = x$}}.
  \end{equation}
  This is only possible because $T_x\br^d \cong \br^d$, though, since the definition of $\gamma'(0)$ is given by $\lim_{t \to 0} \frac{\gamma(t) - \gamma(0)}{t}$, but subtracting $\gamma(t)$ and $\gamma(0)$ does not necessarily make sense if we consider just $X$.
\end{prop}

\begin{prop} \label{smooth_restriction}
  Let $f:X \to Y$ be a smooth function and $S \subset X$ be an embedded manifold.
  Then the restriction $f|_S:S \to Y$ is also smooth.
\end{prop}

\begin{prop} \label{smooth_codomain}
  Let $X$ be a smooth manifold and $S \subset X$ be an embedded manifold.
  Then every smooth function $f:Y \to X$ whose image is contained in $S$ is also smooth as a function from $Y$ to $S$.
\end{prop}

\begin{obs}
  These propositions are extremely useful to prove that functions from and to embedded manifolds are smooth.
  Keep them in mind throughout this whole chapter.
\end{obs}

\begin{defi}
  Given a smooth function $f:X \to Y$ and $x \in X$, we say that $x$ is a \emph{regular point of $f$}\index{regular!point} if $T_xf$ is surjective.
\end{defi}

\begin{defi} \label{regular_value}
  Given a smooth function $f:X \to Y$ and $y \in Y$, we say that $y$ is a \emph{regular value of $f$}\index{regular!value} if every $x \in f^{-1}(y)$ is a regular point of $f$.
\end{defi}

\begin{defi}
  Let $f:X \to Y$ be a smooth function and $S \subset X$ be an embedded manifold such that $S = f^{-1}(y)$ for some regular value $y \in Y$.
  Then we say that $f$ is a \emph{defining function for $S$}\index{defining function}.
\end{defi}

\begin{prop} \label{tangent_embedded}
  Let $X$ be a manifold and $S \subset X$ be an embedded manifold.
  If $f:X \to Y$ is a defining function for $S$, then $T_xS = \ker{T_xf}$ for every $x \in S$.
\end{prop}

\begin{thm}[Preimage Theorem] \label{preimage_theorem}
  Given a smooth function $f:X \to Y$ and a regular value $y \in Y$, its preimage $f^{-1}(y)$ is an embedded manifold of dimension $\dim{X} - \dim{Y}$.
\end{thm}

\begin{obs}
  The previous theorem is also known as \emph{Regular Level Set Theorem}.
\end{obs}

This is one of the most useful and powerful theorems to show certain subsets of $\br^d$ or $\mat{d}{N}$ are manifolds.
Let us provide an example that is very important to the present work and which is actually three examples at the same time.
\begin{exm}[Stiefel manifold] \label{stiefel_manifold}
  Given a symmetric and positive-definite matrix $S \in \mat{d}{d}$, consider the following set:
  \begin{equation}
    \st{N}{d} \ceq \qty{C \in \mat{d}{N} : C^{\top}SC = \Id_N}.
  \end{equation}
  Observe that when $N = d$ and $S = \Id_N$ , we have $\st{N}{N} = O(N)$, being $O(N)$ the orthogonal group.
  And when $S = \Id_d$ and $N = 1$, then $\st{1}{d} = \bs^{d-1}$, being $\bs^{d-1}$ the $(d-1)$-dimensional sphere.
  The set $\st{N}{d}$ is called \emph{(generalized) Stiefel manifold}\index{Stiefel manifold} and now let us prove that it is a manifold using the previous theorem.

  Consider the following function (which is smooth because of Proposition \ref{smooth_codomain})
  \begin{align}
    \begin{split}
      f:\mat{d}{N} & \to \sym(N) \\
      C & \mapsto C^{\top}SC - \Id_N.
    \end{split}
  \end{align}
  Observe that $\st{N}{d} = f^{-1}(0_N)$.
  Now, if we prove that $0_N$ is a regular value of $f$, we have the desired result.
  Let $C \in \st{N}{d}$.
  Then, since these manifolds are vector spaces, the directional derivative is computed in the traditional way:
  \begin{align}
    \begin{split}
      T_Cf(\mu) & = \lim_{t \to 0} \frac{f(C + t\mu) - f(C)}{t} \\
              & = \lim_{t \to 0} \frac{(C + t\mu)^{\top}S(C + t\mu) - C^{\top}SC}{t} \\
              & = \lim_{t \to 0} C^{\top}S\mu + \mu^{\top}SC + t\mu^{\top}S\mu \\
              & = C^{\top}S\mu + \mu^{\top}SC.
    \end{split}
  \end{align}
  Now we need to show that this map is surjective.
  Given $A \in \sym(N)$, if we let $\mu = \frac{1}{2}CA$, then
  \begin{equation}
    T_Cf(\mu)
    = \frac{1}{2}C^{\top}SCA + \frac{1}{2}A^{\top}C^{\top}SC
    = \frac{1}{2}A + \frac{1}{2}A^{\top}
    = A
  \end{equation}
  because $A = A^{\top}$ and $C^{\top}SC = \Id_N$.
  The result follows.
  To conclude, let us compute its tangent space.
  Since $f$ is a defining function for the Stiefel manifold, Proposition \ref{tangent_embedded} tells us that
  \begin{equation}
    T_C\st{N}{d}
    = \ker{T_Cf}
    = \qty{\mu \in \mat{d}{N} : C^{\top}S\mu + \mu^{\top}SC = 0_N}.
  \end{equation}
\end{exm}

\begin{prop} \label{stiefel_compact}
  The Stiefel manifold $\st{N}{d}$ is compact.
\end{prop}

\begin{proof}
  According to the Heine--Borel Theorem (\ref{heine_borel}), we need to prove that $\st{N}{d}$ is a closed and bounded subset of $\mat{d}{N}$.
  Now, we will consider in $\mat{d}{N}$ the topology induced by the inner product $\braket{C_1}{C_2} \ceq \tr(C_1^{\top}SC_2)$, but, since all norms in $\mat{d}{N}$ are equivalent (for a proof, see \cite{kconradnorm}), the topology induced is the same and Heine--Borel still holds in this case.
  With that said, since the defining function $f$ for the Stiefel manifold is continuous and since $\st{N}{d} = f^{-1}(0_N)$, we already have that the Stiefel is closed.
  To show that it is bounded, observe that, if $C \in \st{N}{d}$, then $\norm{C}^2 = \tr(C^{\top}SC) = \tr(\Id_N) = N$.
  Consequently, the Stiefel manifold is contained in the ball $B(0, \sqrt{N} + \eps)$ for any $\eps > 0$, which means it is bounded.
  The result follows.
\end{proof}

\begin{exm}
  Given $k$ manifolds $X_1, \ldots, X_k$, we can define the \emph{product manifold (of $X_1, \ldots, X_k$)}\index{manifold!product} as the set $X_1 \x \ldots \x X_k$ with the product topology (see Definition \ref{product_topology}) and with the charts constructed as follows: given charts $(U_i, \phi_i)$ of $X_i$, a chart of $X_1 \x \ldots \x X_k$ is defined by $(U_1 \x \ldots \x U_k, \phi_1 \x \ldots \x \phi_k)$, being $(\phi_1 \x \ldots \x \phi_k)(x_1, \ldots, x_k) \ceq (\phi_1(x_1), \ldots, \phi_k(x_k))$.
\end{exm}

\begin{obs}
  From now on when we consider the product of two or more manifolds, the smooth structure we will consider is always the one described above.
\end{obs}

\begin{prop} \label{tangent_space_product}
  If $X_1, \ldots, X_k$ are manifolds, then
  \begin{equation}
    T_{(x_1, \ldots, x_k)}(X_1 \x \ldots \x X_k) \cong T_{x_1}X_1 \x \ldots \x T_{x_k}X_k
  \end{equation}
  because the function $[\gamma] \mapsto ([\pi_1 \circ \gamma], \ldots, [\pi_k \circ \gamma])$, being $\pi_i$ the canonical projection, is an isomorphism.
\end{prop}

\begin{defi}
  A smooth manifold $G$ that is also a group (see Definition \ref{group}) is called a \emph{Lie group}\index{Lie group} if the group operation $m:G \x G \to G$ and the inverse $(\cdot)^{-1}:G \to G$ are smooth functions.
\end{defi}

\begin{exm}
  The orthogonal group $O(N)$ is a Lie group because the product of matrices and the inverse of a matrix are smooth functions since the component functions are polynomials (for the inverse recall Cramer's rule).
  Observe that we are using Propositions \ref{smooth_restriction} and \ref{smooth_codomain} because $O(N)$ is an embedded manifold.
\end{exm}

\begin{defi}
  An \emph{action}\index{group!action}\index{action} of a group $G$ on a set $X$ is a map $\theta:X \x G \to X$ such that $\theta(x, e) = x$ and $\theta(\theta(x, g), h) = \theta(x, gh)$ for every $x \in X$ and $g, h \in G$.
  Recall that $e \in G$ represents the identity of the group.
\end{defi}

\begin{obs}
  The definition above is actually called a \emph{right action}, but we will call it just action because we will not use left actions in this work.
\end{obs}

\begin{exm} \label{action_stiefel}
  Let $X = \st{N}{d}$ and $G = O(N)$.
  The function
  \begin{align}
    \begin{split}
      \theta:\st{N}{d} \x O(N) & \to \st{N}{d} \\
      (C, M) & \mapsto CM
    \end{split}
  \end{align}
  defines an action of $O(N)$ on $\st{N}{d}$.
\end{exm}

\begin{defi}
  Given a set $X$, $x \in X$ and an action $\theta:X \x G \to X$, we define the \emph{stabilizer of $x$}\index{stabilizer} to be the set $\stab(x) \ceq \qty{g \in G : \theta(x, g) = x}$.
\end{defi}

\begin{defi}
  A group action is called \emph{free}\index{action!free} if $\stab(x) = \qty{e}$ for every $x \in X$.
\end{defi}

\begin{exm}
  The action defined in Example \ref{action_stiefel} is free.
  Indeed, $CM = C$ \tiff $C^{\top}SCM = C^{\top}SC$.
  Therefore, since $C^{\top}SC = \Id_N$, we have $M = \Id_N$ and the result follows.
\end{exm}

\begin{defi}
  Given a set $X$, $x \in X$ and an action $\theta:X \x G \to X$, we define the \emph{orbit of $x$}\index{orbit} to be the set $\orb(x) \ceq \qty{\theta(x, g) : g \in G}$.
\end{defi}

\begin{defi}
  If we have a group $G$ acting on $X$, it is possible to define an equivalence relation on $X$ as follows: $x \sim y$ \tiff $y \in \orb(x)$.
  The quotient $X/\sim$ will be denoted by $X/G$ and we will call it \emph{orbit space}\index{space!orbit}.
\end{defi}

\begin{defi}
  Given a manifold $X$ and a Lie group $G$ acting on $X$, we say that the action $\theta$ is \emph{smooth}\index{action!smooth} if the function $\theta:X \x G \to X$ is smooth.
\end{defi}

\begin{exm}
  The action of Example \ref{action_stiefel} is smooth because multiplying matrices is smooth (again, recall that the manifolds are embedded).
\end{exm}

\begin{defi}
  Given a manifold $X$ and a Lie group $G$ acting on $X$, we say that the action $\theta$ is \emph{proper}\index{action!proper} if the function $\id_X \x \theta:X \x G \to X \x X$ given by $(\id_X \x \theta)(p, g) = (p, \theta(p, g))$ is proper.
\end{defi}

\begin{prop}
  Every continuous action by a compact Lie group on a manifold is proper.
\end{prop}

\begin{cor}
  Since we already showed that the Stiefel manifold is compact and that the orthogonal group is just a particular case of a Stiefel manifold, the previous proposition shows that the action defined in \ref{action_stiefel} is proper.
\end{cor}

\begin{thm}[Quotient Manifold Theorem] \label{qmt}
  Suppose $G$ is a Lie group acting smoothly, freely and properly on a smooth manifold $X$.
  Then the orbit space $X/G$ admits a unique smooth structure such that the canonical projection $\pi:X \to X/G$ is a submersion.
  Also, $\dim{X/G} = \dim{X} - \dim{G}$.
\end{thm}

\begin{obs}
  We are considering the quotient topology (Example \ref{quotient_topology}) on $X/G$.
\end{obs}

\begin{exm}[Grassmannian] \label{grassmannian}
  According to the previous theorem (and from what we have showed previously), the quotient $\st{N}{d} / O(N)$ is a manifold and we will call it \emph{(generalized) Grassmannian}\index{grassmannian} or \emph{(generalized) Grassmann manifold}\index{Grassmann manifold}.
  The generalized part is because in the definition of Stiefel manifold we are considering matrices $C \in \mat{d}{N}$ such that $C^{\top}SC = \Id_N$, being $S \in \mat{d}{d}$ a symmetric and positive-definite matrix.

  The above definition is the one we will use in Chapter \ref{optimization}, but the most popular definition of Grassmannian and the one we will use in Chapter \ref{qm} is the following: let $X$ be a $d$-dimensional vector space.
  Then, given $0 \leq N \leq d$, the Grassmannian can also be defined as $\gr{N}{X} \ceq \qty{L \subset X : \text{$L$ is a subspace of $X$ and $\dim{L} = N$}}$.
  Now, it is worth mentioning how we can move from one definition to the other.
  If we start with an equivalence class $[C] \in \st{N}{d}/O(N)$, then to obtain a subspace of dimension $N$ we just need to consider the subspace spanned by the columns of $C$.
  On the other hand, if we start with a subspace $L \subset X$, then we need first to choose a basis $B = \qty{v_1, \ldots, v_d}$ for $X$ such that $S = \qty(\braket{v_i}{v_j})_{i,j}$.
  Secondly, we choose an orthonormal basis for $L$: $B_L = \qty{\l_1, \ldots, \l_N}$.
  Then, we write the vectors of $B_L$ as linear combinations of the vectors in $B$: $\l_j = c_{1j}v_1 + \ldots + c_{dj}v_d$.
  And, to conclude, we consider the equivalence class $[C]$, being $C = \qty(c_{ij})_{i,j}$.
  Observe that this matrix $C$ satisfies $C^{\top}SC = \Id_N$ because we assumed $B_L$ to be orthonormal, even if the basis $B$ is not (which gives rise to the matrix $S$).
\end{exm}

\begin{exm}[Projective space] \label{projective}
  The manifold $\gr{1}{d}$ or, more generally, $\gr{1}{X}$, is also known as \emph{projective space}\index{projective space} and it is often denoted by $\bp{X}$.
\end{exm}

\begin{obs} \label{abstract_concrete}
  We will usually say that the Grassmannian defined as a quotient of the Stiefel manifold is the \emph{concrete Grassmannian}\index{concrete Grassmannian}\index{Grassmannian!concrete} and the definition as a set of subspaces is the \emph{abstract Grassmannian}\index{abstract Grassmannian}\index{Grassmannian!abstract}.
\end{obs}

\begin{obs}
  Observe that the abstract definition of the Grassmannian also works for the case in which $X$ is infinite-dimensional (always assuming $N \in \bn$), but it is not possible to establish the correspondence with the quotient of the Stiefel manifold anymore, at least not as we defined the Stiefel (there is also an abstract version of this manifold, see \cite{grossi2012}).
  It is important to mention this because in Quantum Mechanics (Chapter \ref{qm}) we will usually deal with the projective space of infinite-dimensional vector spaces and the definition we are considering there is the abstract one.
  However, we will not use the (infinite-dimensional) manifold structure of these spaces.
\end{obs}

\begin{lem}
  Given a manifold $X$ and a Lie group $G$ acting on $X$, the canonical projection $\pi:X \to X/G$ is an open map.
\end{lem}

\begin{thm}
  If $X$ is a compact manifold and $G$ is a Lie group acting on $X$, then $X/G$ is compact.
\end{thm}

\begin{proof}
  If $\mc{U}$ is a cover of $X/G$, then, since the projection is continuous, $\qty{\pi^{-1}(U) : U \in \mc{U}}$ is a cover of $X$.
  Now, using the compactness of $X$, there exists $U_1, \ldots, U_k \in \mc{U}$ such that $X = \bigcup_{i=1}^k \pi^{-1}(U_i)$.
  However, if we use the previous lemma and the surjectivity of $\pi$, we know that $\pi(\pi^{-1}(U_i)) = U_i$ are open in $X/G$.
  That is it, $X/G = \bigcup_{i=1}^k U_i$ because $\pi$ is surjective.
\end{proof}

\begin{cor} \label{grassmannian_compact}
  The Grassmannian is compact.
\end{cor}

\begin{proof}
  Using the fact that the Stiefel manifold is compact and the previous theorem the result follows.
\end{proof}

\begin{cor} \label{product_compact}
  Finite products of the Grassmannian and of the Stiefel manifold are compact.
\end{cor}

\begin{proof}
  It follows from the fact that the Grassmannian and the Stiefel manifold are compact together with the fact that finite products of compact spaces are compact (Proposition \ref{product_compact}).
\end{proof}

\begin{defi}
  Given a manifold $X$ and an equivalence relation $\sim$, we say that the quotient $X/\sim$ is a \emph{quotient manifold}\index{manifold!quotient}\index{quotient manifold} if it has a smooth structure such that the canonical projection $\pi:X \to X/\sim$ is a submersion.
\end{defi}

\begin{obs}
  Theorem \ref{qmt} is basically saying that $X/G$ is a quotient manifold.
\end{obs}

Now, there is a very useful result that allows us to obtain smooth functions defined on quotient manifolds (and, consequently, on the Grassmannian):
\begin{thm} \label{smooth_quotient}
  Given a quotient manifold $X/\sim$, a function $f:X/\sim \to Y$ is smooth \tiff $f \circ \pi:X \to Y$ is smooth, being $\pi:X \to X/\sim$ the canonical projection.
\end{thm}

The next theorem is useful to find the tangent space to a quotient manifold, which is necessary to do Riemannian Geometry:
\begin{thm}
  Given a quotient manifold $X/\sim$ and $x \in X$, the set $\pi^{-1}(\pi(x))$ is called \emph{fiber of $x$} and (1) it is a closed set of $X$, (2) it is an embedded manifold of $X$, (3) its tangent spaces are given by
  \begin{equation}
    T_y\pi^{-1}(\pi(x)) = \ker{T_y\pi},
  \end{equation}
  being $y \in \pi^{-1}(\pi(x))$ and $\pi$ the canonical projection.
  When $y = x$, the tangent space is called \emph{vertical space}\index{vertical space} and it is denoted by $V_x$.
\end{thm}

\begin{obs} \label{quotient_vertical}
  Now, since $T_x\pi:T_xX \to T_{[x]}(X/\sim)$ is a surjective linear transformation, we can use the Universal Property of the Quotient (Proposition \ref{univ_quotient}) to obtain $T_{[x]}(X/\sim) \cong T_xX/V_x$, which means we can compute the tangent space to the quotient manifold using $X$ and $\pi$.
  This is already useful, but when the manifold is Riemannian we can describe this quotient more concretely.
  So, let us move to Riemannian manifolds.
\end{obs}

\section{Riemannian Geometry}

Again, all manifolds are assumed to be smooth.
Also, since in this section we will use tangent vectors multiple times, specially in quotient manifolds, we will switch the notation from $[\gamma]$ to $v$ to improve readability.

\begin{defi}
  Given a manifold $X$, a \emph{(Riemannian) metric on $X$}\index{Riemannian!metric}\index{metric} is a function $x \mapsto \braket{\cdot}_x$ that chooses an inner product on $T_xX$ smoothly, \ie, $\braket{\cdot}_x:T_xX \x T_xX \to \br$ is an inner product and, given two vector fields $V_1, V_2$ on $X$, the function $X \ni x \mapsto \braket{V_1(x)}{V_2(x)}_x \in \br$ should be smooth.
  A manifold endowed with a Riemannian metric is called a \emph{Riemannian manifold}\index{manifold!Riemannian}\index{Riemannian!manifold} and we will usually denote the Riemannian metric by $g$ or $\braket{\cdot}$.
\end{defi}

\begin{obs}
  We are slightly abusing the notation in the definition above because $V_i(x)$ is actually representing a vector in $T_xX$ instead of a pair $(x, V_i(x))$ in $TX$.
\end{obs}

\begin{obs}
  Given the fact that we have an inner product, we can induce a norm in each tangent space in the natural way: $\norm{v}_x = \sqrt{\braket{v}_x}$, being $v \in T_xX$.
\end{obs}

\begin{obs}
  When the metric does not depend explicitly on $x$, as we will see in the next example, we will drop the subscript and denote $\braket{\cdot}_x$ by $\braket{\cdot}$.
\end{obs}

\begin{exm}[Vector spaces]
  Any vector space endowed with an inner product $(X, \braket{\cdot})$ is a Riemannian manifold because, since $T_xX \cong X$ for every $x \in X$, we can use $\braket{\cdot}$ also as a Riemannian metric.
  Observe that in this case the inner product is the same for the manifold and all its tangent spaces.
  Therefore, we will ignore the subscript.
\end{exm}

\begin{exm}[Euclidean space]
  When $X = \br^d$, we will always consider the inner product given by $\braket{u}{v} = u^{\top}v$.
\end{exm}

\begin{exm}[Matrix space]
  Observe that, if $X = \mat{d}{N}$, then $\braket{M_1}{M_2} = \tr(M_1^{\top}M_2) = \vec(M_1)^{\top}\vec(M_2)$, being vec the columnwise vectorization (see Definition \ref{vectorization}).
  So, the inner product in the matrix space $\mat{d}{N}$ and in the Euclidean space $\br^{dN}$ is actually the same if we adopt this columnwise convention.
  That said, we will mostly use the following inner product when we are dealing with matrices: $\braket{M_1}{M_2} = \tr(M_1^{\top}SM_2)$, being $S \in \mat{d}{d}$ a symmetric and positive-definite matrix.
\end{exm}

\begin{exm}
  If $(X, g^X)$ and $(Y, g^Y)$ are Riemannian manifolds, then the product $X \x Y$ is also a Riemannian manifold and the metric is given by:
  \begin{equation}
    g_{(x,y)}^{X \x Y}((v_x, v_y), (w_x, w_y))
    \ceq g_x^X(v_x, w_x) + g_y^Y(v_y, w_y).
  \end{equation}
\end{exm}

\begin{exm}[Embedded manifolds]
  If $f:X \to Y$ is an immersion and $(Y, g^Y)$ is a Riemannian manifold, we can define a metric on $X$ as follows:
  \begin{equation}
    g_x^X(v, w)
    \ceq g_{f(x)}^Y(T_xf(v), T_xf(w)).
  \end{equation}
  This metric $g^X$ is usually called \emph{pullback metric (of $g^Y$ along $f$)} and it is denoted by $f^*g^Y$.
\end{exm}

\begin{defi}
  If $(X, g^X)$ is a Riemannian manifold and $S \subset X$ is an embedded manifold, then $S$ endowed with the pullback metric along the inclusion is called a \emph{Riemannian submanifold (of $X$)}.
\end{defi}

\begin{exm}[Stiefel manifold]
  We already showed that the Stiefel manifold $\st{N}{d}$ is embedded in $\mat{d}{N}$ and we also computed its tangent spaces in Example \ref{stiefel_manifold}.
  So, if we consider the pullback metric on the Stiefel, we have a more ``exotic'' example of Riemannian manifold.
  That said, since we are working with a generalized Stiefel, we have to fix the inner product a little bit.
  So, to us, the inner product on $\mat{d}{N}$ and $\st{N}{d}$ is given by $\braket{M_1}{M_2} \ceq \tr(M_1^{\top}SM_2)$, being $S$ the symmetric and positive-definite matrix underlying the generalized Stiefel.
  Some people claim that this metric is not the best one to consider in the Stiefel manifold and this makes a big difference when one is implementing Riemannian algorithms.
  However, since we will implement algorithms on the Grassmannian using the Stiefel just as a helper, we will keep using this metric.
  If the reader is more interested in the Stiefel manifold, though, check \cite{edelman1998} to see a discussion about the metric issue and for algorithms.
\end{exm}

In Observation \ref{quotient_vertical} we said that the tangent space to a quotient manifold $X/\sim$ is isomorphic to $T_xX/V_x$, being $V_x = \ker{T_x\pi}$ and $\pi$ the canonical projection.
However, when $X$ is a Riemannian manifold, $T_xX/V_x$ is isomorphic to the orthogonal complement of $V_x$ and this space is easier to work with.\footnote{The isomorphism is obtained by considering the orthogonal projection $\pi:T_xX \to V_x^{\perp}$ and then using the Universal Property of the Quotient (\ref{univ_quotient}).}

\begin{defi}
  The orthogonal complement of $V_x$ is called \emph{horizontal space} and we will denote it by $H_x$.
  It is worth emphasizing that $H_x \cong T_{[x]}(X/\sim)$.
\end{defi}

\begin{obs}
  $H_x \cong H_y$ for every $x, y \in X$ such that $x \sim y$.
\end{obs}

\begin{defi}
  Given a quotient manifold $X/\sim$, $[x] \in X/\sim$ and $v \in T_{[x]}(X/\sim)$, we define the \emph{horizontal lift of $v$} to be the unique vector $\lift_x(v) \in H_x$ such that $T_x\pi(\lift_x(v)) = v$.
  The uniqueness is given by the fact that $H_x$ is isomorphic to $T_{[x]}(X/\sim)$.
  Also, if $V \in \mf{X}(X/\sim)$ is a vector field in the quotient, then the \emph{horizontal lift of $V$}, which we will denote by $\ol{V}$, is the vector field $\ol{V} \in \mf{X}(X)$ such that $\ol{V}(x) = \lift_x(V([x]))$.
\end{defi}

\begin{defi}
  If $(X, \braket{\cdot}^X)$ is a Riemannian manifold, then a quotient manifold $X/\sim$ with the metric $\braket{\cdot}^{X/\sim}$ defined by
  \begin{equation}
    \braket{v}{w}_{[x]}^{X/\sim} \ceq \braket{\lift_x(v)}{\lift_x(w)}_x^X
  \end{equation}
  is called a \emph{Riemannian quotient manifold}\index{quotient manifold!Riemannian}.
  However, this metric is well-defined only if
  \begin{equation} \label{metric_well_defined}
    \braket{\lift_x(v)}{\lift_x(w)}_x^X = \braket{\lift_y(v)}{\lift_y(w)}_y^X
  \end{equation}
  for every $y \in X$ such that $y \sim x$.
\end{defi}

\begin{exm}[Grassmannian] \label{grassmannian_riemannian}
  Changing the notation from $x$ to $C$, let us compute the vertical space $V_C$ first.
  Recall that the fiber $\pi^{-1}(\pi(C))$ is an embedded manifold of $\st{N}{d}$ (Quotient Manifold Theorem \ref{qmt}).
  Therefore, we can compute its tangent space (which is, by definition, $V_C$) in the usual way, \ie, we see $\pi^{-1}(\pi(C))$ as an embedded manifold of $\mat{d}{N}$ (Observation \ref{embedded_inside_embedded}) and then consider the derivative of a curve $\gamma:(-\eps, \eps) \to \pi^{-1}(\pi(C))$ computed at $0$ (Proposition \ref{tangent_space_embedding}).
  Now, since the Grassmannian is defined as $\st{N}{d}/O(N)$ (Example \ref{grassmannian}), it is not hard to see that $\pi^{-1}(\pi(C)) = C \cdot O(N) \ceq \qty{CM : M \in O(N)}$.
  Consequently, the curve $\gamma$ can be written as $\gamma(t) = CM(t)$, being $M(t) \in O(N)$ for every $t \in (-\eps, \eps)$.
  Besides, since we also know that $\gamma(0) = C$, we have $M(0) = \Id_N$.
  So, this tells us that $T_C\pi^{-1}(\pi(C)) \cong C \cdot T_{\Id_N}O(N)$.
  Recalling that $T_{\Id_N}O(N) = \qty{\nu \in \mat{N}{N} : \nu = -\nu^{\top}}$ (this is a particular case of the Stiefel manifold and we already computed the tangent space to the Stiefel manifold at Example \ref{stiefel_manifold}), we have
  \begin{equation}
    V_C = T_C\pi^{-1}(\pi(C)) \cong C \cdot T_{\Id_N}O(N)
    = \qty{C\nu \in \mat{d}{N} : \nu = -\nu^{\top}}
  \end{equation}
  and
  \begin{equation}
    H_C = \qty{\eta \in T_C\st{N}{d} : \text{$\braket{C\nu}{\eta} = \tr(\nu^{\top}C^{\top}S\eta) = 0$ for every $C\nu \in V_C$}}.
  \end{equation}
  However, if $\tr(\nu^{\top}C^{\top}S\eta) = 0$ for every $\nu$ antisymmetric, then $C^{\top}S\eta$ is a symmetric matrix because, if we choose $\nu = E_{ij} - E_{ij}^{\top}$ ($E_{ij}$ is defined in \ref{canonical_matrix}), then $\tr(E_{ij}^{\top}C^{\top}S\eta) - \tr(E_{ij}C^{\top}S\eta) = (C^{\top}S\eta)_{ij} - (C^{\top}S\eta)_{ji} = 0$.
  On the other hand, $\eta \in T_C\st{N}{d}$ means, by definition, that $C^{\top}S\eta$ is antisymmetric.
  Consequently, since the only matrix that is symmetric and antisymmetric at the same time is $0$, we have
  \begin{equation}
    H_C = \qty{\eta \in \mat{d}{N} : C^{\top}S\eta = 0_N}.
  \end{equation}
  
  To conclude we have to verify that Equation \ref{metric_well_defined} is satisfied in the Grassmannian, which is the same as proving that
  \begin{equation} \label{metric_grassmannian_well_defined}
    \braket{\lift_C(\eta_1)}{\lift_C(\eta_2)}_C^{\St}
    = \braket{\lift_{CM}(\eta_1)}{\lift_{CM}(\eta_2)}_{CM}^{\St}
  \end{equation}
  for every $M \in O(N)$ because $[C] = \qty{CM : M \in O(N)}$.
  Now, given $M \in O(N)$ and $\eta \in T_{[C]}\gr{N}{d}$, let $\gamma:(-\eps,\eps) \to \st{N}{d}$ be a curve such that $\gamma(0) = C$ and $\gamma'(0) = \lift_C(\eta)$.
  Consider also the curve $\rho:(-\eps,\eps) \to \st{N}{d}$ given by $\rho(t) = \gamma(t)M$.
  Then we have that $\rho(0) = CM$ and $\rho'(0) = \gamma'(0)M = \lift_C(\eta)M$.
  Now, according to the Chain Rule (Theorem \ref{chain_rule}), we have $(\pi \circ \gamma)'(0) = T_{\gamma(0)}\pi(\gamma'(0)) = T_C\pi(\lift_C(\eta)) = \eta$.
  However, since $\pi \circ \gamma$ and $\pi \circ \rho$ are the same curve on $\gr{N}{d}$, we can conclude that
  \begin{equation}
    \eta = (\pi \circ \gamma)'(0) = (\pi \circ \rho)'(0)
    = T_{\rho(0)}\pi(\rho'(0)) = T_{CM}\pi(\lift_C(\eta)M).
  \end{equation}
  Then, using the uniqueness of the horizontal lift we conclude that
  \begin{equation}
    \lift_{CM}(\eta) = \lift_C(\eta)M.
  \end{equation}
  Finally, using the fact that $\tr(A_1A_2A_3A_4A_5) = \tr(A_5A_1A_2A_3A_4)$ and the identity above, we can show that Equation \ref{metric_grassmannian_well_defined} is valid:
  \begin{align}
    \begin{split}
      \braket{\lift_{CM}(\eta_1)}{\lift_{CM}(\eta_2)}_{CM}^{\St}
      & = \braket{\lift_C(\eta_1)M}{\lift_C(\eta_2)M}_{CM}^{\St} \\
      & = \tr(M^{\top}\lift_C(\eta_1)^{\top}S\lift_C(\eta_2)M) \\
      & = \tr(\lift_C(\eta_1)^{\top}S\lift_C(\eta_2)) \\
      & = \braket{\lift_C(\eta_1)}{\lift_C(\eta_2)}_C^{\St}.
    \end{split}
  \end{align}
  Consequently, the Grassmannian is indeed a Riemannian quotient manifold of the Stiefel.
  That is it, from now on we can ignore the abstract definition of $T_{[C]}\gr{N}{d}$ and work only with $H_C$.
\end{exm}

\begin{defi}
  Given a Riemannian manifold $(X, \braket{\cdot})$ and a smooth function $f:X \to \br$, the \emph{(Riemannian) gradient of $f$}\index{gradient}\index{Riemannian!gradient} is the unique vector field $\grad{f}:X \to TX$ that satisfies
  \begin{equation}
    \braket{v}{\grad{f}(x)}_x = T_xf(v)
  \end{equation}
  for every $v \in T_xX$.
\end{defi}

\begin{exm} \label{gradient_vector_spaces}
  When $X$ is a vector space, we can choose an orthonormal basis $B = \qty{v_1, \ldots, v_d}$ for $X$ and then a concrete description for the gradient is given by the directional derivatives:
  \begin{equation}
    \grad{f}(x) = \pdv{f}{v_1}\qty(x)v_1 + \ldots + \pdv{f}{v_d}\qty(x)v_d,
  \end{equation}
  being
  \begin{equation}
    \pdv{f}{v_i}\qty(x) \ceq \lim_{t \to 0} \frac{f(x + tv_i) - f(x)}{t}.
  \end{equation}
  Observe that when $X = \br^d$ and we choose the canonical basis $\qty{e_1, \ldots, e_d}$, we obtain the gradient from Calculus.
  It is also worth mentioning that if $X = \mat{d}{N}$, then the gradient can also be written as a matrix whose entry $(i,j)$ is $\pdv{f}{v_{ij}}\qty(x)$.
\end{exm}

\begin{exm} \label{gradient_nonorthogonal}
  Now let us see how we can change the gradient from one basis to another because we will need this for the Grassmannian and the Stiefel.
  Suppose we have a vector space $X$ and two basis for it: $B_v = \qty{v_1, \ldots, v_d}$ and $B_w = \qty{w_1, \ldots, w_d}$.
  If we write $w_j = a_{1j}v_1 + \ldots + a_{dj}v_d$ and assume $B_v$ is orthonormal, then $a_{ij} = \braket{v_i}{w_j}$ and we can build the matrix $A = \qty(a_{ij})_{i,j}$.
  Also, if we write $A^{-1} = \qty(a^{ij})_{i,j}$, then $v_j = a^{1j}w_1 + \ldots + a^{dj}w_d$.
  Now, using the linearity of the directional derivative, we have:
  \begin{align}
    \begin{split}
      \grad{f}(x)
      & = \sum_{k=1}^d \pdv{f}{v_k}\qty(x)v_k \\
      & = \sum_{k=1}^d \pdv{f}{\sum_{i=1}^d a^{ik}w_i}\qty(x)
        \cdot \qty(\sum_{j=1}^d a^{jk}w_j) \\
      & = \sum_{i,j,k=1}^d a^{ik}a^{jk}\pdv{f}{w_i}\qty(x)w_j.
    \end{split}
  \end{align}
  However, if we use the following identity
  \begin{equation}
    \braket{w_i}{w_j} = \sum_{k=1}^d \braket{\braket{v_k}{w_i}v_k}{w_j}
    = \sum_{k=1}^d \braket{v_k}{w_i}\braket{v_k}{w_j}
    = \sum_{k=1}^d a_{ki}a_{kj},
  \end{equation}
  then the Gram matrix $G_w \ceq \qty(\braket{w_i}{w_j})_{i,j}$ is equal to $A^{\top}A$.
  Consequently, $G_w^{-1} = (A^{\top}A)^{-1} = A^{-1}A^{-\top}$ and that means the gradient in the basis $B_w$ is given by:
  \begin{equation}
    \grad{f}(x) = G_w^{-1}
    \begin{bmatrix}
      \pdv{f}{w_1}\qty(x) \\ \vdots \\ \pdv{f}{w_d}\qty(x)
    \end{bmatrix}_{B_w},
  \end{equation}
  where the subscript $B_w$ means we are writing the vector in the basis $B_w$.
  Observe that this also works for the case in which $X = \mat{d}{N}$, \ie, we have $G_w^{-1}$ multiplying the matrix $\qty(\pdv{f}{w_{ij}}\qty(x))_{i,j}$.
\end{exm}

\begin{prop} \label{grad_product}
  Given a smooth function $f:X \x Y \to \br$, the Riemannian gradient on this product is given by
  \begin{equation}
    \grad{f}(x, y) = \qty(\grad{f}(\cdot, y)(x), \grad{f}(x, \cdot)(y)).
  \end{equation}
\end{prop}

\begin{prop}
  Let $X$ be a Riemannian submanifold of $\br^d$, $f:X \to \br$ be a smooth function and $\ol{f}:\br^d \to \br$ be an extension of $f$.
  Then, the Riemannian gradient of $f$ is given by:
  \begin{equation}
    \grad{f}(x) = \proj_x^X(\grad{\ol{f}}(x)),
  \end{equation}
  being $\proj_x^X:\br^d \to \br^d$ the orthogonal projection to $T_xX$.
\end{prop}

\begin{obs}
  The idea behind this result is: compute the gradient in the traditional way by considering an extension and then project the result.
  In other words, the intuition is that the Riemannian gradient is the projection of the Euclidean gradient.
  Sometimes, as we will see in Section \ref{hf_ch3}, $\ol{f}$ and $f$ have the same expression.
\end{obs}

\begin{exm}
  When $X = \st{N}{d}$, the orthogonal projection is given by
  \begin{align}
    \begin{split}
      \proj_C^{\St}:\mat{d}{N} & \to \mat{d}{N} \\
      \mu & \mapsto (\Id_d - CC^{\top}S)\mu + \frac{CC^{\top}S\mu - C\mu^{\top}SC}{2}
    \end{split}
  \end{align}
  and the reader can find a full derivation of this projection in \cite[Section 7.3]{boumal2022}.
  Now, computing the gradient is a little trickier because we are working with the generalized Stiefel.
  Basically, in our case the Stiefel is built in the following way: we start with a non-orthogonal basis $B_v = \qty{v_1, \ldots, v_d}$ of $\br^d$ and then we extend this to a basis of $\mat{d}{N}$ by considering $v_{ij} \ceq \qty(0, \ldots, v_i, \ldots, 0)$, being $v_i$ located in the $j$-th entry.
  We also have that $S = \qty(v_i^{\top}v_j)_{i,j}$, but in this case we are considering $v_i$ and $v_j$ as linear combinations of the canonical basis of $\br^d$.
  With all that said, to compute the gradient of a function $f:\st{N}{d} \to \br$, we need to extend it to $\ol{f}:\mat{d}{N} \to \br$ and then recall what we have discussed in Examples \ref{gradient_vector_spaces} and \ref{gradient_nonorthogonal} to obtain the following gradient for a function defined in the Stiefel:
  {\small
    \begin{align}
      \begin{split}
        \grad{f}(C) & = \proj_C^{\St}(\grad{\ol{f}}(C)) \\
                    & = (\Id_d - CC^{\top}S)S^{-1}\grad{\ol{f}}(C)
                      + \frac{CC^{\top}SS^{-1}\grad{\ol{f}}(C)
                      - C\grad{\ol{f}}(C)^{\top}S^{-1}SC}{2} \\
                    & = (\Id_d - CC^{\top}S)S^{-1}\grad{\ol{f}}(C)
                      + \frac{CC^{\top}\grad{\ol{f}}(C)
                      - C\grad{\ol{f}}(C)^{\top}C}{2}.
      \end{split}
    \end{align}
  }%
  So, the intuition is: compute the gradient of $\ol{f}$ as we learn in Calculus.
  Fix it multiplying by $S^{-1}$ on the left to get back to the non-orthogonal basis we started with.
  To conclude, project the result to obtain a tangent vector to the Stiefel.

  If the reader is alert for the details, then using the matrix $S^{-1}$ may seem suspicious because the matrix $S$ is the Gram matrix for $\br^d$ and not for $\mat{d}{N}$, which is what we need.
  However, since the basis we are considering for $\mat{d}{N}$ is built using a basis of $\br^d$ as we explained previously, the Gram matrix for $\mat{d}{N}$ is $S \otimes \Id_N$, being $\otimes$ the Kronecker product (see Definition \ref{kronecker_product}).
  As the reader can notice, the dimension now is correct, but to use $S \otimes \Id_N$ we would need to vectorize the matrices from $\mat{d}{N}$.
  Now, a very fortunate result says that we can keep using just $S$ because, given $C \in \mat{d}{N}$, we have $(S \otimes \Id_N)\vec(C) = SC$ (see Appendix \ref{matrix_ids}).
\end{exm}

\begin{prop}
  Let $X/\sim$ be a Riemannian quotient manifold, $X$ be embedded in $\br^d$ and $f:X/\sim \to \br$ be a smooth function.
  Then, the Riemannian gradient of $f$ is given by:
  \begin{equation}
    \grad{f}([x]) = \proj_x^{\hor}(\grad{\oll{f}}(x)),
  \end{equation}
  being $\oll{f}:\br^d \to \br$ a smooth extension of $\ol{f}:X \to \br$; $\ol{f}$ a lift of $f$, \ie, $\ol{f} = f \circ \pi$; and $\proj_x^{\hor}:\br^d \to \br^d$ the orthogonal projection to $H_x$.
\end{prop}

\begin{obs}
  Again, the main idea is finding a way to use the regular Euclidean gradient, since this gradient we already know how to compute.
  In this case the easiest way requires three steps, though: lifting the original function, extending the lift and then projecting the Euclidean gradient to the horizontal space because this is the space tangent to the quotient manifold.
\end{obs}

\begin{exm} \label{grad_grassmannian}
  For the Grassmannian, the projection to the horizontal space is given by
  \begin{align}
    \begin{split}
      \proj_{[C]}^{\hor}:\mat{d}{N} & \to \mat{d}{N} \\
      \eta & \mapsto (\Id_d - CC^{\top}S)\eta.
    \end{split}
  \end{align}
  Observe that this is obtained by ignoring the antisymmetric part of the projection to the Stiefel because this part maps vectors to $V_C$ and $V_C \cap H_C = \qty{0_{d \x N}}$.
  Consequently,
  \begin{align}
    \grad{f}([C]) = (\Id_d - CC^{\top}S)S^{-1}\grad{\oll{f}}(C).
  \end{align}
\end{exm}

\begin{exm} \label{grad_product_grassmannians}
  Using Proposition \ref{grad_product} and the previous example, we obtain the following equation for the gradient of a function $f:\gr{N_1}{d_1} \x \gr{N_2}{d_2} \to \br$:
  \begin{equation} \label{gradient}
    \grad{f}([C_1], [C_2])
    = \qty(\proj_1S_1^{-1}\grad{\oll{f}}(\cdot, C_2)(C_1), \
    \proj_2S_2^{-1}\grad{\oll{f}}(C_1, \cdot)(C_2)),
\end{equation}
being $\proj_i = \Id_{d_i} - C_iC_i^{\top}S_i$ and $\grad\oll{f}$ the Euclidean gradient of the extension of the lifting.
\end{exm}

\begin{defi} \label{derivation}
  Given a vector field $V$, we define a function called \emph{derivation}\index{derivation} in the following way:
  \begin{align}
    \begin{split}
      DV:C^{\infty}(X) & \to C^{\infty}(X) \\
      f & \mapsto \qty(x \mapsto T_xf(V(x))).
    \end{split}
  \end{align}
  $DV$ is also a linear transformation over $\br$ and it satisfies the following Leibniz rule:
  \begin{equation}
    DV(fg) = fDV(g) + gDV(f).
  \end{equation}
\end{defi}

\begin{defi}
  Given a manifold $X$, an \emph{(affine) connection on $X$}\index{affine connection}\index{connection} is a function
  \begin{align}
    \begin{split}
      \nabla:\mf{X}(X) \x \mf{X}(X) & \to \mf{X}(X) \\
      (V_1, V_2) & \mapsto \nabla_{V_1}V_2
    \end{split}
  \end{align}
  that satisfies the following axioms for every $V_1, V_2, V_3 \in \mf{X}(X)$ and $f,g \in C^{\infty}(X)$:
  \begin{enumerate}

  \item $\nabla_{fV_1 + gV_2}V_3 = f\nabla_{V_1}V_3 + g\nabla_{V_2}V_3$.

  \item $\nabla_{V_1}(fV_2) = f\nabla_{V_1}V_2 + DV_1(f)V_2$.

  \end{enumerate}
  The vector field $\nabla_{V_1}V_2$ is called the \emph{covariant derivative of $V_2$ in the direction $V_1$}\index{covariant derivative}.
\end{defi}

\begin{obs}
  The axioms above are saying that the connection is $C^{\infty}(X)$-linear in the first entry and that it satisfies a generalized Leibniz rule in the second.
\end{obs}

\begin{prop} \label{independence_connection}
  Given three vector fields $V_1, V_2, V_3 \in \mf{X}(X)$ such that $V_1(x) = V_2(x)$, then $(\nabla_{V_1}V_3)(x) = (\nabla_{V_2}V_3)(x)$.
  In other words, the vector field $\nabla_{V_1}V_3$ computed at $x$ just depends on the vector $V_1(x)$ and not on the whole vector field $V_1$.
\end{prop}

\begin{exm}
  When $X$ is a vector space, the affine connection is given by the directional derivative of vector fields.
  So, if $V_1, V_2 \in \mf{X}(X)$, we define:
  \begin{equation}
    \nabla_{V_1}V_2(x) \ceq \lim_{t \to 0} \frac{V_2(x + tV_1(x)) - V_2(x)}{t}.
  \end{equation}
  Again, this is only possible because we identify $T_xX \cong X$, otherwise $x + tV_1(x)$ does not make any sense.
\end{exm}

\begin{exm}
  Let $X$, $Y$ be manifolds and $\nabla^X$, $\nabla^Y$ be connections on $X$ and $Y$, respectively.
  Observe that there exists a diffeomorphism between the tangent bundles $T(X \x Y)$ and $TX \x TY$ given by $((x, y), [\gamma]) \mapsto ((x, [\pi_X \circ \gamma]), (y, [\pi_Y \circ \gamma]))$.
  Consequently, we can assume vector fields $V \in \mf{X}(X \x Y)$ are given by $(V_X, V_Y)$, being $V_X:X \x Y \to TX$ and $V_Y:X \x Y \to TY$.
  Having said that, if we fix $(x, y) \in X \x Y$, we define a connection on $X \x Y$ as follows:
  \begin{multline}
    \qty(\nabla_{(V_X, V_Y)}^{X \x Y}(W_X, W_Y))(x, y)
    \ceq \bigg(\nabla_{V_X(x,y)}^XW_X(\cdot, y) + T_yW_X(x, \cdot)(V_Y(x, y)), \\
    T_xW_Y(\cdot, y)(V_X(x, y)) + \nabla_{V_Y(x,y)}^YW_Y(x, \cdot)\bigg).
  \end{multline}
  This is a little cumbersome, so, let us break into pieces.
  First, $W(\cdot, y):X \to TX$ is a vector field on $X$ and we are abusing notation and considering $V_X(x,y)$ as a tangent vector in $T_xX$.
  So, $\nabla_{V_X(x,y)}^X W_X(\cdot, y)$ is the usual covariant derivative on $X$.
  Now, $T_yW_X(x, \cdot)(V_Y(x, y))$ is trickier.
  Observe that $W_X(x, \cdot):Y \to T_xX$ can be seen a smooth function because $T_xX$ is a vector space and, consequently, a manifold.
  But that means we can compute the directional derivative of $W_X(x, \cdot)$, which is $T_yW_X(x, \cdot):T_yY \to T_{W_X(x, y)}T_xX$.
  Using the isomorphism $T_{W_X(x, y)}T_xX \cong T_xX$, we obtain the desired, since now $T_yW_X(x, \cdot)(V_Y(x, y))$ is indeed a vector on $T_xX$.
\end{exm}

\begin{defi}
  Given two vector fields $V_1, V_2 \in \mf{X}(X)$, we define the \emph{Lie bracket of $V_1$ and $V_2$}\index{Lie bracket} as follows:
  \begin{align}
    \begin{split}
      [V_1, V_2]:C^{\infty}(X) & \to C^{\infty}(X) \\
      f & \mapsto DV_1(DV_2(f)) - DV_2(DV_1(f)),
    \end{split}
  \end{align}
  being $DV_i$ the derivation induced by $V_i$.
\end{defi}

\begin{defi}
  Given a Riemannian manifold $(X, \braket{\cdot})$, there exists a unique connection $\nabla$ that satisfies the following two additional axioms for every $V_1, V_2, V_3 \in \mf{X}(X)$:
  \begin{enumerate}

  \item $[V_1, V_2] = D\nabla_{V_1}V_2 - D\nabla_{V_2}V_1$.

  \item $DV_1\qty(\braket{V_2}{V_3}) = \braket{\nabla_{V_1}V_2}{V_3}
    + \braket{V_2}{\nabla_{V_1}V_3}$.

  \end{enumerate}
  This unique connection is called \emph{Levi-Civita connection}\index{connection!Levi-Civita}\index{Levi-Civita connection}.
\end{defi}

\begin{obs}
  In the first axiom the notation $D\nabla_{V_i}V_j$ represents the derivation induced by the field $\nabla_{V_i}V_j$.
  In the second axiom $DV_1\qty(\braket{V_2}{V_3})$ should be interpreted as follows: $DV_1$ is a derivation and $\braket{V_2}{V_3} \in C^{\infty}(X)$ is the function $x \mapsto \braket{V_2(x)}{V_3(x)}_x$.
  The right-hand side is also interpreted as a smooth function: $x \mapsto \braket{\nabla_{V_1(x)}V_2}{V_3(x)}
  + \braket{V_2(x)}{\nabla_{V_1(x)}V_3}$.
\end{obs}

\begin{exm}
  The connection described for vector spaces is the Levi-Civita connection.
\end{exm}

\begin{prop}
  If $X$ is a Riemannian submanifold of $\br^d$, then its Levi-Civita connection is given by:
  \begin{equation}
    \qty(\nabla_{V_1}^XV_2)(x) = \proj_x^X\nabla_{V_1(x)}^{\br^d}\ol{V}_2,
  \end{equation}
  being $\ol{V}_2:\br^d \to \br^d$ an extension of $V_2$.
\end{prop}

\begin{prop} \label{connection_quotient}
  If $X/\sim$ is a Riemannian quotient manifold of $X$ and $X$ is a Riemannian submanifold of $\br^d$, then the Levi-Civita connection on $X/\sim$ is given by:
  \begin{equation}
    \qty(\nabla_{V_1}^{X/\sim}V_2)([x])
    = \proj_{x}^{\hor}\nabla_{\lift_x(V_1(x))}^{X}\ol{V}_2,
  \end{equation}
  being $\ol{V}_2$ a horizontal lift of $V_2$.
\end{prop}

\begin{defi}
  Let $X$ be a Riemannian manifold, $\nabla$ be its Levi-Civita connection and $f:X \to \br$ be a smooth function.
  The \emph{(Riemannian) Hessian of $f$}\index{Riemannian!Hessian}\index{Hessian} is the (linear) operator
  \begin{align}
    \begin{split}
      \hess{f}(x):T_xM & \to T_xM \\
      v & \mapsto \nabla_{v}\grad{f}.
    \end{split}
  \end{align}
  This operator is well-defined because of Proposition \ref{independence_connection}.
\end{defi}

\begin{exm}
  When $X$ is a vector space, the Hessian is just
  \begin{equation}
    \hess{f}(x)(v) = \lim_{t \to 0} \frac{\grad{f}(x + tv) - \grad{f}(x)}{t}.
  \end{equation}
  However, if we choose an orthonormal basis $B = \qty{v_1, \ldots, v_d}$ for $X$, we can represent the Hessian in the usual way as a matrix:
  \begin{equation}
    \hess{f}(x)(v) =
    \begin{bmatrix}
      \pdv{f}{v_1}{v_1}\qty(x) & \ldots & \pdv{f}{v_1}{v_d}\qty(x) \\
      \vdots                   & \ddots & \vdots                   \\
      \pdv{f}{v_d}{v_1}\qty(x) & \ldots & \pdv{f}{v_d}{v_d}\qty(x) \\
    \end{bmatrix}
    \begin{bmatrix}
      c_1 \\ \vdots \\ c_d
    \end{bmatrix},
  \end{equation}
  being $v = c_1v_1 + \ldots + c_dv_d$.
\end{exm}

\begin{prop}
  Let $f:X \x Y \to \br$ be a smooth function defined on the product of two Riemannian manifolds.
  Fixing $(x, y) \in X \x Y$, let $f_1$ be the function $f(\cdot, y)$, $f_2$ be the function $f(x, \cdot)$, and write $\grad{f}(x,y) = (G_X(x, y), G_Y(x, y))$.
  Then, the Riemannian Hessian of $f$ is given by:
  \begin{align}
    \begin{split}
      \hess{f}(x, y)&:T_xX \x T_yY \to T_xX \x T_yY \\
      (v, w) \mapsto \bigg(&\hess{f_1}(x)(v) + T_y(G_X(x, \cdot))(w), \ \\
                    & T_x(G_Y(\cdot, y))(v) + \hess{f_2}(y)(w)\bigg).
    \end{split}
  \end{align}
\end{prop}

\begin{exm} \label{hessian_product_grassmannians}
  Let us consider the case in which $X = \gr{N_1}{d_1}$ and $Y = \gr{N_2}{d_2}$.
  By definition,
  \begin{equation}
    \hess{f_1}([C_1])(\eta_1) = \nabla_{\eta_1}^{\gr{N_1}{d_1}}\grad{f_1}.
  \end{equation}
  Now, according to Proposition \ref{connection_quotient},
  \begin{equation}
    \nabla_{\eta_1}^{\gr{N_1}{d_1}}\grad{f_1} = \proj_{C_1}^{\hor}\nabla_{\lift_{C_1}(\eta_1)}^{\st{N_1}{d_1}}\grad{\ol{f}_1}.
  \end{equation}
  However,
  \begin{equation}
    \grad{\ol{f}_1}(\cdot)
    = (\proj_{(\cdot)}^{\hor} \circ S_1^{-1}\grad{\oll{f}_1})(\cdot).
  \end{equation}
  So, in the end we have:
  \begin{equation}
    \hess{f_1}([C_1])(\eta_1) = \proj_{C_1}^{\hor}
    \nabla_{\lift_{C_1}(\eta_1)}^{\mat{d_1}{N_1}}\qty(\proj_{(\cdot)}^{\hor} \circ S_1^{-1}\grad{\oll{f}_1}(\cdot)).
  \end{equation}
  Now we can compute this covariant derivative using Calculus because it is defined in the Euclidean space.
  Since $\proj_{C_1}^{\hor} = \Id_{d_1} - C_1C_1^{\top}S_1$ and $S_1^{-1}\grad{\oll{f}_1}$ are matrices, we can basically use the product rule for derivative and the final result is:
  {\small
    \begin{equation}
      \proj_{C_1}^{\hor}
      \nabla_{\lift_{C_1}(\eta_1)}^{\mat{d_1}{N_1}}\qty(\proj_{(\cdot)}^{\hor} \circ S_1^{-1}\grad{\oll{f}_1}(\cdot))
      = \proj_{C_1}^{\hor}\hess{\oll{f}_1}(C_1)\eta_1 - \eta_1C_1^{\top}\grad{\oll{f}_1}(C_1).
    \end{equation}
  }%
  Now let us compute $T_{C_2}\qty(G_{\gr{N_1}{d_1}}(C_1, \cdot))(\eta_2)$.
  Using Proposition \ref{grad_product} and Example \ref{grad_grassmannian}, we have
  \begin{equation}
    G_{\gr{N_1}{d_1}}(C_1, \cdot) = \proj_{C_1}^{\hor}S_1^{-1}\grad{\oll{f}}(C_1, \cdot).
  \end{equation}
  However, since we are computing the directional derivative with respect to $C_2$, in this case we do not need to differentiate the projection.
  Consequently,
  \begin{equation}
    T_{C_2}\qty(G_{\gr{N_1}{d_1}}(C_1, \cdot))(\eta_2)
    = \proj_{C_1}^{\hor}S_1^{-1}\hess{\oll{f}}(C_1, \cdot)(C_2)\eta_2.
  \end{equation}
  So, if we define $\hess_{11}{\oll{f}}(C_1, C_2) \ceq \hess{\oll{f}_1}(C_1)$, $\hess_{12}{\oll{f}}(C_1, C_2) \ceq \hess{\oll{f}}(C_1, \cdot)(C_2)$ and $\proj_i \ceq \proj_{C_i}^{\hor}$, the final result is:
{\footnotesize
  \begin{align} \label{full_hessian}
    \begin{split}
      \hess{f}([C_1], [C_2])
      &:T_{[C_1]}\gr{N_1}{d_1} \x T_{[C_2]}\gr{N_2}{d_2}
        \to T_{[C_1]}\gr{N_1}{d_1} \x T_{[C_2]}\gr{N_2}{d_2} \\
      (\eta_1, \eta_2)
      \mapsto \Big(
      & \proj_1S_1^{-1}\big(\hess_{11}\oll{f}(C_1, C_2)\eta_1
        + \hess_{12}\oll{f}(C_1, C_2)\eta_2\big)
        - \eta_1C_1^{\top}\grad\oll{f}_1(C_1, C_2), \\
      & \proj_2S_2^{-1}\big(\hess_{21}\oll{f}(C_1, C_2)\eta_1 
        + \hess_{22}\oll{f}(C_1, C_2)\eta_2\big)
        - \eta_2C_2^{\top}\grad\oll{f}_2(C_1, C_2)
        \Big),
    \end{split}
  \end{align}
}%
\end{exm}

\begin{defi}
  Given a manifold $X$ and a (smooth) curve $\gamma:(a,b) \to X$, a \emph{(smooth) vector field along $\gamma$}\index{vector field!along a curve} is a smooth function $V:(a,b) \to TX$ such that $V(t) \in T_{\gamma(t)}X$ for every $t \in (a,b)$.
  The set of all vector fields along $\gamma$ will be denoted by $\mf{X}(\gamma)$.
\end{defi}

\begin{obs}
  Again, $\mf{X}(\gamma)$ is actually a $C^{\infty}((a,b))$-module.
\end{obs}

\begin{defi}
  Given a Riemannian manifold $X$, a curve $\gamma:(a,b) \to X$ and the Levi-Civita connection $\nabla$, there exists a unique (linear) operator $\frac{\nabla}{\dd{t}}:\mf{X}(\gamma) \to \mf{X}(\gamma)$ such that the following axioms hold for every $V_1 \in \mf{X}(\gamma)$, $V_2 \in \mf{X}(X)$ and $f \in C^{\infty}((a,b))$:
  \begin{enumerate}

  \item $\frac{\nabla}{\dd{t}}(fV_1) = f'V_1 + f\frac{\nabla}{\dd{t}}(V_1)$.

  \item $\qty(\frac{\nabla}{\dd{t}}(V_2 \circ \gamma))(t) = \nabla_{\gamma'(t)}V_2$.

  \end{enumerate}
  The operator $\frac{\nabla}{\dd{t}}$ is called \emph{covariant derivative along $\gamma$}\index{covariant derivative!along a curve}.
\end{defi}

\begin{exm} \label{cov_vec}
  When $X$ is a vector space, we can compute the covariant derivative along a curve as a regular derivative, \ie, given $V \in \mf{X}(\gamma)$, we have
  \begin{equation}
    \frac{\nabla}{\dd{t}}(V)(t) = V'(t) \ceq \lim_{h \to 0} \frac{V(t + h) - V(t)}{h}.
  \end{equation}
\end{exm}

\begin{prop}
  If $X$ is a Riemannian submanifold of $\br^d$, the covariant derivative of $V$ along $\gamma$ is given by:
  \begin{equation}
    \frac{\nabla}{\dd{t}}(V)(t) = \proj_{\gamma(t)}^XV'(t),
  \end{equation}
  being $\proj_{\gamma(t)}^X$ the projection to $T_{\gamma(t)}X$.
\end{prop}

\begin{prop} \label{cov_quotient}
  If $X/\sim$ is a Riemannian quotient manifold of $X$ and $X$ is a Riemannian submanifold of $\br^d$, the covariant derivative of $V$ along $\gamma$ is given by:
  \begin{equation}
    \frac{\nabla}{\dd{t}}(V)(t) = \proj_{\gamma(t)}^{\hor}\ol{V}'(t),
  \end{equation}
  being $\proj_{\gamma(t)}^{\hor}$ the projection to $H_{\gamma(t)}$ and $\ol{V}$ the horizontal lift of $V$ along $\gamma$, \ie, $\ol{V}(t) \ceq \lift_{\gamma(t)}(V(t))$.
\end{prop}

\begin{defi}
  Given a curve $\gamma:(a,b) \to X$, we define the \emph{acceleration of $\gamma$}\index{acceleration} to be $\frac{\nabla}{\dd{t}}(\gamma')$.
\end{defi}

\begin{defi}
  A curve $\gamma:(a,b) \to X$ is said to be a \emph{geodesic}\index{geodesic} if $\frac{\nabla}{\dd{t}}(\gamma')(t) = 0$ for every $t \in (a,b)$.
\end{defi}

\begin{defi}
  A vector field $V$ along $\gamma:(a,b) \to X$ is said to be \emph{parallel}\index{vector field!parallel} if $\frac{\nabla}{dt}(V)(t) = 0$ for every $t \in (a,b)$.
\end{defi}

\begin{thm}
  Given a manifold $X$, a connection $\nabla$ on $X$, a curve $\gamma:(a,b) \to X$, $t_0 \in (a,b)$ and $v \in T_{\gamma(t_0)}X$, there exists a unique parallel vector field $V \in \mf{X}(\gamma)$ such that $V(t_0) = v$.
\end{thm}

\begin{defi}
  Given a curve $\gamma:(a,b) \to X$ and $t_0, t_1 \in (a,b)$, we define the \emph{parallel transport along $\gamma$}\index{parallel transport} to be the linear transformation defined by:
  \begin{align}
    \begin{split}
      \pt_{t_0 \to t_1}^{\gamma}:T_{\gamma(t_0)}X & \to T_{\gamma(t_1)}X \\
      v & \mapsto V(t_1),
    \end{split}
  \end{align}
  being $V \in \mf{X}(\gamma)$ the unique parallel vector field along $\gamma$ such that $V(t_0) = v$ (which exists by the previous theorem).
\end{defi}

\begin{obs}
  From now on we will assume $[0,1] \subset (a,b)$, $t_0 = 0$ and $t_1 = 1$.
\end{obs}

\begin{exm}
  Geodesics on vector spaces are straight lines.
  Indeed, according to Example \ref{cov_vec}, $\frac{\nabla}{\dd{t}}\gamma' = \gamma''$.
  Consequently, if we assume $\gamma(0) = x$ and $\gamma'(0) = v$, we have that $\gamma(t) = x + tv$ by the Picard--Lindelöf Theorem (also known as Existence and Uniqueness of Solutions to Ordinary Differential Equations).
  Now, the parallel transport is just the identity function because the vector field obtained in the previous theorem is unique and the constant vector field obviously satisfies $V(0) = v$.
\end{exm}

\begin{exm} \label{geodesic_grassmannian}
  Now let us compute geodesics and the parallel transport for the Grassmannian.
  According to Proposition \ref{cov_quotient}, the covariant derivative of a vector field $W$ along a curve $\gamma$ in the Grassmannian $\gr{N}{d}$ is given by:
  \begin{equation}
    \qty(\frac{\nabla}{\dd{t}}W)(t) = \proj_{\gamma(t)}^{\hor}\ol{W}'(t).
  \end{equation}
  Therefore, this is $0$ (that is, $W$ is parallel) if, and only if, $\ol{W}'(t) \in V_{\gamma(t)}$ (recall that $V_{\gamma(t)} = H_{\gamma(t)}^{\perp}$), which means
  \begin{equation}
    \ol{W}'(t) = \gamma(t)\nu(t)
  \end{equation}
  for $\nu(t)$ antisymmetric.
  However, $\ol{W}(t) \in H_{\gamma(t)}$ for every $t$, which means
  \begin{equation}
    \gamma(t)^{\top}S\ol{W}(t) = 0.
  \end{equation}
  Differentiating this, we obtain
  \begin{equation}
    \gamma'(t)^{\top}S\ol{W}(t) + \gamma(t)^{\top}S\ol{W}'(t) = 0
  \end{equation}
  and, consequently,
  \begin{equation}
    \nu(t) = -\gamma'(t)^{\top}S\ol{W}(t)
  \end{equation}
  because $\gamma(t)^{\top}S\gamma(t) = \Id_N$.
  In other words, $\ol{W}$ is parallel to $\gamma$ \tiff
  \begin{equation}
    \ol{W}'(t) = -\gamma(t)\gamma'(t)^{\top}S\ol{W}(t).
  \end{equation}
  Now, if $\ol{W} = \gamma'$, this amounts to
  \begin{equation}
    \gamma''(t) = -\gamma(t)\gamma'(t)^{\top}S\gamma'(t).
  \end{equation}
  Therefore, assuming that $\gamma$ is a geodesic, the above holds and $\gamma(t)^{\top}S\ol{W}(t) = \gamma(t)^{\top}S\gamma'(t) = 0$ because $\ol{W}(t) \in H_{\gamma(t)}$.
  But that means $\gamma'(t)^{\top}S\ol{W}(t)$ is constant because
  \begin{multline}
    \dv{t}\gamma'(t)^{\top}S\ol{W}(t) = \gamma''(t)^{\top}S\ol{W}(t) + \gamma'(t)^{\top}S\ol{W}'(t)
    = -\gamma'(t)^{\top}S\gamma'(t)\ub{\gamma(t)^{\top}S\ol{W}(t)}_{0} \\
    - \ub{\gamma'(t)^{\top}S\gamma(t)}_{0}\gamma'(t)^{\top}S\ol{W}(t).
  \end{multline}
  Consequently, the ODE we need to solve for the geodesic is
  \begin{equation}
    \gamma''(t) = -\gamma(t)\gamma'(0)^{\top}S\gamma'(0) = -\gamma(t)\kappa,
  \end{equation}
  which is an harmonic oscillator.
  Now, suppose we have an invertible matrix $O$ such that $O^{\top}SO = \Id_d$.
  Then $S = O^{-\top}O^{-1}$.
  So, let $O^{-1}\gamma'(0) = UDV$ be a \emph{thin Singular Value Decomposition}\footnote{Assuming $N \leq d$, which is always the case for us, $U \in \mat{d}{N}$ and $D, V \in \mat{N}{N}$ are matrices such that $D$ is diagonal and $U^{\top}U = V^{\top}V = VV^{\top} = \Id_N$. The matrix $U$ is also known as semi-orthogonal because it is not true that $UU^{\top} = \Id_d$ if $N < d$.} (see \cite[Chapter 2]{golub2013}) and observe that
  \begin{equation}
    \gamma(t) = \qty(\gamma(0)V^{\top}\cos(tD) + OU\sin(tD))V
  \end{equation}
  satisfies the equation above.
  Indeed, using that $V^{\top} = V^{-1}$ and that
  \begin{align}
    \begin{split}
      \gamma'(0)^{\top}S\gamma'(0)
      & = \gamma'(0)^{\top}O^{-\top}O^{-1}\gamma'(0) \\
      & = (O^{-1}\gamma'(0))^{\top}O^{-1}\gamma'(0) \\
      & = (UDV)^{\top}UDV \\
      & = V^{\top}DU^{\top}UDV \\
      & = V^{\top}D^2V,
    \end{split}
  \end{align}
we have
  \begin{align}
    \begin{split}
      \gamma''(t)
      & = \qty(-\gamma(0)V^{\top}\cos(tD)D^2 - OU\sin(tD)D^2)V \\
      & = -\qty(\gamma(0)V^{\top}\cos(tD) + OU\sin(tD))D^2V \\
      & = -\qty(\gamma(0)V^{\top}\cos(tD) + OU\sin(tD))VV^{-1}D^2V \\
      & = -\qty(\gamma(0)V^{\top}\cos(tD) + OU\sin(tD))VV^{\top}D^2V \\
      & = -\gamma(t)\gamma'(0)^{\top}S\gamma'(0).
    \end{split}
  \end{align}
  So, that is it, we obtained a closed formula for the geodesic.

  Now, the ODE we need to solve for the parallel transport along a geodesic is
  \begin{align}
    \begin{split}
      \ol{W}'(t)
      & = -\gamma(t)\gamma'(0)^{\top}S\ol{W}(0) \\
      & = -\qty(\gamma(0)V^{\top}\cos(tD) + OU\sin(tD))V\gamma'(0)^{\top}S\ol{W}(0).
    \end{split}
  \end{align}
  Integrating this, we obtain:
  \begin{multline}
    \ol{W}(t) = \qty(-\gamma(0)V^{\top}\sin(tD)D^{-1}
    + OU\cos(tD)D^{-1})V\gamma'(0)^{\top}S\ol{W}(0) + \kappa \\
    = \qty(-\gamma(0)V^{\top}\sin(tD)D^{-1}
    + OU\cos(tD)D^{-1})V(O^{-1}\gamma'(0))^{\top}O^{-1}\ol{W}(0) + \kappa \\
    = \qty(-\gamma(0)V^{\top}\sin(tD)D^{-1}
    + OU\cos(tD)D^{-1})DU^{\top}O^{-1}\ol{W}(0) + \kappa \\
    = \qty(-\gamma(0)V^{\top}\sin(tD)
    + OU\cos(tD))U^{\top}O^{-1}\ol{W}(0) + \kappa.
  \end{multline}
  But $\ol{W}(0) = \ol{W}(0)$, therefore,
  \begin{multline}
    \ol{W}(0) = \qty(-\gamma(0)V^{\top}\sin(0 \cdot D)
    + OU\cos(0 \cdot D))U^{\top}O^{-1}\ol{W}(0) + \kappa \\
    = OUU^{\top}O^{-1}\ol{W}(0) + \kappa.
  \end{multline}
  Since the solution is unique, if we choose $\kappa = \qty(\Id - OUU^{\top}O^{-1})\ol{W}(0)$, we obtain the result that the parallel transport of a vector $\eta$ along a geodesic is given by:
  \begin{equation}
    P_{0 \to 1}^{\gamma}(\eta)
    = \qty(\qty(-\gamma(0)V^{\top}\sin(tD)
    + OU\cos(tD))U^{\top}O^{-1}
    + \qty(\Id - OUU^{\top}O^{-1}))\eta
  \end{equation}
  When $\eta = \gamma'(0)$, the transport becomes
  \begin{equation}
    P_{0 \to 1}^{\gamma}(\gamma'(0))
    = \qty(-\gamma(0)V^{\top}\sin(tD)
    + OU\cos(tD))DV.
  \end{equation}
\end{exm}

\begin{defi} \label{retraction}
  A \emph{retraction on $X$}\index{retraction} is a smooth function $R:TX \to X$ such that the curve $\rho(t) \ceq R(x, t[\gamma])$ satisfies $[\gamma] = [\rho]$ for every $(x, [\gamma]) \in TX$.
\end{defi}

\begin{thm}
  Given a Riemannian manifold $X$ and $(x, v) \in TX$, there exists a unique \emph{maximal} geodesic $\gamma_v:(a,b) \to X$ such that (1) $0 \in (a,b)$, (2) $\gamma_v(0) = x$ and (3) $\gamma_v'(0) = v$.
  By maximal it is meant that, if $\rho_v:(c,d) \to X$ also satisfies these three properties, then $(c,d) \subset (a,b)$.
\end{thm}

\begin{defi}
  Let
  \begin{equation}
    \mc{O} \ceq \qty{(x, v) \in TX : \text{$\gamma_v$ is defined on an interval containing $[0,1]$}}.
  \end{equation}
  The \emph{(Riemannian) exponential}\index{Riemannian!exponential}\index{Riemannian exponential} is the function $\exp:\mc{O} \to X$ defined by
  \begin{equation}
    \exp(x,tv) = \gamma_v(t).
  \end{equation}
  When $x$ is fixed, we will denote the function $\exp(x, \cdot)$ by $\exp_x$.
\end{defi}

\begin{defi}
  A Riemannian manifold $X$ is said to \emph{complete}\index{complete manifold} if for every $(x,v) \in TX$ we have $\gamma_v$ defined on the whole $\br$.
\end{defi}

\begin{obs}
  If the manifold $X$ is complete, then $\mc{O} = TX$.
\end{obs}

\begin{prop}
  The Riemannian exponential is a retraction.
\end{prop}

\begin{exm} \label{exp_euclidean}
  If $X$ is a vector space, the Riemannian exponential is given by
  \begin{equation}
    \exp(x,tv) = \exp_x(v) = x + tv.
  \end{equation}
\end{exm}

\begin{exm} \label{exp_grassmannian}
  In the case of the Grassmannian $\gr{N}{d}$, the exponential is given by
  \begin{equation}
    \exp([C], t\eta) = \exp_{[C]}(t\eta) = \qty(CV^{\top}\cos(tD) + OU\sin(tD))V,
  \end{equation}
  being $O^{-1}\eta = UDV$ a thin Singular Value Decomposition.
  See Example \ref{geodesic_grassmannian} for the details.
\end{exm}

\begin{prop} \label{exp_product}
  The Riemannian exponential in the product of manifolds $X \x Y$ is given by
  \begin{equation}
    \exp^{X \x Y}((x, y), t(v, w)) = (\exp^X(x, tv), \exp^Y(y, tw)).
  \end{equation}
\end{prop}

To conclude, let us state without proof that the Grassmannian is a complete manifold.
This fact is important because one of the hypothesis of the theorems that assert the algorithms we will implement converge is that the manifold should be complete.

\begin{thm}[Hopf--Rinow]
  A compact and connected Riemannian manifold is complete.
\end{thm}

\begin{obs}
  This is actually a corollary of Hopf--Rinow that is more suited to our purposes.
  If the reader is interested in the ``original'' theorem, see \cite[Theorem 6.19]{lee2018}.
\end{obs}

\begin{obs}
  Connectedness is a property of topological spaces that says we cannot break the space into two or more disjoint pieces (if the space is connected, of course).
  A geometric example of a space that is \emph{not} connected is two disjoint open balls in the Euclidean plane, such as $B((-2,0), 1) \cup B((2,0), 1)$.
  The Grassmannian is connected, but we do not know any elementary proof of this fact and we refer the interested reader to \cite[Section 2.4]{piccione2009} for a proof.
\end{obs}

\begin{cor} \label{grassmannian_complete}
  The Grassmannian and the product of Grassmannians is complete.
\end{cor}

\begin{proof}
  It follows from Hopf--Rinow together with the fact that the Grassmannian and its finite products are connected and compact spaces (Corollary \ref{product_compact}).
\end{proof}


\chapter{Quantum Mechanics} \label{qm}

\epigraph{Shut up and calculate.}{David Mermin}

\section{Introduction}

The goal of this chapter is to introduce Quantum Mechanics to the reader that may not be familiar with this field and then explain in a detailed manner how the Hartree--Fock Method can be thought of as a Riemannian optimization problem.
We will not pursue a very rigorous approach to Quantum Mechanics (at least from a mathematician's perspective) because most of the difficult problems of the field can be ignored in our case since we will end up working with finite-dimensional vector spaces.
This chapter required many references and they are usually credited along the text, but for the Theory section (\ref{qm_theory}) the main references are \cite{tannoudji2020_1, tannoudji2020_2, hall2013, kostrikin1997, takhtajan2008} and for the Practice section (\ref{qm_practice}) the main reference is \cite{szabo1996}.
It should be noted, though, that the Practice section is one of the contributions of this work because we do not know a reference that works out the details of the connection between the Grassmannian and Slater determinants.

\subsection{History}

In high school we learn Newton's laws of motion and how these laws model the behavior of macroscopic objects.
However, there are some limits in which Newton's laws fail, such as when the objects are moving with speed close to the speed of light, when the objects are too massive, and when objects are very small.
In the first case, the theory that explains how objects behave is called Special Relativity, in the second case General Relativity, and in the last case Quantum Mechanics.
So, as already said, in the present work we will explore Quantum Mechanics and it is interesting to start with a brief overview of the history of this field in order to later understand the motivation for some of the mathematical tools used.
If the reader is interested in the details of the history, see \cite{mehra1982}.

Arguably, the history starts with Max Planck in 1900 when he tried to explain black body radiation, that is, the energy emitted by an idealized object that absorbs every incident light, independent of the wavelength of the light.\footnote{Simplifying things a little bit, he was trying to describe the change of color of an object when you heat it. First, the object has its natural color; then, when you heat it, it starts to emit an orange-reddish color; in the end, if you increase the temperature enough, it will emit a white-blueish color. So, the question was: what is the formula that describes this change of color?}
This was a very important topic back then because some of the most important and accurate scientific theories were closely related to this phenomenon, such as electromagnetism and thermodynamics.
So, Planck was trying to obtain a formula for the energy emitted by a black body and the story of how he achieved this is quite interesting: at first, he did not believe in the existence of atoms and he had discovered a formula for black body radiation which turned out to be wrong after some experimental results came out.
Then, analyzing these results, he realized he had to come up with a new formula and, to obtain this new formula, he studied the work of L. E. Boltzmann, which assumed the existence of atoms.
As it turned out, Planck obtained a new formula using Boltzmann's work, which is known as \emph{Planck's radiation law}, and this formula explained the black body radiation and agreed with the experimental results.
So, after this achievement, Planck started to believe in the atomic theory.
Having said that, though, the physical interpretation behind his formula was not completely clear to him:
\begin{quote}
  I also knew the formula that expresses the energy distribution in the normal spectrum. A theoretical interpretation therefore had to be found at any cost, no matter how high. It was clear to me that classical physics could offer no solution to this problem, and would have meant that all energy would eventually transfer from matter to radiation. (...) This approach was opened to me by maintaining the two laws of thermodynamics. The two laws, it seems to me, must be upheld under all circumstances. For the rest, I was ready to sacrifice every one of my previous convictions about physical laws. (...) [One] finds that the continuous loss of energy into radiation can be prevented by assuming that energy is forced at the outset to remain together in certain quanta. This was purely a formal assumption and I really did not give it much thought except that no matter what the cost, I must bring about a positive result. (Planck to Wood, 7 October 1931)\footnote{This letter is in the Archive for the History of Quantum Physics, Microfilm 66, 5 (ref. 58). A translation of the entire letter can be found in \cite{hermann1971}.}
\end{quote}
So, he knew the mathematical aspects and the experimental results behind his formula, but not entirely the physical aspects.
Actually, to this day it is open to debate the extent to which he knew a good physical interpretation and the consequences of his work while he was doing it (not in retrospect), see \cite{gearhart2002}.
Knowing the exact explanation behind his formula or not, Planck introduced a key component while deriving it: the discretization of the energy, that is, in his formula the energy exchanged between the black body and the incident or emitted light could not have any value, it could only be multiple of a constant that is now called \emph{Planck's constant}.
This discretized quantity absorbed or emitted he called \emph{elementary quantum of action} and the process of discretization was later called \emph{quantization}.
However, it was not until the work of Einstein in 1905 that people fully realized the implications of quantization.
Actually, even Einstein did not fully understand Planck's work, as he says in \cite[On the Theory of Light Production and Light Absorption]{einstein1989}:
\begin{quote}
  At that time it seemed to me that in a certain respect Planck's theory of radiation constituted a counterpart to my work. New considerations (...) showed me, however, that the theoretical foundation on which Mr. Planck's radiation theory is based differs from the one that would emerge from Maxwell's theory and the theory of electrons, precisely because Planck's theory makes implicit use of the aforementioned hypothesis of light quanta.
\end{quote}
So, in 1905 Einstein introduced new ideas of quantization that paved the way of quantum theory because (1) experiments proved that his hypothesis of quantization were correct, and (2) it explained better than the other proposals phenomena such as photoluminescence, the photoelectric effect and the ionization of gases by ultraviolet light.
What Einstein did, specifically? He quantized light itself, not just the energy exchanged, as did Planck.
To explain the photoelectric effect,\footnote{The photoelectric effect is the observation that electrons eject from a metal plate when light hits the plate and that the energy of the ejected electrons is not proportional to the intensity of the incident light, as Classical Mechanics explained, but to the frequency.} for example, he says in \cite[On a Heuristic Point of View Concerning the Production and Transformation of Light]{einstein1989}:
\begin{quote}
  According to the conception that the exciting light consists of energy quanta of energy $(R/N)\beta\nu$, the production of cathode rays [\ie, streams of electrons] by light can be conceived in the following way. The body's surface layer is penetrated by energy quanta whose energy is converted at least partially to kinetic energy of electrons. The simplest possibility is that a light quantum transfers its entire energy to a single electron; we will assume that this can occur. However, we will not exclude the possibility that the electrons absorb only a part of the energy of the light quanta.
\end{quote}
So, what he called \emph{light quantum} later became known as \emph{photon} and this paper made people realize that this new viewpoint that considers light as made of particles was completely different from the physical theories known back then.

To conclude this brief historical overview, we have to talk about \emph{wave-particle duality}.
Before Einstein, some experiments showed that light behaved as waves and Maxwell's equations explained this behavior quite well.
However, other experiments required Einstein's proposal that light is made of particles and this is one of the reasons why Einstein's idea of quantizing light was revolutionary (Einstein was awarded the Nobel Prize for this idea): the theories describing particles and waves were completely different.
As Einstein says in \cite{einstein1938}:
\begin{quote}
  It seems as though we must use sometimes the one theory and sometimes the other, while at times we may use either. We are faced with a new kind of difficulty. We have two contradictory pictures of reality; separately neither of them fully explains the phenomena of light, but together they do.
\end{quote}
Now, despite the awkwardness of this contradictory behavior, Einstein came up with some formulas to describe it and in 1924 de Broglie went even further to formulate the \emph{de Broglie hypothesis}, which asserts that not just light, but actually all matter behave as a wave and as a particle at the same time.
This aspect of nature, after being verified experimentally by Millikan \cite{millikan1914} in the case of photons and by the \emph{Davisson--Germer experiment} in the case of electrons, became known as the wave-particle duality and it is this duality that motivates a very important mathematical object behind Quantum Mechanics.

\section{Theory} \label{qm_theory}

\subsection{Axioms}

For our purposes, all we need to understand Quantum Mechanics is the Dirac--von Neumann axioms:
\begin{description} \label{axioms}

\item[First Axiom.] To every quantum system there is a complex Hilbert space associated, which we will usually denote by $H$ or $(H, \braket{\cdot}{\cdot})$ to emphasize the inner product.
  The projective space of $H$, which we will denote by $\bp{H}$, is called the \emph{state space}\index{space!state} of the system.

\item[Second Axiom.] The \emph{composite}\index{system!composite} of two quantum systems is described by the completion of the tensor product of the respective Hilbert spaces.
  In other words, if we have two quantum systems with associated Hilbert spaces $\qty(H_A, \braket{\cdot}{\cdot}_A)$ and $\qty(H_B, \braket{\cdot}{\cdot}_B)$, then the space associated to the composite system is $H_A \wten H_B$ with the inner product given by the multiplication of the inner products, \ie, $\braket{\phi_A \otimes \phi_B}{\psi_A \otimes \psi_B}_{AB} \ceq \braket{\phi_A}{\psi_A}_A\braket{\phi_B}{\psi_B}_B$.

\item[Third Axiom.] The \emph{states}\index{state} of a quantum system (associated to the Hilbert space) $H$ are the elements of the projective space of $H$.
  We will denote states by \emph{kets}, \ie, given $\psi \in H$, the state associated to $\psi$ is $\ket{\psi} \in \bp{H}$.
  In some contexts they are also called \emph{pure states}.

\item[Fourth Axiom.] The \emph{observables}\index{observable} of a quantum system $H$ are the self-adjoint operators of $H$.
  We will denote observables using capital letters with circumflex on top such as $\what{O}$.

\item[Fifth Axiom.] Given a quantum system $H$, a state $\ket{\psi}$ and an observable $\what{O}$, the \emph{expected value}\index{expected value} or \emph{expectation value} of this observable in the state $\ket{\psi}$ is given by the number $\frac{\mel{\psi}{\what{O}}{\psi}}{\braket{\psi}{\psi}}$.

\end{description}

\begin{obs}
  Before we move forward to explain the axioms above and give some examples, it should be said that we will not provide the most rigorous or broad formulation of Quantum Mechanics in this chapter.
  If the reader is interested in such a formulation, some good references are \cite{hall2013, takhtajan2008}.
  So, that means we will not dive in the technical aspects of non-separable Hilbert spaces, unbounded operators and the measure-theoretical description of how to compute the value of an observable in a certain state because in practice we will only work with finite-dimensional spaces, the time-independent Schrödinger equation and pure states.
  That means the axioms listed above are enough for our purposes.
  It is also worth mentioning that the word \emph{postulate} is more commonly used in the literature, but here we will use \emph{axiom} instead.
\end{obs}

Now let us explain the axioms.
By \emph{quantum system}, one usually means a collection of molecules, atoms, subatomic particles, antiparticles, quasiparticles, spin etc.
\emph{Composite systems} are also examples of quantum systems and we can see molecules as a composite system of atoms, for example.
A good rule of thumb for what can be considered a quantum system is: if the object is small, then some quantum effects should be taken into account in the interactions of this object with other objects.
A rationale for this is that in real life we can only observe an object by letting it interact with other objects and these interactions always disturbs the object being studied.
So, if the disturbance caused by observations can be neglected, the object is considered big, otherwise it is small.
This phenomenological assumption was used by Dirac in \cite{dirac1958} to explain what are the objects of study of Quantum Mechanics.
In practice it is hard to tell when something is small or big and experimental results should be the ultimate guide, which means that if quantum effects are not being taken into account in the description of a phenomenon and the results obtained are not good, then quantum considerations may help.
Moving forward, by \emph{state} we roughly mean the most accurate knowledge we have of how a quantum system was prepared.
Preparation here has the same meaning used in experiments, specially because quantum systems are so susceptible to changes that a preparation of the system is usually required to measure a property of it.
To conclude, we theoretically compute properties of a system using \emph{observables} and the last axiom.
Essentially, we associate a quantity we want to measure with a self-adjoint operator and then use the last axiom to obtain the value represented by the state of the quantum system with respect to the observable we want to measure.
Examples of observables are the momentum, the position and the energy of an object.
We will define them later.

\begin{obs}
  The last axiom explains why a state is an element of the projective space instead of being an element of the Hilbert space itself.
  Indeed, if we have an observable $\what{O}$ and a system in the state $\ket{\psi}$ or $\ket{\lambda\psi}$, being $\lambda \in \bc^\x$, then both states have the same expected value:
  \begin{equation}
    \frac{
      \mel{\lambda\psi}{\what{O}}{\lambda\psi}
    }{
      \braket{\lambda\psi}{\lambda\psi}
    }
    = \frac{
      \ol{\lambda}{\lambda}\mel{\psi}{\what{O}}{\psi}
    }{
      \ol{\lambda}{\lambda}\braket{\psi}{\psi}
    }
    = \frac{\mel{\psi}{\what{O}}{\psi}}{\braket{\psi}{\psi}}.
  \end{equation}
\end{obs}

\begin{obs} \label{obs_normalized}
  We will usually ignore the equivalence class notation of a state and just write it as a ket because in practice we only consider normalized states, \ie, the representative of the class is a normalized vector.
  This consideration gives us a natural representative of the class up to a constant $e^{i\theta}$, $\theta \in [0,2\pi)$.
  This constant is called \emph{phase factor} and it does not influence the computation of the expected value, as we just showed.
\end{obs}

Now let us give a concrete description of the objects defined in the axioms.
In the historical subsection, we said that every object in the universe behaves as a particle and as a wave at the same time.
In mathematical terms, that means we can describe every object in the universe using the following Hilbert space:
\begin{equation}
  L^2(\br^n) \ceq \qty{\psi:\br^n \to \bc : \intl_{\br^n} \abs{\psi(\vb{r})}^2 \dd\vb{r} < \infty}.
\end{equation}
This space encodes the wave-particle duality of an object in the following way: the wave behavior comes from the fact that the time evolution of a state $\ket{\psi(\vb{r}, t)}$ representing the object is described by the Schrödinger equation, which is an equation whose solutions have wave characteristics, and the particle behavior is given by the fact that the probability of finding the object described by the state $\ket{\psi(\vb{r})}$ in some region of space $U \subset \br^3$ is $\int_U \abs{\psi(\vb{r})}^2 \dd\vb{r}$.
In other words, objects evolve in time as waves and at the same time can be localized (probabilistically) in space as a particle.
This is a very broad formulation, which means that in practice this may not be the exact space in which the observables act on and we will see a counterexample in a moment, but it is a representative formulation, at least for our purposes, because this space is the role model we will always keep in mind in the present work.

There is another important Hilbert space used in practice, which is $\bc^m$.
This is the space that describes the spin of objects and when one wants to describe an object and its spin at the same time, the Hilbert space that should be considered is $L^2(\br^n) \otimes \bc^m$, \ie, it is a composite system.
So, now that we know the two main state spaces we will use, let us talk about wave functions and observables in order to compute things.

In Observation \ref{obs_normalized} we said we will consider only normalized vectors to be representatives of a state and there is a good reason for that, as we will see.
First recall that a state $\ket{\psi}$ is an abstract element of $\bp{H}$.
However, in practice we want a quantitative description of the state and we can obtain this description in the following way: suppose we are interested in describing the state $\ket{\psi}$ with respect to some observables $\what{O}_1, \ldots, \what{O}_n$.
Suppose also that there exists a common orthonormal eigenbasis $\qty{\psi_i}_{i \in I}$ for these observables, \ie, $\what{O}_k\psi_i = \lambda_{k,i}\psi_i$ for every $i \in I$ and $k \in \qty{1, \ldots, n}$.
Then notice that we can write any representative of the state $\ket{\psi}$ as a linear combination given by: $\psi = \sum_{i \in I} \wtil{\psi}(i)\psi_i$.
However, since the basis is assumed to be orthonormal, we have $\wtil{\psi}(i) = \braket{\psi_i}{\psi}$ and from this observation we obtain the following definition:
\begin{defi} \label{wave_function}
  The function $\wtil{\psi}:I \to \bc$ defined by $\wtil{\psi}(i) \ceq \braket{\psi_i}{\psi}$ is called \emph{wave function (of the state $\ket{\psi}$ with respect to the observables $\what{O}_1, \ldots, \what{O}_n$)}\index{wave function}.
  In other words, the wave function is a choice of coordinates for a state with respect to some observables that allows us to move quantitatively from $\bp{H}$ to $H$.
\end{defi}
There is also a probabilistic interpretation behind the wave function, which is the following: the probability of measuring the value $\lambda_{k,i}$ when we compute the expected value of the state $\ket{\psi}$ with respect to the observable $\what{O}_k$ is $\abs{\wtil{\psi}(i)}^2$.
Now, since, the sum of the probabilities should be $1$ and since
\begin{equation}
  \braket{\psi}{\psi}
  = \sum_{i, j \in I} \ol{\wtil{\psi}(i)}\wtil{\psi}(j)\ub{\braket{\psi_i}{\psi_j}}_{\delta_{ij}}
  = \sum_{i \in I} \abs{\wtil{\psi}(i)}^2,
\end{equation}
we have a motivation to consider only normalized states.
We will also give a physical motivation in a moment.
This probabilistic interpretation comes from experiments, \ie, it was observed that if we have a system in the state $\ket{\psi}$ and we want to measure some property of it represented by an observable $\what{O}$, then the only values obtained by our measuring apparatus are the eigenvalues of $\what{O}$ (which are real values since $\what{O}$ is self-adjoint).
So, immediately after the measurement, the initial state $\ket{\psi}$ of the system will change to one of the eigenvectors of $\what{O}$, say $\ket{\psi_i}$, and this eigenvector is associated to the the eigenvalue measured by the apparatus, say $\lambda_i$.
This phenomenon is called \emph{wave function collapse} and observe that it is compatible with the last axiom because, if we interpret the expected value as the weighted average of the possible outcomes with weights given by their probabilities, then:
\begin{equation}
  \mel{\psi}{\what{O}}{\psi}
  = \sum_{i, j \in I} \ol{\psi(i)}\psi(j)\mel{\psi_i}{\what{O}}{\psi_j}
  = \sum_{i, j \in I}
  \ol{\psi(i)}\psi(j)\lambda_j\ub{\braket{\psi_i}{\psi_j}}_{\delta_{ij}}
  = \sum_{i \in I} \ub{\abs{\psi(i)}^2}_{\text{prob.}}\ob{\lambda_i}^{\text{out.}}.
\end{equation}
Another interpretation of the expected value is that if we always prepare our system in the same (normalized) state $\ket{\psi}$ and always measure the same property of it described by the observable $\what{O}$, then the expected value after a large number of measurements is $\mel{\psi}{\what{O}}{\psi}$.
However, and this is very important to the theory, in each measurement we will only obtain as a result one of the eigenvalues of $\what{O}$.
The mathematical formulation that explains this phenomenon of measuring only discrete values is one of the main achievements of Quantum Mechanics and, as an example, the quantization of the energy can be explained by the fact that we only measure the discrete spectrum of the observable associated to the energy.

\begin{exm}
  The goal of this example is to give an interpretation for the correspondence between the representative $\psi \in L^2(\br^n)$ of a state $\ket{\psi}$ and its wave function with respect to the position observables.
  This example will be somewhat technical and the reader can skip it, but keep in mind that from now on wave functions will always be considered with respect to the position observables.
  In other words, we will always assume that the function $\psi(\vb{r})$ encodes the information about an object, for example an electron, with respect to the position $\vb{r} \in \br^n$ of this object in space.
  Also, note that this gives us a physical interpretation for considering only normalized vectors, since now $\braket{\psi}{\psi}$ represents the probability of finding the object somewhere in space, which should be $1$.
  The last disclaimer is that we will not prove the claims made here, the goal is just to explain how wave functions are usually interpreted.
  The reader interested in the formal aspects of this example and the next can take a look at \cite{madrid2001, hall2013}.

  The \emph{$i$-th position operator} $\what{Q}_i$ is defined as $\psi(\vb{r}) \mapsto r_i\psi(\vb{r})$.
  However, there are two problems with this operator: (1) it is not defined in the whole $L^2(\br^n)$ because the function $r_i\psi(\vb{r})$ may not be in $L^2(\br^n)$; and (2) it does not have eigenvectors in $L^2(\br^n)$.
  The first problem is usually ignored in practice because it is possible to find functions that are in the domain of the operator, as we will see in a moment.
  However, observe that this is an example of what we said earlier: the observable acts in a (dense) subspace of $L^2(\br^n)$ and not on the whole space.
  It is also worth noticing that, from a theoretical perspective, for every $\psi \in L^2(\br^n)$ the function $\vb{r} \mapsto \abs{\psi(\vb{r})}^2$ is a well-defined probability density for the position of the object, and what the previous comment highlights is that the position of a state described by a wave function that is not on the domain of $\what{Q}_i$ cannot be measured in practice.
  Another important observation is that the domain of $\what{Q}_i$ is dense, which means the transpose and the adjoint of this operator exists.
  We will need the transpose for the second problem, but it is the existence of the adjoint that assures us $\what{Q}_i$ is a true observable.
  Now let us deal with the second problem.
  Observe that, if $\what{Q}\psi(\vb{r}) \ceq r_i\psi(\vb{r}) = \lambda\psi(\vb{r})$, then $\psi(\vb{r})(r_i - \lambda) = 0$ for every $\vb{r} \in \br^n$.
  But this only happens in $L^2(\br^n)$ if $\psi \equiv 0$ because $r_i$ varies with $\vb{r}$ and $\lambda$ is fixed (remember that $L^2(\br^n)$ is actually a quotient and the characteristic functions that are $0$ everywhere except when $r_i = \lambda$ are equivalent to $0$ because subspaces with dimension less than $n$ have measure $0$).
  So, since the $0$ function is never an eigenvector by definition and even if it was, it would not have any physical meaning, we discard this solution.
  Now, to overcome the problem of $\what{Q}_i$ not having eigenvectors, we use the fact that $L^2(\br^n)$ is (anti-)isometric\footnote{The map is antilinear, not linear.} to its dual via $\psi \mapsto \bra{\psi}$ and that the Schwartz space\footnote{We do not provide a definition for this space in this work, but, intuitively, it is a subspace of $L^2(\br^n)$ made of functions whose derivatives are rapidly decreasing. Its dual is called space of tempered distributions.} is contained in the domain of $\what{Q}_i$.
  By doing this, we can consider the transpose $\what{Q}_i^t$ of the position operator acting in the space of tempered distributions and in this space we have Dirac delta functions $\delta_{\vb{r}}$, \ie, the linear functional defined by $\phi \mapsto \phi(\vb{r})$.
  Now, the transpose $\what{Q}_i^t$ acts on tempered distributions in the following way: $\bra{\psi} \mapsto r_i\bra{\psi}$, being this a product in the distributional sense, which is defined by: $\qty(r_i\bra{\psi})(\phi(\vb{r})) \ceq \bra{\psi}(r_i\phi(\vb{r})) \ceq \braket{\psi}{r_i\phi(\vb{r})}$.
  Now observe that Dirac delta functions are eigenvectors of the transpose $\what{Q}_i^t$.
  Indeed, for any $\lambda \in \br^n$ and $\psi \in L^2(\br^n)$, we have
  \begin{equation}
    \qty(\what{Q}_i^t\delta_{\lambda})(\psi(\vb{r}))
    = \qty(r_i\delta_{\lambda})(\psi(\vb{r}))
    = \delta_{\lambda}(r_i\psi(\vb{r}))
    = \lambda_i\psi(\lambda)
    = \qty(\lambda_i\delta_{\lambda})(\psi(\vb{r})),
  \end{equation}
  and this proves that $\delta_{\lambda}$ is an eigenvector with eigenvalue $\lambda_i$.
  We can now interpret wave functions with respect to the position in the following way: $\wtil{\psi}(\vb{r}) \ceq \delta_{\vb{r}}(\psi) = \psi(\vb{r})$.
  In other words, the wave function $\wtil{\psi}$ coincides with the representative of the state $\ket{\psi}$.
  Actually, in the literature it is common to denote a linear functional $\bra{\psi}$ applied to a function $\phi$ as $\braket{\psi}{\phi}$ instead of $\psi(\phi)$, and this allows us to write the wave function as it was in Definition \ref{wave_function}: $\psi(\vb{r}) = \braket{\delta_{\vb{r}}}{\psi}$.
  This notation is a little misleading, since there is no natural inner product in the space of tempered distributions, but the reader will probably find wave functions written like this in the literature.
  Actually, most of the time $\delta_{\vb{r}}$ is denoted by $\bra{\vb{r}}$ and then $\psi(\vb{r})$ becomes $\braket{\vb{r}}{\psi}$.
  To conclude this example, observe that, despite the fact that the eigenvectors of the position operator are not in $L^2(\br^n)$, the states are (or at least they are in a dense subspace) and we can keep $L^2(\br^n)$ in mind instead of tempered distributions.
\end{exm}

\begin{exm}
  This example is here just for the sake of completeness and to provide more examples to the reader not used to Quantum Mechanics.
  It is not an important example for the present work, though.
  Once again, the rigorous formulation can be found in \cite{madrid2001, hall2013}.

  The \emph{momentum operator} for an object moving in $\br^n$ is defined as
  \begin{equation}
    \what{P}\psi = -i\hbar\grad{\psi} = -i\hbar\sum_{i=1}^n \pdv{\psi}{x_i},
  \end{equation}
  being $\hbar = \frac{h}{2\pi}$ the \emph{reduced Planck's constant} and $h$ the \emph{Planck's constant}.
  Observe that we have the same problem as before: the domain is not $L^2(\br^n)$ because a function whose squared norm is integrable is not necessarily differentiable and the eigenvectors of this operator are not in $L^2(\br^n)$, as we will see.
  In practice, the problems are handled in the same way as the previous example.
  Now observe that, given a vector $\vb{p} \in \br^n$, the following functions, which are usually called \emph{plane waves}, are eigenvectors of $\what{P}$:
  \begin{equation}
    \psi(\vb{r}) = e^{2\pi{i}(p_1r_1 + \ldots + p_nr_n)} = e^{2\pi{i}\vb{p} \cdot \vb{r}}.
  \end{equation}
  Indeed,
  \begin{equation}
    \what{P}\psi(\vb{r})
    = -i\hbar\sum_{i=1}^n \pdv{\psi(\vb{r})}{x_i}
    = -i\hbar\sum_{i=1}^n \pdv{e^{2\pi{i}\vb{p} \cdot \vb{r}}}{x_i}
    = 2\pi\hbar\sum_{i=1}^n p_ie^{2\pi{i}\vb{p} \cdot \vb{r}}
    = \sum_{i=1}^n hp_i\psi(\vb{r}),
  \end{equation}
  which means the eigenvalue is $h(p_1 + \ldots + p_n)$.
  Now, the wave function of the state $\ket{\psi}$ with respect to the momentum operator is the Fourier transform of $\psi$:
  \begin{equation}
    \wtil{\psi}(\vb{p}) \ceq \braket{e^{2\pi{i}\vb{p} \cdot \vb{r}}}{\psi}
    = \intl_{\br^n} \ol{e^{2\pi{i}\vb{p} \cdot \vb{r}}}\psi(\vb{r})\dd\vb{r}
    = \what{\psi}(\vb{p}).
  \end{equation}
  To conclude, Plancherel Theorem guarantees that $\what{\psi} \in L^2(\br^n)$, which means that, despite the fact that plane waves are not in $L^2(\br^n)$, the Fourier transform of a wave function $\psi$ is and we can once again keep $L^2(\br^n)$ as our role model of Hilbert space.
\end{exm}

Now let us finally move to the main example of this work.

\subsection{Molecules (Electronic Structure)} \label{molecules}

As we already know, the Hilbert space associated to molecules should be $L^2(\br^n)$, but we will refine this description because, ultimately, the Hilbert space depends on the physical phenomena one is interested in.
So, for example, if one wishes to see molecules as an object by itself located somewhere in space, the Hilbert space that should be considered is $L^2(\br^3)$.
However, sometimes the molecule is subject to constraints, such as being in an oscillatory motion, trapped in a box or for some reason it is contained in a plane, and then one should consider $L^2(\br)$ or $L^2(\br^2)$ instead.
For our purposes, though, we need another kind of refinement: we will see molecules as being composed of atoms and atoms as being composed of electrons and nuclei.
Now a question may arise: should not we consider the nuclei itself as a composite system made of elementary particles, \ie, particles that are not made of other particles?
So, for example, atoms are composed of electrons, protons and neutrons, but then protons and neutrons are composed of quarks and so on and so forth.
The answer is affirmative, this is the most accurate representation we can have of nature, but, philosophical questions apart, we want to compute and predict phenomena that happens in real life and in order to do that simplifications should be made if the theoretical results obtained agrees with the experimental ones.
Consequently, to us, a good description of a molecule with $M$ nuclei and $N$ electrons is $\what{\bigotimes}_{i=1}^{M + N} L^2(\br^3) \cong L^2(\br^{3(M + N)})$ (recall the second axiom in \ref{axioms}).
With that said, the working scientist is not that interested in the Hilbert space describing a composite system, what matters the most in real life are the interactions between particles and how one can describe and compute them using observables.
Now, remember, we want to compute the energy of a molecule and then find its minimum.
How can we do this?
A standard approach is quantizing the energy observable used in Classical Mechanics, \ie, transforming this classical observable into a self-adjoint operator.\footnote{In Classical Mechanics observables are real-valued smooth functions defined on the phase space $T^*M$ of a configuration space $M$, being $M$ a smooth manifold and $T^*M$ its cotangent bundle.}
In Classical Mechanics, the energy of a system is usually identified with the Hamiltonian of the system, which, to our purposes, can be defined as the sum of the kinetic and potential energies of the system (there are exceptions to this definition of energy but we will not consider them here).
So, for a composite system made of $n$ interacting particles, the formula for the energy/Hamiltonian is
\begin{equation}
  \mc{H}(\vb{r}_1(t), \ldots, \vb{r}_n(t)) \ceq \sum_{i=1}^n
  \frac{1}{2}m_i\abs{\dot{\vb{r}}_i(t)}^2 + V(\vb{r}_1(t), \ldots, \vb{r}_n(t)),
\end{equation}
being $\vb{r}_i(t)$ and $\frac{1}{2}m_i\abs{\dot{\vb{r}}_i(t)}^2$ the trajectory and the kinetic energy of the $i$-th particle, while $V(\vb{r}_1(t), \ldots, \vb{r}_n(t))$ is the potential energy of the system.
The example of potential energy we will consider is given by Coulomb's law:
\begin{equation}
  V(\vb{r}_1(t), \ldots, \vb{r}_n(t)) \ceq \sum_{1 \leq i < j \leq n}
  \frac{q_iq_j}{4\pi\eps_0}\frac{
    \what{\vb{r}}_{ij}
  }{
    \abs{\vb{r}_i(t) - \vb{r}_j(t)}^2
  },
\end{equation}
being $\what{\vb{r}}_{ij} \ceq \frac{{\vb{r}_i(t) - \vb{r}_j(t)}}{\abs{\vb{r}_i(t) - \vb{r}_j(t)}}$, $\eps_0$ the \emph{vacuum permittivity} and $q_i$ the charge of the $i$-th particle.
Now enters quantization.
The above Hamiltonian becomes the following operator in the quantum realm:
\begin{multline} \label{hamiltonian}
  \what{H}\psi \ceq
  - \sum_{i=1}^N \frac{\hbar^2}{2m_e}\nabla_i^2\psi
  - \sum_{j=1}^M \frac{\hbar^2}{2m_j}\nabla_j^2\psi
  - \sum_{i=1}^N \sum_{j=1}^M
  \frac{Z_jq_e^2}{4\pi\eps_0}\frac{\psi}{\abs{\vb{r}_i - \vb{R}_j}} \\
  + \sum_{1 \leq i < k \leq N}
  \frac{q_e^2}{4\pi\eps_0}\frac{\psi}{\abs{\vb{r}_i - \vb{r}_k}}
  + \sum_{1 \leq j < k \leq M}
  \frac{Z_jZ_kq_e^2}{4\pi\eps_0}\frac{\psi}{\abs{\vb{R}_j - \vb{R}_k}}.
\end{multline}
Let us explain the details: $\hbar$ is, again, the \emph{reduced Planck's constant} and its value is $\frac{h}{2\pi}$, being $h$ the \emph{Planck's constant}.
The mass of the $j$-th nuclei is $m_j$ and $m_e$ is the mass of electron.
The $\nabla^2$ symbol represents the \emph{Laplacian} of a function and in this case it is given by
\begin{equation}
  \nabla_i^2\psi
  \ceq \pdv[2]{\re\psi}{(\vb{r}_i)_x}
  + i\pdv[2]{\im\psi}{(\vb{r}_i)_x}
  + \pdv[2]{\re\psi}{(\vb{r}_i)_y}
  + i\pdv[2]{\im\psi}{(\vb{r}_i)_y}
  + \pdv[2]{\re\psi}{(\vb{r}_i)_z}
  + i\pdv[2]{\im\psi}{(\vb{r}_i)_z}
\end{equation}
and
\begin{equation}
  \nabla_j^2\psi
  \ceq \pdv[2]{\re\psi}{(\vb{R}_j)_x}
  + i\pdv[2]{\im\psi}{(\vb{R}_j)_x}
  + \pdv[2]{\re\psi}{(\vb{R}_j)_y}
  + i\pdv[2]{\im\psi}{(\vb{R}_j)_y}
  + \pdv[2]{\re\psi}{(\vb{R}_j)_z}
  + i\pdv[2]{\im\psi}{(\vb{R}_j)_z}.
\end{equation}
Recall that $\vb{r}_i, \vb{R}_j \in \br^3$ and that $\psi$ is a complex-valued function.
Finally, $\eps_0$ is again the vacuum permittivity, $Z_j$ is the atomic number of the $j$-th nuclei, which means $Z_jq_e$ is its charge, and $-q_e$ is the charge of the electron.
Physically, this operator is capturing the fact that the energy of a molecule is given by the sum of the kinetic energy of the nuclei, the kinetic energy of the electrons, the Coulomb attraction between electrons and nuclei, the Coulomb repulsion between nuclei, and the Coulomb repulsion between electrons.
Notice how similar to the classical Hamiltonian it is, but in this case the Hamiltonian is time-independent, also known as stationary.
And how do we know this Hamiltonian is correct? Being pragmatic, because when one uses it to compute the energy of molecules, the results agrees with experiments.

OK, now that we know the Hamiltonian is the observable that computes the energy and recalling that we want to find the state with minimum energy, which is called \emph{ground state}, how can we proceed?
Using the last axiom, we have two ways of finding the ground state: looking for a state $\ket{\psi_0} \in \bp L^2(\br^{3(M + N)})$ such that
\begin{equation}
  \mel{\psi_0}{\what{H}}{\psi_0} \leq \mel{\psi}{\what{H}}{\psi}
\end{equation}
for every other state $\ket{\psi} \in \bp L^2(\br^{3(M + N)})$, or solving the equation
\begin{equation} \label{schrodinger}
  \what{H}\psi = E\psi
\end{equation}
and choosing an eigenvector associated to the lowest eigenvalue among the solutions.
The equation above is the famous \emph{time-independent Schrödinger equation}\index{Schrödinger equation} and we call its eigenvalues \emph{energy levels} and its eigenvectors \emph{stationary states}.
Also, the states that are not the ground state are called \emph{excited states}.
The time-dependent version that is responsible for the wave behavior of molecules as time evolves is
\begin{equation}
  i\hbar\pdv{t}\psi(t) = \what{H}\psi(t).
\end{equation}
Observe that the $1$-dimensional wave equation
\begin{equation}
  \pdv[2]{t}\psi(x, t) = \lambda^2\pdv[2]{x}\psi(x,t)
\end{equation}
is very similar to the time-dependent Schrödinger equation for a particle in a box moving in only one dimension and without potential energy:
\begin{equation}
  i\hbar\pdv{t}\psi(x, t) = -\frac{\hbar^2}{2m}\pdv[2]{x}\psi(x, t).
\end{equation}
However, as already said, we will not consider the time-dependent Schrödinger equation in the present work, so let us get back to the time-independent version.

\begin{obs}
  Notice that we are considering second-order partial derivatives in the Hamiltonian, which, again, means that the Hilbert space we should consider is not exactly $L^2(\br^{3(M + N)})$ but a subspace of it.
  Another relevant question is whether these equations have a solution or not, \ie, is there a ground state?
  In our case, as we will see in Section \ref{qm_practice}, a solution exist, so we do not have to worry.
  However, if the reader is interested in the mathematical aspects of this problem, in \cite[Chapter 3]{takhtajan2008} the author deals with it in a very detailed manner.
\end{obs}

The question now is: how do we solve the (time-independent) Schrödinger equation?
Let us start with Equation \ref{schrodinger}.
Writing the Hamiltonian defined in \ref{hamiltonian} in atomic units, \ie, making every constant equal to $1$, the Schrödinger equation becomes
{\small\begin{equation}
  \qty(
    - \sum_{i=1}^N \frac{\nabla_i^2}{2}
    - \sum_{j=1}^M \frac{\nabla_j^2}{2m_j}
    - \sum_{i=1}^N \sum_{j=1}^M \frac{Z_j}{\abs{\vb{r}_i - \vb{R}_j}}
    + \sum_{1 \leq i < k \leq N} \frac{1}{\abs{\vb{r}_i - \vb{r}_k}}
    + \sum_{1 \leq j < k \leq M} \frac{Z_jZ_k}{\abs{\vb{R}_j - \vb{R}_k}}
  )\psi = E\psi.
\end{equation}}%
This is a partial differential equation in $3(M+N)$ variables and, to this day, we only know an analytical solution for the hydrogen atom, which indicates that this equation is pretty hard to solve.\footnote{There are analytical solutions to other systems such as the quantum harmonic oscillator, quantum pendulum, free particles etc, but not for atoms and molecules in general.}
A rigorous solution for the hydrogen atom can be found in \cite[Chapter 3]{takhtajan2008} and the standard solution can be found in \cite[Chapter 7]{mcquarrie2008}.
Let us then consider an approximation to make things easier.
Since the nuclei of molecules are way heavier than the electrons, an approximation that is usually quite good in practice is considering the nuclei of the atoms fixed, which means they do not have kinetic energy and the Coulomb repulsion between the nuclei is constant.
Now, notice that, if we have an operator $T$, then summing it with a multiple of the identity does not change its eigenvectors and the eigenvalues are just the sum of the eigenvalues of $T$ with the constant.
In other words:
\begin{equation}
  Tv = \lambda{v} \; \; \text{\tiff} \; (T + c\Id)v = (\lambda + c)v.
\end{equation}
Using this fact, we can ignore the constant given by nuclei repulsion when we consider the nuclei fixed because we can compute it separately and then sum to the solution we find for the rest of the Hamiltonian.
Considering this approximation, which is called \emph{Born--Oppenheimer approximation}\index{Born--Oppenheimer approximation}, the equation we have to solve now is:
\begin{equation} \label{bo_hamiltonian}
  \what{H}_{el}\psi
  \ceq \qty(
    - \sum_{i=1}^N \frac{\nabla_i^2}{2}
    - \sum_{i=1}^N \sum_{j=1}^M \frac{Z_j}{\abs{\vb{r}_i - \vb{R}_j}}
    + \sum_{1 \leq i < k \leq N} \frac{1}{\abs{\vb{r}_i - \vb{r}_k}}
  )\psi = E\psi.
\end{equation}
The Hamiltonian we are considering now is usually called \emph{electronic Hamiltonian}\index{hamiltonian!electronic}.
This equation is still pretty hard to solve and the only analytical solution known to it other than the hydrogen atom is H$_2^+$, \ie, two hydrogen bonded, but with just one electron orbiting the nuclei.
By the way, notice that in this approximation we reduced the number of variables in the equation to $3N$ variables since $\vb{R}_j$ is fixed.
However, since there is not a general method to solve this equation analytically other than the hydrogen atom or H$_2^+$, and since our goal is not to come up with such a method, we will consider the computational approach and try to solve this equation numerically.
We explain how this can be done in Section \ref{qm_practice}, but let us first rule out a final theoretical consideration.
If we have a wave function $\psi$ such that $\what{H}\psi = \lambda\psi$, then its real part $\re\psi$ also satisfies $\what{H}\qty(\re\psi) = \lambda\qty(\re\psi)$.
Indeed, since $\re\psi = \frac{\psi + \ol{\psi}}{2}$, then $\what{H}\qty(\re\psi) = \frac{\what{H}\psi + \what{H}\ol{\psi}}{2}$.
However, using the explicit Hamiltonian \ref{hamiltonian} and using the fact that
\begin{equation}
  \pdv[2]{\ol{\psi}}{x_i} = \pdv[2]{x_i}\qty(\re\psi - i\im\psi)
  = \pdv[2]{\re\psi}{x_i} - i\pdv[2]{\im\psi}{x_i}
  = \ol{\pdv[2]{\psi}{x_i}},
\end{equation}
we obtain
\begin{equation}
  \what{H}\ol{\psi} = \ol{\what{H}\psi} = \ol{\lambda\psi}
\end{equation}
and this implies
\begin{equation} \label{from_complex_to_real}
  \what{H}\qty(\re\psi)
  = \frac{\what{H}\psi + \what{H}\ol{\psi}}{2}
  = \lambda\frac{\psi + \ol{\psi}}{2}
  = \lambda\qty(\re\psi),
\end{equation}
as desired.
Observe that in the second step we also used the fact that observables are self-adjoint and therefore have real eigenvalues, \ie, $\lambda = \ol{\lambda}$.
The above computation shows us that in our search for the ground state we need to consider only real-valued wave functions since if a complex-valued wave function is a solution to the Schrödinger equation, then its real part is also a solution.
It is also worth mentioning that, if $\psi = \im\psi$, then $\re\psi \equiv 0$, which we do not want.
However, in this case we can consider $i\psi$ as our solution.
That is it, the computation above also works for $\what{H}_{el}$ and we will use this fact in Section \ref{qm_practice} and consider only real-valued wave functions in our search for the ground state.

\subsection{Spin} \label{spin}

Now let us talk about spin because it plays a fundamental role in Quantum Mechanics and in our work.
There are three common approaches to present and understand spin: using representation theory, Pauli's equation or Dirac's equation.
However, all these approaches requires advanced mathematical tools that will not be important to us and would lead us too far from our goal.
Therefore, here we will only highlight the basics of the theory behind spin and focus on what we need.
If the reader is interested in understanding more of the physics behind spin, see \cite[Chapter~6]{tannoudji2020_1} and \cite[Chapter~9]{tannoudji2020_2}; and if the reader is interested in the mathematics behind spin, see \cite[Chapter~17]{hall2013} and \cite[Chapter~4]{takhtajan2008}.

From a historical perspective, spin was discovered in the \emph{Stern--Gerlach experiment}, an experiment in which Stern and Gerlach shot silver atoms in an inhomogeneous magnetic field because the inhomogeneity of the field deflected the atoms before they hit the detector screen.
From a Classical Mechanics perspective, a random and continuous distribution of the atoms in the detector was expected, but in this experiment the atoms deflected only up or down, which was interpreted as the atoms having an intrinsic and quantized angular momentum.
Pauli explained this experiment by quantizing the classical Hamiltonian for a charged particle interacting with an electromagnetic field and he obtained what is now known as \emph{Pauli's equation}.
This equation requires the wave function to have two components to explain the up and down phenomenon observed experimentally.
Mathematically, that means $\psi \in L^2(\br^3, \bc^2) \ceq \qty{\psi:\br^3 \to \bc^2 : \intl_{\br^3} \abs{\psi(\vb{r})}^2 \dd\vb{r} < \infty}$ or $\psi \in L^2(\br^3) \otimes \bc^2$, since $L^2(\br^3) \otimes \bc^2 \cong L^2(\br^3, \bc^2)$.
Another confirmation of spin came indirectly from \emph{Dirac's equation}.
Dirac was trying to explain the spectral lines of the hydrogen at low temperature, \ie, which frequencies of light are absorbed or emitted by the hydrogen atom at low temperature, see \cite{shrum1924}, and, in order to explain this, he quantized some equations from Special Relativity.
His equation is a generalization of Pauli's equation, which can be thought of as a non-relativistic limit of Dirac's equation, and this equation required the existence of four components in the wave function, two accounting for spin and two that now we know accounts for antiparticles.\footnote{A couple of years after Dirac came up with his equation, Carl Anderson discovered experimentally the existence of \emph{positron}, the antiparticle of electron, see \cite{anderson1933}.}
So, from a physical perspective, spin is the intrinsic angular momentum of a particle and, simplifying things a little, the Hilbert space that describes the spin of a particle is $V_s$, being $s \in \qty{0, \frac{1}{2}, 1, \frac{3}{2}, \ldots}$ and $V_s$ a $(2s + 1)$-dimensional complex vector space carrying an irreducible representation of $SU(2)$.
There are a couple of different but equivalent ways of describing $V_s$ and the irreducible representation explicitly, see \cite[Chapter~4]{takhtajan2008}, and in the next section we will see one of them for $V_{\frac{1}{2}}$, which is the case we are interested in.
Now, an important classification that comes from spin is the following:

\begin{description}

\item[Sixth Axiom.] Every elementary particle has spin and, if the value $s$ that defines the spin is a natural number, the particle is called a \emph{boson}\index{boson}, otherwise it is called a \emph{fermion}\index{fermion}.
  Consequently, the state space describing a particle with spin $s$ is $L^2(\br^3) \what{\otimes} V_s$.

\end{description}

\begin{exm}[Fermions]
  Electron, proton and neutron are examples of fermions.
  More generally, leptons and quarks are fermions with spin $\frac{1}{2}$.
\end{exm}

\begin{exm}[Bosons]
  Photon, gluon and $W$ and $Z$ bosons are bosons with spin $1$ and the Higgs boson has spin $0$.
\end{exm}

To conclude the discussion about spin, let us talk about systems with identical particles because we are interested in molecules, which, potentially, have a big number of electrons to be considered.
Experimental evidence and the \emph{Spin-Statistics Theorem} indicates that composite systems made up of identical particles obeys either the \emph{Bose--Einstein statistics}, if the particle is a boson, or the \emph{Fermi--Dirac statistics}, if the particle is a fermion.
A proof of this theorem can be found in \cite[Chapter~4]{streater2000}, but what this theorem tells us is that we do not need to consider the whole space $\what{\bigotimes}^N L^2(\br^3) \wten V_s$ to describe the composite system with $N$ particles, but only its symmetric subspace $\sym^N\qty(L^2(\br^3) \wten V_s)$, if the particle is a boson, or the antisymmetric subspace $\bigwedge^N \qty(L^2(\br^3) \wten V_s)$, if the particle is a fermion.
These (sub)spaces were defined in Section \ref{multilinear_algebra} and they are usually called symmetric and exterior power, respectively.
For the particular case of electrons, this is also known as the \emph{Pauli exclusion principle} and physically it means two electrons cannot occupy the same state.
In mathematical notation, if we have $\psi \otimes \gamma \in L^2(\br^3) \wten V_{\frac{1}{2}}$, then $\qty(\psi(\vb{r}_1) \otimes \gamma) \otimes \qty(\psi(\vb{r}_2) \otimes \gamma) \notin \qty(L^2(\br^3) \wten V_{\frac{1}{2}}) \wedge \qty(L^2(\br^3) \wten V_{\frac{1}{2}})$.
Now we have all the tools we need to consider the Hartree--Fock Method.

\section{Practice (Hartree--Fock)} \label{qm_practice}

Recall that we want to solve the Schrödinger equation \ref{bo_hamiltonian} numerically and that there are (at least) two approaches to tackle this equation: (1) seeing it as an eigenvalue problem
\begin{equation}
  \what{H}_{el}\Psi = E\Psi;
\end{equation}
and (2) searching $\Psi_0 \in \bigwedge^N \qty(L^2(\br^3) \otimes V_{\frac{1}{2}})$ (we will also consider the spin of electrons) such that
\begin{equation}
  \mel{\Psi_0}{\what{H}_{el}}{\Psi_0} \leq \mel{\Psi}{\what{H}_{el}}{\Psi}
\end{equation}
for every other $\Psi \in \bigwedge^N \qty(L^2(\br^3) \otimes V_{\frac{1}{2}})$.
As it is, both perspectives are not suited to be handled computationally and the first restriction we will consider is choosing a finite-dimensional subspace $W$ of $L^2(\br^3)$ with dimension $d$.
So, that means $\dim(W \otimes V_{\frac{1}{2}}) = 2d$ and $\dim{V} = \binom{2d}{N}$, being $V \ceq \bigwedge^N \qty(W \otimes V_{\frac{1}{2}})$.
However, if we have a molecule such as benzene with $42$ electrons and a subspace $W$ of dimension $114$, then $\dim{V} = \binom{228}{42} = 13673073388289473024722279629504446017541199759$.
That means if we want to store a wave function in $V$ written as a linear combination with respect to the basis set of size $114$ using the number format \emph{double} to represent the coefficients, we would need $99484702428813565527687603240370270$ terabytes of memory.
This is infeasible in real life and, since the first approach to tackle the Schrödinger equation as an eigenvalue problem would require storing a $\dim{V} \x \dim{V}$ matrix in order to find its eigenvectors, we have to discard it.\footnote{For a small number of electrons and a small dimension $d$ this is feasible, of course, but we will not pursue this approach anyway. If the reader is interested in it, see \cite[Chapter 11]{helgaker2000}.}
So, let us consider the second approach.
The reader may think that we will face the same problem since we are still dealing with $V$, but now we can restrict our search to subsets of $V$, which was not possible in the first approach because in order to build the matrix representation of the Hamiltonian we would need to consider the whole space $V$.
Now, there are many possible ways of restricting the search space and each way leads to a different optimization method that has its own advantages and disadvantages from a computational and theoretical perspective.
The most accurate result, of course, is obtained by considering the whole $V$ and the solution we would obtain in this search is usually called \emph{Full Configuration Interaction} or just \emph{Full CI wave function}.
Other popular methods are the \emph{Configuration Interaction Method}, the \emph{Coupled Cluster Method}, the \emph{M{\o}ller--Plesset Method} etc.
All of these methods are called \emph{post-Hartree--Fock methods} because they use the solution obtained in the \emph{Hartree--Fock Method} as a starting point.
We will define what we mean by Hartree--Fock Method in a minute, but it is worth to mention that despite the fact that the Hartree--Fock Method is not the most accurate method used in Computational Quantum Chemistry these days, it is still a foundational method because, as already mentioned, the other state-of-the-art methods uses it as a first step.
So, having a robust implementation of Hartree--Fock and understanding it deeply is essential to Computational Quantum Chemistry.
Let us now see the mathematical aspects of the method.

First, a concrete description for $V_{\frac{1}{2}}$ is given by
\begin{equation}
  V_{\frac{1}{2}} \ceq \qty{\gamma:\qty{-1/2, \ 1/2} \to \bc}.
\end{equation}
Then, we can describe spin $\alpha$ (up) and $\beta$ (down) using the functions defined by $\alpha(1/2) = \beta(-1/2) = 1$ and $\alpha(-1/2) = \beta(1/2) = 0$.
Now, we will represent the finite-dimensional space of \emph{atomic orbitals} (\ie, wave functions for only one electron) by
\begin{equation}
  W \ceq \Span{w_1^{\alpha}, \ldots, w_{d_{\alpha}}^{\alpha},
    w_1^{\beta}, \ldots, w_{d_{\beta}}^{\beta}},
\end{equation}
being $w_i^{\gamma} \ceq \psi_i^{\gamma} \otimes \gamma$, $\psi_i^{\gamma} \in L^2(\br^3)$ and $\gamma \in \qty{\alpha, \beta}$.
In the literature $\psi_i^{\gamma}$ is also called \emph{spatial orbital} and $w_i^{\gamma}$ is also called \emph{spin-orbital}.
Besides, the symbol $\otimes$ is usually ignored, so, $\psi_i^{\gamma} \otimes \gamma$ is written as $\psi_i^{\gamma}\gamma$.
With that said, if we are interested in a system with $N$ electrons represented by $V \ceq \bigwedge^N W$, we can define $W_{\gamma} \ceq \Span{w_1^{\gamma}, \ldots, w_{d_{\gamma}}^{\gamma}}$ and then it is possible to see $\gr{\NA}{W_{\alpha}} \x \gr{\NB}{W_{\beta}}$, being $\NA$ the number of electrons with spin $\alpha$ and $\NB$ the number of electrons with spin $\beta$, as a submanifold of $\bp{V}$ by considering the function known as \emph{Plücker embedding}\index{Plücker embedding}:
\begin{align} \label{plucker_embedding}
  \begin{split}
    \Pl:\gr{\NA}{W_{\alpha}} \x \gr{\NB}{W_{\beta}}
    & \to \bp{V} \\
    \qty(\Span{v_1^{\alpha}, \ldots, v_{N_\alpha}^{\alpha}}, \; \Span{v_1^{\beta}, \ldots, v_{N_\beta}^{\beta}})
    & \mapsto \ket{\wed{v_1^{\alpha}}{v_{\NA}^{\alpha}}
      \wedge \wed{v_1^{\beta}}{v_{\NB}^{\beta}}}.
  \end{split}
\end{align}
And now we can define Hartree--Fock precisely:

\begin{defi} \label{hf}
  The \emph{Hartree--Fock Method}\index{method!Hartree--Fock} is the optimization problem of minimizing the expected value of the Hamiltonian (\ref{hamiltonian} or \ref{bo_hamiltonian}) of a system with $N$ electrons given the constraint that the solution is in the manifold $\gr{\NA}{W_{\alpha}} \x \gr{\NB}{W_{\beta}}$.
  In other words, we want to minimize the function $f \ceq E \circ \Pl:\gr{\NA}{W_{\alpha}} \x \gr{\NB}{W_{\beta}} \to \br$, being $E:\bp{V} \to \br$ the energy of a system, \ie, $E(\ket{\Psi}) = \mel{\Psi}{\what{H}_{el}}{\Psi}$.
\end{defi}

It may seem that this is not a natural constraint, but let us recapitulate what we know so far to see why it actually is: first, if we want to solve the Schrödinger equation computationally, we have to work with a finite-dimensional subspace of $L^2(\br^3)$.
Moreover, since spin is really important because it explains many phenomena, we have to actually consider $L^2(\br^3) \otimes V_{\frac{1}{2}}$ as our state space.
Now, if we have the orbitals $\psi_1, \ldots, \psi_N \in L^2(\br^3)$, then a natural guess for the (spatial part of the) wave function of a system with $N$ electrons is $\psi(\vb{r}_1, \ldots, \vb{r}_N) \ceq \psi_1(\vb{r}_1) \cdot \ldots \cdot \psi_N(\vb{r}_N) \in L^2(\br^{3N})$.
Notice that in this wave function it is as if we have one orbital for each electron, with $\psi_i(\vb{r}_i)$ describing the $i$-th electron.
However, and that is another reason why spin is important, the above product does not take into account the Pauli exclusion principle.
To fix this problem, we have to consider an antisymmetric product and this leads us to the definition of a very special and important type of wave function:

\begin{defi}
  The wave functions $\wed{v_1}{v_N} \in \bigwedge^N \qty(L^2(\br^3) \wten V_{\frac{1}{2}})$ and $\wed{\psi_1}{\psi_N} \in \bigwedge^N L^2(\br^3)$ are called \emph{Slater determinants}\index{Slater determinant}.
\end{defi}

So, as we saw, Slater determinants are the most basic wave functions we have that take all the quantum mechanical aspects of electrons into account and, therefore, the Hartree--Fock Method can be considered the most basic quantum mechanical method.
Some of the post-Hartree--Fock methods mentioned previously search for solutions that are linear combinations of Slater determinants, for example.

\begin{obs}
  In the Hartree--Fock Method we consider only Slater determinants, but not every Slater determinant is allowed, since elements with ``mixed spin'', such as $\frac{1}{\sqrt{2}}\qty(\psi \otimes (\alpha + \beta))$, do not make sense from a physics perspective.
\end{obs}

\begin{defi} \label{restricted_hf}
  When a molecule has an even number of electrons and we impose the restrictions that $\NA = \NB$ and that the Slater determinants for both spins are exactly the same, we have what is called \emph{Restricted Hartree--Fock Method} (RHF).
  If we do not impose any restriction, the method is sometimes called \emph{Unrestricted Hartree--Fock} (UHF).
\end{defi}

\begin{obs}
  In the RHF, the manifold we look for solutions is $\gr{N/2}{W}$, with $W \ceq \Span{\psi_1, \ldots, \psi_d}$.
  In other words, we only look for Slater determinants given by $\wed{\phi_1}{\phi_{\frac{N}{2}}}$, being $\phi_i = c_{1i}\psi_1 + \ldots + c_{di}\psi_d$.
\end{obs}

\begin{obs}
  The dimension of $\gr{\NA}{W_{\alpha}} \x \gr{\NB}{W_{\beta}}$ is $\NA(d_{\alpha} - \NA) + \NB(d_{\beta} - \NB)$.
  So, if we go back to the benzene example assuming, again, that $\NA = \NB = 21$ and $d_{\alpha} = d_{\beta} = 114$, the dimension of the space in which we are looking for solutions is $3.906$ and that means we can store a Slater determinant using only $31$ kilobytes of memory.
\end{obs}

With all that said we hopefully convinced the reader that, if we want to consider two fundamental restrictions in our search for the solution of the Schrödinger equation, one computational and the other physical, we naturally obtain the Grassmannian as our search space.
This space has a structure that makes it into a Riemannian manifold that we already explored throughout Chapter \ref{geometry}, and in Chapter \ref{optimization} we will (1) see how to implement optimization algorithms on this manifold, and (2) provide pseudocodes for the (Unrestricted) Hartree--Fock.

\begin{obs}
  As we said in the beginning of this section, we are considering only real-valued orbitals in this work.
  However, since observables are self-adjoint, the expected value is always a real-valued function and this implies we do not need to see the Grassmannian (or any other manifold, as a matter of fact) as a complex manifold since only its real manifold structure will play a role.
  That means it is possible to extend the formulas obtained in the rest of this section and in the next chapter to the complex case if the reader think it is worth.
\end{obs}

To conclude this section, let us prove that the Hartree--Fock Method has a solution.
First, it is worth mentioning something important: the basis $\qty{w_1^{\alpha}, \ldots, w_{d_{\alpha}}^{\alpha}, w_1^{\beta}, \ldots, w_{d_{\beta}}^{\beta}}$ that span $W$ is not necessarily orthonormal, but we will impose that the vectors
\begin{equation} \label{lin_comb_slater}
  v_i^{\gamma} = \qty(c_{1i}^{\gamma}\psi_1^{\gamma} + \ldots + c_{{d_\gamma}i}^{\gamma}\psi_{d_\gamma}^{\gamma})\gamma,
\end{equation}
being $\gamma \in \qty{\alpha, \beta}$, should be.
Consequently, if we represent the Slater determinant $\wed{v_1^{\gamma}}{v_{N_{\gamma}}^{\gamma}}$ as the matrix
\begin{equation}
  C^{\gamma} =
  \begin{bmatrix}
    c_{11}^{\gamma} & \ldots & c_{1N_{\gamma}}^{\gamma} \\
    \vdots          & \ddots & \vdots                   \\
    c_{d_{\gamma}1}^{\gamma} & \ldots & c_{d_{\gamma}N_{\gamma}}^{\gamma} \\
  \end{bmatrix},
\end{equation}
this matrix satisfies $\qty(C^{\gamma})^{\top}SC^{\gamma} = \Id_{N_{\gamma}}$, being $S_{\gamma} = \qty(\braket{\psi_i^{\gamma}}{\psi_j^{\gamma}})_{i,j}$ the overlap matrix.
Now, observe that $\abs{\wed{v_1^{\gamma}}{v_{N_{\gamma}}^{\gamma}}} = 1$ because
\begin{align*}
  \abs{\wed{v_1^{\gamma}}{v_{N_{\gamma}}^{\gamma}}}^2
  & = \braket{
    \wed{v_1^{\gamma}}{v_{N_{\gamma}}^{\gamma}}
    }{
    \wed{v_1^{\gamma}}{v_{N_{\gamma}}^{\gamma}}
    } \\
  & = \det(\braket{v_i^{\gamma}}{v_j^{\gamma}})_{i,j} \\
  & = \det(\Id_{N_{\gamma}}) \\
  & = 1.
\end{align*}
We also have $\braket{v_i^{\alpha}}{v_j^{\beta}} \ceq \braket{v_i}{v_j}\braket{\alpha}{\beta} = 0$ because $\braket{\alpha}{\beta} = 0$.
And the last important observation is that the constraint $\qty(C^{\gamma})^{\top}SC^{\gamma} = \Id_{N_{\gamma}}$ says that $C^{\gamma}$ is actually an element of the Stiefel manifold and, if we multiply $C^{\gamma}$ on the right by any matrix $M^{\gamma} \in O(N_{\gamma})$, the orthonormality is also preserved:
\begin{equation}
  \qty(C^{\gamma}M^{\gamma})^{\top}SC^{\gamma}M^{\gamma}
  = \qty(M^{\gamma})^{\top}\qty(C^{\gamma})^{\top}SC^{\gamma}M^{\gamma}
  = \qty(M^{\gamma})^{\top}\Id_{N_{\gamma}}M^{\gamma}
  = \Id_{N_{\gamma}}.
\end{equation}
Therefore, we can actually use the matrix representation of the Grassmannian defined in Example \ref{grassmannian} to represent the orbitals and that is what we will do in Chapter \ref{optimization}.
Now let us write the formula for the energy with the disclaimer that, if the reader is interested, the complete and detailed derivation is in Subsection \ref{energy_derivation}.
Defining the \emph{one-electron operator} as
\begin{equation}
  H(\vb{r})\psi(\vb{r}) \ceq -\frac{1}{2}\nabla^2\psi(\vb{r})
  - \sum_{j=1}^M \frac{Z_j}{\abs{\vb{r} - \vb{R}_j}}\psi(\vb{r}),
\end{equation}
the expression for
\begin{equation}
  \mel{\wed{v_1^{\alpha}}{v_{\NA}^{\alpha}} \wedge \wed{v_1^{\beta}}{v_{\NB}^{\beta}}}
  {\what{H}_{el}}
  {\wed{v_1^{\alpha}}{v_{\NA}^{\alpha}} \wedge \wed{v_1^{\beta}}{v_{\NB}^{\beta}}},
\end{equation}
if we use the coefficients of the matrix $C^{\gamma}$ defined above, is the sum of
\begin{equation}
  \sum_{\mu=1}^{\NA} \sum_{i,j=1}^{d_{\alpha}}
  c_{i\mu}^{\alpha}c_{j\mu}^{\alpha}
  \intl_{\br^3} \psi_i^{\alpha}(\vb{r})H(\vb{r})\psi_j^{\alpha}(\vb{r})
  \dd\vb{r}
  +
  \sum_{\mu=1}^{\NB} \sum_{i,j=1}^{d_{\beta}}
  c_{i\mu}^{\beta}c_{j\mu}^{\beta}
  \intl_{\br^3} \psi_i^{\beta}(\vb{r})H(\vb{r})\psi_j^{\beta}(\vb{r}) \dd\vb{r}
\end{equation}
with
\begin{multline}
  \frac{1}{2} \sum_{\mu,\nu=1}^{\NA} \sum_{i,j,k,l=1}^{d_{\alpha}}
  c_{i\mu}^{\alpha}c_{j\nu}^{\alpha}c_{k\mu}^{\alpha}c_{l\nu}^{\alpha}
  \intl_{\br^6} \psi_i^{\alpha}(\vb{r}_1)\psi_j^{\alpha}(\vb{r}_2)
  \frac{1}{\vb{r}_{12}}
  \qty(\psi_k^{\alpha}(\vb{r}_1)\psi_l^{\alpha}(\vb{r}_2)
  - \psi_l^{\alpha}(\vb{r}_1)\psi_k^{\alpha}(\vb{r}_2))
  \dd\vb{r}_1\dd\vb{r}_2 \\
  + \frac{1}{2} \sum_{\mu,\nu=1}^{\NB} \sum_{i,j,k,l=1}^{d_{\beta}}
  c_{i\mu}^{\beta}c_{j\nu}^{\beta}c_{k\mu}^{\beta}c_{l\nu}^{\beta}
  \intl_{\br^6} \psi_i^{\beta}(\vb{r}_1)\psi_j^{\beta}(\vb{r}_2)
  \frac{1}{\vb{r}_{12}}
  \qty(\psi_k^{\beta}(\vb{r}_1)\psi_l^{\beta}(\vb{r}_2)
  - \psi_l^{\beta}(\vb{r}_1)\psi_k^{\beta}(\vb{r}_2))
  \dd\vb{r}_1\dd\vb{r}_2 \\
  + \sum_{\mu=1}^{\NA} \sum_{\nu=1}^{\NB}
  \sum_{i,k=1}^{d_{\alpha}} \sum_{j,l=1}^{d_{\beta}}
  c_{i\mu}^{\alpha}c_{j\nu}^{\beta}c_{k\mu}^{\alpha}c_{l\nu}^{\beta}
  \intl_{\br^6} \psi_i^{\alpha}(\vb{r}_1)\psi_j^{\beta}(\vb{r}_2)
  \frac{1}{\vb{r}_{12}}
  \psi_k^{\alpha}(\vb{r}_1)\psi_l^{\beta}(\vb{r}_2)
  \dd\vb{r}_1\dd\vb{r}_2.
\end{multline}

\begin{obs}
  One of the considerations that should be taken into account in the choice of a basis is the computation of the above integrals.
\end{obs}

And now we can finally prove the following theorem:

\begin{thm} \label{hf_welldefined}
  The Hartree--Fock Method has a minimum.
\end{thm}

\begin{proof}
  Observe that we want to find a Slater determinant in $\gr{\NA}{W_{\alpha}} \x \gr{\NB}{W_{\beta}}$ that minimizes the energy.
  However, the manifold $\gr{\NA}{W_{\alpha}} \x \gr{\NB}{W_{\beta}}$ is a compact topological space, as we proved in Corollary \ref{product_compact}, and the energy function described above depends only on the coefficients $c_{i\mu}^{\gamma}$ and not on the integrals.
  That means the formula for the energy is a polynomial of degree $4$ and, consequently, it is a smooth and continuous function.\footnote{To obtain the smoothness of the energy we are actually using Proposition \ref{smooth_restriction} to obtain the smoothness for the Stiefel and then Theorem \ref{smooth_quotient} to obtain the result for the Grassmannian.}
  Then, by Theorem \ref{compact_maxmin}, which asserts that continuous functions defined in a compact space has minimum and maximum, the result follows.
\end{proof}

And that is it for the basics of Quantum Mechanics.

\subsection{Computing the Energy} \label{energy_derivation}

Here we will provide a detailed derivation of the formula for the energy of a system with $N$ electrons when computed in the set of Slater determinants.
To obtain this formula we will use many tools from Multilinear Algebra, so, if the reader is not familiar with these tools, we suggest reading Section \ref{multilinear_algebra} first.

Recall that the Hilbert space associated to this system is
\begin{equation}
  \what{\bigotimes}^N \qty(L^2(\br^3) \wten V_{\frac{1}{2}})
  \cong L^2(\br^{3N}) \wten \ten{V_{\frac{1}{2}}}{V_{\frac{1}{2}}},
\end{equation}
being
\begin{equation}
  V_{\frac{1}{2}} = \qty{\gamma:\qty{-1/2, 1/2} \to \bc},
\end{equation}
and that the Hamiltonian we are interested in is
\begin{equation}
  \what{H}_{el} \ceq
  - \sum_{i=1}^N \frac{1}{2}\nabla_i^2
  - \sum_{i=1}^N \sum_{j=1}^M \frac{Z_j}{\vb{r}_{ij}}
  + \sum_{1 \leq i < k \leq N} \frac{1}{\vb{r}_{ik}}.
\end{equation}
Note that we are already considering the Born--Oppenheimer approximation and, to shorten the notation, we will use $\vb{r}_{ij}$ and $\vb{r}_{ik}$ instead of $\abs{\vb{r}_i - \vb{R}_j}$ and $\abs{\vb{r}_i - \vb{r}_k}$, respectively.
Also note that this Hamiltonian acts on $L^2(\br^{3N}) \wten \ten{V_{\frac{1}{2}}}{V_{\frac{1}{2}}}$, but we should understand how it acts on $\what{\bigotimes}^N \qty(L^2(\br^3) \otimes V_{\frac{1}{2}})$, since this is the space we will work with.
Another important observation is that the Hamiltonian $\what{H}_{el}$ is actually $\what{H}_{el} \otimes \tenw{\id}{\id}$, but we will keep writing it as $\what{H}_{el}$.
Now, the natural way of switching between these spaces is considering the linear transformation
\begin{align}
  T:\what{\bigotimes}^N \qty(L^2(\br^3) \wten V_{\frac{1}{2}}) & \to L^2\qty(\br^{3N}) \wten \ten{V_{\frac{1}{2}}}{V_{\frac{1}{2}}} \\
  \tenw{\psi_1\gamma_1}{\psi_N\gamma_N} & \mapsto \psi_1 \cdot \ldots \cdot \psi_N \otimes \tenw{\gamma_1}{\gamma_N},
\end{align}
being $\gamma_1,\ldots,\gamma_N \in V_{\frac{1}{2}}$, and then working with the composite $\what{H}_{el} \circ T$.
So, let us assume that the Slater determinants are given by $\wed{v_1}{v_N} \ceq \wed{\phi_1\gamma_1}{\phi_N\gamma_N}$, being $\phi_i \in L^2(\br^3)$ and $\gamma_i \in V_{\frac{1}{2}}$, and that $\braket{\phi_i}{\phi_j} = \delta_{ij}$ because that is what we need in practice.
Now let us compute the energy for these Slater determinants, \ie, find a formula for
\begin{equation} \label{raw_energy}
  \mel{T(\wed{\phi_1\gamma_1}{\phi_N\gamma_N})}
  {\what{H}_{el}}
  {T(\wed{\phi_1\gamma_1}{\phi_N\gamma_N})}.
\end{equation}
First, if we use the definition of elementary wedge product, which is
\begin{equation}
  \wed{\phi_1\gamma_1}{\phi_N\gamma_N} = \frac{1}{\sqrt{N!}} \sum_{\sigma \in S_N}
  \sign(\sigma) \tenw{\phi_{\sigma(1)}\gamma_{\sigma(1)}}{\phi_{\sigma(N)}\gamma_{\sigma(N)}},
\end{equation}
and then use the linearity of the inner product, \ref{raw_energy} becomes
{\small
  \begin{equation}
    \frac{1}{N!} \sum_{\sigma,\tau \in S_N} \sign(\sigma \circ \tau)
    \mel{T\qty(\tenw{\phi_{\sigma(1)}\gamma_{\sigma(1)}}
      {\phi_{\sigma(N)}\gamma_{\sigma(N)}})}
    {\what{H}_{el}}
    {T\qty(\tenw{\phi_{\tau(1)}\gamma_{\tau(1)}}
      {\phi_{\tau(N)}\gamma_{\tau(N)}})}.
  \end{equation}
}%
Recall that $\sign(\sigma)\sign(\tau) = \sign(\sigma \circ \tau)$ (Proposition \ref{sign_composition}).
Now, applying $T$:
{\footnotesize\begin{equation}
  \frac{1}{N!} \sum_{\sigma,\tau \in S_N} \sign(\sigma \circ \tau)
  \mel{\phi_{\sigma(1)} \cdot \ldots \cdot \phi_{\sigma(N)}
    \otimes \tenw{\gamma_{\sigma(1)}}{\gamma_{\sigma(N)}}}
  {\what{H}_{el}}
  {\phi_{\tau(1)} \cdot \ldots \cdot \phi_{\tau(N)}
    \otimes \tenw{\gamma_{\tau(1)}}{\gamma_{\tau(N)}}}.
\end{equation}}%
Finally, since the inner product of the tensor product is the product of the inner products, we obtain:
\begin{multline} \label{first_energy}
  \frac{1}{N!} \sum_{\sigma,\tau \in S_N} \sign(\sigma \circ \tau)
  \braket{\gamma_{\sigma(1)}}{\gamma_{\tau(1)}}
  \ldots
  \braket{\gamma_{\sigma(N)}}{\gamma_{\tau(N)}} \\
  \intl_{\br^{3N}}
  \phi_{\sigma(1)}(\vb{r}_1) \ldots \phi_{\sigma(N)}(\vb{r}_N)
  \what{H}_{el}
  \phi_{\tau(1)}(\vb{r}_1) \ldots \phi_{\tau(N)}(\vb{r}_N)
  \dd\vb{r}_1 \ldots \dd\vb{r}_N.
\end{multline}
Now, if we consider the one-electron operator
\begin{equation}
  H(\vb{r})\phi(\vb{r})
  \ceq -\frac{1}{2}\nabla^2\phi(\vb{r})
  - \sum_{j=1}^M \frac{Z_j}{\abs{\vb{r} - \vb{R}_j}}\phi(\vb{r}),
\end{equation}
we can write $\what{H}_{el}$ as
\begin{equation}
  \sum_{i=1}^N H(\vb{r}_i) + \sum_{1 \leq i < k \leq N} \frac{1}{\vb{r}_{ik}}
  = \sum_{i=1}^N H(\vb{r}_i) + \frac{1}{2}\sum_{^{i,k=1}_{k \neq i}}^N \frac{1}{\vb{r}_{ik}},
\end{equation}
being $H(\vb{r}_i)$ actually $\id \otimes \ldots \otimes H(\vb{r}_i) \otimes \ldots \otimes \id$.
Computing the integral above separating $\what{H}_{el}$ in this sum, we first obtain, for the one-electron,
\begin{multline} \label{one_electron}
  \intl_{\br^{3N}}
  \phi_{\sigma(1)}(\vb{r}_1) \ldots \phi_{\sigma(N)}(\vb{r}_N) H(\vb{r}_i)
  \phi_{\tau(1)}(\vb{r}_1) \ldots \phi_{\tau(N)}(\vb{r}_N)
  \dd\vb{r}_1 \ldots \dd\vb{r}_N \\
  =
  \intl_{\br^3} \phi_{\sigma(1)}(\vb{r}_1)\phi_{\tau(1)}(\vb{r}_1)\dd\vb{r}_1
  \cdot \ldots \cdot
  \intl_{\br^3}
  \phi_{\sigma(i)}(\vb{r}_i)H(\vb{r}_i)\phi_{\tau(i)}(\vb{r}_i)\dd\vb{r}_i
  \cdot \ldots \cdot
  \intl_{\br^3} \phi_{\sigma(N)}(\vb{r}_N)\phi_{\tau(N)}(\vb{r}_N)\dd\vb{r}_N \\
  =
  \braket{\phi_{\sigma(1)}}{\phi_{\tau(1)}}
  \cdot \ldots \cdot
  \intl_{\br^3}
  \phi_{\sigma(i)}(\vb{r}_i)H(\vb{r}_i)\phi_{\tau(i)}(\vb{r}_i)\dd\vb{r}_i
  \cdot \ldots \cdot
  \braket{\phi_{\sigma(N)}}{\phi_{\tau(N)}} \\
  = \delta_{\sigma\tau} \intl_{\br^3} \phi_{\sigma(i)}(\vb{r}_i) H(\vb{r}_i)
  \phi_{\tau(i)}(\vb{r}_i) \dd\vb{r}_i,
\end{multline}
since $\braket{\phi_i}{\phi_j} = \delta_{ij}$ by hypothesis, and the other integral is a little more tricky to analyze, but the result is:
\begin{multline} \label{two_electrons}
  \intl_{\br^{3N}}
  \phi_{\sigma(1)}(\vb{r}_1) \ldots \phi_{\sigma(N)}(\vb{r}_N)
  \frac{1}{\vb{r}_{ik}}
  \phi_{\tau(1)}(\vb{r}_1) \ldots \phi_{\tau(N)}(\vb{r}_N)
  \dd\vb{r}_1 \ldots \dd\vb{r}_N \\
  =
  \braket{\phi_{\sigma(1)}}{\phi_{\tau(1)}}
  \cdot \ldots \cdot
  \qty(\intl_{\br^6} \phi_{\sigma(i)}(\vb{r}_i)\phi_{\sigma(k)}(\vb{r}_k)
  \frac{1}{\vb{r}_{ik}}
  \phi_{\tau(i)}(\vb{r}_i)\phi_{\tau(k)}(\vb{r}_k)
  \dd\vb{r}_i\dd\vb{r}_k)
  \cdot \ldots \cdot
  \braket{\phi_{\sigma(N)}}{\phi_{\tau(N)}} \\
  =
  \begin{cases}
    \intl_{\br^6}
    \phi_{\sigma(i)}(\vb{r}_i)\phi_{\sigma(k)}(\vb{r}_k)
    \frac{1}{\vb{r}_{ik}}
    \phi_{\tau(i)}(\vb{r}_i)\phi_{\tau(k)}(\vb{r}_k) \dd\vb{r}_i\dd\vb{r}_k,
    & \text{if $\sigma = \tau$ in $\qty{1, \ldots, N}$} \ssm \qty{i, k}, \\
    0,
    & \text{otherwise}.
  \end{cases}
\end{multline}
Let us now define and analyze
\begin{equation}
  \Xi \ceq \sum_{\sigma \in S_N} \sum_{i=1}^N
  \intl_{\br^3} \phi_{\sigma(i)}(\vb{r}_i)
  H(\vb{r}_i)
  \phi_{\sigma(i)}(\vb{r}_i) \dd\vb{r}_i
\end{equation}
and
\begin{equation} \label{upsilon}
  \Upsilon \ceq \frac{1}{2} \sum_{\sigma \in S_N} \sum_{^{i,k=1}_{k \neq i}}^N
  \intl_{\br^6}
  \phi_{\sigma(i)}(\vb{r}_i)\phi_{\sigma(k)}(\vb{r}_k) \frac{1}{\vb{r}_{ik}}
  \qty(\phi_{\sigma(i)}(\vb{r}_i)\phi_{\sigma(k)}(\vb{r}_k)
  - \braket{\gamma_{\sigma(k)}}{\gamma_{\sigma(i)}}\phi_{\sigma(k)}(\vb{r}_i)\phi_{\sigma(i)}(\vb{r}_k)) \dd\vb{r}_i\dd\vb{r}_k
\end{equation}
because, using \ref{one_electron} and \ref{two_electrons}, we have that
\begin{multline}
  \frac{1}{N!} \sum_{\sigma,\tau \in S_N} \sign(\sigma \circ \tau)
  \braket{\gamma_{\sigma(1)}}{\gamma_{\tau(1)}}
  \ldots
  \braket{\gamma_{\sigma(N)}}{\gamma_{\tau(N)}} \\
  \intl_{\br^{3N}}
  \phi_{\sigma(1)}(\vb{r}_1) \ldots \phi_{\sigma(N)}(\vb{r}_N)
  \what{H}_{el}
  \phi_{\tau(1)}(\vb{r}_1) \ldots \phi_{\tau(N)}(\vb{r}_N)
  \dd\vb{r}_1 \ldots \dd\vb{r}_N
  = \frac{1}{N!}(\Xi + \Upsilon).
\end{multline}
First observe that
\begin{equation}
  \intl_{\br^3} \phi(\vb{r}_i) H(\vb{r}_i) \phi(\vb{r}_i) \dd\vb{r}_i
  =
  \intl_{\br^3} \phi(\vb{r}_j) H(\vb{r}_j) \phi(\vb{r}_j) \dd\vb{r}_j
\end{equation}
for $i,j=1,\ldots,N$.
Consequently,
\begin{align}
  \sum_{\sigma \in S_N}
  \intl_{\br^3} \phi_{\sigma(j)}(\vb{r}_j)H(\vb{r}_j)\phi_{\sigma(j)}(\vb{r}_j) \dd\vb{r}_j
  & = \sum_{\sigma \in S_N}
    \intl_{\br^3} \phi_{\sigma(j)}(\vb{r})H(\vb{r})\phi_{\sigma(j)}(\vb{r}) \dd\vb{r} \\
  & = \sum_{j=1}^N (N-1)! \intl_{\br^3} \phi_j(\vb{r}) H(\vb{r})
    \phi_j(\vb{r}) \dd{\vb{r}}
\end{align}
and, therefore,
\begin{equation}
  \Xi = \sum_{i=1}^N \sum_{j=1}^N (N-1)!
  \intl_{\br^3} \phi_j(\vb{r}) H(\vb{r}) \phi_j(\vb{r}) \dd\vb{r}
  = N! \sum_{i=1}^N \intl_{\br^3} \phi_i(\vb{r}) H(\vb{r}) \phi_i(\vb{r}) \dd\vb{r}.
\end{equation}
Let us now analyze $\Upsilon$.
First observe that
\begin{multline}
  \sum_{\sigma \in S_N} \intl_{\br^6}
  \phi_{\sigma(i)}(\vb{r}_i)\phi_{\sigma(k)}(\vb{r}_k)
  \frac{1}{\vb{r}_{ik}}
  \qty(\phi_{\sigma(i)}(\vb{r}_i)\phi_{\sigma(k)}(\vb{r}_k)
    - \braket{\gamma_{\sigma(k)}}{\gamma_{\sigma(i)}}
    \phi_{\sigma(k)}(\vb{r}_i)\phi_{\sigma(i)}(\vb{r}_k))
  \dd\vb{r}_i\dd\vb{r}_k \\
  = \sum_{l=1}^N \sum_{\sigma \in (S_{\qty{1, \ldots, N} \ssm \qty{l}})}
  \intl_{\br^6}
  \phi_l(\vb{r}_i)\phi_{\sigma(k)}(\vb{r}_k)
  \frac{1}{\vb{r}_{ik}}
  \qty(\phi_l(\vb{r}_i)\phi_{\sigma(k)}(\vb{r}_k)
    - \braket{\gamma_{\sigma(k)}}{\gamma_l}
    \phi_{\sigma(k)}(\vb{r}_i)\phi_l(\vb{r}_k))
  \dd\vb{r}_i\dd\vb{r}_k.
\end{multline}
However, if we fix $l$ and compute the integral summing over $\sigma \in S_{\qty{1, \ldots, N} \ssm \qty{l}}$, we obtain
\begin{equation}
  (N-2)! \sum_{^{j=1}_{j \neq l}}^N \intl_{\br^6} \phi_l(\vb{r}_i)\phi_j(\vb{r}_k)
  \frac{1}{\vb{r}_{ik}} \qty(\phi_l(\vb{r}_i)\phi_j(\vb{r}_k)
    - \braket{\gamma_j}{\gamma_l}\phi_j(\vb{r}_i)\phi_l(\vb{r}_k))
  \dd\vb{r}_i\dd\vb{r}_k.
\end{equation}
If we substitute this in the original integral \ref{upsilon}, we obtain
\begin{equation}
  \frac{(N-2)!}{2} \sum_{i=1}^N \sum_{l=1}^N \sum_{^{k=1}_{k \neq i}}^N
  \sum_{^{j=1}_{j \neq l}}^N \intl_{\br^6} \phi_l(\vb{r}_i)\phi_j(\vb{r}_k)
  \frac{1}{\vb{r}_{ik}} \qty(\phi_l(\vb{r}_i)\phi_j(\vb{r}_k)
  - \braket{\gamma_j}{\gamma_l}\phi_j(\vb{r}_i)\phi_l(\vb{r}_k))
  \dd\vb{r}_i\dd\vb{r}_k.
\end{equation}
Now observe that, if we fix $i$, $l$ and $j$, the integral does not change if we ``ignore'' the variable $\vb{r}_k$ we are integrating over.
So, switching $\vb{r}_k$ for $\vb{r}_j$, the summation over $k$ vanishes and, since this summation appears $N-1$ times, we obtain
\begin{equation}
  \frac{(N-2)!}{2} \sum_{i=1}^N \sum_{l=1}^N \sum_{^{j=1}_{j \neq l}}^N (N-1)
  \intl_{\br^6} \phi_l(\vb{r}_i)\phi_j(\vb{r}_j) \frac{1}{\vb{r}_{ij}}
  \qty(\phi_l(\vb{r}_i)\phi_j(\vb{r}_j)
  - \braket{\gamma_j}{\gamma_l}\phi_j(\vb{r}_i)\phi_l(\vb{r}_j))
  \dd\vb{r}_i\dd\vb{r}_j.
\end{equation}
The same is true for $i$: if we fix the other variables, the integral does not depend on $\vb{r}_i$.
However, this time the integral appears $N$ times and we obtain
\begin{equation}
  \frac{(N-2)!}{2} \sum_{l=1}^N \sum_{^{j=1}_{j \neq l}}^N N(N-1)
  \intl_{\br^6} \phi_l(\vb{r}_l)\phi_j(\vb{r}_j) \frac{1}{\vb{r}_{lj}}
  \qty(\phi_l(\vb{r}_l)\phi_j(\vb{r}_j)
  - \braket{\gamma_j}{\gamma_l}\phi_j(\vb{r}_l)\phi_l(\vb{r}_j))
  \dd\vb{r}_l\dd\vb{r}_j.
\end{equation}
Finally, if we write the variables of integration as $\vb{r}_1$ and $\vb{r}_2$, if we factor out $N(N-1)$, and if we go back to indices $i$ and $j$, we obtain
\begin{equation}
  \Upsilon = \frac{N!}{2} \sum_{i=1}^N \sum_{^{j=1}_{j \neq i}}^N
  \intl_{\br^6} \phi_{i}(\vb{r}_1)\phi_{j}(\vb{r}_2) \frac{1}{\vb{r}_{12}}
  \qty(\phi_{i}(\vb{r}_1)\phi_{j}(\vb{r}_2)
  - \braket{\gamma_j}{\gamma_i}\phi_{j}(\vb{r}_1)\phi_{i}(\vb{r}_2))
  \dd{\vb{r}_1}\dd{\vb{r}_2}.
\end{equation}
Therefore, summing this with the first integral and noticing that we do not need the restriction $i \neq j$ anymore, we obtain the final expression for the energy:
{\small\begin{equation} \label{final_energy}
  \sum_{i=1}^N \intl_{\br^3} \phi_{i}(\vb{r})
  H(\vb{r})
  \phi_{i}(\vb{r})
  \dd\vb{r}
  +
  \frac{1}{2} \sum_{i,j=1}^N \intl_{\br^6}
  \phi_{i}(\vb{r}_1)\phi_{j}(\vb{r}_2)
  \frac{1}{\vb{r}_{12}}
  \qty(\phi_{i}(\vb{r}_1)\phi_{j}(\vb{r}_2)
  - \braket{\gamma_j}{\gamma_i}\phi_{j}(\vb{r}_1)\phi_{i}(\vb{r}_2))
  \dd\vb{r}_1\dd\vb{r}_2.
\end{equation}}%
To conclude, observe that, if we write $\phi_i$ as a linear combination $c_{1j}^{\gamma}\psi_1^{\gamma} + \ldots + c_{nj}^{\gamma}\psi_{d_{\gamma}}^{\gamma}$, being $\gamma \in \qty{\alpha, \beta}$, then we can expand the formula above and obtain the desired expression for the energy:
\begin{multline}
  \sum_{\mu=1}^{\NA} \sum_{i,j=1}^{d_{\alpha}} c_{i\mu}^{\alpha}c_{j\mu}^{\alpha}
  \intl_{\br^3} \psi_i^{\alpha}(\vb{r})H(\vb{r})\psi_j^{\alpha}(\vb{r}) \dd\vb{r}
  +
  \sum_{\mu=1}^{\NB} \sum_{i,j=1}^{d_{\beta}} c_{i\mu}^{\beta}c_{j\mu}^{\beta}
  \intl_{\br^3} \psi_i^{\beta}(\vb{r})H(\vb{r})\psi_j^{\beta}(\vb{r}) \dd\vb{r} \\
  +
  \frac{1}{2} \sum_{\mu,\nu=1}^{\NA} \sum_{i,j,k,l=1}^{d_{\alpha}}
  c_{i\mu}^{\alpha}c_{j\nu}^{\alpha}c_{k\mu}^{\alpha}c_{l\nu}^{\alpha}
  \intl_{\br^6} \psi_i^{\alpha}(\vb{r}_1)\psi_j^{\alpha}(\vb{r}_2)
  \frac{1}{\vb{r}_{12}}
  \qty(\psi_k^{\alpha}(\vb{r}_1)\psi_l^{\alpha}(\vb{r}_2)
  - \psi_l^{\alpha}(\vb{r}_1)\psi_k^{\alpha}(\vb{r}_2))
  \dd\vb{r}_1\dd\vb{r}_2 \\
  + \frac{1}{2} \sum_{\mu,\nu=1}^{\NB} \sum_{i,j,k,l=1}^{d_{\alpha}}
  c_{i\mu}^{\beta}c_{j\nu}^{\beta}c_{k\mu}^{\beta}c_{l\nu}^{\beta}
  \intl_{\br^6} \psi_i^{\beta}(\vb{r}_1)\psi_j^{\beta}(\vb{r}_2)
  \frac{1}{\vb{r}_{12}}
  \qty(\psi_k^{\beta}(\vb{r}_1)\psi_l^{\beta}(\vb{r}_2)
  - \psi_l^{\beta}(\vb{r}_1)\psi_k^{\beta}(\vb{r}_2))
  \dd\vb{r}_1\dd\vb{r}_2 \\
  + \sum_{\mu=1}^{\NA} \sum_{\nu=1}^{\NB}
  \sum_{i,k=1}^{d_{\alpha}} \sum_{j,l=1}^{d_{\beta}}
  c_{i\mu}^{\alpha}c_{j\nu}^{\beta}c_{k\mu}^{\alpha}c_{l\nu}^{\beta}
  \intl_{\br^6} \psi_i^{\alpha}(\vb{r}_1)\psi_j^{\beta}(\vb{r}_2)
  \frac{1}{\vb{r}_{12}}
  \psi_k^{\alpha}(\vb{r}_1)\psi_l^{\beta}(\vb{r}_2)
  \dd\vb{r}_1\dd\vb{r}_2.
\end{multline}
Notice that in the last integral we do not have the term $\braket{\gamma_j}{\gamma_i}\phi_j(\vb{r}_1)\phi_i(\vb{r}_2)$ because the spin is always swapped, \ie, $\braket{\gamma_j}{\gamma_i} = \braket{\alpha}{\beta}$ or $\braket{\gamma_j}{\gamma_i} = \braket{\beta}{\alpha}$, which is $0$ in both cases.


\chapter{Riemannian Optimization} \label{optimization}

\epigraph{Premature optimization is the root of all evil.}{Donald Knuth}

\section{Introduction}

In this chapter we will present Riemannian Optimization to the reader and then implement three algorithms to solve the Hartree--Fock Method: Gradient Descent, Newton--Raphson and Conjugate Gradient.
The results we obtained and a comparison with other methods can be found in the Results section (\ref{results}).
Also, the main references for this chapter are \cite{boumal2022, edelman1998, smith2014}.

\subsection{Riemannian Optimization}

Riemannian Optimization is a subfield of Mathematical Optimization that aims to generalize optimization methods from Euclidean spaces to Riemannian manifolds.
Consequently, in theory Riemannian Optimization allows us to tackle a wider number of problems, but this should be taken with a grain of salt because in practice one usually starts with a traditional constrained optimization problem instead of a function defined in an explicit manifold and, in order to implement a Riemannian algorithm, it is necessary to encode the constraints as manifold embedded in an Euclidean space or as a quotient of an embedded manifold.
This encoding can be hard to obtain, but if one has a problem defined in an Euclidean space and the constraints are given by a smooth function, then the chances that this is actually a problem defined on a Riemannian manifold are very high.
To see why this happens, consider the following problem:
\begin{equation} \label{min_problem}
  \text{minimize $\ol{f}:\br^d \to \br$ subject to $c(x) = 0$, being $c:\br^d \to \br^k$.}
\end{equation}
Then, if $c$ satisfies some hypothesis,\footnote{Since we are considering just smooth manifolds in this work, one of the hypothesis is that $c$ should be smooth and another one is that $0$ should be a regular value so we can use the Preimage Theorem \ref{preimage_theorem}.} $X = c^{-1}(0)$ is a manifold and we obtain the following equivalent optimization problem defined on a manifold:
\begin{equation} \label{riemannian_problem}
  \text{minimize $f:X \to \br$, being $f = \ol{f}|_X$}.
\end{equation}
A simple example is the sphere $X = \bs^{d-1}$, which can be encoded as $c(x) = \braket{x} - 1$, but the reader can check a vast list of manifolds for which there is already an implementation of many Riemannian algorithms at the website of the Manopt package \cite{boumal2014}.
By the way, it is also possible to consider the case in which there is more than one function encoding the constraints, say $\qty{c_1, \ldots, c_m}$, but then one should be more careful because $\bigcap_{i=1}^m c_i^{-1}(0)$ is not necessarily a manifold.\footnote{If each $c_i^{-1}(0)$ is a manifold and they intersect transversally, then the intersection is a manifold (see \cite[Theorem 6.30]{lee2012}), but otherwise it is more difficult to analyze.}
Now, it should be emphasized that in some sense Riemannian Optimization is more powerful than regular constrained optimization because we can consider manifolds which are not (easily) described by constraint functions and, as already mentioned in the Introduction, we can also think of it as a generalization of unconstrained optimization.
And the Grassmannian, at least in the parametrization we are considering in this work, is one example of a manifold that is not described by constraint functions.

Moving forward, in the present work we will implement the three algorithms described in the seminal papers \cite{edelman1998, smith2014}: Gradient Descent, Newton--Raphson and Conjugate Gradient.
The goal of these methods is to solve optimization problems posed as \ref{riemannian_problem} and the main idea behind them is that they find \emph{stationary points} (also called \emph{critical points}) of $f$ (usually called \emph{cost} or \emph{objective function})\index{cost function}\index{objective function} by iteratively finding directions in which the function is minimized (known as \emph{descent directions})\index{descent direction} and then updating the point by moving in that direction.\footnote{From now on we are always going to consider minimization problems but the exact same ideas works for maximization problems if we multiply the cost function by $-1$.}
One may then ask what are the differences between these methods and the only one is in how they find the descent directions.
At the core, many computational methods (not just optimization methods) are obtained by approximating a function with its Taylor expansion because this allows one to move from a continuous problem to a discrete problem.
In the case of these three algorithms the same happens: all of them can be described using a Taylor expansion of the cost function, as we will see in a moment.
Now, what is common in all of them is that they find the descent direction by solving a system of linear equations and we can describe them using the following blueprint:
\begin{algo}[Blueprint] \label{blueprint}
  Given a manifold $X$, a function $f:X \to \br$ and a starting point $x_0 \in X$, the goal of these algorithms is to construct a sequence $\qty{x_k \in X : k \in \bn}$ that converges to a stationary point by following these two steps:
  \begin{enumerate}

  \item Solve $B_kv_k = y_k$, being $B_k$ a matrix and $v_k, y_k$ vectors.

  \item Update the point by moving in the direction $v_k$.

  \end{enumerate}
  At the end, under some hypothesis on $X$, $f$ and $x_0$, the sequence converges to a local minimum, \ie, a point $x^* \in X$ such that $f(x^*) \leq f(x)$ for $x \in U$, being $U$ an open set containing $x^*$.
\end{algo}

\begin{obs}
  Now is a good time to make some disclaimers.
  Using numerical methods to solve real-world problems is very hard and, although theorems asserting the convergence of the three algorithms we are going to study here exist, the goal of the present work is to implement these algorithms and see how they perform in the problem we are interested in.
  That means we will not prove the convergence of these algorithms, but, again, the proofs exist and the interested reader can check them in the paper \cite{smith2014} or in the book \cite{boumal2022}, which has more in-depth discussion.
  The reader may be skeptical about this approach, but the two main hypothesis of the convergence theorems are satisfied in our case: the manifold we are working with is a complete manifold (Corollary \ref{grassmannian_complete}) and the function we want to minimize is smooth (Theorem \ref{hf_welldefined}).
  A third hypothesis, which is the assumption that the starting point is close to a stationary point so the sequence can converge to it, is very unrealistic because (1) in practice we do not know where a stationary point is and we are actually using these algorithms to find one; and (2) the open set that the starting point should be in to converge to the stationary point is obtained indirectly in the proofs of the convergence theorems.
  That means we know the open set exists, but we do not know, in general, an explicit radius that gives us a ball containing the stationary point and contained in the open set.
  Consequently, the existence of this open set does not help in practice at all.
  So, from now on the reader can assume that the manifolds are complete and that the cost functions are smooth or at least $C^2$ (\ie, the second-order partial derivatives exists and are continuous).
\end{obs}

Getting back to the blueprint, we stated the two steps in a very informal way, so, let us formalize them.
Given $f:X \to \br$ and a retraction $R:TX \to X$, if we define $\what{f} \ceq f \circ R$, then a second-order Taylor expansion of $\what{f}$ around $(x, v) \in TX$ is given by
\begin{equation} \label{taylor}
  \what{f}(x,v)
  \approx f(x) + \braket{\grad{f}(x)}{v}_x + \frac{1}{2}\braket{\hess{f}(x)(v)}{v}_x.
\end{equation}
Now, recall that we want to find a descent direction and, to achieve this, we will try to find a vector $v$ that minimizes the right-hand side, which we will call $m_x$.
So, since
\begin{equation} \label{grad_mx}
  \grad{m_x}(v) = \grad{f}(x) + \hess{f}(x)(v),
\end{equation}
a critical point of $m_x$ should be a solution to
\begin{equation} \label{newton_system}
  \hess{f}(x)(v) = -\grad{f}(x)
\end{equation}
and this is the equation we have to solve in the Newton--Raphson Method.
Observe that in this case $B_k = \hess{f}(x_k)$, $y_k = -\grad{f}(x_k)$ and $x_{k+1} = R(x_k, v_k)$, being $v_k$ the solution to the equation above.
The other methods are variations of this.
In the Gradient Descent, \eg, we replace the Hessian by $B_k = t_k \cdot \Id$, with $t_k > 0$, in which case the system above is trivial and the descent direction is just $-\frac{1}{t_k}\grad{f}(x_k)$.
And in the Conjugate Gradient we solve Equation \ref{newton_system} above using an approximation we will explain in Section \ref{conjugate_gradient}.
It is worth mentioning that $B$ should be positive-definite and self-adjoint because otherwise Equation \ref{grad_mx} does not hold and a minimizer to $m_x$ does not necessarily exist, which already shows a disadvantage of Newton--Raphson, since the Hessian is not necessarily positive-definite when $x_k$ is far from a critical point (in practice we observed that Newton--Raphson diverges when the starting point is far from a critical point).
Another important observation is that, although we described the update using an arbitrary retraction $R$, we will just use the Riemannian exponential, which means the points are updated by moving along a geodesic.
And that is it, these are the key features of Riemannian methods that distinguishes them from regular constrained optimization methods: the descent direction is in the tangent space to the manifold instead of being a direction that may lead the sequence to outside the manifold and updating the point along a geodesic also keeps the sequence in the manifold.
That means both steps respects the (geometry of the) constraints given by the problem.

The following figure illustrates how these algorithms works.
In this figure the manifold is $\gr{1}{3}$, also known as \emph{Real Projective Plane}, and the red curve represents what a geodesic in this manifold looks like.
\begin{center}
  \begin{tikzpicture}
    \draw[semithick] (1, 0) arc (0:180:1);
    \draw[semithick] (0, 0) ellipse (1 and 0.3);
    \draw[semithick] (1, -0.8) arc (0:250:1cm and 0.3cm);
    \draw[semithick] (0.88, -1.2) arc (10:133:0.9cm and 0.3cm);
    \draw[semithick, -] (-0.34, -1.63) to[out=180, in=-50] (-1, -1.39);
    \draw[semithick, -] (-0.34, -1.1) -- (-0.34, -1.63);
    \draw[semithick, -] (1, -0.8) -- (1, -1.5);
    \draw[semithick, -] (-1, -0.8) -- (-1, -1.4);
    \draw[semithick, -] (1, -0.8) to[out=270, in=-4] (-0.35, -1.63);
    \draw[semithick, -] (1, -1.5) to[out=190, in=-5] (-0.35, -1.08);
    \draw (0, 1.3) node {$x_k$};
    \filldraw (0, 1) circle (1pt);
  \end{tikzpicture} \qquad
  \begin{tikzpicture}
    \draw[semithick] (1, 0) arc (0:180:1);
    \draw[semithick] (0, 0) ellipse (1 and 0.3);
    \draw[semithick] (1, -0.8) arc (0:250:1cm and 0.3cm);
    \draw[semithick] (0.88, -1.2) arc (10:133:0.9cm and 0.3cm);
    \draw[semithick, -] (-0.34, -1.1) -- (-0.34, -1.63);
    \draw[semithick, -] (-0.34, -1.63) to[out=180, in=-50] (-1, -1.39);
    \draw[semithick, -] (1, -0.8) -- (1, -1.5);
    \draw[semithick, -] (-1, -0.8) -- (-1, -1.4);
    \draw[semithick, -] (1, -0.8) to[out=270, in=-4] (-0.35, -1.63);
    \draw[semithick, -] (1, -1.5) to[out=190, in=-5] (-0.35, -1.08);
    \draw (0, 1.3) node {$x_k$};
    \filldraw (0, 1) circle (1pt);
    \draw[semithick, -] (-1.1, 0.7) -- (-0.9, 1.1);
    \draw[semithick, -] (-0.9, 1.1) -- (1.1, 1.1);
    \draw[semithick, -] (-1.1, 0.7) -- (0.93, 0.7);
    \draw[semithick, -] (0.93, 0.7) -- (1.1, 1.1);
    \draw[red, ->] (0, 1) -- (0.3, 0.8);
    \draw (0, 0.83) node {\scriptsize $v_k$};
  \end{tikzpicture} \qquad
  \begin{tikzpicture}
    \draw[semithick] (1, 0) arc (0:180:1);
    \draw[semithick] (0, 0) ellipse (1 and 0.3);
    \draw[semithick] (1, -0.8) arc (0:250:1cm and 0.3cm);
    \draw[semithick] (0.88, -1.2) arc (10:133:0.9cm and 0.3cm);
    \draw[semithick, -] (-0.34, -1.63) to[out=180, in=-50] (-1, -1.39);
    \draw[semithick, -] (-0.34, -1.1) -- (-0.34, -1.63);
    \draw[semithick, -] (1, -0.8) -- (1, -1.5);
    \draw[semithick, -] (-1, -0.8) -- (-1, -1.4);
    \draw[semithick, -] (1, -0.8) to[out=270, in=-4] (-0.35, -1.63);
    \draw[semithick, -] (1, -1.5) to[out=190, in=-5] (-0.35, -1.08);
    \draw[semithick, red] (0.8, -0.18) arc (0:80:1cm and 1.2cm);
    \draw (0, 1.3) node {$x_k$};
    \filldraw (0, 1) circle (1pt);
    \draw (1.2, -0.4) node {$x_{k+1}$};
    \filldraw (0.8,  -0.18) circle (1pt);
  \end{tikzpicture}
\end{center}

Some of the motivations for using Riemannian Optimization are given by the fact that respecting the geometry of an optimization problem can lead to the design of more principled algorithms and, as practitioners observed along the years \cite[Chapter 1]{boumal2022}, to better results.
Other important, but less quantitative, motivations are: (1) these algorithms are aesthetically pleasing; (2) they can provide deeper insights about numerical methods that would be hard to obtain otherwise or would be unthinkable in the Euclidean constrained optimization framework.
A good example of the latter in Quantum Chemistry is done in \cite{aoto2022} and it is one of the doors that this work aims to open, \ie, we would like to spark the curiosity of the reader to look for geometry in unexpected fields.
Some other examples of optimization problems defined on Riemannian manifolds can be found in \cite[Chapter 2]{boumal2022}, but it should be mentioned that implementing Riemannian algorithms can be quite tricky because, since the manifolds are usually embedded in an Euclidean space or are quotient of embedded manifolds, the distinction between working in the manifold and in the extrinsic Euclidean space can become very blurred.
That means a good grasp of geometry is essential to implement the algorithms correctly.

\section{Algorithms} \label{algorithms}

The goal now is to be quite explicit on how to implement the three algorithms we are interested, but in the present work we will consider just the case in which the manifold is the product of two Grassmannians because this is the manifold we are going to use in the Hartree--Fock Method.
It should be mentioned, though, that the Manopt package \cite{boumal2014} was a source of inspiration for the design of the pseudocodes we will see and these pseudocodes can be adapted to work with other manifolds.

An advantage of writing the pseudocodes for arbitrary Grassmannians and cost functions instead of considering just the Hartree--Fock Method is that the reader can just plug the partial derivatives of any cost function into the pseudocode to obtain an optimization procedure to their particular case (as long as the cost function is defined in the product of two Grassmannians, of course).
So, for example, the reader interested in the Density Functional Theory framework can extend the pseudocodes presented here to work with any functional as long as the partial derivatives are known and another interesting optimization problem defined on the product of two Grassmannians is the search for the Slater determinant that minimizes the distance to the Full Configuration Interaction wave function.
There are some very interesting physical and chemical interpretations behind this problem, see \cite{aoto2020, ding2020}.

Moving to the practical aspects, recall that we are considering the Grassmannian as a quotient of the Stiefel manifold, \ie, given the manifold
\begin{equation}
  \st{N}{d} \ceq \qty{C \in \mat{d}{N} : C^TSC = \Id_N},
\end{equation}
being $S$ a symmetric and positive-definite matrix, and given the orthogonal group
\begin{equation}
  O(N) \ceq \qty{M \in \mat{N}{N} : M^{\top}M = MM^{\top} = \Id_N},
\end{equation}
the Grassmannian is defined as
\begin{equation}
  \gr{N}{d} \ceq \st{N}{d} / O(N).
\end{equation}
Another important concept we will need is the lifting of functions defined in the Grassmannian: given $f:\gr{N}{d} \to \br$, the lifting of $f$ is a function $\ol{f}:\st{N}{d} \to \br$ such that $\ol{f} = f \circ \pi$, being $\pi:\st{N}{d} \to \gr{N}{d}$ the canonical projection.
We also need to consider extensions of $\ol{f}$ to the Euclidean space $\mat{d}{N}$, \ie, functions $\oll{f}:\mat{d}{N} \to \br$ such that $\oll{f}(C) = \ol{f}(C)$ for every $C \in \st{N}{d}$.
We will use these three different notations because it allows us to keep track of the space we are working with (Grassmannian, Stiefel or the Euclidean space).

\subsection{Gradient Descent}

As already mentioned in page \pageref{newton_system}, the descent direction of the Gradient Descent is obtained by choosing a positive multiple of the identity instead of the Hessian in the Taylor expansion.
Therefore, the explicit algorithm for an arbitrary manifold is:
\begin{algo}[Riemannian Gradient Descent] \label{rgd}
  Given a Riemannian manifold $X$, a cost function $f:X \to \br$ and a starting point $x_0 \in X$, do the following:
  \begin{enumerate}
    
  \item Compute $\grad{f}(x_k)$.

  \item Update the point according to
    $x_{k+1} \ceq \exp_{x_k}(-t_k\grad{f}(x_k))$ for some $t_k > 0$.

  \end{enumerate}
\end{algo}
There are many strategies to choose $t_k$ and, from a theoretical perspective, the best value is the minimum of the function $t \mapsto f(\exp_{x_k}(-t\grad{f}(x_k)))$.
However, it is extremely difficult to find an explicit solution to this optimization problem when $f$ and $\exp$ are complicated functions, as is our case.
Some of the state-of-the-art techniques to obtain $t_k$ in the Euclidean case, such as Adam, Adagrad, Amsgrad etc, are not so straightforward to adapt to the Riemannian framework (see \cite{becigneul2018}) and we decided to avoid them.
Having said that, we adopted the very common and simple strategy of using small constant values for $t_k$ (the values we have used are in the Results section \ref{results}).
So, now that we have an explicit description of the algorithm, let us see how it works for the product of Grassmannians.

Considering $X = \gr{N_1}{d_1} \x \gr{N_2}{d_2}$ and $f:X \to \br$, fixing $([C_1], [C_2]) \in X$, defining $f_1 \ceq f(\cdot, C_2)$ and $f_2 \ceq f(C_1, \cdot)$, we know from Example \ref{grad_product_grassmannians} that
\begin{equation} \label{gradient}
  \grad{f}([C_1], [C_2])
  = \qty(\proj_1S_1^{-1}\grad{\oll{f}_1}(C_1, C_2), \
  \proj_2S_2^{-1}\grad{\oll{f}_2}(C_1, C_2)),
\end{equation}
being $\proj_i = \Id_{d_i} - C_iC_i^{\top}S_i$ and $\grad\oll{f}_i$ the matrix representation of the partial derivatives of $\oll{f}_i$.
And that is it for the gradient, we already have an expression that can be implemented.
The second step of the algorithm is also straightforward to implement, one just need to recall from Example \ref{exp_grassmannian} that
\begin{equation}
  \exp_{[C_i]}(t\eta_i) = \qty(C_iV_i^{\top}\cos(tD_i) + O_iU_i\sin(tD_i))V_i,
\end{equation}
being $O_i$ a matrix such that $O_i^{\top}S_iO_i = \Id_{d_i}$, and $O_i^{-1}\eta_i = U_iD_iV_i$ a thin Singular Value Decomposition (tSVD).
To conclude, we need a stopping criterion because we cannot handle infinite sequences in practice.
Since we want the gradient to be $0$ at the end, one criterion adopted was checking if the norm of the gradient is (close to) $0$, and the other one was checking if the difference of the cost function computed in two consecutive iterations did not change (much).
So, since the norm we are using is the one induced by the inner product, which was defined in Example \ref{grassmannian_riemannian} for the Grassmannian, we have:
\begin{equation}
  \norm{(\eta_1, \eta_2)}_{\Gr}
  = \sqrt{\tr(\eta_1^{\top}S_1\eta_1) + \tr(\eta_2^{\top}S_2\eta_2)}.
\end{equation}
That is it, we can implement the first algorithm:\footnote{This and all the other pseudocodes were written in a syntax very similar to Python.}
\begin{algbox}{Riemannian Gradient Descent (RGD)}{RGD}
  input: $max\_iter, \ step\_size, \ C_i^{(0)}, \ S_i, \ S_i^{-1}, \ O_i, \ O_i^{-1}, \ i=1,2$
  set: $k, \ tol\_grad, \ tol\_val, \ E_{prev}, \ E^{(0)}, \ E_{diff}, \ \riem{G}^{(0)} \ = \ 0, \ 10^{-8}, \ 10^{-10}, \ \infty, \ 0, \ 1, \ 1$
  precompute: $\Id_{d_1}, \ \Id_{d_2}$

  def rip($(\eta_1, \ \eta_2), \ (\mu_1, \ \mu_2)$): # Riemannian inner product
      return $\tr(\eta_1^{\top}S_1\mu_1) + \tr(\eta_2^{\top}S_2\mu_2)$

  def riemannian_norm($\eta_1, \ \eta_2$):
      return sqrt(rip($(\eta_1, \ \eta_2), \ (\eta_1, \ \eta_2)$))

  def geodesic($C, \ \eta, \ O, \ O^{-1}, \ step\_size$):
      $U, D, V$ = tSVD($O^{-1}\eta$)
      return $(CV^{\top}\cos(step\_size \cdot D) + OU\sin(step\_size \cdot D))V$

  def riemannian_gradient($C_1, \ C_2, \ \euc{G}_1, \ \euc{G}_2$):
      return $(\Id_{d_1} - C_1C_1^{\top}S_1)S_1^{-1}\euc{G}_1, \ (\Id_{d_2} - C_2C_2^{\top}S_2)S_2^{-1}\euc{G}_2$
  
  while $k \leq max\_iter$ and $tol\_grad < \riem{G}^{(k)}$ and $tol\_val < E_{diff}$:
      $E^{(k)}$ = f$([C_1^{(k)}], \ [C_2^{(k)}])$
      $\euc{G}_1^{(k)}, \ \euc{G}_2^{(k)}$ = euclidean_gradient($C_1^{(k)}, \ C_2^{(k)}$)
      $\riem{G}_1^{(k)}, \ \riem{G}_2^{(k)}$ = riemannian_gradient($C_1^{(k)}, \ C_2^{(k)}, \ \euc{G}_1^{(k)}, \ \euc{G}_2^{(k)})$
      $C_1^{(k+1)}$ = geodesic($C_1^{(k)}, \ -\riem{G}_1^{(k)}, \ O_1, \ O_1^{-1}, \ step\_size$)
      $C_2^{(k+1)}$ = geodesic($C_2^{(k)}, \ -\riem{G}_2^{(k)}, \ O_2, \ O_2^{-1}, \ step\_size$)
      $\riem{G}^{(k)}$ = riemannian_norm($\riem{G}_1^{(k)}, \ \riem{G}_2^{(k)}$)
      $E_{diff} \ = \ |E_{prev} - E^{(k)}|$; $E_{prev} \ = \ E^{(k)}$; $k \ = \ k + 1$
\end{algbox}

Before we move to Newton--Raphson, let us mention some implementation details: the maximum number of iterations depends a lot on the cost function and we used $1000$ in our implementation.
The tolerances are also very dependent on the accuracy one wants to achieve and we have used some canonical values found throughout the literature.
Now, the reader may have noticed that there are two functions we left in blank: the cost function $f$ and the \verb|euclidean_gradient|.
These are the two functions that depends on the specific problem one is trying to solve and the reader can implement them separately and leave the rest of the pseudocode as it is to obtain a full-fledged optimization algorithm.
It is also worth mentioning that the \verb|euclidean_gradient| is the function that computes the usual partial derivatives of $\oll{f}$, which can be done analytically, using automatic differentiation or using finite differences.
That is it for the RGD.

\subsection{Newton--Raphson} \label{rnr}

We already described how the Newton--Raphson Method works previously (page \pageref{newton_system}), but let us state it precisely in here:
\begin{algo}[Riemannian Newton--Raphson Method]
  Given a Riemannian manifold $X$, a function $f:X \to \br$ and a starting point $x_0 \in X$, do the following:
  \begin{enumerate}
    
  \item Solve $\hess{f}(x_k)(v_k) = -\grad{f}(x_k)$ for $v_k \in T_{x_k}X$.

  \item Update the point according to $x_{k+1} \ceq \exp_{x_k}(v_k)$.

  \end{enumerate}
\end{algo}
Now, since we already know from the Gradient Descent how to compute the gradient and the exponential, what is left is the Hessian.
So, recall that in Example \ref{hessian_product_grassmannians} we obtained the following expression for the Hessian of $f:\gr{N_1}{d_1} \x \gr{N_2}{d_2} \to \br$:
{\footnotesize
  \begin{align} \label{full_hessian}
    \begin{split}
      \hess{f}([C_1], [C_2])
      &:T_{[C_1]}\gr{N_1}{d_1} \x T_{[C_2]}\gr{N_2}{d_2}
        \to T_{[C_1]}\gr{N_1}{d_1} \x T_{[C_2]}\gr{N_2}{d_2} \\
      (\eta_1, \eta_2)
      \mapsto \Big(
      & \proj_1S_1^{-1}\big(\hess_{11}\oll{f}(C_1, C_2)\eta_1
        + \hess_{12}\oll{f}(C_1, C_2)\eta_2\big)
        - \eta_1C_1^{\top}\grad\oll{f}_1(C_1, C_2), \\
      & \proj_2S_2^{-1}\big(\hess_{21}\oll{f}(C_1, C_2)\eta_1 
        + \hess_{22}\oll{f}(C_1, C_2)\eta_2\big)
        - \eta_2C_2^{\top}\grad\oll{f}_2(C_1, C_2)
        \Big),
    \end{split}
  \end{align}
}%
being $\proj_i = \Id_{d_i} - C_iC_i^{\top}S_i$ and $\hess_{ij}\oll{f}$ the matrix with the second-order partial derivatives of $\oll{f}$ with respect to $C_j$ first and then $C_i$.
However, the above expression assumes that the Hessian is in a matrix representation that is hard to use in practice for two reasons: first, since we want to solve a system of equations computationally, we have to use a vector representation of $\eta_i$ because scientific computing packages usually requires equations to be in the format $Ax = b$, being $A$ a matrix and $x$ and $b$ vectors.
The second reason is that it can be very hard to obtain a matrix representation of $\hess_{ij}\oll{f}$ with the correct dimensions because, if we look carefully to the formula above, we have $\eta_j \in \mat{N_j}{d_j}$ while $\hess_{ij}\oll{f}$ should be in $\mat{N_id_i}{N_jd_j}$.
Sometimes it is possible to obtain $\hess_{ij}\oll{f} \in \mat{m}{N_i}$, being $m$ any positive natural number, but, since we could not find such a representation for our particular case and since at the end we will need to vectorize the formula above for the first reason, let us do that.

Assuming that $\hess_{ij}\oll{f} \in \mat{N_id_i}{N_jd_j}$ and $\eta_i \in \br^{d_iN_i}$, the term
\begin{equation}
  \hess_{ii}\oll{f}(C_1, C_2)\eta_i + \hess_{ij}\oll{f}(C_1, C_2)\eta_j
\end{equation}
is already correct but we have to take care of
\begin{equation}
  \eta_iC_i^{\top}\grad\oll{f}_i(C_1, C_2)
\end{equation}
and of the outer projection, since $\proj_iS_i^{-1} \in \mat{d_i}{d_i}$.
So, to state explicitly what we have to do, the goal is to find an expression that makes the linear transformations
\begin{align} \label{matrix_rep_grad}
  \begin{split}
    \mat{d_i}{N_i} & \to \mat{d_i}{N_i} \\
    \eta_i & \mapsto \eta_iC_i^{\top}\grad\oll{f}_i(C_1, C_2)
  \end{split}
\end{align}
and
\begin{align} \label{matrix_rep_proj}
  \begin{split}
    \mat{d_i}{N_i} & \to \mat{d_i}{N_i} \\
    \eta_i & \mapsto \proj_iS_i^{-1}\eta_i
  \end{split}
\end{align}
act on $\eta_i \in \br^{d_iN_i}$ instead of $\eta_i \in \mat{d_i}{N_i}$.
The idea to obtain this expression is to recall from Linear Algebra that the matrix representation of a linear transformation is given by computing it in a vector of a basis and then setting the resulting vector as a column of the matrix.
So, we opted for stating just the final result in here, but the full computation is in Appendix \ref{matrix_ids}.
With that said, \ref{matrix_rep_grad} becomes
\begin{align}
  \begin{split}
    \br^{d_iN_i} & \to \br^{d_iN_i} \\
    \eta_i & \mapsto
             \qty(\grad\oll{f}_i(C_1, C_2)^{\top}C_i \otimes \Id_{d_i})\eta_i
  \end{split}
\end{align}
and \ref{matrix_rep_proj} becomes
\begin{align}
  \begin{split}
    \br^{d_iN_i} & \to \br^{d_iN_i} \\
    \eta_i & \mapsto
             \qty(\Id_{N_i} \otimes (\Id_{d_i}
             - C_iC_i^{\top}S_i)S_i^{-1})\eta_i,
  \end{split}
\end{align}
being $\otimes$ the Kronecker product.
Putting all the pieces together, the (semi)final matrix representation of the Hessian is:
\begin{equation}
  \begin{bsmallmatrix}
    (\Id_{N_1} \otimes \proj_1S_1^{-1})\hess_{11}\oll{f}
    - (\grad\oll{f}_1)^{\top}C_1 \otimes \Id_{d_1}
    & (\Id_{N_1} \otimes \proj_1S_1^{-1})\hess_{12}\oll{f} \\
    (\Id_{N_2} \otimes \proj_2S_2^{-1})\hess_{21}\oll{f}
    & (\Id_{d_2} \otimes \proj_2S_2^{-1})\hess_{22}\oll{f}
      - (\grad\oll{f}_2)^{\top}C_2 \otimes \Id_{d_2}
  \end{bsmallmatrix}.
\end{equation}
Observe that this matrix can be built more efficiently if we start with
\begin{equation} \label{generic_euc_hess}
  \begin{bsmallmatrix}
    \hess_{11}\oll{f} & (\hess_{21}\oll{f})^{\top} \\
    \hess_{21}\oll{f} & \hess_{22}\oll{f}
  \end{bsmallmatrix},
\end{equation}
then multiply by
\begin{equation}
  \begin{bsmallmatrix}
    \Id_{N_1} \otimes \proj_1S_1^{-1} & 0_{N_1 \cdot d_1 \x N_2 \cdot d_2} \\
    0_{N_2 \cdot d_2 \x N_1 \cdot d_1} & \Id_{N_2} \otimes \proj_2S_2^{-1}
  \end{bsmallmatrix}
\end{equation}
on the left and, finally, subtract from this product the following:
\begin{equation}
  \begin{bsmallmatrix}
    (\grad_1\oll{f})^{\top}C_1 \otimes \Id_{d_1} & 0_{N_1 \cdot d_1 \x N_2 \cdot d_2} \\
    0_{N_2 \cdot d_2 \x N_1 \cdot d_1} & (\grad_2\oll{f})^{\top}C_2 \otimes \Id_{d_2}
  \end{bsmallmatrix}.
\end{equation}

Now, there is one last problem we have to address: we have actually built an extension of the Hessian that is defined in the whole $\br^{d_1 \cdot N_1 + d_2 \cdot N_2}$ instead of just in the tangent space of $\gr{N_1}{d_1} \x \gr{N_2}{d_2}$.
However, if we recall (Example \ref{grassmannian_riemannian}) that
\begin{equation} \label{gradient_lagrangian}
  T_{[C_1]}\gr{N_1}{d_1} \x T_{[C_2]}\gr{N_2}{d_2}
  = \qty{(\eta_1, \eta_2) : \eta_i \in \mat{d_i}{N_i}, \ C_i^{\top}S_i\eta_i = 0_{N_i}},
\end{equation}
then notice that this space is the kernel of the linear transformation given by
\begin{align} \label{L}
  \begin{split}
    L:\mat{d_1}{N_1} \x \mat{d_2}{N_2} & \to \mat{N_1}{N_1} \x \mat{N_2}{N_2} \\
    (\eta_1, \eta_2) & \mapsto (C_1^{\top}S_1\eta_1, \ C_2^{\top}S_2\eta_2).
  \end{split}
\end{align}
Consequently, Proposition \ref{equivalence_subspace} tells us that solving the main equation of Newton--Raphson, namely,
\begin{equation}
  \hess{f}([C_1], [C_2])(\eta_1, \eta_2) = -\grad{f}([C_1], [C_2]),
\end{equation}
assuming that $\eta_i \in T_{[C_i]}\gr{N_i}{d_i}$, is equivalent to solving the augmented problem
\begin{equation} \label{aug_nt_step}
  \qty(\hess{f}([C_1], [C_2]) \oplus L)(\eta_1, \eta_2)
  = \qty(-\grad{f}([C_1], [C_2]), \ 0_{N_1}, \ 0_{N_2})
\end{equation}
assuming $\eta_i \in \mat{d_i}{N_i}$.
So, vectorizing \ref{L} (the computation is in Appendix \ref{matrix_ids}), we obtain $\Id_{N_i} \otimes C_i^{\top}S_i$ and the final augmented Hessian is
\begin{equation}
  \begin{bsmallmatrix}
    (\Id_{N_1} \otimes \proj_1S_1^{-1})\hess_{11}\oll{f}
    - (\grad\oll{f}_1)^{\top}C_1 \otimes \Id_{n_1}
    & (\Id_{N_1} \otimes \proj_1S_1^{-1})\hess_{21}\oll{f} \\
    (\Id_{N_2} \otimes \proj_2S_2^{-1})\hess_{12}\oll{f}
    & (\Id_{N_2} \otimes \proj_2S_2^{-1})\hess_{22}\oll{f}
      - (\grad\oll{f}_2)^{\top}C_2 \otimes \Id_{d_2} \\
    \Id_{N_1} \otimes C_1^{\top}S_1 & 0_{N_1N_1 \x N_2d_2} \\
    0_{N_2N_2 \x N_1d_1} & \Id_{N_2} \otimes C_2^{\top}S_2
  \end{bsmallmatrix}.
\end{equation}
Now, since vectorizing the right-hand side of \ref{aug_nt_step} is straightforward, we have our second algorithm:
\begin{algbox}{Riemannian Newton--Raphson (RNR)}{RNR}
  input: $max\_iter, \ \ C_i^{(0)}, \ S_i, \ S_i^{-1}, \ O_i, \ O_i^{-1}, \ i=1,2$
  set: $k, \ tol\_grad, \ tol\_val, \ E_{prev}, \ E^{(0)}, \ E_{diff}, \ \riem{G}^{(0)} \ = \ 0, \ 10^{-8}, \ 10^{-10}, \ \infty, \ 0, \ 1, \ 1$
  precompute: $\Id_{N_1}, \ \Id_{N_2}$

  def augment_gradient($\riem{G}_1, \ \riem{G}_2$):
      return $\vstack\qty(\vec(\riem{G}_1), \ \vec(\riem{G}_2), \ 0_{N_1 \cdot N_1 + N_2 \cdot N_2})$

  def augment_hessian($C_1, \ C_2, \ \riem\hess{f}$):
      $\hor_1, \ \hor_2 \ = \ \begin{bsmallmatrix}
\Id_{N_1} \otimes C_1^{\top}S_1 & & 0_{N_1 \cdot N_1 \x d_2 \cdot N_2}
\end{bsmallmatrix}, \
\begin{bsmallmatrix}
  0_{N_2 \cdot N_2 \x d_1 \cdot N_1} & & \Id_{N_2} \otimes C_2^{\top}S_2
\end{bsmallmatrix}$
      return $\vstack(\riem\hess{f}, \ \hor_1, \ \hor_2)$

  def riemannian_hessian($C_1, \ C_2, \ \euc{G}_1, \ \euc{G}_2, \ \euc\hess{f}$):
      $\riem\hess{f} \ = \
\begin{bsmallmatrix}
\Id_{N_1} \otimes (\Id_{d_1} - C_1C_1^{\top}S_1)S_1^{-1} & 0_{d_1 \cdot N_1 \x d_2 \cdot N_2} \\
0_{d_2 \cdot N_2 \x d_1 \cdot N_1} & \Id_{N_2} \otimes (\Id_{d_2} - C_2C_2^{\top}S_2)S_2^{-1}
\end{bsmallmatrix} \cdot \euc\hess{f}$
      $\riem\hess{f} \ = \ \riem\hess{f} -
\begin{bsmallmatrix}
\euc{G}_1^{\top}C_1 \otimes \Id_{d_1} & 0_{d_1 \cdot N_1 \x d_2 \cdot N_2} \\
0_{d_2 \cdot N_2 \x d_1 \cdot N_1} & \euc{G}_2^{\top}C_2 \otimes \Id_{d_2}
\end{bsmallmatrix}$
      return $\riem\hess{f}$

  while $k \leq max\_iter$ and $tol\_grad < \riem{G}^{(k)}$  and $tol\_val < E_{diff}$:
      $E^{(k)}$ = f($[C_1^{(k)}], \ [C_2^{(k)}]$)
      $\euc{G}_1^{(k)}, \ \euc{G}_2^{(k)}$ = euclidean_gradient($C_1^{(k)}, \ C_2^{(k)}$)
      $\riem{G}_1^{(k)}, \ \riem{G}_2^{(k)}$ = riemannian_gradient($C_1^{(k)}, \ C_2^{(k)}, \ \euc{G}_1^{(k)}, \ \euc{G}_2^{(k)}$)
      $\text{aug}\grad{f}^{(k)}$ = augment_gradient($\riem{G}_1^{(k)}, \ \riem{G}_2^{(k)}$)
      $\euc\hess{f}^{(k)}$ = euclidean_hessian($C_1^{(k)}, \ C_2^{(k)}$)
      $\riem\hess{f}^{(k)}$ = riemannian_hessian($C_1^{(k)}, \ C_2^{(k)}, \ \euc{G}_1^{(k)}, \ \euc{G}_2^{(k)}, \ \euc\hess{f}^{(k)}$)
      $\text{aug}\hess{f}^{(k)}$ = augment_hessian($C_1^{(k)}, \ C_2^{(k)}, \ \riem\hess{f}^{(k)}$)
      solve: $\text{aug}\hess{f}^{(k)} \cdot
      \begin{bsmallmatrix}
        \eta_1 \\ \eta_2
      \end{bsmallmatrix}
      \ = \ -\text{aug}\grad{f}^{(k)}$
      $\eta_1^{(k)}, \ \eta_2^{(k)} \ = \ \unvec(\eta_1, \ N_1 \x d_1), \ \unvec(\eta_2, \ N_2 \x d_2)$
      $C_1^{(k+1)}$ = geodesic($C_1^{(k)}, \ \eta_1^{(k)}, \ O_1, \ O_1^{-1}, \ 1.0$)
      $C_2^{(k+1)}$ = geodesic($C_2^{(k)}, \ \eta_2^{(k)}, \ O_2, \ O_2^{-1}, \ 1.0$)
      $\riem{G}^{(k)}$ = riemannian_norm($\riem{G}_1^{(k)}, \ \riem{G}_2^{(k)}$)
      $E_{diff} \ = \ |E_{prev} - E^{(k)}|$; $E_{prev} \ = \ E^{(k)}$; $k \ = \ k + 1$
\end{algbox}
Some of the functions were defined in RGD (\ref{RGD}) and the \verb|vstack| function was defined in \ref{vstack}.
Observe that the \verb|euclidean_hessian| was not implemented because it also depends on the specific problem one is trying to solve.

\subsection{Conjugate Gradient} \label{conjugate_gradient}

The goal of Conjugate Gradient is to fix the most expensive steps of Newton--Raphson: computing the Hessian and solving $\hess{f}(x_k)(v_k) = -\grad{f}(x_k)$.
We will describe very briefly how this algorithm works, but the reader interested in the details can check \cite{shewchuk1994}.

Given two vectors $u, v$ and a positive-definite (symmetric) matrix $A$, we say that $u$ and $v$ are \emph{conjugate with respect to $A$} if $\braket{u}{Av} = 0$.
It is not hard to see that, if we have a set $\qty{v_1, \ldots, v_d}$ such that $v_i$ and $v_j$ are conjugate for every $i \neq j$, then the vectors of this set are linearly independent.
Consequently, if we want to solve a system $Ax = b$, we can obtain its unique solution in the following way: supposing $x^*$ is the solution, we write $x^* = \alpha_1v_1 + \ldots + \alpha_dv_d$ and then we have $Ax^* = \alpha_1Av_1 + \ldots + \alpha_dAv_d$.
Now, computing $\braket{v_k}{b}$, we obtain $\braket{v_k}{b} = \braket{v_k}{Ax^*} = \alpha_1\braket{v_k}{Av_1} + \ldots + \alpha_d\braket{v_k}{Av_d} = \alpha_k\braket{v_k}{Av_k}$.
Therefore, the solution can be written as
\begin{equation}
  x^* = \frac{\braket{v_1}{b}}{\braket{v_1}{Av_1}}v_1 + \ldots + \frac{\braket{v_d}{b}}{\braket{v_d}{Av_d}}v_d.
\end{equation}
So, if we find a basis of conjugate vectors, we can easily compute the solution of $Ax = b$.
The method known as \emph{Conjugate Gradient} aims to find these conjugate directions using a generalization of the Gram--Schmidt algorithm.
With that said, in our case the system is $\hess{f}(x_k)(v_k) = -\grad{f}(x_k)$ and the generalized Gram--Schmidt uses linear combinations of the gradient of $f$ in each step as its conjugate directions.
It is also worth mentioning that, since the dimensions of the manifold and its tangent spaces (at each point) are the same, we cannot have more conjugate directions than this dimension because the conjugate vectors are linearly independent.
In the linear case where the system $Ax = b$ is fixed we do not have this problem, but when we adapt CG to solve $\hess{f}(x_k)(v_k) = -\grad{f}(x_k)$, we cannot be sure that the algorithm will converge in less iterations than the dimension of the manifold.
For this reason we reset the descent direction to be the gradient each time we reach the dimension of the manifold (first step).
This adaptation of CG to minimizing nonlinear functions $f$ is also known as \emph{Nonlinear Conjugate Gradient} and this is the algorithm we will actually implement next:
\begin{algo}[Riemannian Conjugate Gradient]
  Given a Riemannian manifold $X$, a function $f:X \to \br$ and a starting point $x_0 \in X$, do the following:
  \begin{enumerate}

  \item If $k = n \cdot \dim{X}$ for some $n \in \bn$, set $v_k = -\grad{f}(x_k)$.

  \item Update the point according to
    $x_{k+1} \ceq \exp_{x_k}(t_kv_k)$ for some $t_k > 0$.

  \item Let $\pt_{0 \to 1}^{\gamma}$ be the parallel transport along the geodesic $\gamma$ starting in $x_k$ with direction $v_k$.

  \item Compute $\alpha_k^{\text{PR}} = \frac{\braket{\grad{f}(x_{k+1}) - \pt_{0 \to 1}^{\gamma}(t_k\grad{f}(x_k))}{\grad{f}(x_{k+1})}}{\braket{\grad{f}(x_k)}{\grad{f}(x_k)}}$
    or $\alpha_k^{\text{FR}} = \frac{\braket{\grad{f}(x_{k+1})}{\grad{f}(x_{k+1})}}{\braket{\grad{f}(x_k)}{\grad{f}(x_k)}}$.

  \item Define $v_{k+1} \ceq -\grad{f}(x_{k+1}) + \alpha_k^X\pt_{0 \to 1}^{\gamma}(t_kv_k)$, being $X \in \qty{\text{PR}, \ \text{FR}}$.

  \end{enumerate}
\end{algo}
In the second step we adopted the same strategy used in the Gradient Descent, \ie, we used constant values for $t_k$.
Now, to implement the third step we need to recall (Example \ref{geodesic_grassmannian}) that the parallel transport in $\gr{N_i}{d_i}$ is given by
{\small
  \begin{equation}
    \pt_{0 \to t}^{\gamma_i}(t\mu_i)
    = \qty(\qty(-\gamma_i(0)V_i^{\top}\sin(tD_i)
    + O_iU_i\cos(tD_i))U_i^{\top}O_i^{-1}
    + \Id_{d_i} - O_iU_iU_i^{\top}O_i^{-1})\mu_i.
  \end{equation}
}%
Here $\mu_i, \eta_i \in T_{[C_i]}\gr{N_i}{d_i}$; $\gamma_i$ is a geodesic such that $\gamma_i(0) = C_i$ and $\gamma_i'(0) = \eta_i$; $O_i$ is a matrix such that $O_i^{\top}S_iO_i = \Id_{d_i}$; and $O_i^{-1}\eta_i = U_iD_iV_i$ is a tSVD.
To conclude, the fourth step is computing the coefficients known as \emph{Polak--Ribière} (PR) and \emph{Fletcher--Reeves} (FR), respectively.
Choosing which one is the best is usually done experimentally because it depends on the function one wants to optimize and Fletcher--Reeves performed better for Hartree--Fock.
So, this is the one we will write in the pseudocode, but the reader can implement the other one quite easily.
We have our last algorithm:
\begin{algbox}{Riemannian Conjugate Gradient (RCG)}{RCG}
  input: $max\_iter, \ step\_size, \ \ C_i^{(0)}, \ S_i, \ S_i^{-1}, \ O_i, \ O_i^{-1}, \ i=1,2$
  set: $k, \ dim, \ tol\_grad, \ tol\_val, \ E_{prev}, \ E^{(0)}, \ E_{diff}, \ \riem{G}^{(0)} \ = \ 0, \ N_1(d_1 - N_1) + N_2(d_2 - N_2), \ 10^{-8}, \ 10^{-10}, \ \infty, \ 0, \ 1, \ 1$
  precompute: $\Id_{d_1}, \ \Id_{d_2}$

  def pt($\mu, \ \eta, \ C, \ Id, \ O, \ O^{-1}, \ step\_size$): # parallel transport
      $U, D, V$ = tSVD($O^{-1}\eta$)
      return $\qty(\qty(-CV^{\top}\sin(step\_size \cdot D) + OU\cos(step\_size \cdot D))U^{\top}O^{-1} + \Id - OUU^{\top}O^{-1})\mu$

  while $k \leq max\_iter$ and $tol\_grad < \riem{G}^{(k)}$ and $tol\_val < E_{diff}$:
      if $k \ \% \ dim \ == \ 0$:
          if $k \ == \ 0$:
              $\euc{G}_1^{(k)}, \ \euc{G}_2^{(k)}$ = euclidean_gradient($C_1^{(k)}, \ C_2^{(k)}$)
              $\riem{G}_1^{(k)}, \ \riem{G}_2^{(k)}$ = riemannian_gradient($C_1^{(k)}, \ C_2^{(k)}, \ \euc{G}_1^{(k)}, \ \euc{G}_2^{(k)}$)
              $\eta_1^{(k)}, \ \eta_2^{(k)} \ = \ -\riem{G}_1^{(k)}, \ -\riem{G}_2^{(k)}$
      else:
          $\eta_1^{(k)} \ = \ -\riem{G}_1^{(k)} + \alpha^{(k-1)}$pt($\eta_1^{(k-1)}, \ \eta_1^{(k-1)}, \ C_1^{(k-1)}, \ \Id_{d_1}, \ O_1, \ O_1^{-1}, \ step\_size$)
          $\eta_2^{(k)} \ = \ -\riem{G}_2^{(k)} + \alpha^{(k-1)}$pt($\eta_2^{(k-1)}, \ \eta_2^{(k-1)}, \ C_2^{(k-1)}, \ \Id_{d_2}, \ O_2, \ O_2^{-1}, \ step\_size$)
          $E^{(k)}$ = f($[C_1^{(k)}], \ [C_2^{(k)}]$)
          $C_1^{(k+1)}$ = geodesic($C_1^{(k)}, \ \eta_1^{(k)}, \ O_1, \ O_1^{-1}, \ step\_size$)
          $C_2^{(k+1)}$ = geodesic($C_2^{(k)}, \ \eta_2^{(k)}, \ O_2, \ O_2^{-1}, \ step\_size$)
          $E_{diff} \ = \ |E_{prev} - E^{(k)}|$; $E_{prev} \ = \ E^{(k)}$
          $\euc{G}_1^{(k+1)}, \ \euc{G}_2^{(k+1)}$ = euclidean_gradient($C_1^{(k+1)}, \ C_2^{(k+1)}$)
          $\riem{G}_1^{(k+1)}, \ \riem{G}_2^{(k+1)}$ = riemannian_gradient($C_1^{(k+1)}, \ C_2^{(k+1)}, \ \euc{G}_1^{(k+1)}, \ \euc{G}_2^{(k+1)}$)
          $\riem{G}^{(k+1)}$ = riemannian_norm($\riem{G}_1^{(k)}, \ \riem{G}_2^{(k)}$)
          $\alpha^{(k)}$ = rip($(\riem{G}_1^{(k+1)}, \ \riem{G}_2^{(k+1)}), \ (\riem{G}_1^{(k+1)}, \ \riem{G}_2^{(k+1)})$)
          $\alpha^{(k)}$ /= rip($(\riem{G}_1^{(k)}, \ \riem{G}_2^{(k)}), \ (\riem{G}_1^{(k)}, \ \riem{G}_2^{(k)})$)
          $k \ = \ k + 1$
\end{algbox}
Again, some of the functions were defined in RGD (\ref{RGD}).

\section{Case Study (Hartree--Fock)} \label{hf_ch3}

In Chapter \ref{qm} we motivated and defined the Hartree--Fock Method (see Section \ref{qm_practice}) and at the end we obtained a concrete optimization problem that we can now solve using the algorithms described previously.
So, let us briefly recapitulate what Hartree--Fock (HF) is and then compute the partial derivatives we need to fill in the blank spaces we left in the pseudocodes.
To put it shortly, HF is the optimization problem given by
\begin{equation}
  \text{minimize $f \ceq E \circ \Pl:\gr{\NA}{d_{\alpha}} \x \gr{\NB}{d_{\beta}} \to \br$},
\end{equation}
being $\NA$ and $\NB$ the numbers of spin $\alpha$ and spin $\beta$ electrons in our quantum system (say, a molecule), $d_{\alpha}$ and $d_{\beta}$ the sizes of the basis set for each spin, $\Pl$ the Plücker embedding (see \ref{plucker_embedding}) and $E$ the energy function $\ket{\Psi} \mapsto \mel{\Psi}{\what{H}_{el}}{\Psi}$.
However, while we used the abstract description of the Grassmannian (Observation \ref{abstract_concrete}) in Chapter \ref{qm} because it was easier to describe the method starting with a finite-dimensional subspace $W \ceq \Span{w_1^{\alpha}, \ldots, w_{d_{\alpha}}^{\alpha}, w_1^{\beta}, \ldots, w_{d_{\beta}}^{\beta}} \subset L^2(\br^3) \otimes V_{\frac{1}{2}}$, in this chapter we are using the quotient definition of the Grassmannian.
It is quite easy to move from one to the other, though, just consider the function:
  \begin{equation}
    \begin{bsmallmatrix}
      c_{11}^{\gamma} & \ldots & c_{1N_{\gamma}}^{\gamma} \\
      \vdots & \ddots & \vdots \\
      c_{d_{\gamma}1}^{\gamma} & \ldots & c_{d_{\gamma}N_{\gamma}}^{\gamma}
    \end{bsmallmatrix}
    \mapsto
    \Span{c_{11}^{\gamma}w_1^{\gamma}
      + \ldots
      + c_{d_{\gamma}1}^{\gamma}w_{d_{\gamma}}^{\gamma},
      \ldots,
      c_{1N_{\gamma}}^{\gamma}w_1^{\gamma}
      + \ldots
      + c_{d_{\gamma}}^{\gamma}N_{\gamma}w_{d_{\gamma}}^{\gamma}},
  \end{equation}
being $\gamma \in \qty{\alpha, \beta}$.
So, since now we know how to write the cost function in the quotient representation of the Grassmannian, let us recall its concrete expression: given the one-electron operator $H(\vb{r})\psi(\vb{r}) \ceq -\frac{1}{2}\nabla^2\psi(\vb{r}) - \sum_{j=1}^M \frac{Z_j}{\abs{\vb{r} - \vb{R}_j}}\psi(\vb{r})$, defining $\vb{r}_{12} \ceq \abs{\vb{r}_1 - \vb{r}_2}$ and assuming $w_i^{\gamma} = \psi_i^{\gamma}\gamma \ceq \psi_i^{\gamma} \otimes \gamma$, $f$ can be written as:
\begin{multline}
  \sum_{\mu=1}^{\NA} \sum_{i,j=1}^{d_{\alpha}}
  c_{i\mu}^{\alpha}c_{j\mu}^{\alpha}
  \intl_{\br^3} \psi_i^{\alpha}(\vb{r})H(\vb{r})\psi_j^{\alpha}(\vb{r})
  \dd\vb{r}
  +
  \sum_{\mu=1}^{\NB} \sum_{i,j=1}^{d_{\beta}}
  c_{i\mu}^{\beta}c_{j\mu}^{\beta}
  \intl_{\br^3} \psi_i^{\beta}(\vb{r})H(\vb{r})\psi_j^{\beta}(\vb{r}) \dd\vb{r} \\
  + \frac{1}{2} \sum_{\mu,\nu=1}^{\NA} \sum_{i,j,k,l=1}^{d_{\alpha}}
  c_{i\mu}^{\alpha}c_{j\nu}^{\alpha}c_{k\mu}^{\alpha}c_{l\nu}^{\alpha}
  \intl_{\br^6} \psi_i^{\alpha}(\vb{r}_1)\psi_j^{\alpha}(\vb{r}_2)
  \frac{1}{\vb{r}_{12}}
  \qty(\psi_k^{\alpha}(\vb{r}_1)\psi_l^{\alpha}(\vb{r}_2)
  - \psi_l^{\alpha}(\vb{r}_1)\psi_k^{\alpha}(\vb{r}_2))
  \dd\vb{r}_1\dd\vb{r}_2 \\
  + \frac{1}{2} \sum_{\mu,\nu=1}^{\NB} \sum_{i,j,k,l=1}^{d_{\beta}}
  c_{i\mu}^{\beta}c_{j\nu}^{\beta}c_{k\mu}^{\beta}c_{l\nu}^{\beta}
  \intl_{\br^6} \psi_i^{\beta}(\vb{r}_1)\psi_j^{\beta}(\vb{r}_2)
  \frac{1}{\vb{r}_{12}}
  \qty(\psi_k^{\beta}(\vb{r}_1)\psi_l^{\beta}(\vb{r}_2)
  - \psi_l^{\beta}(\vb{r}_1)\psi_k^{\beta}(\vb{r}_2))
  \dd\vb{r}_1\dd\vb{r}_2 \\
  + \sum_{\mu=1}^{\NA} \sum_{\nu=1}^{\NB}
  \sum_{i,k=1}^{d_{\alpha}}\sum_{j,l=1}^{d_{\beta}}
  c_{i\mu}^{\alpha}c_{j\nu}^{\beta}c_{k\mu}^{\alpha}c_{l\nu}^{\beta}
  \intl_{\br^6} \psi_i^{\alpha}(\vb{r}_1)\psi_j^{\beta}(\vb{r}_2)
  \frac{1}{\vb{r}_{12}}
  \psi_k^{\alpha}(\vb{r}_1)\psi_l^{\beta}(\vb{r}_2)
  \dd\vb{r}_1\dd\vb{r}_2.
\end{multline}
However, since this expression is too cumbersome to work with, let us encapsulate some of these objects into matrices to improve the implementation and the readability.
First, in practice one always consider $d_{\alpha} = d_{\beta}$ and $\psi_i^{\alpha} = \psi_i^{\beta}$.
Therefore, replacing $d_{\gamma}$ by $d$ and $\psi_i^{\gamma}$ by $\psi_i$ and defining
\begin{equation}
  h_{ij} \ceq \intl_{\br^3} \psi_i(\vb{r})H(\vb{r})\psi_j(\vb{r})\dd\vb{r}
\end{equation}
and
\begin{equation}
  g_{ijkl} \ceq \intl_{\br^6} \psi_i(\vb{r}_1)\psi_j(\vb{r}_2)
  \frac{1}{\vb{r}_{12}} \psi_k(\vb{r}_1)\psi_l(\vb{r}_2)
  \dd\vb{r}_1\dd\vb{r}_2,
\end{equation}
the energy becomes
\begin{multline} \label{energy_coeffs}
  \sum_{\mu=1}^{\NA} \sum_{i,j=1}^d
  c_{i\mu}^{\alpha}c_{j\mu}^{\alpha}h_{ij}
  +
  \sum_{\mu=1}^{\NB} \sum_{i,j=1}^d
  c_{i\mu}^{\beta}c_{j\mu}^{\beta}h_{ij} +
  \frac{1}{2} \sum_{\mu,\nu=1}^{\NA} \sum_{i,j,k,l=1}^d
  c_{i\mu}^{\alpha}c_{j\nu}^{\alpha}c_{k\mu}^{\alpha}c_{l\nu}^{\alpha}
  (g_{ijkl} - g_{ijlk}) \\
  + \frac{1}{2} \sum_{\mu,\nu=1}^{\NB} \sum_{i,j,k,l=1}^d
  c_{i\mu}^{\beta}c_{j\nu}^{\beta}c_{k\mu}^{\beta}c_{l\nu}^{\beta}
  (g_{ijkl} - g_{ijlk})
  + \sum_{\mu=1}^{\NA} \sum_{\nu=1}^{\NB} \sum_{i,j,k,l=1}^d
  c_{i\mu}^{\alpha}c_{j\nu}^{\beta}c_{k\mu}^{\alpha}c_{l\nu}^{\beta}
  g_{ijkl}.
\end{multline}
Now, defining what is known as \emph{density matrix}\index{density matrix}:
\begin{equation} \label{density_matrix}
  P^{\gamma} = C^{\gamma}\qty(C^{\gamma})^{\top},
\end{equation}
we can simplify the energy even more to
\begin{equation}
  \sum_{i,j=1}^d \qty(P_{ij}^{\alpha} + P_{ij}^{\beta})h_{ij}
  +
  \frac{1}{2} \sum_{i,j,k,l=1}^d \qty(P_{ik}^{\alpha}P_{jl}^{\alpha} +
  P_{ik}^{\beta}P_{jl}^{\beta})\qty(g_{ijkl} - g_{ijlk})
  + \sum_{i,j,k,l=1}^d P_{ik}^{\alpha}P_{jl}^{\beta}g_{ijkl}.
\end{equation}
And, to conclude, we are going to define the \emph{Fock matrix}\index{Fock matrix}
\begin{equation} \label{fock_matrix}
  F_{ij}^{\gamma} \ceq h_{ij} + \sum_{k,l=1}^d
  \qty(\qty(P_{kl}^{\alpha} + P_{kl}^{\beta})g_{ikjl}
  - P_{kl}^{\gamma}g_{ijkl})
\end{equation}
because this matrix is widely used in the literature and we are using it to show how our equations change (or not) from what the community is used to.
So, with all these matrices defined, the final expression for the energy is
\begin{equation} \label{final_energy}
  f\qty([C^{\alpha}], \ [C^{\beta}])
  = \frac{1}{2}\sum_{i,j=1}^d \qty(\qty(P_{ij}^{\alpha} + P_{ij}^{\beta})h_{ij}
  + P_{ij}^{\alpha}F_{ij}^{\alpha} + P_{ij}^{\beta}F_{ij}^{\beta}).
\end{equation}
Observe that both the density and Fock matrices depend on $C^{\gamma}$, which is usually called \emph{coefficients matrix} or \emph{orbitals' coefficients}.
It is also worth mentioning that we can lift the energy because $\ol{f}$ is invariant under $O(\NA) \x O(\NB)$, that is, $\ol{f}(C^{\alpha}M^{\alpha}, \ C^{\beta}M^{\beta}) = \ol{f}(C^{\alpha}, \ C^{\beta})$ for $(M^{\alpha}, \ M^{\beta}) \in O(\NA) \x O(\NB)$.
Indeed, if we notice that the energy and the Fock matrix can be written only in terms of the density matrix, then the invariance is a direct consequence of the following: $P^{\gamma} = C^{\gamma}(C^{\gamma})^{\top} = C^{\gamma}M^{\gamma}(C^{\gamma}M^{\gamma})^{\top} = C^{\gamma}M^{\gamma}(M^{\gamma})^{\top}(C^{\gamma})^{\top}$ because $M^{\gamma}(M^{\gamma})^{\top} = \Id_{\NC}$ by definition of $M^{\gamma} \in O(\NC)$.

Now, all we have to do is computing the partial derivatives of $\oll{f}$, which, in practice, is also given by Equation \ref{final_energy}.
So, if we go back to Equation \ref{energy_coeffs} and observe that the energy is a fourth-degree polynomial because the integrals are constants, the first-order partial derivatives are straightforward to compute and the result in the direction $x_{pq}^{\alpha}$ (for spin $\beta$ is analogous) is:
\begin{equation} \label{gradient_poly}
  \pdv{\oll{f}}{x_{pq}^{\alpha}}\qty(C^{\alpha}, \ C^{\beta})
  = \sum_{i=1}^d 2c_{iq}^{\alpha}h_{ip}
  + \sum_{i,j,k=1}^d 2c_{iq}^{\alpha}\qty(
  \sum_{^{\mu=1}_{\mu \neq q}}^{\NA}
  c_{j\mu}^{\alpha}c_{k\mu}^{\alpha}(g_{pjik} - g_{pikj})
  + P_{jk}^{\beta}g_{pjik}).
\end{equation}
However, here something interesting happens: when $\mu = q$, the spin $\alpha$ term inside the parenthesis is $0$.
Indeed, if we define $\wtil{\psi}_q \ceq \sum_{i=1}^d c_{iq}\psi_i$, then
\begin{multline}
  \sum_{i,j,k=1}^d 2c_{iq}^{\alpha}c_{jq}^{\alpha}c_{kq}^{\alpha}
  (g_{pjik} - g_{pikj})
  = \sum_{i,j,k=1}^d 2c_{iq}^{\alpha}c_{jq}^{\alpha}c_{kq}^{\alpha}
  (g_{pjik} - g_{pjki}) \\
  = \sum_{i,j,k=1}^d 2c_{iq}^{\alpha}c_{jq}^{\alpha}c_{kq}^{\alpha}
  \qty(\int_{\br^6} \psi_p(\vb{r}_1)\psi_j(\vb{r}_2) \frac{1}{\vb{r}_{12}}
  \qty(\psi_i(\vb{r}_1)\psi_k(\vb{r}_2) - \psi_k(\vb{r}_1)\psi_i(\vb{r}_2))
  \dd\vb{r}_1\dd\vb{r}_2) \\
  = 2\int_{\br^6} \psi_p(\vb{r}_1)\wtil{\psi}_q(\vb{r}_2) \frac{1}{\vb{r}_{12}}
  \qty(\wtil{\psi}_q(\vb{r}_1)\wtil{\psi}_q(\vb{r}_2)
  - \wtil{\psi}_q(\vb{r}_1)\wtil{\psi}_q(\vb{r}_2))
  \dd\vb{r}_1\dd\vb{r}_2 = 0.
\end{multline}
Therefore, ignoring the restriction $\mu \neq q$ in \ref{gradient_poly}, we have
\begin{equation} \label{gradient_fock}
  \pdv{\oll{f}}{x_{pq}^{\alpha}}\qty(C^{\alpha}, \ C^{\beta})
  = 2(F^{\alpha}C^{\alpha})_{pq}.
\end{equation}

Now we need to compute the derivative of \ref{gradient_fock} with respect to $C^{\alpha}$ and $C^{\beta}$.
Let us start with $\beta$ because it is easier.
Ignoring all the terms of \ref{gradient_poly} that does not depend on $\beta$, we obtain
\begin{equation}
  \sum_{i,j,k=1}^d \sum_{\mu=1}^{\NB}
  2c_{iq}^{\alpha}c_{j\mu}^{\beta}c_{k\mu}^{\beta}g_{pjik}.
\end{equation}
However, if $\mu \neq s$, the partial derivative with respect to $x_{rs}^{\beta}$ will be $0$ and that means we only have to deal with
\begin{equation}
  \sum_{i,j,k=1}^d 2c_{iq}^{\alpha}c_{js}^{\beta}c_{ks}^{\beta}g_{pjik}.
\end{equation}
The idea now is to break the summation above into three cases: $j = k = r$; $j = r$, but $k \neq r$; and $k = r$, but $j \neq r$.
In the first case we have
\begin{equation}
  \sum_{i=1}^d 2c_{iq}^{\alpha}c_{rs}^{\beta}c_{rs}^{\beta}g_{prir}
\end{equation}
and the derivative is just
\begin{equation}
  \sum_{i=1}^d 4c_{iq}^{\alpha}c_{rs}^{\beta}g_{prir}.
\end{equation}
In the second case we have
\begin{equation}
  \sum_{^{i,k=1}_{k \neq r}}^d 2c_{iq}^{\alpha}c_{rs}^{\beta}c_{ks}^{\beta}g_{prik}
\end{equation}
and the derivative is
\begin{equation}
  \sum_{^{i,k=1}_{k \neq r}}^d 2c_{iq}^{\alpha}c_{ks}^{\beta}g_{prik}
\end{equation}
because $k \neq r$.
The third case is analogous and the derivative is
\begin{equation}
  \sum_{^{i,j=1}_{j \neq r}}^n 2c_{iq}^{\alpha}c_{js}^{\beta}g_{pjir}.
\end{equation}
Now, if we rename the variable $k$ in the second case to $j$ and sum the three cases using the identity $g_{prij} = g_{pjir}$, we obtain the final result:
\begin{equation} \label{beta_alpha}
  \pdv{\oll{f}}{x_{rs}^{\beta}}{x_{pq}^{\alpha}}\qty(C^{\alpha}, \ C^{\beta})
  = \sum_{i,j=1}^d 4c_{iq}^{\alpha}c_{js}^{\beta}g_{pjir}.
\end{equation}
Now let us compute the partial derivative with respect to $x_{rs}^{\alpha}$.
The idea here is to break the computation into two cases: $s = q$ and $s \neq q$.
In the first case the computation is straightforward and the result is
\begin{align}
  \begin{split}
    \pdv{\oll{f}}{x_{rs}^{\alpha}}{x_{pq}^{\alpha}}
    & = 2h_{rp} + \sum_{i,j=1}^d 2\qty(
      \sum_{^{\mu=1}_{\mu \neq q}}^{\NA}
      c_{i\mu}^{\alpha}c_{j\mu}^{\alpha}(g_{pirj} - g_{prji})
      + P_{ij}^{\beta}g_{pirj}
      ) \\
    & = 2h_{pr} + \sum_{i,j=1}^d 2\qty(
      \sum_{\mu=1}^{\NA}
      P_{ij}^{\alpha}(g_{pirj} - g_{prji})
      + P_{ij}^{\beta}g_{pirj})
      - \sum_{i,j=1}^d 2c_{iq}^{\alpha}c_{jq}^{\alpha}(g_{pirj} - g_{prji}) \\
    & = 2F_{pr}^{\alpha}
      - \sum_{i,j=1}^d 2c_{iq}^{\alpha}c_{jq}^{\alpha}(g_{pirj} - g_{prji}).
  \end{split}
\end{align}
Observe that we used the identity $h_{pr} = h_{rp}$.
And for the case in which $s \neq q$ the result is
\begin{equation}
  \pdv{\oll{f}}{x_{rs}^{\alpha}}{x_{pq}^{\alpha}}
  = \sum_{i,j=1}^d 2c_{iq}^{\alpha}c_{js}^{\alpha}\qty(
  2g_{pjir} - g_{pijr} - g_{pirj}
  ).
\end{equation}
To compute the derivative in this case one can use the same idea we used in the $x_{rs}^{\beta}$ case: first, notice that the terms with the one-electron integral and with spin $\beta$ vanish because $s \neq q$.
Then, we can get rid of the summation over $\mu$ because only the case in which $\mu = s$ interests.
Finally, compute the partial derivative breaking it into three cases: $j = k = r$; $j = r$ and $k \neq r$; $k = r$ and $j \neq r$.
Sum the result using the identity $g_{pjir} = g_{prij}$.

Now let us implement these matrices.
As one can see, $h_{ij}$ and $g_{ijkl}$ are shared between the first and second-order partial derivatives.
Therefore, an option is to store these numbers in a $2$-dimensional matrix $h$ and in a $4$-dimensional matrix $g$ and then pass them as parameters to the \verb|euclidean_gradient| and the \verb|euclidean_hessian| functions.
It is also worth mentioning that we will use these integrals as a black box because it is not trivial to compute them, specially $g$.
So, although the gradient and the Hessian should depend only on $C^{\alpha}$ and $C^{\beta}$, from a computational perspective it is worth adapting them to our case.
With all that said, the \verb|euclidean_gradient| is given by
\begin{algbox}{Euclidean Gradient}{EG}
  input: $C^{\alpha}, \ C^{\beta}, \ h, \ g$

  def build_fock($C^{\alpha}, \ C^{\beta}, \ h, \ g$):
      $F^{\alpha}, \ F^{\beta}, \ P^{\alpha}, \ P^{\beta} \ = \ h, \ h, \ C^{\alpha}(C^{\alpha})^{\top}, \ C^{\beta}(C^{\beta})^{\top}$
      for $i, j, k, l \ = \ 0, \ldots, d-1$:
           $F^{\alpha}[i, j]$ += $(P^{\alpha}[k, l] + P^{\beta}[k, l]) \cdot g[i, k, j, l] - P^{\alpha}[k, l] \cdot g[i, j, k, l]$
           $F^{\beta}[i, j]$ += $(P^{\alpha}[k, l] + P^{\beta}[k, l]) \cdot g[i, k, j, l] - P^{\beta}[k, l] \cdot g[i, j, k, l]$
      return $F^{\alpha}, \ F^{\beta}$

  $F^{\alpha}, \ F^{\beta}$ = build_fock($C^{\alpha}, \ C^{\beta}, \ h, \ g$)
  return $2F^{\alpha}C^{\alpha}, \ 2F^{\beta}C^{\beta}$
\end{algbox}

Now, the Hessian. If we rename $1$ and $2$ to $\alpha$ and $\beta$, respectively, then, according to Equation \ref{generic_euc_hess}, the Hessian can be divided into four blocks:
{\small
  \begin{equation}
    \begin{bmatrix}
      H_{\alpha\alpha} & H_{\beta\alpha}^{\top} \\
      H_{\beta\alpha}  & H_{\beta\beta}
    \end{bmatrix}.
  \end{equation}
}
So, a good strategy is to build these blocks separately.
Let us start with $H_{\beta\alpha}$.
The formula for this was already computed in \ref{beta_alpha}, but we need to parse how this goes into the matrix.
The idea is simple: for every $p$ and $q$ we have a variable $x_{pq}^{\alpha}$ that will be responsible for a column of $H_{\beta\alpha}$.
Then, when we vary $r$ and $s$ (in that order because we are vectorizing matrices columnwise, see \ref{vectorization}), we obtain a column vector of size $d \cdot \NB$.
To illustrate this, suppose $d = \NA = \NB = 2$.
Then, if we represent $\pdv{\oll{f}}{x_{rs}^{\beta}}{x_{pq}^{\alpha}}$ using the number $pqrs$, we have the following matrix:
{\small
  \begin{equation}
    H_{\beta\alpha} =
    \begin{bmatrix}
      0000 & 1000 & 0100 & 1100 \\
      0010 & 1010 & 0110 & 1110 \\
      0001 & 1001 & 0101 & 1101 \\
      0011 & 1011 & 0111 & 1111 \\
    \end{bmatrix}.
  \end{equation}
}
Now, the trick to implement this is to notice that we can convert four indices to two by doing $H_{\beta\alpha}[r + d \cdot s, \ p + d \cdot q]$.
And that is it, we can now implement $H_{\beta\alpha}$ using a bunch of \emph{loops}.
The blocks $H_{\alpha\alpha}$ and $H_{\beta\beta}$ are analogous, we just need to add an \emph{if statement} to check whether $s = q$ or not.
\begin{algbox}{Euclidean Hessian}{EH}
  input: $C^{\alpha}, \ C^{\beta}, \ F^{\alpha}, \ F^{\beta}, \ g$

  def mixed_spin($C^{\alpha}, \ C^{\beta}, \ g$):
      $H_{\beta\alpha}$ = zeros($d \cdot \NB \x d \cdot \NA$)
      for $s \ = \ 0, \ldots, \NB-1$:
          for $r \ = \ 0, \ldots, d-1$:
              for $q \ = \ 0, \ldots, \NA-1$:
                  for $p \ = \ 0, \ldots, d-1$:
                      for $i, j \ = \ 0, \ldots, d-1$:
                          $H_{\beta\alpha}[r + d \cdot s, \ p + d \cdot q]$ += $4 \cdot C^{\alpha}[i, q] \cdot C^{\beta}[j, s] \cdot g[p, j, i, r]$
      return $H_{\beta\alpha}$

  def same_spin($C^{\gamma}, \ F^{\gamma}, \ g, \ \NC$):
      $H_{\gamma\gamma}$ = zeros($d \cdot \NC \x d \cdot \NC$)
      for $s \ = \ 0, \ldots, \NC-1$:
          for $r \ = \ 0, \ldots, d-1$:
              for $q \ = \ 0, \ldots, \NC-1$:
                  for $p \ = \ 0, \ldots, d-1$:
                      if $s \ == \ q$:
                          $H_{\gamma\gamma}[r + d \cdot s, p + d \cdot q] \ = \ 2 \cdot F^{\gamma}[p, r]$
                          for $i, j \ = \ 0, \ldots, d-1$:
                              $H_{\gamma\gamma}[r + d \cdot s, p + d \cdot q]$ -= $2 \cdot C^{\gamma}[i, q] \cdot C^{\gamma}[j, q] \cdot (g[p, i, r, j] - g[p, r, j, i])$
                      else:
                          for $i, j \ = \ 0, \ldots, d-1$:
                              $H_{\gamma\gamma}[r + d \cdot s, p + d \cdot q]$ += $2 \cdot C^{\gamma}[i, q] \cdot C^{\gamma}[j, s] \cdot (2 \cdot g[p, j, i, r] - g[p, i, j, r] - g[p, i, r, j])$
      return $H_{\gamma\gamma}$

  $H_{\alpha\alpha}, \ H_{\beta\beta}$ = same_spin($C^{\alpha}, \ F^{\alpha}, \ g, \ N_{\alpha})$, same_spin($C^{\beta}, \ F^{\beta}, \ g, \ N_{\beta}$)
  $H_{\beta\alpha}$ = mixed_spin($C^{\alpha}, \ C^{\beta}, \ g$)
  return $
  \begin{bsmallmatrix}
    H_{\alpha\alpha} & H_{\beta\alpha}^{\top} \\ H_{\beta\alpha} & H_{\beta\beta}
  \end{bsmallmatrix}
  $
\end{algbox}

One last observation worth mentioning is that it is possible to remove the \emph{if statement} inside the \verb|same_spin| function and this is desirable because \emph{if}'s inside loops are very inefficient when one is dealing with (potentially) big problems.
We will not provide a complete pseudocode for this implementation, but let us show the idea of how to do this.
Basically, we can break $H_{\gamma\gamma}$ into smaller blocks and then build it using these blocks.
Let us show an example.
Assuming $d = 6$ and $\NA = 3$, $H_{\alpha\alpha}$ is a $6 \cdot 3 \x 6 \cdot 3$ matrix.
However, if we, again, represent $\pdv{\oll{f}}{x_{rs}^{\alpha}}{x_{pq}^{\alpha}}$ using the number $pqrs$, this matrix has the following format:
{\small
  \begin{equation}
    H_{\alpha\alpha} =
    \begin{bsmallmatrix}
      0000 & 1000 & 2000 & 3000 & 4000 & 5000 & & \vline & & 0100 & 1100 & 2100 &
                                                                                  3100 & 4100 & 5100 & & \vline & & 0200 & 1200 & 2200 & 3200 & 4200 & 5200 \\
      0010 & 1010 & 2010 & 3010 & 4010 & 5010 & & \vline & & 0110 & 1110 & 2110 &
                                                                                  3110 & 4110 & 5110 & & \vline & & 0210 & 1210 & 2210 & 3210 & 4210 & 5210 \\
      0020 & 1020 & 2020 & 3020 & 4020 & 5020 & & \vline & & 0120 & 1120 & 2120 &
                                                                                  3120 & 4120 & 5120 & & \vline & & 0220 & 1220 & 2220 & 3220 & 4220 & 5220 \\
      0030 & 1030 & 2030 & 3030 & 4030 & 5030 & & \vline & & 0130 & 1130 & 2130 &
                                                                                  3130 & 4130 & 5130 & & \vline & & 0230 & 1230 & 2230 & 3230 & 4230 & 5230 \\
      0040 & 1040 & 2040 & 3040 & 4040 & 5040 & & \vline & & 0140 & 1140 & 2140 &
                                                                                  3140 & 4140 & 5140 & & \vline & & 0240 & 1240 & 2240 & 3240 & 4240 & 5240 \\
      0050 & 1050 & 2050 & 3050 & 4050 & 5050 & & \vline & & 0150 & 1150 & 2150 &
                                                                                  3150 & 4150 & 5150 & & \vline & & 0250 & 1250 & 2250 & 3250 & 4250 & 5250 \\
      \medskip \\ \hline \\ \medskip \\
      0001 & 1001 & 2001 & 3001 & 4001 & 5001 & & \vline & & 0101 & 1101 & 2101 &
                                                                                  3101 & 4101 & 5101 & & \vline & & 0201 & 1201 & 2201 & 3201 & 4201 & 5201 \\
      0011 & 1011 & 2011 & 3011 & 4011 & 5011 & & \vline & & 0111 & 1111 & 2111 &
                                                                                  3111 & 4111 & 5111 & & \vline & & 0211 & 1211 & 2211 & 3211 & 4211 & 5211 \\
      0021 & 1021 & 2021 & 3021 & 4021 & 5021 & & \vline & & 0121 & 1121 & 2121 &
                                                                                  3121 & 4121 & 5121 & & \vline & & 0221 & 1221 & 2221 & 3221 & 4221 & 5221 \\
      0031 & 1031 & 2031 & 3031 & 4031 & 5031 & & \vline & & 0131 & 1131 & 2131 &
                                                                                  3131 & 4131 & 5131 & & \vline & & 0231 & 1231 & 2231 & 3231 & 4231 & 5231 \\
      0041 & 1041 & 2041 & 3041 & 4041 & 5041 & & \vline & & 0141 & 1141 & 2141 &
                                                                                  3141 & 4141 & 5141 & & \vline & & 0241 & 1241 & 2241 & 3241 & 4241 & 5241 \\
      0051 & 1051 & 2051 & 3051 & 4051 & 5051 & & \vline & & 0151 & 1151 & 2151 &
                                                                                  3151 & 4111 & 5151 & & \vline & & 0251 & 1251 & 2251 & 3251 & 4251 & 5251 \\
      \medskip \\ \hline \\ \medskip \\
      0002 & 1002 & 2002 & 3002 & 4002 & 5002 & & \vline & & 0102 & 1102 & 2102 &
                                                                                  3102 & 4102 & 5102 & & \vline & & 0202 & 1202 & 2202 & 3202 & 4202 & 5202 \\
      0012 & 1012 & 2012 & 3012 & 4012 & 5012 & & \vline & & 0112 & 1112 & 2112 &
                                                                                  3112 & 4112 & 5112 & & \vline & & 0212 & 1212 & 2212 & 3212 & 4212 & 5212 \\
      0022 & 1022 & 2022 & 3022 & 4022 & 5022 & & \vline & & 0122 & 1122 & 2122 &
                                                                                  3122 & 4122 & 5122 & & \vline & & 0222 & 1222 & 2222 & 3222 & 4222 & 5222 \\
      0032 & 1032 & 2032 & 3032 & 4032 & 5032 & & \vline & & 0132 & 1132 & 2132 &
                                                                                  3132 & 4132 & 5132 & & \vline & & 0232 & 1232 & 2232 & 3232 & 4232 & 5232 \\
      0042 & 1042 & 2042 & 3042 & 4042 & 5042 & & \vline & & 0142 & 1142 & 2142 &
                                                                                  3142 & 4142 & 5142 & & \vline & & 0242 & 1242 & 2242 & 3242 & 4242 & 5242 \\
      0052 & 1052 & 2052 & 3052 & 4052 & 5052 & & \vline & & 0152 & 1152 & 2152 &
                                                                                  3152 & 4152 & 5152 & & \vline & & 0252 & 1252 & 2252 & 3252 & 4252 & 5252
    \end{bsmallmatrix}.
  \end{equation}
}%
So, observe that the case in which $s = q$ gives us the $d \x d$ blocks in the main diagonal and the other $d \x d$ blocks represents the case in which $s \neq q$.
On top of that, since the Hessian is symmetric, we only need to compute the lower blocks outside the diagonal and fill them in the upper half of the matrix or vice versa.
This implementation is left as an exercise to the reader.

\section{Results} \label{results}

Let us finally see the results of how Riemannian Optimization algorithms perform in a real-world and difficult problem such as Hartree--Fock.
All the algorithms mentioned below were benchmarked in the \emph{G2/97} dataset \cite{curtiss1997}, which contains 148 molecules, and the basis set used was the 6-31G \cite{hehre1972}.
For the starting point we used a sum of atomic electron densities as described in \cite{vanlenthe2006} and the integrals were computed using the \emph{ir/wmme} program.\footnote{Which can be found at \url{https://sites.psu.edu/knizia/software/}}
It is also worth emphasizing that the code of all the algorithms is available in the GitHub repository \cite{aotograssmann} and that these tests were done in two regular laptops (it can take a couple of days to run everything depending on the specifications of the laptop, though).
Now, the convergence was measured in relative terms, \ie, we run many algorithms for the same molecule and the lowest (converged) energy was the reference to check whether the other algorithms converged to the (relative) minimum.
It is important to emphasize two things: first, the algorithms implemented are not global minimization algorithms, which means the minimum can be local.
Second, since we had to augment the (Riemannian) Hessian in order to obtain the right algorithm, it is not clear to us how to do a second partial derivative test to check if the convergence was to an actual minimum.
This is why we measured in relative terms.
The tolerance used to compare the converged energies was $10^{-5}$ and by lowest converged energy we mean that we also used two criteria to check whether the algorithms converged: the norm of the gradient and the norm of the difference of the energy between two consecutive iterations.
The tolerances used in the stopping criteria were $10^{-8}$ and $10^{-10}$, respectively.

Let us start with the results of the Gradient Descent and Conjugate Gradient:
\begin{table}[h]
  \begin{center}
    \begin{tabular}{|c|c|c|c|c|c|}
      \hline
      Step size & Converged         & Avg. n. iter. \\
      \hline
      0.005     & 0.7\%  \ (1/148)  & 535.0         \\
      \hline
      0.01      & 25.0\% \ (37/148) & 822.8         \\
      \hline
      0.02      & 41.2\% \ (61/148) & 562.1         \\
      \hline
      0.03      & 31.1\% \ (46/148) & 425.3         \\
      \hline
      0.04      & 20.9\% \ (31/148) & 304.2         \\
      \hline
    \end{tabular}
    \caption{Riemannian Gradient Descent (RGD).}
  \end{center}
\end{table}
\begin{table}[h]
  \begin{center}
    \begin{tabular}{|c|c|c|c|c|c|}
      \hline
      Conjugacy        & Step size & Converged          & Avg. n. iter. \\
      \hline
      Fletcher--Reeves & 0.005     & 89.2\% \ (132/148) & 185.8         \\
      \hline
      Fletcher--Reeves & 0.0075    & 91.9\% \ (136/148) & 156.0         \\
      \hline
      Fletcher--Reeves & 0.01      & 93.2\% \ (138/148) & 137.6         \\
      \hline
      Fletcher--Reeves & 0.02      & 81.1\% \ (120/148) & 104.0         \\
      \hline
      Fletcher--Reeves & 0.025     & 68.9\% \ (102/148) & 88.8          \\
      \hline
      Polak--Ribière   & 0.05      & 41.2\% \ (61/148)  & 200.9         \\
      \hline
      Polak--Ribière   & 0.06      & 43.2\% \ (64/148)  & 183.8         \\
      \hline
      Polak--Ribière   & 0.07      & 45.3\% \ (67/148)  & 146.8         \\
      \hline
      Polak--Ribière   & 0.08      & 31.8\% \ (47/148)  & 139.0         \\
      \hline
      Polak--Ribière   & 0.09      & 27.7\% \ (41/148)  & 115.5         \\
      \hline
    \end{tabular}
    \caption{Riemannian Conjugate Gradient (RCG).}
  \end{center}
\end{table}

As one can see, we tried many different step sizes for both algorithms and the maximum number of iterations used was 1000 for Gradient Descent (although this number is considered low for many usages of this method) and 300 for Conjugate Gradient.
We noticed that Gradient Descent would probably converge for many other molecules if we increased this limit, but since the computational cost (asymptotically) of Conjugate Gradient is the same, we judged it was not worth increasing this limit because the results presented by RCG were much better, as we can see.
It is also important to say that the average number of iterations was measured only for the cases in which the algorithm converged.
Now, the results of RCG were quite surprising because it was the best (single, without combining with others,) algorithm we tested in terms of molecules converged, but there are some caveats to it that we will discuss after the next table.
\begin{table}[h]
  \begin{center}
    \begin{tabular}{|c|c|c|c|c|c|}
      \hline
      Method              & Converged          & Avg. n. iter. & Complexity  & $\norm{\grad{f}}$ \\
      \hline
      RGD (0.02)          & 41.2\% \ (61/148)  & 562.1         & $O(d^4)$    & $6.94 \x 10^{-5}$ \\
      \hline
      RCG-FR (0.01)       & 93.2\% \ (138/148) & 137.6         & $O(d^4)$    & $3.21 \x 10^{-5}$ \\
      \hline
      RCG-PR (0.07)       & 45.3\% \ (67/148)  & 146.8         & $O(d^4)$    & $2.95 \x 10^{-5}$ \\
      \hline
      RNR                 & 83.8\% \ (124/148) & 5.0           & $O(d^4 + d^3N^3)$ & $6.29 \x 10^{-10}$ \\
      \hline \hline
      NRLM                & 53.4\% \ (79/148)  & 8.5           & $O(d^4 + d^3N^3)$ & $7.04 \x 10^{-10}$ \\
      \hline
      SCF + 2 DIIS        & 91.2\% \ (135/148) & 33.6          & $O(d^4)$    & $7.24 \x 10^{-9}$ \\
      \hline \hline
      SCF + 2 DIIS + RNR  & 91.9\% \ (136/148) & 54.1          & $O(d^6N^6)$ & $6.40 \x 10^{-9}$ \\
      \hline
      RCG-FR (0.01) + RNR & 95.9\% \ (142/148) & 102.1         & $O(d^6N^6)$ & $2.72 \x 10^{-6}$ \\
      \hline
    \end{tabular}
    \caption{\tbf{RGD (0.02)}: Riemannian Gradient Descent using 0.02 as the step size. \tbf{RCG-FR (0.01)}: Riemannian Conjugate Gradient using Fletcher--Reeves and 0.01 as the step size. \tbf{RCG-PR (0.07)}: Riemannian Conjugate Gradient using Polak--Ribière and 0.07 as the step size. \tbf{RNR}: Riemannian Newton--Raphson. \tbf{NRLM}: Newton--Raphson with Lagrange multipliers. \tbf{SCF + 2 DIIS}: Self-Consistent Field using 2 previous steps in the DIIS technique. $d$: size of basis set. $N$: number of electrons. The column $\norm{\grad{f}}$ represents the mean of the gradients of the last iteration of each converged molecule.}
    \label{table_three}
  \end{center}
\end{table}

As one can see, there are two other algorithms besides the ones described in this chapter that we implemented and tested: Newton--Raphson with Lagrange multipliers, which was described in Appendix \ref{lagrange_multipliers}, and the Self-Consistent Field with the DIIS technique, which is probably the most well-established algorithm in the Hartree--Fock literature (see \cite[Chapter 3]{szabo1996} and \cite[Chapter 10]{helgaker2000}).
The NRLM we implemented because we wanted to compare the performance of the regular Euclidean constrained optimization technique against the Riemannian version of the same algorithm in the same problem (recall that in the Introduction we said Riemannian algorithms can be thought of as alternatives to constrained optimization algorithms).
And the SCF was implemented because we wanted to compare the results with a well-established algorithm that the community is familiar with.
Now let us start analyzing this table.

First of all, the best single algorithm in terms of molecules converged was RCG followed by SCF using 2 steps in the DIIS technique.
Since both algorithms have the same complexity, in this regard using Riemannian algorithms shows a real gain.
We can also see that the Riemannian version of the Newton--Raphson Method is way better than the regular Euclidean version, with an improvement of 30\% more molecules converged.
Again, since the computational complexity of both methods is the same, this is also a real advantage of using Riemannian algorithms.
Combining different methods is a very common practice and we tried to combine the best SCF and the best RCG with RNR.\footnote{We tested different numbers of steps in the DIIS and 2 was the best one.}
As we can see, combining SCF and RNR did not present a real gain while combining RCG with RNR was more interesting.
Having said that, since the convergence of RCG was already high, more tests are required to confirm that this is indeed a good combination.

Let us now analyze the SCF and the RCG more carefully because they were the best algorithms and because they have the same complexity.
First of all, the complexity of both algorithms (and also of RGD) is $O(d^4)$ and the bottleneck is constructing the Fock matrix.
So, since the gradient of the energy requires computing this matrix, there is no first-order method that can improve this complexity without approximating the gradient and that means these three algorithms are the best we can obtain in terms of complexity.
Having said that, in practice we noticed that RCG was almost two times faster than SCF,\footnote{RGD is the fastest algorithm, but since it requires many iterations to converge, in the end it is better to choose between RCG and SCF.} but since the current implementation of RCG is not optimal, we did not benchmark this aspect very carefully.
One possible explanation is that RCG requires two SVD decompositions while SCF requires diagonalizing the Fock matrix, what makes RCG a $O(d^4 + 2dN^2)$ algorithm while SCF is $O(d^4 + d^3)$.
Now, since $d$ is usually bigger than $N$, this might explain the observed performance.
However, as we can see in the table, the average number of iterations (until convergence is reached) required by RCG is more than four times the average number of SCF.
Consequently, SCF will probably converge two times faster than RCG.
Again, it is worth emphasizing that the implementation of RCG is not optimal and this difference in time can actually decrease.
So, even RCG being slower than SCF if we consider all iterations, since it converges in more cases, it may be worth using it instead of SCF.

Now we need to talk about the accuracy of these algorithms.
Recall that we used two stopping criteria: the norm of the gradient and the norm of the difference of the energy between two consecutive iterations.
In most cases RCG converged under the second criterion and not under the first.
This could mean that the algorithm did not converge to a minimum, but, since the convergence was measured with respect to the other algorithms, as already explained in the beginning of this section, we can be sure that it did (at least according to our criteria and tolerances).
That said, the tolerance used in the convergence comparison was $10^{-5}$, which may be too big for some purposes.
So, if the reader needs accuracy, the best option is combining RCG with RNR or, since RNR is quite expensive, combining RCG with SCF (we did not test this, though).
Another option is ignoring the criterion based on the difference of energies, but then RCG takes many more iterations to converge.
The following graph illustrates quite well the convergence of the algorithms implemented for the C$_2$H$_4$ molecule (again, the energy error is measured with respect to the lowest converged energy):
\begin{center}
  \includegraphics[scale=0.37]{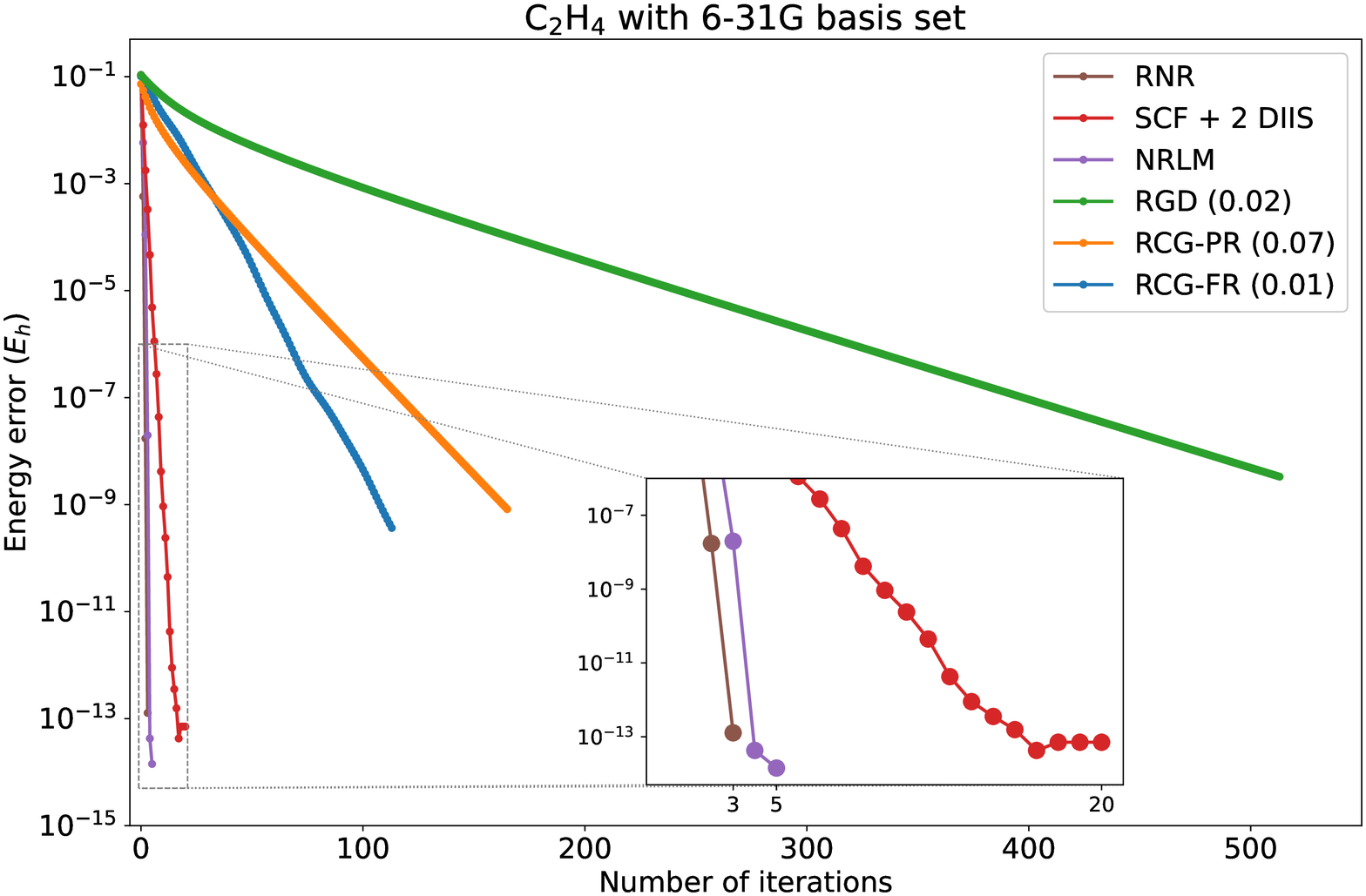}
\end{center}
To conclude, it is worth mentioning that when we combined RCG with RNR, the vast majority of molecules converged really well with $\norm{\grad{f}} \approx 10^{-10}$.
However, there were 10 molecules that converged using only RCG and this is why the mean norm of the gradient is high in Table \ref{table_three}.
That points to the fact that switching from RCG to RNR when the gradient is approximately $10^{-3}$ or $10^{-4}$ may be be a good option.
That is it.


\chapter{Conclusion and Perspectives} \label{conclusion}

As already discussed in the Introduction (\ref{intro}), having robust implementations for Hartree--Fock (HF) is highly desirable because it is a method that underlies many state-of-the-art computations done by chemists, material engineers, physicists etc.
With that said, in this work (Chapter \ref{qm} and Section \ref{hf_ch3}) we showed how HF can be interpreted as a Riemannian optimization problem and, hopefully, we also proved to the reader that the field of Riemannian Optimization can provide new robust implementations for HF.
In Chapter \ref{optimization} we provided pseudocodes for all the algorithms we implemented and in Section \ref{results} we showed that a combination of two Riemannian algorithms (the Conjugate Gradient together with the Newton--Raphson) achieved the best performance in terms of number of molecules converged even when compared to one of the most well-established algorithms in the literature.
There are many interesting aspects to investigate about these Riemannian algorithms that we did not explore, though, such as trying to find theoretical and/or experimental heuristics to improve the average number of iterations until convergence is reached or using other algorithms that are currently known to perform even better than Conjugate Gradient.
Most of the tools and references necessary to do such investigations were provided if the reader is interested, though.

Another goal of this work was to present the field of Riemannian Optimization to the reader because it is an interesting field that can be used not just for Hartree--Fock, but possibly for other \emph{ab initio} methods.
In \cite{aoto2022}, for example, the author shows that both the \emph{Configuration Interaction} and the \emph{Coupled Cluster} methods have an underlying manifold as their constraints, but for the ``Coupled Cluster manifold'' we do not know if there are closed formulas for the objects necessary to implement Riemannian algorithms for it, such as the gradient, Hessian, geodesic etc.
Another place where Riemannian Optimization can be used is in Density Functional Theory methods.
In \cite{headgordon2002} the authors implemented a Riemannian quasi-Newton to the \emph{B3LYP} functional and the results obtained were very interesting.

Now, to the reader interested mainly in Riemannian Optimization, there are many other applications outside (Quantum) Chemistry, see \cite{boumal2022, edelman1998}, but two very active fields in which Riemannian Optimization is gaining traction are Data Science \cite{trendafilov2022} and Machine Learning \cite{weber2021}.
This is promising because these two references shows how Riemannian Optimization can help these fields, but the other way around is also true.
In \cite{becigneul2018}, for example, the authors propose how to translate state-of-the-art adaptive methods used by Machine Learning practitioners to a Riemannian framework.
Having said that, the main interests of people working specifically with Riemannian Optimization these days are (1) improving and developing new algorithms, specially for nonsmooth cost functions \cite[A Collection of Nonsmooth Riemannian Optimization Problems]{hosseini2019}, and (2) understanding the role that the curvature of the manifold play in the complexity of algorithms \cite{criscitiello2022}.

Now let us show some ideas of what can be improved (or extended) in this work:
\begin{enumerate}

\item One thing that can be done is tweaking the algorithms to use cheaper retractions instead of the Riemannian exponential.
  By doing this the points are not updated along a geodesic and this may increase the number of iterations until convergence is reached or even decrease the number of molecules converged, but some retractions are way cheaper to compute than the exponential, so, it is a trade-off.

\item All the algorithms implemented are called \emph{line search} algorithms but there is a different approach we did not explore at all called \emph{trust region} with many other algorithms one can try.
  Some of these algorithms are described in the book \cite{boumal2022}, but many variants can be found in the Manopt package \cite{boumal2014}.

\item Another idea to explore is using different parametrizations of the Grassmannian.
  Some parametrizations were already explored for Hartree--Fock, see \cite{feldman2022, headgordon2002}, but the one we used here (quotient of the Stiefel manifold) and the one described in \cite{lai2020} were not, as far as we know.

\item As already mentioned, finding theoretical or experimental heuristics that shows which combination of algorithms to use or when it is better to switch between algorithms is something worthwhile pursuing if one wants an implementation that is going to be used by the community.

\item For quantum chemists, a couple of ideas are (1) extending the present work to the Density Functional Theory framework; (2) improving the current implementation to explore the symmetries of molecules, for example; and (3)
implementing the restricted version of HF (see Definition \ref{restricted_hf}).

\end{enumerate}

And, to conclude, some general remarks and advices about Riemannian Optimization in general.
This field provides many powerful tools to model and tackle optimization problems, as we already saw, but there are some difficulties one has to overcome to use these tools:
\begin{enumerate}

\item First of all, to use Riemannian algorithms the reader should have an optimization problem defined on a manifold and it may not be easy to spot a manifold if the reader is not familiar with Riemannian Geometry.

\item The manifold can have different parametrizations and choosing the best one usually requires experience with numerical methods.

\item The manifold may not have been explored in the literature, specially from a numerical perspective.

\item In order to implement a full-fledged Riemannian algorithm one needs to have a good grasp of Riemannian Geometry, which is not an easy topic.

\end{enumerate}
Some of these problems were partially addressed in the present work.
In Chapter \ref{geometry}, for example, we stated one of the most powerful tools to ``spot'' manifolds, which is the Preimage Theorem \ref{preimage_theorem}, and we used this theorem to show that the Stiefel manifold is, indeed, a manifold.
The last item is also partially addressed by Chapter \ref{geometry}, since we provided all the tools necessary to implement at least three Riemannian algorithms.
We did not address the second problem, but for the case of the Grassmannian we refer the reader to \cite{lai2020} for a very detailed numerical comparison between different parametrizations.
The third problem is probably the worst because it means one need to find the formulas for itself and this can be quite challenging.
Again, in Chapter \ref{geometry} we provided some theorems and a couple of examples of what one should compute and how the computation can be done, but experience tells that a good grasp of Riemannian Geometry is required to compute everything correctly.
Two references to help with this problem are \cite{absil2008, boumal2022}.
Having said that, it is worth mentioning that one can find a very big list of manifolds for which there are already algorithms implemented at the website of the Manopt package \cite{boumal2014}.
The last advice is that if the reader is interested in using Riemannian Optimization but do not know Riemannian Geometry, then using algorithms from packages such as Manopt is a good option because it has many algorithms already implemented for lots of manifolds.
So, that means one can check which algorithm is the best for their problem without having to spend months or even years implementing them.
On top of that, the source code of this package is very good and it can help with the implementation if the reader is interested in pursuing one.


\appendix

\chapter{Topology} \label{top}

\epigraph{Topology belonged to the stratosphere of human thought. It might conceivably turn out to be of some use in the twenty-fourth century, but for the time being...}{Aleksandr Solzhenitsyn}

The goal of this appendix is to define the topological tools used in this work and to state Proposition \ref{compact_maxmin}, which is the result that guarantees the Hartree--Fock Method has at least one solution.
The proofs were omitted, but they can be found in main references for this appendix, which are \cite{munkres2014, rudin1976, sutherland2009}.

\section{Metric spaces} \label{metric_spaces}

\begin{defi}
  A set $X$ is said to be a \emph{metric space}\index{space!metric} if it is possible to define a function $d:X \x X \to \br$ such that the following axioms are valid for every $x, y, z \in X$:
  \begin{enumerate}
    
  \item $d(x, y) = 0$ \tiff $x = y$.

  \item $d(x, y) = d(y, x)$.

  \item $d(x, z) \leq d(x, y) + d(y, z)$.

  \end{enumerate}
  The function $d$ is called \emph{metric}\index{metric} or \emph{distance}\index{distance}, the second axiom is called \emph{symmetry of the metric} and the third is called \emph{triangle inequality}.
  The elements of $X$ are also usually called \emph{points}\index{point}.
\end{defi}

\begin{obs}
  Note that $d(x, y) > 0$ if $x \neq y$.
  Indeed, using the three axioms above, we obtain the following inequality: $0 = d(x, x) \leq d(x, y) + d(y, x) = 2d(x, y)$.
  Dividing this inequality by $2$, we obtain $0 \leq d(x, y)$.
  Using the first axiom again, the result follows because $d(x, y) = 0$ \tiff $x = y$, which is not the case since we assumed $x \neq y$.
\end{obs}

\begin{exm} \label{discrete_metric}
  Every nonempty set $X$ can be endowed with a metric, just consider the function defined by $d(x, y) = 1$, if $x \neq y$, and $d(x, y) = 0$, if $x = y$.
  This metric is called \emph{discrete metric}\index{metric!discrete}.
\end{exm}

\begin{exm} \label{std_metric}
  $\bq$, $\br$ and $\bc$ are metric spaces with the absolute value as the metric, \ie, $d(x, y) = \abs{x - y}$.
  Remember, the absolute value for $\bq$ and $\br$ is defined as: $\abs{x} = x$, if $x \geq 0$, and $\abs{x} = -x$, if $x < 0$.
  The absolute value for $\bc$ is defined as $\sqrt{z\ol{z}}$, being $\ol{z}$ the complex conjugate of $z$, that is, if $z = x + iy$, then $\ol{z} \ceq z - iy$.
\end{exm}

\begin{exm} \label{metrics_rn}
  $\bq^d$, $\br^d$ and $\bc^d$ are metric spaces and the most common metrics used are:
  \begin{enumerate}

  \item $d((x_1, \ldots, x_d), (y_1, \ldots, y_d)) \ceq \sqrt{(x_1 - y_1)^2 + \ldots + (x_d - y_d)^2}$.
    This metric is usually called \emph{Euclidean metric}\index{metric!Euclidean}.

  \item $d((x_1, \ldots, x_d), (y_1, \ldots, y_d)) \ceq \max\qty{\abs{x_i - y_i} : i = 1, \ldots, d}$.
    This metric is usually called \emph{maximum metric}\index{metric!maximum}.

  \item $d((x_1, \ldots, x_d), (y_1, \ldots, y_d)) \ceq \abs{x_1 - y_1} + \ldots + \abs{x_d - y_d}$.
    This metric is usually called \emph{taxicab metric}\index{metric!taxicab}.

  \end{enumerate}
  All these metrics are equivalent, that is, they induce the same topology (see Proposition \ref{induced_topology} for the definition of induced topology).
  Actually, these metrics are induced by norms (see Example \ref{norms_equivalent}) and every two metrics in a finite-dimensional vector space that are induced by two norms are equivalent to each other (for a proof of this, see \cite{kconradnorm}).
\end{exm}

\begin{defi}
  Given two metric spaces $(X, d_X)$ and $(Y, d_Y)$ such that $X \subset Y$, we say that $X$ is a \emph{subspace of $Y$}\index{metric!subspace} if $d_Y|_{X \x X} = d_X$.
\end{defi}

\begin{exm}
  $\bq^d$ is a subspace of $\br^d$ and $\bc^d$.
\end{exm}

\begin{exm}
  $\br^d$ is a subspace of $\bc^d$.
\end{exm}

\begin{defi}
  Given a nonempty set $X$, a \emph{sequence in $X$}\index{sequence} is a function $f:\bn \to X$.
  The elements $f(n)$ are called \emph{terms of the sequence}\index{sequence!term}.
\end{defi}

\begin{obs}
  Given a sequence $f$, we usually denote the term $f(n)$ as $x_n$ and the sequence as $\qty{x_n}$.
  Note that, using this notation, the function $f$ is implicitly described by the terms $x_n$.
\end{obs}

\begin{defi}
  Given a metric space $X$, a sequence $\qty{x_n}$ and an element $x \in X$, we say that the sequence \emph{$\qty{x_n}$ converges to $x$}\index{sequence!convergence} if for every $\eps > 0$ there exists $n_0 \in \bn$ such that $d(x_n, x) < \eps$ for every $n \geq n_0$.
  The element $x$ is called the \emph{limit of the sequence $\qty{x_n}$}\index{sequence!limit}.
\end{defi}

\begin{obs}
  A common notation to represent the convergence of a sequence $\qty{x_n}$ to $x$ is $\qty{x_n} \to x$.
\end{obs}

\begin{defi}
  Given a metric space $X$ and a sequence $\qty{x_n}$, the sequence is said to be a \emph{Cauchy sequence}\index{sequence!Cauchy} if for every $\eps > 0$ there exists $n_0 \in \bn$ such that $d(x_n, x_m) < \eps$ for every $n, m \geq n_0$.
\end{defi}

\begin{prop}
  Given a metric space $X$ and a sequence $\qty{x_n}$ in $X$, if the sequence converges to a point $x$, then it is a Cauchy sequence.
\end{prop}

\begin{exm}\label{incompleteness_of_q}
  The converse of the previous proposition is not true.
  Indeed, assuming the Euler number $e$ is irrational, \ie, $e \in \br \ssm \bq$, we can consider the metric space $\bq$ and the sequence $x_n = \qty(1 + \frac{1}{n})^n$ as our example.
  It is well known (and it is proven in \cite[Chapter~3]{rudin1976}) that the previous sequence converges to $e$ and note that $x_n \in \bq$ for every $n$.
  So, the sequence is a Cauchy sequence, since it converges, it is in $\bq$, but it does not converge to a point in $\bq$.
\end{exm}

The last example motivates the following definition:
\begin{defi} \label{complete_space}
  A metric space $X$ is said to be a \emph{complete space}\index{space!complete} if every Cauchy sequence in $X$ converges to a limit in $X$.
\end{defi}

With this definition, it is possible to define Hilbert spaces, see \ref{hilbert_space}.

\begin{exm}
  $\br^d$ is complete, but $\bq^d$ is not.
\end{exm}

\begin{defi}
  Given a metric space $X$, a point $x \in X$ and a real number $r > 0$, let us define three very important sets:
  \begin{enumerate}

  \item $B(x, r) \ceq \qty{y \in X : d(x, y) < r}$ is called \emph{open ball centered in $x$ of radius $r$}\index{ball!open}.

  \item $B[x, r] \ceq \qty{y \in X : d(x, y) \leq r}$ is called \emph{closed ball centered in $x$ of radius $r$}\index{ball!closed}

  \item $S(x, r) \ceq \qty{y \in X : d(x, y) = r}$ is called \emph{sphere centered in $x$ of radius $r$}\index{sphere}

  \end{enumerate}
\end{defi}

\begin{exm}
  If $X = \br$, then an open (resp. closed) ball is just an open (resp. closed) interval, \ie, $B(x, r) = (x-r, x+r)$ and $B[x, r] = [x-r, x+r]$.
  The sphere in this case consists of two points: $S(x, r) = \qty{x-r, x+r}$.
\end{exm}

\begin{defi} \label{open_metric}
  Given a metric space $X$ and a subset $U \subset X$, we call $U$ an \emph{open set}\index{set!open} if for every $x \in U$ there exists $r > 0$ such that $B(x, r) \subset U$.
  A \emph{closed set}\index{set!closed} is a set $F$ whose complement $X \ssm F$ is open.
\end{defi}

\begin{exm} \label{open_ball_is_open}
  Open balls are open sets.
\end{exm}

\begin{defi}
  Given a metric space $X$ and a subset $S \subset X$, we say that $S$ is \emph{bounded}\index{set!bounded} if there exists $x \in X$ and $r > 0$ such that $S \subset B(x, r)$.
  Sets that are not bounded are called \emph{unbounded}\index{set!unbounded}.
\end{defi}

\begin{obs}
  If $S$ is contained in the open ball $B(x, r)$, then it is also contained in the closed ball $B[x, r]$.
\end{obs}

\begin{exm}
  The intervals $[a, b]$, $(a, b)$, $[a, b)$, $(a, b]$ are bounded, just consider $x = 0$ and $r = 1 + \max\qty{\abs{a}, \abs{b}}$ in the previous definition.
  However, the intervals $(-\infty, x]$, $(-\infty, x)$, $(x, +\infty)$, $[x, +\infty)$ are unbounded.
\end{exm}

\begin{defi}
  Consider two metric spaces $(X, d_X)$ and $(Y, d_Y)$.
  Given $x \in X$ and a function $f:X \to Y$, we say that \emph{$f$ is continuous at $x$}\index{function!continuous (metric)} if for every $\eps > 0$ there exists $\delta > 0$ such that, if $x' \in X$ satisfies $d_X(x, x') < \delta$, then $f(x')$ satisfies $d_Y(f(x), f(x')) < \eps$.
  We say that $f$ is \emph{continuous} if it is continuous at every point of its domain.
\end{defi}

\begin{obs} \label{obs_metric_continuous}
  We can rephrase the above definition as: for every $\eps > 0$ there exists $\delta > 0$ such that $f(B(x, \delta)) \subset B(f(x), \eps)$.
\end{obs}

\begin{defi}
  Given two metric spaces $(X, d_X)$, $(Y, d_Y)$ and a function $f:X \to Y$, we say that $f$ is a \emph{(metric) isometry}\index{metric!isometry} if $d_X(x,x') = d_Y(f(x), f(x'))$ for every $x, x' \in X$.
\end{defi}

\begin{obs}
  Isometries are continuous and injective functions.
\end{obs}

\section{Topological spaces}

\begin{defi}
  Given a set $X$, a \emph{topology on $X$}\index{topology} is a set $\mc{T}$ whose elements are subsets of $X$ and such that the following axioms are true:
  \begin{enumerate}

  \item $\emptyset, X \in \mc{T}$.

  \item $U \cap V \in \mc{T}$ for every $U, V \in \mc{T}$.

  \item $\bigcup_{U \in \, \mc{U}} U \in \mc{T}$ for every $\mc{U} \subset \mc{T}$.

  \end{enumerate}
  The elements of $\mc{T}$ are called \emph{open sets}\index{set!open}, the complement of the open sets are called \emph{closed sets}\index{set!closed} and the pair $(X, \mc{T})$ is called \emph{topological space}\index{space!topological}.
\end{defi}

\begin{obs}
  When the topology is not important or it is clear which topology we are talking about, we will denote a topological space by $X$ instead of $(X, \mc{T})$.
\end{obs}

\begin{obs}
  The second axiom can be extended to the intersection of finite sets by induction.
  Indeed, if $U_1, \ldots, U_n$ are open sets and we know the result is true for $n-1$, then it is true for $n$ because $U_1 \cap \ldots \cap U_n = (U_1 \cap \ldots \cap U_{n-1}) \cap U_n$.
\end{obs}

\begin{obs}
  $\emptyset$ and $X$ are both closed and open sets.
\end{obs}

\begin{exm} \label{trivial_topologies}
  Any set $X$ has at least two topologies: $\qty{\emptyset, X}$ and $\wp(X)$, being $\wp(X)$ the set of all subsets of $X$, also called \emph{power set of $X$}\index{set!power}\index{power set}.
  The first topology is called \emph{trivial}\index{topology!trivial} and the second is called \emph{discrete}\index{topology!discrete}.
\end{exm}

\begin{exm} \label{quotient_topology}
  Given a topological space $(X, \mc{T}_X)$ and an equivalence relation $\sim$, we can define a topology $\mc{T}_{X/\sim}$ on the quotient $X/\sim$, called \emph{quotient topology}\index{quotient topology}\index{topology!quotient}, as follows: $U \in \mc{T}_{X/\sim}$ \tiff $\pi^{-1}(U) \in \mc{T}_X$, being $\pi:X \to X/\sim$ the canonical projection.
\end{exm}

\begin{exm}
  Given a topological space $(X, \mc{T}_X)$ and a subset $Y \subset X$, we can define the following topology on $Y$: $\mc{T}_Y \ceq \qty{U \cap Y : U \in \mc{T}_X}$.
  This topology is called \emph{subspace topology}\index{topology!subspace}.
\end{exm}

\begin{obs}
  A set can be open or closed with respect to a subspace without being open or closed with respect to the whole space.
  An example of this is: let $X = \br$, $Y = [0, 2]$, $U = [-1, 1)$, $F = [1, 3)$.
  Note that $U \cap Y = [0, 1) = (-1, 1) \cap Y$, so, the set $U$ is open in $Y$, but it is not open in $X$.
  The same is true for $F$, that is, $F \cap Y = [1, 2] = [1, 3] \cap Y$, which means it is closed in $Y$, but not in $X$.
\end{obs}

\begin{prop} \label{induced_topology}
  Given a metric space $(X, d)$, the set
  \begin{equation*}
    \mc{T}_d \ceq \qty{U \subset X : \text{for every $x \in U$ there exists $r > 0$ such that $B(x, r) \subset U$}}
  \end{equation*}
  is a topology on $X$ called \emph{topology induced by the metric $d$}\index{topology!induced}.
\end{prop}

\begin{obs}
  The open sets in this case are exactly the same as we defined in Definition \ref{open_metric}.
\end{obs}

\begin{exm}
  The discrete metric defined in \ref{discrete_metric} induces the discrete topology defined in \ref{trivial_topologies}.
\end{exm}

\begin{defi}
  Given a set $X$, a \emph{basis}\index{topology!basis} for a topology in $X$ is a collection $\mc{B}$ of subsets of $X$ (called \emph{basis elements}) such that the following is true for every $x \in X$ and $B_1, B_2 \in \mc{B}$:
  \begin{enumerate}

  \item There exists at least one $B \in \mc{B}$ such that $x \in B$.

  \item If $x \in B_1 \cap B_2$, then there exists $B_3 \in \mc{B}$ such that $x \in B_3 \subset B_1 \cap B_2$.

  \end{enumerate}
  If $\mc{B}$ satisfies these two axioms, then we can define the \emph{topology $\mc{T}$ generated by $\mc{B}$} as follows: a subset $U \subset X$ is an open set, \ie, $U \in \mc{T}$, \tiff for every $x \in U$ there exists $B \in \mc{B}$ such that $x \in B \subset U$.
\end{defi}

\begin{obs}
  The basis elements are open sets in the topology generated by the basis.
\end{obs}

\begin{defi} \label{product_topology}
  Let $(X_1, \mc{T}_{X_1}), \ldots, (X_n, \mc{T}_{X_n})$ be topological spaces.
  The \emph{product topology (on $X_1 \x \ldots \x X_n$)}\index{topology!product} is the topology generated by the basis $\mc{B} \ceq \qty{U_1 \x \ldots \x U_n : U_i \in \mc{T}_{X_i}}$.
\end{defi}

\begin{defi}
  A topological space $X$ is said to be \emph{second-countable}\index{space!second-countable}\index{second-countable} if it has a countable\footnote{By that we mean that the cardinality of the basis is countable.} basis for its topology.
\end{defi}

\begin{exm}
  $\bq^d$, $\br^d$ and $\bc^d$ are second-countable spaces and
  \begin{equation}
    \mc{B} \ceq \qty{B((x_1, \ldots, x_d), r) : \re(x_i), \im(x_i) \in \bq, \ i=1,\ldots,d, \ r \in \bq_{>0}}
  \end{equation}
  is a countable basis for all three spaces.
  In the cases $\bq^d$ and $\br^d$ we have $\im(x_i) = 0$.
\end{exm}

\begin{exm}
  Subspaces of second-countable spaces are second-countable.
\end{exm}

\begin{defi}
  Let $(X, \mc{T}_X)$, $(Y, \mc{T}_Y)$ be topological spaces and $f:X \to Y$ be a function.
  We say that \emph{$f$ is continuous}\index{function!continuous (topological)} if $f^{-1}(U) \in \mc{T}_X$ for every $U \in \mc{T}_Y$, being $f^{-1}(U) \ceq \qty{x \in X : f(x) \in U}$.
  This set is called \emph{preimage}\index{preimage} or \emph{inverse image}\index{inverse image} of $U$ and this definition can also be stated as follows: $f$ is continuous if the preimage of open sets are open.
\end{defi}

The reader may ask if the definition of continuous function in the metric sense and in the topological sense are the same, and they are.

\begin{defi}
  Given two topological spaces $(X, \mc{T}_X)$ and $(Y, \mc{T}_Y)$, a point $x \in X$ and a function $f:X \to Y$, we say that \emph{$f$ is continuous at $x$} if for every $V \in \mc{T}_Y$ such that $f(x) \in V$ there exists $U \in \mc{T}_X$ such that $x \in U$ and $f(U) \subset V$.
\end{defi}

\begin{prop}
  A function between topological spaces $f:X \to Y$ is continuous \tiff it is continuous at every point.
\end{prop}

\begin{prop}
  If $f:X \to Y$ is a function between metric spaces, then the metric definition of continuity at a point is equivalent to the topological definition (the topology we are considering is the induced by the metric).
\end{prop}

\begin{exm} \label{vector_space_topology}
  Let $V$ be a $d$-dimensional vector space over $\br$.
  The topology $\mc{T}_V$ we will consider on $V$ is the one that makes every linear transformation $T:\br^d \to V$ a continuous function.
  Explicitly, $U \in \mc{T}_V$ \tiff $T^{-1}(U) \in \mc{T}_{\br^d}$ for every linear transformation $T$.
\end{exm}

\begin{defi}
  Given three topological spaces $X, Y, Z$ and two continuous functions $f:X \to Z$ and $g:Y \to X$, a \emph{lifting of $f$ to $Y$}\index{lift} is a function $h:Y \to Z$ such that $h = f \circ g$.
  This can be represented diagrammatically as:
  \begin{center}
    \begin{tikzcd}
      Y \arrow{d}[left]{g} \arrow[dashed]{dr}[above]{h} & \\
      X \arrow{r}[below]{f} & Z
    \end{tikzcd}
  \end{center}
\end{defi}

\begin{obs}
  The previous definition is not the standard definition of lifting used in the literature, but we could not find a better name to this particular case.
\end{obs}

\begin{defi}
  A function $f:X \to Y$ between topological spaces is said to be a \emph{homeomorphism}\index{homeomorphism} if $f$ is bijective, continuous and its inverse $f^{-1}$ is also continuous.
\end{defi}

\begin{defi}
  A function $f:X \to Y$ between topological spaces is said to be \emph{open} (usually referred as an \emph{open map}\index{open map}) if $f(U)$ is open in $Y$ for every $U$ open in $X$.
\end{defi}

\begin{defi}
  Given a topological space $(X, \mc{T})$ and a subset $S \subset X$, we define the \emph{closure of $S$}\index{closure} to be the intersection of all closed sets in $X$ that contains $S$.
  In other words, if we define $\mc{F}_S \ceq \qty{F \subset X : \text{$F$ is closed and $S \subset F$}}$, then the closure of $S$ is the set $\ol{S} \ceq \bigcap_{F \in \mc{F}_S} F$.
\end{defi}

\begin{defi}
  Given a topological space $(X, \mc{T})$ and a subset $S \subset X$, we say that $S$ is a \emph{dense subset of $X$}\index{dense} if $\ol{S} = X$.
\end{defi}

\begin{defi} \label{completion}
  Given two metric spaces $(X, d)$ and $(\what{X}, \what{d})$, we say that $\what{X}$ is the \emph{completion of $X$}\index{completion}\index{space!completion} if $\what{X}$ is a complete metric space and there exists an isometry $f:X \to \what{X}$ in the metric sense such that $\ol{\im{f}} = \what{X}$.
\end{defi}

\begin{obs}
  If $X$ is already a subspace of $\what{X}$, then we can replace the isometry condition by the hypothesis $\ol{X} = \what{X}$, which is equivalent to considering the isometry to be the inclusion $i:X \to \what{X}$ defined by $i(x) = x$.
\end{obs}

\begin{exm}
  $\br$ is the completion of $\bq$ with the standard metric defined in Example \ref{std_metric}.
\end{exm}

\begin{thm}
  Every metric space has a completion.
\end{thm}

\begin{defi}
  Given $d \in \bn$ and a topological space $X$, we say that $X$ is \emph{locally Euclidean of dimension $d$}\index{space!locally euclidean}\index{locally euclidean} if for every $x \in X$ there exists an open set $U \subset X$ such that $x \in U$ and $U$ is homeomorphic to an open set $V \subset \br^d$.
  (We are considering the subspace topology in $U$ and $V$.)
\end{defi}

\begin{exm}
  The Euclidean space itself is locally Euclidean.
\end{exm}

\begin{defi}
  Given a topological space $(X, \mc{T})$, a subset $S \subset X$ and $\mc{U} \subset \mc{T}$, we say $\mc{U}$ is a \emph{cover of $S$}\index{cover} if $S \subset \bigcup_{U \in \mc{U}} U$.
\end{defi}

\begin{obs}
  Some authors call the above \emph{open cover} instead of cover because they ask $\mc{U}$ to be a subset of $\wp(X)$ and not necessarily of $\mc{T}$ as we did.
  However, in practice we will just use open covers and for the definition below it is open covers that matters, so, every cover is an open cover in the present work.
\end{obs}

\begin{defi}
  Given a topological space $X$ and $K \subset X$, we say that $K$ is \emph{compact}\index{space!compact} if every cover of $K$ has a finite subcover.
  In symbols, for every cover $\mc{U}$ of $K$ there exists $\qty{U_1, \ldots, U_n} \subset \mc{U}$ such that $K \subset \bigcup_{i=1}^n U_i$.
\end{defi}

\begin{exm}
  Finite sets are compact.
\end{exm}

\begin{prop} \label{product_compact}
  Let $X_1, \ldots, X_n$ be compact topological spaces.
  Then $X_1 \x \ldots \x X_n$ with the product topology is also compact.
\end{prop}

\begin{defi}
  A function $f:X \to Y$ between topological spaces is said to be a \emph{proper function}\index{function!proper} or \emph{proper map} if the preimage $f^{-1}(K) \subset X$ is compact for every $K \subset Y$ compact.
  In other words, the preimage of compact sets are compact.
\end{defi}

\begin{prop} \label{compactness_relative}
  Let $(X, \mc{T}_X)$ be a topological space and suppose $K \subset Y \subset X$.
  Then $K$ is compact with respect to covers of open sets of $X$ \tiff $K$ is compact with respect to covers of open sets of $Y$ with the subspace topology.
  In other words, $K$ is compact in $X$ \tiff $K$ is compact in $Y$.
\end{prop}

\begin{defi}
  A topological space $X$ is said to be \emph{Hausdorff}\index{space!Hausdorff} if for every distinct $x, y \in X$ there exists disjoint open sets $U$ and $V$ such that $x \in U$ and $y \in V$.
\end{defi}

\begin{prop} \label{metric_hausdorff}
  Metric spaces are Hausdorff.
\end{prop}

\begin{prop}
  Given a Hausdorff space $X$ and $x \in X$, the set $\qty{x}$ is closed in $X$.
\end{prop}

\begin{prop} \label{compact_is_closed}
  If $X$ is a Hausdorff space and $K \subset X$ is a compact set, then $K$ is closed.
\end{prop}

\begin{prop} \label{compact_is_bounded}
  Given a metric space $X$ and $K \subset X$ compact, then $K$ is bounded.
\end{prop}

\begin{prop} \label{closed_in_compact}
  Given a compact topological space $X$ and $F \subset X$ closed, then $F$ is compact.
\end{prop}

\begin{prop} \label{continuous_function_compact}
  If $f:X \to Y$ is a continuous function between topological spaces and $X$ is compact, then $f(X)$ is also compact.
\end{prop}

\begin{thm}[Heine--Borel] \label{heine_borel}
  A subset $S \subset \br^d$ is compact if, and only if, it is closed and bounded.
\end{thm}

\begin{thm} \label{compact_maxmin}
  If $f:X \to \br$ is a continuous function defined in a nonempty compact space $X$, then $f$ function has maximum and minimum, that is, there exists $x_m, x_M \in X$ such that $f(x_m) \leq f(x) \leq f(x_M)$ for every $x \in X$.
\end{thm}

\begin{proof}
  By Proposition \ref{continuous_function_compact}, $f(X)$ is a compact set, and, by Heine--Borel (Theorem \ref{heine_borel}), $f(X)$ is bounded and closed.
  So, since it is bounded and nonempty, the Least Upper Bound property tells us that there exists the supremum and the infimum of the set $f(X)$.
  If we prove that the supremum and infimum are actually in $f(X)$, the result follows.
  Let $s \ceq \sup{f(X)}$ and suppose $s \notin f(X)$.
  Then, since $f(X)$ is closed, $\br \ssm f(X)$ is open and there exists $r > 0$ such that $B(s, r) \subset \br \ssm f(X)$.
  But now note that, if $r$ exists, then $s-r < s$ and $y < s-r$ for every $y \in f(X)$.
  The first claim is true because we are subtracting a positive number from $s$, so it will be less than $s$ itself.
  And the second claim is true because if there exists $y \in f(X)$ such that $y \geq s-r$, then $y \in B(s, r)$, which is a contradiction with the fact that $B(s, r) \subset \br \ssm f(X)$ (remember $y \leq s$ by definition of $s$ being the supremum of $f(X)$).
  So, we conclude that, if $r$ exists, we have a contradiction, since $s-r$ would be a number less than $s$ that is an upper bound of $f(X)$, which cannot exist by definition of supremum being the least upper bound.
  In other words, $r$ cannot exist and, therefore, $s$ should be in $f(X)$, because what led to the existence of $r$ was supposing $s \notin f(X)$.
  The result follows, since $s$ being in $f(X)$ means, by definition, the existence of $x_M \in X$ such that $f(x_M) = s$ and, since $s$ is the supremum of the image, we know $f(x) \leq s = f(x_M)$ for every $x \in X$.
  The proof for the infimum is analogous.
\end{proof}

The result above shows that Hartree--Fock is well-defined (see Theorem \ref{hf_welldefined}).


\chapter{Linear Algebra} \label{linalg}

\epigraph{We [he and Halmos] share a philosophy about linear algebra: we think basis-free, we write basis-free, but when the chips are down we close the office door and compute with matrices like fury.}{Irving Kaplansky}

The goal of this appendix is to set the language, define and construct some particular vector spaces that we use in this work such as Hilbert spaces and the tensor, exterior and symmetric powers of vector spaces.
Again, we will not prove most of the results stated in here and the main reference for this appendix is \cite{kostrikin1997}.

\section{Basic definitions}

\begin{defi} \label{group}
  A nonempty set $G$ is said to be a \emph{group}\index{group} if it is possible to define three functions, which we are going to call $m:G \x G \to G$, $(\cdot)^{-1}:G \to G$ and $e:\qty{\bullet} \to G$, such that $G$ with these functions satisfies the following axioms for every $x, y, z \in G$:
  \begin{enumerate}

  \item $m(x, m(y, z)) = m(m(x, y), z)$.

  \item $m(e(\bullet), x) = m(x, e(\bullet)) = x$.

  \item $m((\cdot)^{-1}(x), x) = m(x, (\cdot)^{-1}(x)) = e(\bullet)$.

  \end{enumerate}
  We usually write $m(x, y)$ as $x \cdot y$ or just $xy$; $(\cdot)^{-1}(x)$ as $x^{-1}$; and $e(\bullet)$ as $e$ or sometimes as $1$.
  Changing the notation, the axioms are:
  \begin{enumerate}

  \item $x(yz) = (xy)z$.

  \item $xe = ex = x$.

  \item $xx^{-1} = x^{-1}x = e$.

  \end{enumerate}
  The function $m$ is called \emph{group operation} or just \emph{group product}, the first axiom is called \emph{associativity}, the element $e$ is called \emph{identity} and $x^{-1}$ is called \emph{inverse of $x$}.
\end{defi}

\begin{exm}
  Given a set $X$, the set $S_X \ceq \qty{\sigma:X \to X : \text{$\sigma$ is a bijection}}$ with the composition of functions as the operation is a group.
  This group is usually called the \emph{symmetry group of $X$}\index{group!symmetric} and its elements $\sigma$ are also called \emph{permutations}\index{permutation}.
  When the set is finite, we will usually denote $S_X$ by $S_n$, being $n$ the cardinality of $X$.
\end{exm}

\begin{defi}
  A group $G$ is said to be \emph{abelian}\index{group!abelian} or \emph{commutative}\index{group!commutative} if the axiom $xy = yx$ is satisfied for every $x, y \in G$.
\end{defi}

\begin{obs} \label{zero_commutative}
  When a group is commutative, we usually denote its operation by $+$ and the identity by $0$.
\end{obs}

\begin{exm}
  The integers $\bz$ with the operation of addition is a commutative group.
\end{exm}

\begin{exm}
  A non-example, actually, is the group $S_X$ if $X$ has $3$ elements or more.
  To see this, consider $X = \qty{a, b, c}$ and the following functions:
  \begin{align*}
    \sigma:X & \to X & \tau:X & \to X \\
    a & \mapsto b, & a & \mapsto a, \\
    b & \mapsto c, & b & \mapsto c, \\
    c & \mapsto a, & c & \mapsto b.
  \end{align*}
  Note that $\sigma(\tau(a)) = \sigma(a) = b$, but $\tau(\sigma(a)) = \tau(b) = c$, which means $\sigma \circ \tau \neq \tau \circ \sigma$.
\end{exm}

\begin{defi}
  Given a set $X$ with at least two elements, we call $\sigma \in S_X$ a \emph{transposition}\index{transposition} if $\sigma$ moves only two elements of the set, \ie, there exists $X' = \qty{x, y} \subset X$ such that $\sigma(x) = y$, $\sigma(y) = x$ and $\sigma(z) = z$ for every $z \in X \ssm X'$.
\end{defi}

\begin{defi}
  Given a finite set $X$ with at least two elements and a permutation $\sigma \in S_X$, we define the \emph{sign of $\sigma$}\index{sign of permutation} to be $1$ if $\sigma$ can be written as a composition of an even number of transpositions and $-1$ otherwise.
  We will denote the sign of $\sigma$ by $\sign(\sigma)$.
\end{defi}

\begin{obs}
  The sign is well-defined, \ie, a permutation cannot be even and odd at the same time.
\end{obs}

\begin{prop} \label{sign_composition}
  Given $\sigma, \tau \in S_X$, we have $\sign(\sigma)\sign(\tau) = \sign(\sigma \circ \tau)$.
\end{prop}

\begin{defi}
  Given a nonempty set $k$, we say that $k$ is a \emph{field}\index{field} if it is possible to define two functions on $k$, that we shall call $+$ and $\cdot$, such that $k$ with these functions satisfies the following axioms:
  \begin{enumerate}
    
  \item The set $k$ with the function $+$ is a commutative group.

  \item The set $k \ssm \0$ with the function $\cdot$ is also a commutative group.

  \item $x \cdot (y + z) = x \cdot y + x \cdot z$ for every $x, y, z \in k$.

  \end{enumerate}
  The last axiom is called \emph{distributivity}.
\end{defi}

\begin{obs}
  Being pedantic, there are actually six functions involved in the definition above:
  \begin{align}
    + & : k \x k \to k,        & \cdot & : (k \ssm \0) \x (k \ssm \0) \to (k \ssm 0), \\
    - & : k \to k,             & (\cdot)^{-1} & : (k \ssm \0) \to (k \ssm \0), \\
    0 & : \qty{\bullet} \to k, & 1 & : \qty{\bullet} \to k.
  \end{align}
\end{obs}

\begin{exm}
  $\bq$, $\br$ and $\bc$ with the usual operations of addition and multiplication are fields.
\end{exm}

\begin{defi}
  Let $V$ be a nonempty set and $k$ be a field.
  We say that $V$ is a \emph{vector space over $k$}\index{space!vector}\index{vector space} if it is possible to define two functions, $a:V \x V \to V$ and $b:k \x V \to V$, such that $V$ with these functions satisfies the following axioms for every $x, y \in V$ and $r, s \in k$:
  \begin{enumerate}
    
  \item $V$ with the function $a$ is a commutative group.

  \item $b(s, b(r, x)) = b(s \cdot r, x)$.

  \item $b(1, x) = x$, being $1$ the identity of $k$ with respect to the $\cdot$ operation.

  \item $b(r, a(x, y)) = a(b(r, x), b(r, y))$.

  \item $b(r + s, x) = a(b(r, x), b(s, x))$.

  \end{enumerate}
  The function $a$ is called \emph{vector addition} and $b$ is called \emph{scalar multiplication}.
  Usually the vector addition $a(x, y)$ is written as $x + y$ and the scalar multiplication $b(r, x)$ is written as $r \cdot x$ or $rx$.
  However, since we use the same symbols for the operations of $k$ as a field, the reader should keep track of when we are using $\cdot$ as a scalar multiplication and when it is the field operation.
  The same goes for the $+$ symbol, of course.
  Using this notation, the above axioms are:
  \begin{enumerate}
    
  \item $V$ with the function $+$ is a commutative group.

  \item $s(rx) = (s \cdot r)x$.

  \item $1x = x$, being $1$ the identity of $k$ with respect to the $\cdot$ operation.

  \item $r(x + y) = rx + ry$.

  \item $(r + s)x = rx + sx$.

  \end{enumerate}
\end{defi}

\begin{exm}
  $\bq^d$, $\br^d$ and $\bc^d$ with the usual componentwise addition and scalar multiplication, \ie, $(x_1, \ldots, x_d) + (y_1, \ldots, y_d) \ceq (x_1 + y_1, \ldots, x_d + y_d)$ and $r(x_1, \ldots, x_d) \ceq (rx_1, \ldots, rx_d)$, are vector spaces over $\bq$, $\br$ and $\bc$, respectively.
  The case $n = 0$ amounts to the trivial vector space, \ie, $\bq^0 = \br^0 = \bc^0 = \0$.
\end{exm}

\begin{obs}
  When the underlying field of the vector space is not important, we will not mention it.
\end{obs}

\begin{exm}
  Given a set $X$ and a vector space $V$, the set $\mc{F}(X, V) \ceq \qty{f:X \to V}$ is a vector space with the operations defined pointwise, \ie, the sum of two functions $f$ and $g$ is defined as $(f+g)(x) \ceq f(x) + g(x)$ and the scalar multiplication is defined as $(rf)(x) \ceq rf(x)$.
  Observe that in the left of the symbol $\ceq$ we are defining an operation in $\mc{F}(X, V)$ while in the right we are using the given operations of $V$ as a vector space.
\end{exm}

\begin{defi}
  Given a vector space $V$ over $k$ and a non-empty subset $W \subset V$, we say that $W$ is a \emph{subspace of $V$} if $rx+sy \in W$ for every $x,y \in W$ and $r,s \in k$.
\end{defi}

\begin{exm} \label{l2}
  $L^2(\br^n) \ceq \qty{\psi:\br^n \to \bc : \int_{\br^n} \abs{\psi(\vb{r})}^2 \dd\vb{r} < \infty}$ is a subspace of $\mc{F}(\br^n, \bc)$.
\end{exm}

\begin{defi}
  Let $V$ and $W$ be two vector spaces over $k$.
  We say that a function $T:V \to W$ is a \emph{linear transformation}\index{linear transformation} if $T(rx+sy) = rT(x) + sT(y)$ for every $r, s \in k$ and $x,y \in V$.
  If $V = W$, then we also say that $T$ is an \emph{operator}\index{operator}.
\end{defi}

\begin{obs}
  Notice that to consider the vector $rx+sy \in V$ we are using the operations of $V$ while $rT(x) + sT(y) \in W$ uses the operations of $W$.
\end{obs}

\begin{defi}
  A linear transformation $T:V \to W$ is called an \emph{isomorphism}\index{isomorphism} if $T$ is bijective as a function and the inverse is also a linear transformation.
  If there exists an isomorphism between $V$ and $W$, we say that they are \emph{isomorphic}\index{space!isomorphic}.
\end{defi}

\begin{defi}
  Given a linear transformation $T:V \to W$, we define the \emph{kernel of $T$}\index{kernel} to be the set $\ker{T} \ceq \qty{x \in V : T(x) = 0}$ and the \emph{image of $T$}\index{image} to be the set $\im{T} \ceq \qty{y \in W : \text{$y = T(x)$ for some $x \in X$}}$.
\end{defi}

\begin{prop}
  Given a linear transformation $T:V \to W$, $\ker{T}$ is a subspace of $V$ and $\im{T}$ is a subspace of $W$.
\end{prop}

\begin{defi}
  Given an operator $T:V \to V$, an \emph{eigenvector of $T$}\index{eigenvector} is a non-zero vector $x \in V$ such that $Tx = \lambda{x}$ for some $\lambda \in k$.
  The number $\lambda$ is called an \emph{eigenvalue of $T$}\index{eigenvalue}.
\end{defi}

\begin{prop} \label{equivalence_subspace}
  Let (1) $L:U \to V$ be a linear transformation; (2) $W \ceq \ker{L}$; (3) $T:W \to W$ be an operator; (4) $\ol{T}:U \to U$ be an extension of $T$; and fix $b \in W$.
  Then the following is true: solving $Tx = b$ for $x \in W$ is equivalent to solving $(\ol{T} \oplus L)(x) \ceq (\ol{T}x, Lx) = (b, 0)$ for $x \in U$.
\end{prop}

\begin{proof}
  Assuming $Tx = b$, with $x \in W$, we have $x \in \ker{L}$ by definition of $W$ and that means $Lx = 0$.
  We also know that $\ol{T}$ is an extension of $T$, which means $Tx = \ol{T}x$.
  Consequently, if $Tx = b$, then $(\ol{T}x, Lx) = (b, 0)$, as desired.
  On the other hand, if we assume $(\ol{T} \oplus L)(x) \ceq (\ol{T}x, Lx) = (b, 0)$, then $Lx = 0$ and, therefore, $x \in \ker{L} = W$.
  Now, since $\ol{T}x = Tx$ in this case, we have the equivalence.
\end{proof}

\begin{defi}
  Given a set $X$, an \emph{equivalence relation on $X$}\index{equivalence relation} is a binary relation $\sim$ that satisfies the following axioms for every $x,y,z \in X$:
  \begin{enumerate}

  \item $x \sim x$.

  \item $x \sim y$ \tiff $y \sim x$.

  \item If $x \sim y$ and $y \sim z$, then $x \sim z$.

  \end{enumerate}
  These axioms are called \emph{reflexivity}, \emph{symmetry} and \emph{transitivity}, respectively.
\end{defi}

\begin{obs}
  To be pedantic, the reader can think of a binary relation as a subset $R$ of $X \x X$ such that the following axioms holds for every $x,y,z \in X$: $(x,x) \in R$; $(x,y) \in R$ \tiff $(y,x) \in R$; $(x,y), (y,z) \in R$ implies $(x,z) \in R$.
\end{obs}

\begin{defi}
  Given a set $X$, $x \in X$ and an equivalence relation $\sim$ on $X$, we define the \emph{equivalence class of $x$}\index{equivalence class} to be the set $[x] \ceq \qty{y \in X : x \sim y}$.
  Any element of $[x]$ is said to be a \emph{representative of the class $[x]$}\index{representative}.
\end{defi}

\begin{defi}
  Given a set $X$ and an equivalence relation $\sim$ on $X$, we define the \emph{quotient of $X$ by $\sim$} to be the set of equivalence classes of $X$, \ie, $X/\sim \ \ceq \qty{[x] : x \in X}$.
  Observe that the quotient comes with a natural function $\pi:X \to X/\sim$ defined by $\pi(x) = [x]$.
  We usually call this function the \emph{projection (of the quotient)}\index{projection} or \emph{canonical projection}\index{projection!canonical}.
\end{defi}

\begin{defi}
  Let $V$ be a vector space over $k$ and $W$ be a subspace of $V$.
  Define the following equivalence relation on $V$: given $x,y \in V$, we say that $x \sim y$ \tiff $x-y \in W$.
  The quotient of $V$ by $\sim$ is a vector space over $k$ called the \emph{quotient of $V$ by $W$}.
  We usually denote this quotient by $V/W$ instead of $V/\sim$ and the operations are defined as $[x] + [y] \ceq [x+y]$ and $r[x] \ceq [rx]$ for every $[x], [y] \in V/W$ and $r \in k$.
\end{defi}

\begin{prop}[Universal Property of the Quotient] \label{univ_quotient}
  Let $T:V \to W$ be a linear transformation and $U \subset V$ a subspace such that $U \subset \ker{T}$.
  Then there exists a unique linear transformation $\wtil{T}:V/U \to W$ such that $T = \wtil{T} \circ \pi$, being $\pi$ the canonical projection.
  Besides, the spaces $V/\ker{T}$ and $\im{T}$ are isomorphic.
\end{prop}

\begin{defi}
  Given a vector space $V$ and a finite subset $S = \qty{x_1, \ldots, x_d} \subset V$, we say a vector $x \in V$ is a \emph{linear combination of $S$}\index{linear!combination} if there exists a subset $\qty{r_1, \ldots, r_d} \subset k$ such that $x = r_1x_1 + \ldots + r_dx_d$.
  The elements $r_1, \ldots, r_d$ are called the \emph{coefficients of the linear combination}.
\end{defi}

\begin{defi}
  Given a vector space $V$ over $k$ and a subset $S \subset V$, we say that the vectors of $S$ are \emph{linearly independent} if for any finite number of elements of $S$, say $x_1, \ldots, x_d \in S$, the only way to write $0 \in V$ as a linear combination $r_1x_1 + \ldots + r_dx_d$ is when $r_1 = \ldots = r_d = 0 \in k$.
\end{defi}

\begin{defi}
  Given a vector space $V$ over $k$ and a subset $S \subset V$, the \emph{space spanned by $S$}\index{space!spanned} is defined by
  \begin{equation}
    \text{span} \, S \ceq \qty{r_1x_1 + \ldots + r_dx_d : r_i \in k, \ x_i \in S, \ n \in \bn}.
  \end{equation}
  If $\text{span} \, S = V$, we say that \emph{$S$ spans $V$}.
  Observe that $\text{span} \, S$ is a subspace of $V$.
\end{defi}

\begin{defi}
  Given a vector space $V$ and a subset $B \subset V$, we say that $B$ is \emph{basis (for $V$)}\index{basis} if $B$ spans $V$ and if its vectors are linearly independent.
\end{defi}

\begin{obs}
  The definition above is also known as \emph{Hamel basis} because there exists another definition of basis called \emph{Schauder basis} in which we allow infinite linear combinations.
  However, considering infinite linear combinations is more difficult because the vector space structure is not enough for letting us sum an infinite amount of vectors, which means we have to consider series instead of just linear combinations.
  Dealing with series, on the other hand, requires a notion of convergence and, to define convergence, we need a topology in the vector space, which gives rise to the concept of \emph{topological vector space}.
  We will not deal with this in here because in practice we only need finite-dimensional vector spaces, as one can see in Section \ref{qm_practice}.
  However, the vector space of Example \ref{l2}, which is, arguably, the most important of Quantum Mechanics, is an infinite-dimensional vector space and much of the mathematical difficulty of this field comes from dealing with the subtleties that shows up in the infinite-dimensional framework.
  The field that studies these type of vector spaces is called \emph{Functional Analysis} and a good reference for it is \cite{rudin1991}.
  Speaking of dimension, let us define it.
\end{obs}

\begin{defi}
  Given a vector space $V$ and a basis $B$ for $V$, the \emph{dimension of $V$}\index{dimension} is defined as the cardinality of the set $B$ and we will denote it by $\dim{V}$.
  So, if $B$ is a finite set, we say that $V$ is a \emph{finite-dimensional (vector) space}, otherwise we say it is an \emph{infinite-dimensional (vector) space}.
\end{defi}

\begin{thm}
  Every vector space has a basis and the cardinality is well-defined, \ie, if $B$ and $B'$ are two basis for $V$, then $|B| = |B'|$.
\end{thm}

\begin{defi} \label{matrix_repr}
  Given a linear transformation $T:V \to W$ between finite-dimensional spaces and basis $B_V = \qty{v_1, \ldots, v_n}$ and $B_W = \qty{w_1, \ldots, w_m}$ for $V$ and $W$, respectively, the \emph{matrix representation of $T$ with respect to the basis $B_V$ and $B_W$}\index{matrix representation} is given by
  \begin{equation}
    [T]_B =
    \begin{bmatrix}
      r_{11} & \ldots & r_{1n} \\
      \vdots & \ddots & \vdots \\
      r_{m1} & \ldots & r_{mn} \\
    \end{bmatrix},
  \end{equation}
  being $Tv_i = r_{1i}w_1 + \ldots + r_{mi}w_m$.
\end{defi}

\begin{obs}
  The matrix representation of a linear transformation depends on the choice of basis, \ie, if $\wtil{B}_V$ and $\wtil{B}_W$ are two different basis for $V$ and $W$, then $[T]_B \neq [T]_{\wtil{B}}$.
\end{obs}

\begin{obs} \label{transform_span}
  A very common procedure used in Linear Algebra is to define a linear transformation by telling how it acts in a set that spans the vector space.
  Example: let $T:V \to W$ be a linear transformation and $S_V$ be a set that spans $V$.
  Then, if we define or we know that $T(x_i) = y_i$ for $x_i \in S_V$, the linear transformation is automatically given by
  \begin{equation}
    T(r_1x_1 + \ldots r_dx_d) = r_1T(x_1) + \ldots + r_dT(x_d).
  \end{equation}
\end{obs}

\section{Hilbert spaces}

From now on we will assume that $k \in \qty{\br, \bc}$.

\begin{defi}
  Let $V$ be a vector space over $k$.
  A function $\braket{\cdot}:V \x V \to k$ is called an \emph{inner product}\index{inner product} if it satisfies the following axioms for every $x, y, z \in V$ and $r, s \in k$:
  \begin{enumerate}

  \item $\braket{x}{ry + sz} = r\braket{x}{y} + s\braket{x}{z}$.

  \item $\braket{x}{y} = \ol{\braket{y}{x}}$, if $k = \bc$, and $\braket{x}{y} = \braket{y}{x}$, if $k = \br$.

  \item $\braket{x} > 0$, if $x \neq 0$.

  \end{enumerate}
\end{defi}

\begin{exm}
  The function $\braket{(x_1,\ldots,x_d)}{(y_1,\ldots,y_d)} \ceq x_1y_1 + \ldots + x_dy_d$ is an inner product on $\br^d$ and the function $\braket{(x_1,\ldots,x_d)}{(y_1,\ldots,y_d)} \ceq \ol{x_1}y_1 + \ldots + \ol{x_d}y_d$ is an inner product on $\bc^d$.
\end{exm}

\begin{exm}
  The function $\braket{\psi}{\phi} \ceq \int_{\br^n} \ol{\psi(\vb{r})}\phi(\vb{r})\dd{\vb{r}}$ is an inner product on $L^2(\br^n)$.
  Observe that $\braket{\psi} \ceq \int_{\br^n} \ol{\psi(\vb{r})}\psi(\vb{r})\dd{\vb{r}} = \int_{\br^n} \abs{\psi(\vb{r})}^2\dd{\vb{r}}$.
\end{exm}

\begin{defi}
  Given a vector space $V$ and a basis $B$, we say that $B$ is an \emph{orthonormal basis}\index{basis!orthonormal} if $\braket{b_i}{b_j} = \delta_{ij}$, being $\delta_{ij}$ the Kronecker delta, \ie, $\delta_{ij} = 1$ if $i = j$ and $\delta_{ij} = 0$ otherwise.
\end{defi}

\begin{defi}
  Given a vector space with inner product $(V, \braket{\cdot}_V)$ and a subspace $W$, we define the \emph{orthogonal complement of $W$}\index{orthogonal!complement} to be the following subspace of $V$:
  \begin{equation}
    W^{\perp} \ceq \qty{v \in V : \text{$\braket{v}{w}_V = 0$ for every $w \in W$}}.
  \end{equation}
\end{defi}

\begin{defi}
  Given a vector space with inner product $(V, \braket{\cdot}_V)$ and a subspace $W$, we define the \emph{orthogonal projection onto $W$}\index{orthogonal!projection} to be the unique linear transformation $\proj_W:V \to V$ that satisfies the following axioms:
  \begin{enumerate}

  \item $\im(\proj_W) = W$.

  \item $\proj_W \circ \proj_W = \proj_W$.

  \item $\braket{v - \proj_W(v)}{w}_V = 0$ for every $v \in V$ and $w \in W$.

  \end{enumerate}
\end{defi}

\begin{defi}
  Given a vector space with inner product $(V, \braket{\cdot}_V)$ and a basis $B = \qty{v_1, \ldots, v_d}$ for $V$, we define the \emph{Gram matrix of $B$}\index{Gram matrix} to be the following matrix: $G_B \ceq \qty(\braket{v_i}{v_j}_V)_{i,j}$.
\end{defi}

\begin{defi}
  Given a vector space $V$ over $k$, a \emph{norm on $V$}\index{norm} is a function $\norm{\cdot}:V \to \br$ that satisfies the following axioms for every $x, y \in V$ and $r \in k$:
  \begin{enumerate}
    
  \item $\norm{x + y} \leq \norm{x} + \norm{y}$.

  \item $\norm{rx} = \abs{r}\norm{x}$.

  \item If $\norm{x} = 0$, then $x = 0$.

  \end{enumerate}
  The first axiom is called \emph{triangle inequality} and vector spaces with norm are called \emph{normed vector space}.
\end{defi}

\begin{exm}
  $\br^d$ and $\bc^d$ are normed spaces and the most common norms used are:
  \begin{enumerate}
    
  \item $\norm{(x_1, \ldots, x_d)}_{1} \ceq \abs{x_1} + \ldots + \abs{x_d}$.

  \item $\norm{(x_1, \ldots, x_d)}_{2} \ceq \sqrt{x_1^2 + \ldots + x_d^2}$.

  \item $\norm{(x_1, \ldots, x_d)}_{p} \ceq \qty(\abs{x_1}^p + \ldots + \abs{x_d}^p)^{\frac{1}{p}}$, $p \in \br$, $p \geq 1$.

  \item $\norm{(x_1, \ldots, x_d)}_{\infty} \ceq \max\qty{\abs{x_i} : i = 1, \ldots, d}$.

  \end{enumerate}
\end{exm}

\begin{exm} \label{norms_equivalent}
  Consider the following definition: two norms $\norm{\cdot}_1$ and $\norm{\cdot}_2$ are equivalent if there exists constants $A, B > 0$ such that $A\norm{x}_1 \leq \norm{x}_2 \leq B\norm{x}_1$ for every $x \in V$.
  Generalizing the previous example, there is a very strong result that says that every two norms in a finite-dimensional vector space are equivalent.
  A proof of this result can be found in \cite{kconradnorm}.
\end{exm}

\begin{prop}
  Given a vector space $V$ with inner product $\braket{\cdot}$, the function $\norm{x} \ceq \sqrt{\braket{x}}$ is a norm on $V$.
  This norm is called \emph{norm induced by the inner product}.
\end{prop}

The reader can see the definition and some basic properties of metric spaces in Section \ref{metric_spaces}.

\begin{prop}
  Given a normed vector space $(V, \norm{\cdot})$, the function $d:V \x V \to \br$ defined by $d(x, y) \ceq \norm{x-y}$ is a metric on $V$.
  This metric is called \emph{metric induced by the norm}.
\end{prop}

\begin{cor}
  Every vector space with inner product is a metric space.
\end{cor}

\begin{proof}
  Just consider the norm induced by the inner product and use the previous proposition.
\end{proof}

The definition of a complete metric space is at \ref{complete_space}.

\begin{defi} \label{hilbert_space}
  A \emph{Hilbert space}\index{space!Hilbert}\index{Hilbert space} is a vector space equipped with an inner product that is also a complete metric space with the metric induced by the inner product.
\end{defi}

\begin{defi}
  Given a linear transformation $T:(V, \braket{\cdot}_V) \to (W, \braket{\cdot}_W)$, we say that a function $T^*:W \to V$ is an \emph{adjoint of $T$}\index{adjoint} if $\braket{Tx}{y}_W = \braket{x}{T^*y}_V$ for every $x \in V$ and $y \in W$.
\end{defi}

\begin{prop}
  The adjoint of $T$ is actually a linear transformation.
\end{prop}

\begin{defi}
  We say that an operator $T:V \to V$ is \emph{self-adjoint}\index{self-adjoint} if $T = T^*$.
\end{defi}

\begin{prop}
  Every eigenvalue of a self-adjoint operator is real.
\end{prop}

\begin{proof}
  Let $T:V \to V$ be a self-adjoint operator and $x$ be an eigenvector of $T$ with eigenvalue $\lambda$.
  Then notice that
  \begin{equation}
    \lambda\norm{x}^2
    = \braket{x}{\lambda x}
    = \braket{x}{Tx}
    = \braket{Tx}{x}
    = \braket{\lambda x}{x}
    = \ol{\lambda}\norm{x}^2
  \end{equation}
  and, consequently, $\lambda = \ol{\lambda}$, as desired.
\end{proof}

Now let us explain the terminology and notation usually used in Quantum Mechanics (QM).
First of all, we are adopting the convention that inner products are linear in the second argument because then we can use the notation $\bra{x}$ for the linear functional that sends $y \in V$ to $\braket{x}{y} \in k$.
This is quite useful and in the QM literature people usually call $\bra{x}$ as \emph{bra}\index{bra}.
Now, there is a subtle misconception with respect to the $\ket{x}$ notation.
This is called \emph{ket}\index{ket} in the literature and it is meant to represent a \emph{quantum state} (see beginning of Chapter \ref{qm} for the definition of state).
States are elements of $\bp{V}$, \ie, the projective space of $V$ (Example \ref{projective}), but it is very common to find people referring to $\ket{x}$ as an element of $V$ itself in the literature.
We are differentiating these two notations ($x$ and $\ket{x}$) in Chapter \ref{qm} and kets to us are elements of $\bp{V}$ instead of $V$.
Another notation we are adopting from QM is $\mel{x}{T}{y}$ when $T$ is a self-adjoint operator on $V$.
The definition is $\mel{x}{T}{y} \ceq \braket{Tx}{y} = \braket{x}{Ty}$ and this is usually called \emph{matrix element}\index{matrix element} when $x \neq y$ and \emph{expected value (of $T$ with respect to $x$)} when $x = y$.

\section{Multilinear Algebra} \label{multilinear_algebra}

We will keep assuming that $k \in \qty{\br, \bc}$.

\begin{defi}
  Given vector spaces $V_1, \ldots, V_n, V$ over the same field $k$, a function $M:V_1 \x \ldots \x V_n \to V$ is called \emph{$n$-multilinear} or just \emph{multilinear} if $M$ is linear in each entry, \ie,
  \begin{equation}
    M(x_1, \ldots, rx_i + sy_i, \ldots, x_n)
    = rM(x_1, \ldots, x_i, \ldots, x_n) + sM(x_1, \ldots, y_i, \ldots, x_n)
  \end{equation}
  for every $r, s \in k$, $x_j \in V_j$, $y_i \in V_i$ and $i,j = 1, \ldots, n$.
  When $n = 1$ we usually say the function is \emph{linear}, as we already defined, when $n = 2$ we say the function is \emph{bilinear}\index{bilinear}, and when $n = 3$ we say the function is \emph{trilinear}\index{trilinear}.
\end{defi}

\begin{defi}
  Given vector spaces $V_1, \ldots, V_n$ over the same field $k$, the \emph{tensor product (of $V_1, \ldots, V_n$)}\index{tensor!product} is a pair $(\tenw{V_1}{V_n}, \ T:V_1 \x \ldots \x V_n \to \tenw{V_1}{V_n})$, being $\tenw{V_1}{V_n}$ a vector space and $T$ a multilinear function, such that the following (universal) property holds: given any vector space $V$ and any multilinear function $M:V_1 \x \ldots \x V_n \to V$, there exists a unique \emph{linear} transformation $L:\tenw{V_1}{V_n} \to V$ such that $M = L \circ T$.
  This can be represented diagrammatically as
  \begin{center}
    \begin{tikzcd}
      & \tenw{V_1}{V_n} \arrow[dashed]{dr}[above]{L} & \\
      V_1 \x \ldots \x V_n \arrow{ur}[above]{T} \arrow{rr}[below]{M} & & V
    \end{tikzcd}
  \end{center}
  When $V_1 = \ldots = V_n$, we say that $\tenw{V_1}{V_1}$ is the \emph{$n$-th tensor power of $V_1$}\index{tensor!power}.
\end{defi}

\begin{prop}
  The tensor product exists.
\end{prop}

\begin{proof}
  Let $F(V_1 \x \ldots \x V_n)$ be the free vector space on $V_1 \x \ldots \x V_n$: the elements of this space are spanned by the formal symbols $\delta_{(x_1, \ldots, x_n)}$.
  The sum and the scalar product are also defined formally, \ie, the sum of $\delta_{(x_1, \ldots, x_n)}$ with $\delta_{(y_1, \ldots, y_n)}$ is another formal symbol denoted by $\delta_{(x_1, \ldots, x_n)} + \delta_{(y_1, \ldots, y_n)}$ and the product of $\delta_{(x_1, \ldots, x_n)}$ by $r \in k$ is the formal symbol $r\delta_{(x_1, \ldots, x_n)}$.
  Now, let $D$ be the subspace spanned by the vectors
  \begin{equation}
    \begin{matrix}
      \delta_{(x_1, \ldots, rx_i + sy_i, \ldots, x_n)} - r\delta_{(x_1, \ldots, x_i, \ldots, x_n)} - s\delta_{(x_1, \ldots, y_i, \ldots, x_n)}
    \end{matrix}
  \end{equation}
  for every $r, s \in k$, $x_j \in V_j$, $y_i \in V_i$, $i, j = 1, \ldots, n$.
  Then, we define the tensor product to be the quotient $\tenw{V_1}{V_n} \ceq F(V_1 \x \ldots \x V_n) / D$ endowed with the multilinear function $T:V_1 \x \ldots \x V_n \to F(V_1 \x \ldots \x V_n) / D$ given by $T(x_1, \ldots, x_n) = [\delta_{(x_1, \ldots, x_n)}]$.

  Let us prove that this space satisfies the universal property.
  Given a vector space $V$ and a multilinear function $M:V_1 \x \ldots \x V_n \to V$, we need to prove that there exists a unique linear function $L:\tenw{V_1}{V_n} \to V$ such that $M = L \circ T$.
  First notice that we can extend $M$ to a linear function $F(M):F(V_1 \x \ldots \x V_n) \to V$ defined in the generators as $F(M)(\delta_{(x_1, \ldots, x_n)}) = M(x_1, \ldots, x_n)$.
  It is also clear that, by the definition of $F(M)$, we have $D \subset \ker{F(M)}$.
  Consequently, by the Universal Property of the Quotient (\ref{univ_quotient}), there exists a unique linear map $L:\tenw{V_1}{V_n} \to V$ such that $L(T(\delta_{(x_1, \ldots, x_n)})) = L([\delta_{(x_1, \ldots, x_n)}]) = F(M)(\delta_{(x_1, \ldots, x_n)}) = M(x_1, \ldots, x_n)$.
  That is it, the result follows.
\end{proof}

\begin{prop}
  The tensor product is unique up to isomorphism.
  In other words, if $(C_1, T_1)$ and $(C_2, T_2)$ satisfies the universal property of being the tensor product of $V_1, \ldots, V_n$, then $C_1$ and $C_2$ are isomorphic.
\end{prop}

\begin{obs}
  We usually say that $\tenw{V_1}{V_n}$ is \emph{the} tensor product of $V_1, \ldots, V_n$, but any space isomorphic to the quotient we constructed above can represent the tensor product.
  A good practice is to not use any specific construction of the tensor product (there are many) and consider only the universal property instead.
  However, this universal property does not hold topologically when the space is infinite-dimensional, only algebraically \cite{garrett2010}.
\end{obs}

\begin{obs}
  We will also denote $\tenw{V_1}{V_n}$ by $\bigotimes_{i=1}^n V_i$.
\end{obs}

\begin{defi}
  The vectors $T(x_1, \ldots, x_n) = [\delta_{(x_1, \ldots, x_n)}]$ are called \emph{elementary tensors}\index{tensor!elementary} and we will denote them by $\tenw{x_1}{x_n}$.
\end{defi}

\begin{obs}
  There are more vectors in the tensor product than elementary tensors, but the elementary tensors span the tensor product.
  If we consider $\br^2 \otimes \br^2$, for example, then $e_1 \otimes e_1 + e_2 \otimes e_2$ cannot be written as an elementary tensor.
  This gives rise to the concept of \emph{entanglement} in Quantum Mechanics.
\end{obs}

\begin{prop}
  If $\qty{v_{1,j}, \ldots, v_{d_j,j}}$ is a basis for $V_j$, then
  \begin{equation}
    \qty{\tenw{v_{i_1,1}}{v_{i_n,n}} : 1 \leq i_1 \leq d_1, \ldots, 1 \leq i_n \leq d_n}
  \end{equation}
  is a basis for $\tenw{V_1}{V_n}$.
  Consequently, $\dim{\tenw{V_1}{V_n}} = \dim{V_1} \cdot \ldots \cdot \dim{V_n}$.
\end{prop}

\begin{defi} \label{inner_product_tensor}
  Let $(V_1, \braket{\cdot}_1), \ldots, (V_n, \braket{\cdot}_n)$ be vector spaces endowed with inner products.
  Then the inner product we will consider in the tensor product $\tenw{V_1}{V_n}$ is given by
  \begin{equation}
    \braket{\tenw{x_1}{x_n}}{\tenw{y_1}{y_n}}
    \ceq \braket{x_1}{y_1}_1 \cdot \ldots \cdot \braket{x_n}{y_n}_n.
  \end{equation}
\end{defi}

\begin{obs}
  If $V_1, \ldots, V_n$ are finite-dimensional Hilbert spaces, then $\tenw{V_1}{V_n}$ also is.
  However, when the spaces are infinite-dimensional, we should consider the completion (see Definition \ref{completion}) of the tensor product because, although $\tenw{V_1}{V_n}$ is a metric space, it may not be complete.
  We will denote this completion by $\ten{V_1}{V_n}$.
\end{obs}

\begin{defi}
  Given linear transformations $T_i:V_i \to W_i$, we define the \emph{tensor product of linear transformations} as follows:
  \begin{align}
    \begin{split}
      \tenw{T_1}{T_n}:\tenw{V_1}{V_n} & \to \tenw{W_1}{W_n} \\
      \tenw{v_1}{v_n} & \mapsto \tenw{T_1(v_1)}{T_n(v_n)}.
    \end{split}
  \end{align}
\end{defi}

\begin{defi}
  A multilinear function $M:V \x \ldots \x V \to W$ is called \emph{symmetric}\index{symmetric} if
  \begin{equation}
    M(x_{\sigma(1)}, \ldots, x_{\sigma(n)}) = M(x_1, \ldots, x_n)
  \end{equation}
  and \emph{antisymmetric}\index{antisymmetric} if
  \begin{equation}
    M(x_{\sigma(1)}, \ldots, x_{\sigma(n)}) = \sign(\sigma)M(x_1, \ldots, x_n)
  \end{equation}
  for any $\sigma \in S_n$.
\end{defi}

\begin{defi}
  Given a vector space $V$, its \emph{$n$-th exterior power}\index{exterior power} is a pair $(\bigwedge^n V, \ A:V \x \ldots \x V \to \bigwedge^n V)$, being $\bigwedge^n V$ a vector space and $A$ an antisymmetric multilinear function, such that the following (universal) property holds: given any vector space $W$ and any antisymmetric multilinear function $M:V \x \ldots \x V \to W$, there exists a unique \emph{linear} function $L:\bigwedge^n V \to W$ such that $M = L \circ A$.
  This can be represented diagrammatically as
  \begin{center}
    \begin{tikzcd}
      & \bigwedge^n V \arrow[dashed]{dr}[above]{L} & \\
      V \x \ldots \x V \arrow{ur}[above]{A} \arrow{rr}[below]{M} & & W
    \end{tikzcd}
  \end{center}  
\end{defi}

\begin{prop}
  The exterior power exists.
\end{prop}

\begin{proof}
  Consider $\bigwedge^n V$ to be the quotient $\qty(\bigotimes^n V)/J_n$, being $J_n$ the subspace spanned by all $\tenw{x_1}{x_n} \in \bigotimes^n V$ such that $x_i = x_j$ for some $i \neq j$.
  Observe that, if $M:V \x \ldots \x V \to W$ is antisymmetric and $(x_1, \ldots, x_n) \in V \x \ldots \x V$ has two equal entries, say $x_i = x_j$ for $i \neq j$, then $M(x_1, \ldots, x_n) = 0$ because
  \begin{equation}
    M(x_1, \ldots, x_i, \ldots, x_j, \ldots, x_n)
    = -M(x_1, \ldots, x_j, \ldots, x_i, \ldots, x_n)
  \end{equation}
  and this implies $2M(x_1, \ldots, x_n) = 0$ because $x_i = x_j$.
  Consequently, $M(x_1, \ldots, x_n) = 0$.
  Now, using the Universal Property of the Tensor Product, we know that there exists a linear transformation $\ol{L}:\bigotimes^n V \to W$ such that $M = \ol{L} \circ T$.
  But when $M$ is antisymmetric we proved that $J_n \subset \ker\ol{L}$ because $\tenw{x_1}{x_n} = T(x_1, \ldots, x_n)$.
  To conclude, we use the Universal Property of the Quotient to obtain a unique linear transformation $L:\qty(\bigotimes^n V)/J_n \to W$ such that $L([\tenw{x_1}{x_n}]) = \ol{L}(\tenw{x_1}{x_n})$.
  The result follows, \ie, $M = L \circ A$, being $A = \pi \circ T$.
\end{proof}

\begin{defi} \label{wedge_product}
  Considering $\pi:\bigotimes^n V \to \qty(\bigotimes^n V)/J_n$ to be the canonical projection, we will call the vectors $\pi(\tenw{x_1}{x_n})$ \emph{elementary wedge products}\index{wedge!elementary} and we will denote them by $\wed{x_1}{x_n}$.
\end{defi}

\begin{obs}
  It is possible to see the exterior power as a subspace instead of a quotient of the tensor power by considering the map
  \begin{align}
    \begin{split}
      \alt:\qty(\bigotimes^n V)/J_n & \to \bigotimes^n V \\
      [\tenw{x_1}{x_n}]       & \mapsto \frac{1}{\sqrt{n!}} \sum_{\sigma \in S_n} \sign(\sigma)\tenw{x_{\sigma(1)}}{x_{\sigma(n)}}
    \end{split}
  \end{align}
  and defining $\bigwedge^n V \ceq \im(\alt)$.
  In this viewpoint
  \begin{equation}
    \wed{x_1}{x_n} \ceq \frac{1}{\sqrt{n!}} \sum_{\sigma \in S_n}
    \sign(\sigma)\tenw{x_{\sigma(1)}}{x_{\sigma(n)}}.
  \end{equation}
\end{obs}

\begin{obs}
  Using the above identification we can see that the inner product in the exterior power is
  \begin{align}
    \begin{split}
      \braket{\wed{x_1}{x_n}}{\wed{y_1}{y_n}}
      & = \frac{1}{n!}\sum_{\sigma, \tau \in S_n} \sign(\sigma \circ \tau)
        \braket{\tenw{x_{\sigma(1)}}{x_{\sigma(n)}}}
        {\tenw{y_{\tau(1)}}{y_{\tau(n)}}} \\
      & = \frac{1}{n!}\sum_{\sigma, \tau \in S_n} \sign(\sigma \circ \tau)
        \braket{x_{\sigma(1)}}{y_{\tau(1)}} \ldots
        \braket{x_{\sigma(n)}}{y_{\tau(n)}} \\
      & = \frac{1}{n!}\sum_{\sigma, \tau \in S_n} \sign(\sigma^{-1} \circ \tau)
        \braket{x_1}{y_{\sigma^{-1}(\tau(1))}} \ldots
        \braket{x_n}{y_{\sigma^{-1}(\tau(n))}} \\
      & = \frac{1}{n!}\sum_{\sigma, \pi \in S_n} \sign(\pi)
        \braket{x_1}{y_{\pi(1)}} \ldots \braket{x_n}{y_{\pi(n)}} \\
      & = \sum_{\pi \in S_n} \sign(\pi)
        \braket{x_1}{y_{\pi(1)}} \ldots \braket{x_n}{y_{\pi(n)}} \\
      & = \det(\qty(\braket{x_i}{y_j})_{i,j}).
    \end{split}
  \end{align}
\end{obs}

\begin{prop}
  If $\qty{v_1, \ldots, v_d}$ is a basis for $V$, then
  \begin{equation}
    \qty{\wed{v_{i_1}}{v_{i_n}} : 1 \leq i_1 < \ldots < i_n \leq d}
  \end{equation}
  is a basis for $\bigwedge^n(V)$.
  Consequently, $\dim{\bigwedge^n(V)} = \binom{d}{n}$.
\end{prop}

\begin{defi}
  Given a vector space $V$, its \emph{$n$-th symmetric power}\index{symmetric power} is a pair $(\sym^n(V), \ S:V \x \ldots \x V \to \sym^n(V))$, being $\sym^n(V)$ a vector space and $S$ a symmetric multilinear function, such that the following (universal) property holds: given any vector space $W$ and any symmetric multilinear function $M:V \x \ldots \x V \to W$, there exists a unique \emph{linear} function $L:\sym^n(V) \to W$ such that $M = L \circ S$.
  This can be represented diagrammatically as
  \begin{center}
    \begin{tikzcd}
      & \sym^n(V) \arrow[dashed]{dr}[above]{L} & \\
      V \x \ldots \x V \arrow{ur}[above]{S} \arrow{rr}[below]{M} & & W
    \end{tikzcd}
  \end{center}  
\end{defi}

\begin{prop}
  The symmetric power exists.
\end{prop}

\begin{proof}
  Consider $\sym^n(V)$ to be the quotient $\qty(\bigotimes^n V)/J_n$, being $J_n$ the subspace spanned by all $\tenw{x_1}{x_n} - \tenw{x_{\sigma(1)}}{x_{\sigma(n)}} \in \bigotimes^n V$, $\sigma \in S_n$.
  Now, according to the Universal Property of the Tensor Product, we know that there exists a linear transformation $\ol{L}:\bigotimes^n V \to W$ such that $M = \ol{L} \circ T$.
  But when $M$ is symmetric we have $J_n \subset \ker\ol{L}$ because
  \begin{align}
    \begin{split}
      0 & = M(x_1, \ldots, x_n) - M(x_{\sigma(1)}, \ldots, x_{\sigma(n)}) \\
        & = \ol{L}(T(x_1, \ldots, x_n)) - \ol{L}(T(x_{\sigma(1)}, \ldots, x_{\sigma(n)})) \\
        & = \ol{L}(\tenw{x_1}{x_n}) - \ol{L}(\tenw{x_{\sigma(1)}}{x_{\sigma(n)}}) \\
        & = \ol{L}(\tenw{x_1}{x_n} - \tenw{x_{\sigma(1)}}{x_{\sigma(n)}}).
    \end{split}
  \end{align}
  To conclude, we use the Universal Property of the Quotient to obtain a unique linear transformation $L:\qty(\bigotimes^n V)/J_n \to W$ such that $L([\tenw{x_1}{x_n}]) = \ol{L}(\tenw{x_1}{x_n})$.
  The result follows, \ie, $M = L \circ S$, being $S = \pi \circ T$.
\end{proof}

\begin{obs}
  It is also possible to see the symmetric power as a subspace of the tensor power by considering the map
  \begin{align}
    \begin{split}
      \sym:\qty(\bigotimes^n V)/J_n & \to \bigotimes^n V \\
      [\tenw{x_1}{x_n}]             & \mapsto \frac{1}{\sqrt{n!}} \sum_{\sigma \in S_n} \tenw{x_{\sigma(1)}}{x_{\sigma(n)}}
    \end{split}
  \end{align}
  and defining $\sym^n(V) \ceq \im(\sym)$.
\end{obs}

\begin{obs}
  Using the above identification we can see that the inner product in the symmetric power is given by the permanent of the matrix $\qty(\braket{x_i}{y_j})$.
\end{obs}

\begin{prop}
  If $\qty{v_1, \ldots, v_d}$ is a basis for $V$, then
  \begin{equation}
    \qty{\tenw{v_{i_1}}{v_{i_n}} : 1 \leq i_1 \leq \ldots \leq i_n \leq d}
  \end{equation}
  is a basis for $\sym^n(V)$.
  Consequently, $\dim{\sym^n(V)} = \binom{d+n-1}{n}$.
\end{prop}


\chapter{Lagrange Multipliers} \label{lagrange_multipliers}

\epigraph{The methods that I explain require neither geometrical, nor mechanical, constructions or reasoning, but only algebraic operations in accordance with regular and uniform procedure.}{Joseph Louis Lagrange}

The goal of this chapter is to implement the Newton--Raphson Method using Lagrange multipliers because it allows us to compare how the Riemannian and the Euclidean version of the same method perform in the same problem (Hartree--Fock).
The main reference used for Euclidean constrained optimization and Lagrange multipliers was \cite{nocedal2006}.
It is also worth mentioning that we will use the same notation and objects defined in Chapter \ref{optimization}.
So, maybe the reader should check Section \ref{algorithms} before moving forward.

One of the most used approaches to encode the geometry of a constrained optimization problem is the technique of Lagrange multipliers.
This technique is, essentially, the inverse operation of what we did in the beginning of Chapter \ref{optimization} to move from Equation \ref{min_problem} to \ref{riemannian_problem}, \ie, starting with a function $f:X \to \br$, the goal is to consider the equivalent problem of optimizing an extension of $f$ given by $\ol{f}:\br^d \to \br$ subject to the constraint $c:\br^d \to \br^k$ as long as $X = c^{-1}(0)$.\footnote{We will work with equality constraint, \ie, we want the solution of the optimization problem to satisfy $c(x) = 0$.}
It is usually not obvious how one can move from the manifold $X$ to the constraint described by $c$ and vice versa, but the example we will consider may shed some light on how to proceed for examples other than Hartree--Fock.
The theorem that allows us to implement algorithm(s) to minimize $\ol{f}:\br^d \to \br$ subject to $c:\br^d \to \br^k$ is the following:
\begin{thm}[Lagrange Multiplier Theorem]
  Let $\ol{f}:\br^d \to \br$ and $c:\br^d \to \br^k$ be as above and let $x^* \in \br^d$ be an optimal solution to the problem of minimizing $\ol{f}$ subject to $c$.
  Assume that the Jacobian $Jc(x^*)$ has full rank.
  Then there exists a unique $\eps^* \in \br^k$ such that $J\ol{f}(x^*) = (\eps^*)^{\top}Jc(x^*)$.
\end{thm}
A proof of this theorem can be found in \cite[Chapter 12]{nocedal2006}, but what this result asserts is that, if we define what is known as \emph{Lagrangian function}\index{Lagrangian}
\begin{equation} \label{lagrangian}
  \begin{aligned}
    \mc{L}:\br^d \x \br^k & \to \br \\
    (x, \eps) & \mapsto \ol{f}(x) - \eps^{\top}c(x),
  \end{aligned}
\end{equation}
then the set of stationary points of $\ol{f}$ subject to $c$ is a subset of the critical points of $\mc{L}$ (it is well known in the literature \cite{kalman2009}, though, that the critical points of $\mc{L}$ should actually be \emph{saddle points}).
So, that means we can switch from the original constrained optimization problem to the search for saddle points of $\mc{L}$.
Before we move forward, it should be warned that the saddle points of $\mc{L}$ are not necessarily stationary points of $\ol{f}$ subject to $c$, the theorem only work in the other direction.
Notice that we do not incur into this problem when using Riemannian Optimization, but the advantage of this approach is that we can now use regular Euclidean methods to find saddle points of the unconstrained function $\mc{L}$ and there is a vast literature on how to do this.
To be fair with the purpose of this work, we will implement only the Newton--Raphson Method because the other methods we implemented are minimization methods and they cannot be used to find saddle points (it is possible to adapt the Lagrangian in order to use minimization methods, but we did not pursue this because the adaptation makes the comparison with the Riemannian algorithms less intuitive).

Before we move to the details of the implementation, an important observation is that we will actually encode the Stiefel manifold in the Lagrangian instead of the Grassmannian because it is not clear to us how to model the Grassmannian\footnote{There are other parametrizations that would allow us to consider the Grassmannian directly, see the discussion in Chapter \ref{conclusion}.} or a quotient manifold in general using Lagrangians.
Despite this exchange of manifolds, the method still works because in practice we always work with representatives of an equivalence class and, therefore, at the end it does not matter which representative we obtain as long as it is in the Stiefel manifold.
This may raise the question of why using quotient manifolds at all and, to put it shortly, we do not lose much using the Stiefel manifold instead of the Grassmannian in the Lagrange multipliers approach because we will work solely with the Euclidean space.
So, that means we can ignore the Riemannian structure of the Stiefel used to compute the geodesic, gradient and tangent space, but in the Riemannian algorithms this is not the case, all of these objects are quite different in the Stiefel and in the Grassmannian and we exploit this fact in the Riemannian algorithms.
With that said, let us jump to the implementation.

First of all, the constraint we have to impose to be in the product of Stiefel manifolds is given by the function
\begin{equation}
  \begin{split}
    c:\mat{d_1}{N_1} \x \mat{d_2}{N_2} & \to \mat{N_1}{N_1} \x \mat{N_2}{N_2}  \\
    (C_1, C_2) & \mapsto (C_1^{\top}S_1C_1 - \Id_{N_1}, \
                 C_2^{\top}S_2C_2 - \Id_{N_2}).
  \end{split}
\end{equation}
Let us also define the function
\begin{align} \label{restriction}
  \begin{split}    
    c_i:\mat{d_i}{N_i} & \to \mat{N_i}{N_i}  \\
    C_i & \mapsto C_i^{\top}S_iC_i - \Id_{N_i}
  \end{split}
\end{align}
because it will be useful in a moment.
Now, there are two options to encode this in the Lagrangian: vectorizing the above expression or using the trace of a matrix.
Let us describe the two options because both are useful.
Starting with an arbitrary cost function $f:\gr{N_1}{d_1} \x \gr{N_2}{d_2} \to \br$, in the first case the Lagrangian is
\begin{align}
  \begin{split}
    \mc{L}: & \ \mat{d_1}{N_1} \x \mat{d_2}{N_2} \x \br^{N_1 \cdot N_1} \x \br^{N_2 \cdot N_2} \to \br \\
            & \ (C_1, C_2, \eps_1, \eps_2)
              \mapsto \oll{f}(C_1, C_2)
              - \eps_1^{\top}\vec(c_1(C_1))
              - \eps_2^{\top}\vec(c_2(C_2)),
  \end{split}
\end{align}
being vec the columnwise vectorization of a matrix (see Definition \ref{vectorization}), and in the second case we use the identity $\tr(M^{\top}N) = \vec(M)^{\top}\vec(N)$ to obtain the following Lagrangian:
\begin{equation} \label{lagrangian_stiefel}
  \begin{split}
    \mc{L}: & \ \mat{d_1}{N_1} \x \mat{d_2}{N_2} \x \mat{N_1}{N_1} \x \mat{N_2}{N_2} \to \br \\
            & \ (C_1, C_2, \eps_1, \eps_2)
              \mapsto \oll{f}(C_1, C_2)
              - \tr(\eps_1^{\top}c_1(C_1))
              - \tr(\eps_2^{\top}c_2(C_2)).
  \end{split}
\end{equation}
These Lagrangians are equivalent and the choice of using one over the other comes down to the storage of $\eps$ as a matrix or as a vector, which is left to the reader to decide which is best to their purpose.
We will use the second representation because in our case (Hartree--Fock) the matrices $\eps_i$ have a physical interpretation of being the orbitals' energy and because it is also used by other methods of Quantum Chemistry such as the Self-Consistent Field \cite[Chapter 3]{szabo1996}.
Now let us recall the Newton--Raphson Method and then implement it for the Lagrangian defined above.
\begin{algo}[Riemannian Newton--Raphson Method]
  Given a Riemannian manifold $X$, a function $f:X \to \br$ and a starting point $x_0 \in X$, do the following:
  \begin{enumerate}
    
  \item Solve $\hess{f}(x_k)(v_k) = -\grad{f}(x_k)$ for $v_k \in T_{x_k}X$.

  \item Update the point according to $x_{k+1} \ceq \exp_{x_k}(v_k)$.

  \end{enumerate}
\end{algo}
Observe that in our case $X = \mat{d_1}{N_1} \x \mat{d_2}{N_2} \x \mat{N_1}{N_1} \x \mat{N_2}{N_2}$, $f = \mc{L}$ and now the Hessian and the gradient are just the partial derivatives of $\mc{L}$ with respect to $C_i$ and $\eps_i$.
It is also worth recalling (Example \ref{exp_euclidean} and Proposition \ref{exp_product}) that
\begin{equation} \label{euclidean_exp}
  \exp_{(C_1, C_2, \eps_1, \eps_2)}(\eta_1, \eta_2, \mu_1, \mu_2)
  = \qty(C_1 + \eta_1, \ C_2 + \eta_2, \ \eps_1 + \mu_1, \ \eps_2 + \mu_2).
\end{equation}
So, to implement this method we have to compute the gradient and the Hessian of $\mc{L}$.
Let us start with the gradient.
Since we are working with the Euclidean space, the gradient of $\oll{f}$ is computed using the same \verb|euclidean_gradient| function we defined in the Riemannian Gradient Descent Method \ref{rgd}.
However, as already mentioned, we also have to compute the gradient of $\mc{L}$ with respect to $\eps_i$ and that means the final gradient has the following format:
\begin{align}
  \begin{split}
    \grad\mc{L}(C_1, C_2, \eps_1, \eps_2) =
    \begin{bmatrix}  
      \grad\oll{f}_1(C_1) - \unvec(Jc_1(C_1)^{\top}\vec(\eps_1), \ d_1 \x N_1) \\
      \grad\oll{f}_2(C_2) - \unvec(Jc_2(C_2)^{\top}\vec(\eps_2), \ d_2 \x N_2) \\
      -c_1 \\
      -c_2
    \end{bmatrix}
  \end{split}
\end{align}
Here $Jc_i$ is the Jacobian of $c_i$ and unvec is the operation defined in \ref{unvec}.
Now, since $\grad\oll{f}_i$ is computed using the \verb|euclidean_gradient| function, we just need to compute $Jc_i$.
This computation is straightforward, but long, and we opted for writing the result in here and the details in Appendix \ref{matrix_ids}.
So, the final Jacobian is given by
\begin{equation}
  Jc_i(C_1) = \Id_{N_i} \otimes C_i^{\top}S_i +
  \hstack(C_i^{\top}S_i \otimes e_1, \ldots, C_i^{\top}S_i \otimes e_{N_i}),
\end{equation}
being $e_j \in \br^{N_i}$, $\otimes$ the Kronecker product and \verb|hstack| the function defined in \ref{hstack}.
Now let us move to the Hessian.
If we define the auxiliary function
\begin{align}
  \begin{split}
    \wtil{c_i}:\mat{d_i}{N_i} & \to \br \\
    C_i & \mapsto \tr(\eps_i^{\top}(C_i^{\top}S_iC_i - \Id_{N_i})),
  \end{split}
\end{align}
the Hessian of $\mc{L}$ has the following format:
\begin{equation}
  \hess\mc{L}(C_1, C_2, \eps_1, \eps_2) =
  \begin{bsmallmatrix}
    \hess_{11}\oll{f} - \hess\wtil{c}_1 & \hess_{12}\oll{f}
    & -Jc_1^{\top} & 0_{N_1d_1 \x N_2N_2} \\
    \hess_{21}\oll{f} & \hess_{22}\oll{f} - \hess\wtil{c}_2
    & 0_{N_2d_2 \x N_1N_1} & -Jc_2^{\top} \\
    -Jc_1 & 0_{N_1N_1 \x N_2d_2} & 0_{N_1N_1 \x N_1N_1} & 0_{N_1N_1 \x N_2N_2} \\
    0_{N_2N_2 \x N_1d_1} & -Jc_2 & 0_{N_2N_2 \x N_1N_1} & 0_{N_2N_2 \x N_2N_2}
  \end{bsmallmatrix}.
\end{equation}
So, the only object left to compute is $\hess\wtil{c}_i$.
Again, we refer the reader to Appendix \ref{matrix_ids} for the details of this computation.
The final result is
\begin{equation}
  \hess\wtil{c}_i = (\eps_i + \eps_i^{\top}) \otimes S_i
\end{equation}
and we now have everything we need to build the Hessian of the Lagrangian if we use the function \verb|euclidean_hessian| defined in \ref{RNR} to compute $\hess_{ij}\oll{f}$.
That is it, the pseudocode for the Newton--Raphson Method is:
\begin{algbox}{Newton--Raphson with Lagrange Multipliers (NRLM)}{NRLM}
  input: $max\_iter, \ C_i^{(0)}, \ \eps_i^{(0)}, \ S_i, \ i=1,2$
  set: $k, \ tol\_grad, \ tol\_val, \ \mc{L}_{prev}, \ \mc{L}^{(0)}, \ \mc{L}_{diff}, \ \norm{\grad\mc{L}^{(0)}} \ = \ 0, \ 10^{-8}, \ 10^{-10}, \ \infty, \ 0, \ 1, \ 1$
  precompute: $\Id_{N_1}, \ \Id_{N_2}$
  
  def jacobian_restriction($C_1, \ C_2$):
      $\aux_1, \ \aux_2 \ = \ C_1^{\top}S_1, \ C_2^{\top}S_2$
      for $j \ = \ 1, \ldots, N_1$:
          $B_{1,j} \ = \ \aux_1 \otimes e_j$
      for $j \ = \ 1, \ldots, N_2$:
          $B_{2,j} \ = \ \aux_2 \otimes e_j$
      return $\Id_{N_1} \otimes \aux_1 + \hstack(B_{1,1}, \ldots, B_{1,N_1}), \ \Id_{N_2} \otimes \aux_2 + \hstack(B_{2,1}, \ldots, B_{2,N_2})$

  def gradient_lagrangian($C_1, \ C_2, \ \eps_1, \ \eps_2, \ Jc_1, \ Jc_2, \ \euc{G}_1, \ \euc{G}_2$):
      $\grad\wtil{c}_1, \ \grad\wtil{c}_2 \ = \ \unvec(Jc_1^{\top}\vec(\eps_1), \ d_1 \x N_1), \unvec(Jc_2^{\top}\vec(\eps_2), \ d_2 \x N_2)$
      return $\euc{G}_1 - \grad\wtil{c}_1, \ \euc{G}_2 - \grad\wtil{c}_2, \ C_1^{\top}S_1C_1 - \Id_{N_1}, \ C_2^{\top}S_2C_2 - \Id_{N_2}$
  
  def hessian_lagrangian($\hess{f}, \ Jc_1, \ Jc_2, \ \eps_1, \ \eps_2$):
      $\hess\wtil{c} \ = \
  \begin{bmatrix}
    (\eps_1 + \eps_1^{\top}) \otimes S_1 & 0_{N_1d_1 \x N_2d_2} \\
    0_{N_2d_2 \x N_1d_1} & (\eps_2 + \eps_2^{\top}) \otimes S_2
  \end{bmatrix}$
      $Jc \ = \
  \begin{bmatrix}
    Jc_1 & 0_{N_1N_1 \x N_2d_2} \\
    0_{N_2N_2 \x N_1d_1} & Jc_2
  \end{bmatrix}$
      return $
    \begin{bmatrix}
      \hess{f} - \hess\wtil{c} & -Jc^{\top} \\
      -Jc & 0_{N_1N_1 + N_2N_2}
    \end{bmatrix}$

  while $k \leq max\_iter$ and $tol\_grad < \norm{\grad\mc{L}^{(k)}}$ and $tol\_val < \mc{L}_{diff}$:
      $\mc{L}^{(k)} \ = \ \mc{L}(C_1^{(k)}, \ C_2^{(k)}, \ \eps_1^{(k)}, \ \eps_2^{(k)})$
      $\euc{G}_1^{(k)}, \ \euc{G}_2^{(k)}$ = euclidean_gradient($C_1^{(k)}, \ C_2^{(k)}$)
      $Jc_1^{(k)}, \ Jc_2^{(k)}$ = jacobian_restriction($C_1^{(k)}, \ C_2^{(k)}$)
      $\grad\mc{L}^{(k)}$ = gradient_lagrangian($C_1^{(k)}, \ C_2^{(k)}, \ \eps_1^{(k)}, \ \eps_2^{(k)}, \ Jc_1^{(k)}, \ Jc_2^{(k)}, \ \euc{G}_1^{(k)}, \ \euc{G}_2^{(k)}$)
      $\euc\hess{f}^{(k)}$ = euclidean_hessian($C_1^{(k)}, \ C_2^{(k)}$)
      $\hess\mc{L}^{(k)}$ = hessian_lagrangian($\euc\hess{f}^{(k)}, \ Jc_1^{(k)}, \ Jc_2^{(k)}, \ \eps_1^{(k)}, \ \eps_2^{(k)}$)
      solve: $\hess\mc{L}^{(k)} \cdot
  \begin{bsmallmatrix}
    \mu_1^{(k)} \\ \mu_2^{(k)} \\ \upsilon_1^{(k)} \\ \upsilon_2^{(k)}
  \end{bsmallmatrix}
  \ = \
  \vec(\grad\mc{L}^{(k)})
  $
      $\mu_1^{(k)}, \ \mu_2^{(k)}, \ \nu_1^{(k)}, \ \nu_2^{(k)} \ = \ \unvec(\mu_1^{(k)}, \ d_1 \x N_1), \ \unvec(\mu_2^{(k)}, \ d_2 \x N_2), \ \unvec(\nu_1^{(k)}, \ N_1 \x N_1), \ \unvec(\nu_2^{(k)}, \ N_2 \x N_2)$
      $C_1^{(k)}, \ C_2^{(k)}, \ \eps_1^{(k)}, \ \eps_2^{(k)} \ = \ C_1^{(k)} + \mu_1^{(k)}, \ C_2^{(k)} + \mu_2^{(k)}, \ \eps_1^{(k)} + \upsilon_1^{(k)}, \ \eps_2^{(k)} + \upsilon_2^{(k)}$
      $\mc{L}_{diff} \ = \ \mc{L}_{prev} - \mc{L}^{(k)}$; $\mc{L}_{prev} \ = \ \mc{L}^{(k)}$; $k \ = \ k + 1$
\end{algbox}
The reader can check the result we obtained using this algorithm in Section \ref{results} and the implementation of the \verb|euclidean_gradient| and of the \verb|euclidean_hessian| for the Hartree--Fock are in \ref{EG} and \ref{EH}.


\chapter{Matrix Identities} \label{matrix_ids}

\epigraph{Matrices act. They don't just sit there.}{Gilbert Strang}

The goal of this appendix is to state some matrix identities we use in this work and to compute the matrix representation of some objects we need to implement the algorithms presented.
The main reference for this appendix is \cite[Chapter 2]{magnus2007}.

\begin{defi} \label{canonical_matrix}
  We will denote the matrix with $1$ in the entry $(i, j)$ and $0$ in the rest by $E_{ij}$.
\end{defi}

\begin{obs}
  If $E_{ij} \in \mat{M}{N}$, then this matrix can also be written as $e_ie_j^{\top}$, being $e_i \in \br^M$ and $e_j \in \br^N$.
\end{obs}

\begin{defi} \label{hstack}
  Given matrices $A_i \in \mat{M}{N_i}$, $i=1, \ldots, K$, the function \emph{hstack}\index{hstack} is defined by
  \begin{equation}
    \hstack(A_1, \ldots, A_K) \ceq \begin{bmatrix} A_1 & \ldots & A_K \end{bmatrix} \in \mat{M}{(N_1 + \ldots + N_K)}.
  \end{equation}
  In other words, we are stacking the matrices $A_i$ horizontally.
\end{defi}

\begin{defi} \label{vstack}
  Given matrices $A_i \in \mat{M_i}{N}$, $i=1, \ldots, K$, the function \emph{vstack}\index{vstack} is defined by
  \begin{equation}
    \vstack(A_1, \ldots, A_K) \ceq \begin{bmatrix} A_1 \\ \vdots \\ A_K \end{bmatrix} \in \mat{(M_1 + \ldots + M_K)}{N}.
  \end{equation}
  In other words, we are stacking the matrices $A_i$ vertically.
\end{defi}

\begin{defi} \label{kronecker_product}
  Given two matrices $A \in \mat{K}{L}$ and $B \in \mat{P}{Q}$, we define the \emph{Kronecker product}\index{Kronecker product} of these matrices as the following $KP \x LQ$ matrix:
  \begin{equation}
    A \otimes B \ceq
    \begin{bmatrix}
      A_{11}B & \ldots & A_{1L}B \\
      \vdots  & \ddots & \vdots  \\
      A_{K1}B & \ldots & A_{KL}B \\
    \end{bmatrix}.
  \end{equation}
\end{defi}

\begin{defi} \label{vectorization}
  The function
  \begin{align}
    \begin{split}
      \vec:\mat{M}{N} & \to \br^{M \cdot N} \\
      \begin{bmatrix}
        a_{11} & \ldots & a_{1N} \\
        \vdots & \ddots & \vdots \\
        a_{M1} & \ldots & a_{MN} \\
      \end{bmatrix}
      & \mapsto
        \begin{bmatrix}
          a_{11} \\ \vdots \\ a_{M1} \\ \vdots \\ a_{1N} \\ \vdots \\ a_{MN}
        \end{bmatrix}
    \end{split}
  \end{align}
  is called \emph{(columnwise) vectorization}\index{vectorization}.
\end{defi}

\begin{obs}
  Keep in mind that we adopted the convention of vectorizing matrices columnwise.
\end{obs}

\begin{defi} \label{unvec}
  We will represent the reverse operation of vectorization by
  \begin{align}
    \begin{split}
      \unvec(\cdot, M \x N): \br^{M \cdot N} & \to \mat{M}{N} \\
      \begin{bmatrix}
        a_{11} \\ \vdots \\ a_{M1} \\ \vdots \\ a_{1N} \\ \vdots \\ a_{MN}
      \end{bmatrix}
      & \mapsto
        \begin{bmatrix}
          a_{11} & \ldots & a_{1N} \\
          \vdots & \ddots & \vdots \\
          a_{M1} & \ldots & a_{MN} \\
        \end{bmatrix}.
    \end{split}
  \end{align}
  So, applying this function to a matrix $A$ will be denoted by $\unvec(A, M \x N)$.
\end{defi}

\begin{obs}
  There are many different ways to unvectorize a vector and this is why we are annotating the final dimension in the function.
  A $6$-dimensional vector can be unvectorized to a $2 \x 3$ or to a $3 \x 2$ matrix and this makes a huge difference.
\end{obs}

Now let us state without proof some matrix identities we will use throughout the whole work, but specially in the rest of this appendix.

\begin{prop}
  Given matrices $A$, $B$ and $C$, the following identities are true (assume the dimension of these matrices are correct for each item, \ie, if we want to sum $A$ and $B$, they have the same dimension, while if we want to multiply $A$ and $B$, they just need to agree on the number of columns of $A$ with rows of $B$):
  \begin{enumerate}

  \item $\tr(A) = \tr(A^{\top})$.

  \item $\tr(A + B) = \tr(A) + \tr(B)$.

  \item $\tr(AB) = \vec(A^{\top})^{\top}\vec(B)$.

  \item $\vec(ABC) = (C^{\top} \otimes A)\vec(B)$.

  \item $\vec(A + B) = \vec(A) + \vec(B)$.

  \item $(A \otimes B)^{\top} = A^{\top} \otimes B^{\top}$.

  \item $A \otimes (B + C) = A \otimes B + A \otimes C$.

  \end{enumerate}
\end{prop}

Before we move forward to compute some of the objects we need in this work, let us see some concrete examples of three operations that will be quite useful to us.
The first one is the fact that multiplying a matrix $A$ by $e_i$ on the right results in the $i$-th column of this matrix:
\begin{equation}
  \begin{bmatrix}
    a_{11} & a_{12} & a_{13} & a_{14} \\
    a_{21} & a_{22} & a_{23} & a_{24} \\
  \end{bmatrix}
  \begin{bmatrix}
    0 \\ 0 \\ 1 \\ 0
  \end{bmatrix}
  =
  \begin{bmatrix}
    a_{13} \\ a_{23}
  \end{bmatrix}.
\end{equation}
So, observe that if we compute this multiplication varying $i$ (from $1$ to $4$ in this case) and then concatenate the resulting vectors horizontally, we obtain the matrix $A$ again.
The other operation is a variation of this that is trickier to analyze.
We are going to multiply a matrix $A$ by the vectorization of $E_{ij}$ and then run over $i$ and $j$ to build a new matrix.
Let us consider $E_{ij} \in \mat{2}{2}$:
\begin{equation}
  \begin{bmatrix}
    a_{11} & a_{12} & a_{13} & a_{14} \\
    a_{21} & a_{22} & a_{23} & a_{24} \\
  \end{bmatrix}
  \vec(E_{11})
  =
  \begin{bmatrix}
    a_{11} & a_{12} & a_{13} & a_{14} \\
    a_{21} & a_{22} & a_{23} & a_{24} \\
  \end{bmatrix}
  \begin{bmatrix}
    1 \\ 0 \\ 0 \\ 0
  \end{bmatrix}
  =
  \begin{bmatrix}
    a_{11} \\ a_{21}
  \end{bmatrix},
\end{equation}
\begin{equation}
  \begin{bmatrix}
    a_{11} & a_{12} & a_{13} & a_{14} \\
    a_{21} & a_{22} & a_{23} & a_{24} \\
  \end{bmatrix}
  \vec(E_{21})
  =
  \begin{bmatrix}
    a_{11} & a_{12} & a_{13} & a_{14} \\
    a_{21} & a_{22} & a_{23} & a_{24} \\
  \end{bmatrix}
  \begin{bmatrix}
    0 \\ 1 \\ 0 \\ 0
  \end{bmatrix}
  =
  \begin{bmatrix}
    a_{12} \\ a_{22}
  \end{bmatrix},
\end{equation}
\begin{equation}
  \begin{bmatrix}
    a_{11} & a_{12} & a_{13} & a_{14} \\
    a_{21} & a_{22} & a_{23} & a_{24} \\
  \end{bmatrix}
  \vec(E_{12})
  =
  \begin{bmatrix}
    a_{11} & a_{12} & a_{13} & a_{14} \\
    a_{21} & a_{22} & a_{23} & a_{24} \\
  \end{bmatrix}
  \begin{bmatrix}
    0 \\ 0 \\ 1 \\ 0
  \end{bmatrix}
  =
  \begin{bmatrix}
    a_{13} \\ a_{23}
  \end{bmatrix},
\end{equation}
\begin{equation}
  \begin{bmatrix}
    a_{11} & a_{12} & a_{13} & a_{14} \\
    a_{21} & a_{22} & a_{23} & a_{24} \\
  \end{bmatrix}
  \vec(E_{22})
  =
  \begin{bmatrix}
    a_{11} & a_{12} & a_{13} & a_{14} \\
    a_{21} & a_{22} & a_{23} & a_{24} \\
  \end{bmatrix}
  \begin{bmatrix}
    0 \\ 0 \\ 0 \\ 1
  \end{bmatrix}
  =
  \begin{bmatrix}
    a_{14} \\ a_{24}
  \end{bmatrix}.
\end{equation}
If we run \emph{first} over $i$ and then over $j$, the final matrix (after we concatenate these vectors horizontally) is, again, $A$.
It is worth mentioning that the reason why we are running first over $i$ is because we are vectorizing matrices columnwise, \ie, we first run over the rows and then over the columns.
To conclude, we will now consider the operation $\vec(E_{ij})^{\top}A$, but let us assume that $A \in \mat{4}{3}$ this time:
\begin{equation}
  \vec(E_{11})^{\top}
  \begin{bmatrix}
    a_{11} & a_{12} & a_{13} \\
    a_{21} & a_{22} & a_{23} \\
    a_{31} & a_{32} & a_{33} \\
    a_{41} & a_{42} & a_{43} \\
  \end{bmatrix}
  =
  \begin{bmatrix}
    1 & 0 & 0 & 0
  \end{bmatrix}
  \begin{bmatrix}
    a_{11} & a_{12} & a_{13} \\
    a_{21} & a_{22} & a_{23} \\
    a_{31} & a_{32} & a_{33} \\
    a_{41} & a_{42} & a_{43} \\
  \end{bmatrix}
  =
  \begin{bmatrix}
    a_{11} & a_{12} & a_{13}
  \end{bmatrix},
\end{equation}
\begin{equation}
  \vec(E_{21})^{\top}
  \begin{bmatrix}
    a_{11} & a_{12} & a_{13} \\
    a_{21} & a_{22} & a_{23} \\
    a_{31} & a_{32} & a_{33} \\
    a_{41} & a_{42} & a_{43} \\
  \end{bmatrix}
  =
  \begin{bmatrix}
    0 & 1 & 0 & 0
  \end{bmatrix}
  \begin{bmatrix}
    a_{11} & a_{12} & a_{13} \\
    a_{21} & a_{22} & a_{23} \\
    a_{31} & a_{32} & a_{33} \\
    a_{41} & a_{42} & a_{43} \\
  \end{bmatrix}
  =
  \begin{bmatrix}
    a_{21} & a_{22} & a_{23}
  \end{bmatrix},
\end{equation}
\begin{equation}
  \vec(E_{12})^{\top}
  \begin{bmatrix}
    a_{11} & a_{12} & a_{13} \\
    a_{21} & a_{22} & a_{23} \\
    a_{31} & a_{32} & a_{33} \\
    a_{41} & a_{42} & a_{43} \\
  \end{bmatrix}
  =
  \begin{bmatrix}
    0 & 0 & 1 & 0
  \end{bmatrix}
  \begin{bmatrix}
    a_{11} & a_{12} & a_{13} \\
    a_{21} & a_{22} & a_{23} \\
    a_{31} & a_{32} & a_{33} \\
    a_{41} & a_{42} & a_{43} \\
  \end{bmatrix}
  =
  \begin{bmatrix}
    a_{31} & a_{32} & a_{33}
  \end{bmatrix},
\end{equation}
\begin{equation}
  \vec(E_{22})^{\top}
  \begin{bmatrix}
    a_{11} & a_{12} & a_{13} \\
    a_{21} & a_{22} & a_{23} \\
    a_{31} & a_{32} & a_{33} \\
    a_{41} & a_{42} & a_{43} \\
  \end{bmatrix}
  =
  \begin{bmatrix}
    0 & 0 & 0 & 1
  \end{bmatrix}
  \begin{bmatrix}
    a_{11} & a_{12} & a_{13} \\
    a_{21} & a_{22} & a_{23} \\
    a_{31} & a_{32} & a_{33} \\
    a_{41} & a_{42} & a_{43} \\
  \end{bmatrix}
  =
  \begin{bmatrix}
    a_{41} & a_{42} & a_{43}
  \end{bmatrix}.
\end{equation}
Again, if we vary \emph{first} over $i$ and then over $j$, the final matrix is $A$.
OK, now that we have explained these operations, we can build some matrices.

From now on assume that $C \in \mat{d}{N}$, $\eps \in \mat{N}{N}$ and $S \in \mat{d}{d}$ is a symmetric and invertible matrix.
In Chapter \ref{optimization} we have two linear transformations given by
\begin{align}
  \begin{split}
    \mat{d}{N} & \to \mat{d}{N} \\
    \eta & \mapsto A\eta,
  \end{split}
\end{align}
being $A \in \mat{M}{d}$.
However, we want a matrix representation for this transformation when $\eta$ is vectorized, \ie, when $\eta \in \br^{dN}$.
To obtain this, recall from Linear Algebra (see \ref{matrix_repr}) that, to compute this representation, we should apply this linear transformation in the basis vectors, then vectorize the result and use this vector as a column of the final matrix.
So, since a basis for $\mat{d}{N}$ is given by the matrices $E_{ij}$ and since we have to concatenate the vectors horizontally running first over $i$ and then over $j$ to respect the columnwise vectorization, we will use the second operation described above.
Let us build this matrix representation then.
First, we compute the linear transformation in $E_{ij}$:
\begin{equation}
  AE_{ij}.
\end{equation}
Then, we vectorize using the identity $\vec(ABC) = (C^{\top} \otimes A)\vec(B)$:
\begin{align}
  \begin{split}
    \vec\qty(AE_{ij})
    & = \vec\qty(AE_{ij}\Id_N) \\
    & = \qty(\Id_N \otimes A)\vec(E_{ij}).
  \end{split}
\end{align}
To conclude, running over $i$ and then $j$ and concatenating these vectors horizontally gives the final matrix representation of the linear transformation:
\begin{equation}
  \Id_N \otimes A.
\end{equation}
Observe that this matrix lives in $\mat{dN}{dN}$ and, therefore, it makes sense to multiply by $\eta \in \br^{dN}$ now.

The next computation is very similar.
Fixing a matrix $A \in \mat{d}{N}$, we have the linear transformation
\begin{align}
  \begin{split}
    \mat{d}{N} & \to \mat{d}{N} \\
    \eta & \mapsto \eta C^{\top}A
  \end{split}
\end{align}
and we want to obtain the matrix representation of this transformation when $\eta \in \br^{dN}$.
Computing this transformation when $\eta = E_{ij}$ and already vectorizing the result we have
\begin{align}
  \begin{split}
    \vec\qty(E_{ij}C^{\top}A)
    & = \vec(\Id_dE_{ij}C^{\top}A) \\
    & = \qty(A^{\top}C \otimes \Id_d)\vec(E_{ij}).
  \end{split}
\end{align}
Running over $i$ and $j$, the final matrix representation is
\begin{equation}
  A^{\top}C \otimes \Id_d.
\end{equation}

Now, all the computations below are used in Appendix \ref{lagrange_multipliers}. First, we need the Jacobian of the following function:
\begin{align}
  c:\mat{d}{N} & \to \br \\
  C & \mapsto C^{\top}SC - \Id_N.
\end{align}
The strategy to compute this is the following: we compute the partial derivatives of this function, vectorize the result and then concatenate all these vectors horizontally.
It is quite similar to what we did previously, but we need to compute some partial derivatives first.
First observe that we can ignore the identity matrix because it is a constant.
Now, given $V \in \mat{d}{N}$, the directional derivative $T_Cc(V)$ is, by definition:
{\small
  \begin{align}
    \begin{split}
      \lim_{t \to 0} \frac{(C + tV)^{\top}S(C + tV) - C^{\top}SC}{t}
      & = \lim_{t \to 0} \frac{(C^{\top} + tV^{\top})(SC + tSV) - C^{\top}SC}{t} \\
      & = \lim_{t \to 0} \frac{C^{\top}SC + tC^{\top}SV + tV^{\top}SC + t^2V^{\top}SV - C^{\top}SC}{t} \\
      & = \lim_{t \to 0} C^{\top}SV + V^{\top}SC + tV^{\top}SV \\
      & = C^{\top}SV + V^{\top}SC
    \end{split}
\end{align}
}%
Substituting $V$ by $E_{ij} \in \mat{d}{N}$ because we want partial derivatives, we have:
\begin{equation}
  \pdv{c}{X_{ij}}\qty(C) = C^{\top}SE_{ij} + E_{ij}^{\top}SC.
\end{equation}
Now we vectorize this expression.
Using the identities $E_{ij} = e_ie_j^{\top}$, $\vec(ABC) = (C^{\top} \otimes A)\vec(B)$, $\vec(A+B) = \vec(A) + \vec(B)$ and $\vec(e_i) = \vec(e_i^{\top}) = e_i$, we obtain:
\begin{align}
  \begin{split}
    \vec\qty(C^{\top}SE_{ij} + e_je_i^{\top}SC)
    & = \vec\qty(C^{\top}SE_{ij}\Id_N) + \vec\qty(e_je_i^{\top}SC) \\
    & = \qty(\Id_N \otimes C^{\top}S)\vec(E_{ij}) + \qty(C^{\top}S \otimes e_j)e_i.
  \end{split}
\end{align}
To conclude, observe that the formula above gives us a vector in $\br^{N \cdot N}$ and we have to run over $i$ and $j$ to build the matrix that represents the Jacobian.
For the first term we already saw that, after varying $i$ and $j$, we obtain $\Id_N \otimes C^{\top}S$.
For the second we could not find a closed formula, but we can eliminate the multiplication by $e_i$ and compute just $C^{\top}S \otimes e_j$ for every $j$ because multiplying a matrix $A$ by $e_i$ on the right gives the $i$-th column of $A$.
So, if we concatenate the matrices $C^{\top}S \otimes e_j$ horizontally, we have the final result:
\begin{equation}
  Jc(C) = \Id_N \otimes C^{\top}S
  + \hstack(C^{\top}S \otimes e_1, \ldots, C^{\top}S \otimes e_N).
\end{equation}

To conclude, let us compute the Hessian of
\begin{align}
  \begin{split}
    \wtil{c}:\mat{d}{N} & \to \br \\
    C & \mapsto \tr(\eps^{\top}\qty(C^{\top}SC - \Id_N)).
  \end{split}
\end{align}
The idea is to compute the directional derivative with respect to $V$, then compute the directional derivative of the directional derivative with respect to $W$, and, in the end, substitute $V$ by $E_{ij}$ and $W$ by $E_{kl}$ to obtain the second-order partial derivative we actually use in the Hessian.
So, starting with $V$ (and ignoring the identity because it is a constant), we have:
\begin{align}
  \begin{split}
    \lim_{t \to 0} \frac{\tr(\eps^{\top}(C^{\top} + tV^{\top})S(C + tV))
    - \tr(\eps^{\top}C^{\top}SC)}{t}
    & = \lim_{t \to 0} \tr(C^{\top}SV + V^{\top}SC + tV^{\top}SV) \\
    & = \tr(C^{\top}SV + V^{\top}SC).
  \end{split}
\end{align}
Moving to $W$:
\begin{multline}
  \lim_{t \to 0} \frac{\tr(\eps^{\top}(C+tW)^{\top}SV)
    + \tr(\eps^{\top}V^{\top}S(C+tW))
    - \tr(\eps^{\top}C^{\top}SV) - \tr(\eps^{\top}V^{\top}SC)}{t} \\
  = \lim_{t \to 0} \frac{t\tr(\eps^{\top}W^{\top}SV)
    + t\tr(\eps^{\top}V^{\top}SW)}{t} = \tr(\eps^{\top}W^{\top}SV)
  + \tr(\eps^{\top}V^{\top}SW).
\end{multline}
Finally, replacing $V$ by $E_{ij}$, $W$ by $E_{kl}$ and using several of the identities we stated in the beginning of the appendix, we have:
\begin{align}
  \begin{split}
    \pdv{\wtil{c}}{X_{kl}}{X_{ij}}\qty(C)
    & = \tr(\eps^{\top}E_{kl}^{\top}SE_{ij})
      + \tr(\eps^{\top}E_{ij}^{\top}SE_{kl}) \\
    & = \tr(\eps^{\top}E_{kl}^{\top}SE_{ij}) + \tr(E_{kl}^{\top}SE_{ij}\eps) \\
    & = \vec(SE_{kl}\eps)^{\top}\vec(E_{ij})
      + \vec(E_{kl})^{\top}\vec(SE_{ij}\eps) \\
    & = \qty((\eps^{\top} \otimes S)\vec(E_{kl}))^{\top}\vec(E_{ij})
      + \vec(E_{kl})^{\top}\qty((\eps^{\top} \otimes S)\vec(E_{ij})) \\
    & = \vec(E_{kl})^{\top}(\eps^{\top} \otimes S)^{\top}\vec(E_{ij})
      + \vec(E_{kl})^{\top}(\eps^{\top} \otimes S)\vec(E_{ij}) \\
    & = \vec(E_{kl})^{\top}\qty((\eps + \eps^{\top}) \otimes S)\vec(E_{ij}).
  \end{split}
\end{align}
Now, running over $i, j, k, l$, in that order, we obtain the final Hessian:
\begin{equation}
  (\eps + \eps^{\top}) \otimes S.
\end{equation}


\addcontentsline{toc}{chapter}{Bibliography}


\bibliographystyle{plain}
\bibliography{references}

\begin{thebibliography}{10}

\bibitem{absil2008}
P.-A. Absil, R.~Mahony, and R.~Sepulchre.
\newblock {\em Optimization Algorithms on Matrix Manifolds}.
\newblock Princeton University Press, 2008.

\bibitem{grossi2012}
S.~Anan’in and C.~H. Grossi.
\newblock Differential geometry of grassmannians and the plücker map.
\newblock {\em Central European Journal of Mathematics}, 10(3):873–884, Feb
  2012.

\bibitem{anderson1933}
C.~D. Anderson.
\newblock The positive electron.
\newblock {\em Phys. Rev.}, 43:491--494, Mar 1933.

\bibitem{aoto2022}
Y.~A. Aoto.
\newblock Geometric interpretation for coupled-cluster theory. a comparison of
  accuracy with the corresponding configuration interaction model.
\newblock {\em The Journal of Chemical Physics}, 157(8):084109, 2022.

\bibitem{aotograssmann}
Y.~A. Aoto, C.~O. da~Silva, and M.~M.~F. de~Moraes.
\newblock Grassmann.
\newblock Available at \url{https://github.com/YuriAoto/grassmann}, 2023.

\bibitem{aoto2020}
Y.~A. Aoto and M.~F. da~Silva.
\newblock Calculating the distance from an electronic wave function to the
  manifold of slater determinants through the geometry of grassmannians.
\newblock {\em Phys. Rev. A}, 102:052803, Nov 2020.

\bibitem{boumal2022}
N.~Boumal.
\newblock An introduction to optimization on smooth manifolds.
\newblock To appear with Cambridge University Press, Jun 2022.

\bibitem{boumal2014}
N.~Boumal, B.~Mishra, P.-A. Absil, and R.~Sepulchre.
\newblock {M}anopt, a {M}atlab toolbox for optimization on manifolds.
\newblock {\em Journal of Machine Learning Research}, 15(42):1455--1459, 2014.

\bibitem{tannoudji2020_1}
C.~Cohen-Tannoudji, B.~Diu, and F.~Laloë.
\newblock {\em Quantum Mechanics Volume I}.
\newblock Wiley-VCH, 2020.

\bibitem{tannoudji2020_2}
C.~Cohen-Tannoudji, B.~Diu, and F.~Laloë.
\newblock {\em Quantum Mechanics Volume II}.
\newblock Wiley-VCH, 2020.

\bibitem{kconradnorm}
K.~Conrad.
\newblock Equivalence of norms.
\newblock Available at
  \url{https://kconrad.math.uconn.edu/blurbs/gradnumthy/equivnorms.pdf}.

\bibitem{criscitiello2022}
C.~Criscitiello and N.~Boumal.
\newblock An accelerated first-order method for non-convex optimization on
  manifolds.
\newblock {\em Journal of Foundations of Computational Mathematics}, 2022.

\bibitem{curtiss1997}
L.~A. Curtiss, K.~Raghavachari, P.~C. Redfern, and J.~A. Pople.
\newblock Assessment of gaussian-2 and density functional theories for the
  computation of enthalpies of formation.
\newblock {\em The Journal of Chemical Physics}, 106(3):1063--1079, 1997.

\bibitem{kalman2009}
Kalman D.
\newblock Leveling with lagrange: An alternate view of constrained
  optimization.
\newblock {\em Mathematics Magazine}, 82(3):186--196, 2009.

\bibitem{madrid2001}
R.~de~la Madrid.
\newblock {\em Quantum Mechanics in Rigged Hilbert Space Language}.
\newblock PhD thesis, Universidad de Valladolid, 2001.

\bibitem{dirac1958}
P.~Dirac.
\newblock {\em The Principles of Quantum Mechanics}.
\newblock Oxford University Press, 1958.

\bibitem{edelman1998}
A.~Edelman, T.~A. Arias, and S.~T. Smith.
\newblock The geometry of algorithms with orthogonality constraints.
\newblock {\em SIAM Journal on Matrix Analysis and Applications},
  20(2):303--353, 1998.

\bibitem{einstein1938}
A.~Einstein and L.~Infeld.
\newblock {\em The Evolution of Physics: The Growth of Ideas from Early
  Concepts to Relativity and Quanta}.
\newblock Cambridge University Press, 1938.

\bibitem{einstein1989}
A.~Einstein and J.~Stachel~(ed.).
\newblock {\em The Collected Papers of Albert Einstein. Vol. 2: The Swiss
  Years: Writings, 1900-1909 (English translation)}.
\newblock Princeton University Press, 1989.

\bibitem{feldman2022}
R.~Feldmann, A.~Baiardi, and M.~Reiher.
\newblock Quadratically convergent self-consistent field algorithms: from
  classical to quantum nuclei.
\newblock {\em arXiv}, 2022.

\bibitem{becigneul2018}
Becigneul G. and Ganea O.-E.
\newblock Riemannian adaptive optimization methods.
\newblock In {\em International Conference on Learning Representations}, 2019.

\bibitem{garrett2010}
P.~Garrett.
\newblock Non-existence of tensor products of hilbert spaces.
\newblock Available at
  \url{https://www-users.cse.umn.edu/~garrett/m/v/nonexistence_tensors.pdf}.

\bibitem{gearhart2002}
C.~A. Gearhart.
\newblock Planck, the quantum, and the historians.
\newblock {\em Physics in Perspective}, 4:170--215, 2002.

\bibitem{golub2013}
G.~H. Golub and C.~F. Van~Loan.
\newblock {\em Matrix Computations}.
\newblock Johns Hopkins University Press, 2013.

\bibitem{hall2013}
B.~Hall.
\newblock {\em Quantum Theory for Mathematicians}.
\newblock Springer New York, 2013.

\bibitem{hehre1972}
W.~J. Hehre, R.~Ditchfield, and J.~A. Pople.
\newblock Self-consistent molecular orbital methods. xii. further extensions of
  gaussian-type basis sets for use in molecular orbital studies of organic
  molecules.
\newblock {\em The Journal of Chemical Physics}, 56(5):2257--2261, 1972.

\bibitem{helgaker2000}
T.~Helgaker, P.~J{\o}rgensen, and J.~Olsen.
\newblock {\em Molecular Electronic-Structure Theory}.
\newblock John Wiley and Sons, 2000.

\bibitem{hermann1971}
A.~Hermann.
\newblock {\em The Genesis of Quantum Theory (1899-1913)}.
\newblock MIT Press, 1971.

\bibitem{hosseini2019}
S.~(ed.) Hosseini, B.~S.~(ed.) Mordukhovich, and A.~(ed.) Uschmajew.
\newblock {\em Nonsmooth Optimization and Its Applications}.
\newblock Springer International Publishing, 2019.

\bibitem{ding2020}
Ding L. and Schilling C.
\newblock Correlation paradox of the dissociation limit: A quantum information
  perspective.
\newblock {\em Journal of Chemical Theory and Computation}, 16(7):4159--4175,
  may 2020.

\bibitem{lee2012}
J.~M. Lee.
\newblock {\em Introduction to Smooth Manifolds}.
\newblock Springer-Verlag New York, Inc., 2012.

\bibitem{lee2018}
J.~M. Lee.
\newblock {\em Introduction to Riemannian Manifolds}.
\newblock Springer International Publishing AG, 2018.

\bibitem{magnus2007}
J.~R. Magnus and H.~Neudecker.
\newblock {\em Matrix Differential Calculus with Applications in Statistics and
  Econometrics}.
\newblock John Wiley and Sons, 2007.

\bibitem{kostrikin1997}
Y.~Manin and A.~Kostrikin.
\newblock {\em Linear Algebra and Geometry}.
\newblock Gordon and Breach Science Publishers, 1997.

\bibitem{mcquarrie2008}
D.~A. McQuarrie.
\newblock {\em Quantum Chemistry}.
\newblock University Science Books, 2008.

\bibitem{mehra1982}
J.~Mehra and H.~Rechenberg.
\newblock {\em The Historical Development of Quantum Theory}.
\newblock Springer-Verlag New York, Inc., 1982.

\bibitem{millikan1914}
R.~A. Millikan.
\newblock A direct determination of $h$.
\newblock {\em Phys. Rev.}, 4:73--75, Jul 1914.

\bibitem{munkres2014}
J.~Munkres.
\newblock {\em Topology}.
\newblock Pearson Educated Limited, 2014.

\bibitem{nocedal2006}
J.~Nocedal and S.~J Wright.
\newblock {\em Numerical Optimization}.
\newblock Springer-Verlag New York, 2006.

\bibitem{piccione2009}
P.~Piccione and D.~V. Tausk.
\newblock A student’s guide to symplectic spaces, grassmannians and maslov
  index.
\newblock Available at
  \url{https://www.ime.usp.br/~piccione/Downloads/MaslovBook.pdf}.

\bibitem{rudin1976}
W.~Rudin.
\newblock {\em Principles of Mathematical Analysis}.
\newblock McGraw-Hill, Inc., 1976.

\bibitem{rudin1991}
W.~Rudin.
\newblock {\em Functional Analysis}.
\newblock McGraw-Hill, Inc., 1991.

\bibitem{shewchuk1994}
J.~R. Shewchuk.
\newblock An introduction to the conjugate gradient method without the
  agonizing pain.
\newblock Technical report, Carnegie Mellon University, USA, 1994.

\bibitem{shrum1924}
G.~M. Shrum.
\newblock The doublet separation of the balmer lines.
\newblock {\em Proc. R. Soc. Lond. A}, 105:259--270, Mar 1924.

\bibitem{smith2014}
S.~T. Smith.
\newblock Optimization techniques on riemannian manifolds.
\newblock {\em Fields Institute Communications}, 3, 2014.

\bibitem{streater2000}
R.~F. Streater and A.~S. Wightman.
\newblock {\em PCT, Spin and Statistics, and All That}.
\newblock Princeton University Press, 2000.

\bibitem{sutherland2009}
W.~A. Sutherland.
\newblock {\em Introduction to Metric and Topological Spaces}.
\newblock Oxford University Press, 2009.

\bibitem{szabo1996}
A.~Szabo and N.~S. Ostlund.
\newblock {\em Modern Quantum Chemistry}.
\newblock Dover Publications, Inc., 1996.

\bibitem{takhtajan2008}
L.~A. Takhtajan.
\newblock {\em Quantum Mechanics for Mathematicians}.
\newblock American Mathematical Society, 2008.

\bibitem{trendafilov2022}
N.~Trendafilov and M.~Gallo.
\newblock {\em Multivariate Data Analysis on Matrix Manifolds}.
\newblock Springer Cham, 2022.

\bibitem{vanlenthe2006}
J.~H. Van~Lenthe, R.~Zwaans, H.~J.~J. Van~Dam, and M.~F. Guest.
\newblock Starting scf calculations by superposition of atomic densities.
\newblock {\em Journal of Computational Chemistry}, 27(8):926--932, 2006.

\bibitem{headgordon2002}
T.~van Voorhis and M.~Head-Gordon.
\newblock A geometric approach to direct minimization.
\newblock {\em Molecular Physics}, 100(11):1713--1721, 2002.

\bibitem{vargesson2015}
N.~Vargesson.
\newblock Thalidomide-induced teratogenesis: History and mechanisms.
\newblock {\em Birth Defects Research Part C: Embryo Today: Reviews},
  105(2):140--156, 2015.

\bibitem{weber2021}
M.~Weber.
\newblock {\em On Geometric Optimization, Learning and Control}.
\newblock PhD thesis, Princeton University, 2021.

\bibitem{lai2020}
Lai Z., L.-H. Lim, and K.~Ye.
\newblock Simpler grassmannian optimization, 2020.

\end{thebibliography}

\clearpage

\phantomsection

\addcontentsline{toc}{chapter}{Index}

\printindex

\end{document}